\documentclass[11pt,reqno, oneside]{amsart}

\newcommand{\hkra}{\hookrightarrow}

\newcommand{\dimp}{\Leftrightarrow}
\newcommand{\rimp}{\Rightarrow}

\newcommand{\set}[1]{\left\{ #1 \right\}}

\newcommand{\cc}[1]{\overline{#1}}

\newcommand{\union}[2]{\bigcup\limits_{#1}{#2}}
\newcommand{\sub}{\subset}
\newcommand{\super}{\supseteq}

\newcommand{\sm}{\ensuremath{\setminus}}

\newcommand{\lcap}[2]{\bigcap\limits_{#1}^{#2}}

\newcommand{\s}[2]{\sum\limits_{#1}{#2} }
\newcommand{\p}[2]{\prod\limits_{#1}{#2} }
\newcommand{\g}{\circ}
\newcommand{\modulo}[2]{{\raisebox{.2em}{$#1$}\left/\raisebox{-.2em}{$#2$}\right.}}

\newcommand{\units}{^\times}

\newcommand{\inv}{^{-1}}

\newcommand{\norm}[1]{\left\lVert#1\right\rVert}

\newcommand{\cl}{\colon}

\newcommand{\djun}[1]{\bigsqcup\limits_{#1}}

\newcommand{\lbr}[1]{\Bigl(#1\Bigr)}

\newcommand{\emst}{\emptyset}

\newcommand{\iI}{_{i \in I}}

\newcommand{\xra}{\xrightarrow}

\newcommand{\scA}{\mathscr{A}}
\newcommand{\scB}{\mathscr{B}}
\newcommand{\scC}{\mathscr{C}}

\newcommand{\scE}{\mathscr{E}}
\newcommand{\scF}{\mathscr{F}}

\newcommand{\scJ}{\mathscr{J}}
\newcommand{\scK}{\mathscr{K}}
\newcommand{\scM}{\mathscr{M}}
\newcommand{\scN}{\mathscr{N}}

\newcommand{\scS}{\mathscr{S}}
\newcommand{\scT}{\mathscr{T}}

\newcommand{\scL}{\mathscr{L}}

\newcommand{\scW}{\mathscr{W}}
\newcommand{\scX}{\mathscr{X}}
\newcommand{\scY}{\mathscr{Y}}
\newcommand{\scZ}{\mathscr{Z}}

\newcommand{\mini}[2]{\min\limits_{#1}{#2}}

\newcommand{\eva}{\normalfont\text{ev}}
\newcommand{\ii}{\mathbf{i}}

\newcommand{\lspan}[1]{\langle {#1}\rangle}
\newcommand{\pf}{\,\pitchfork\,}

\newcommand{\conn}{\nabla}

\newcommand{\hpd}{\cc{\partial}}

\newcommand{\inn}{^\mathrm{o}}

\newcommand{\lc}{\underline}

\newcommand{\dul}{^\vee}

\newcommand{\vfc}[1]{[#1]^{\virt}}

\newcommand{\ide}{\normalfont\text{id}}

\newcommand{\Pt}{\normalfont\text{pt}}
\newcommand{\im}{\normalfont\text{im}}
\newcommand{\reg}{\normalfont\text{reg}}

\newcommand{\PGL}{\normalfont\text{PGL}}
\newcommand{\PU}{\normalfont\text{PU}}

\newcommand{\coker}{\normalfont\text{coker}}
\newcommand{\pr}{\normalfont\text{pr}}

\newcommand{\Mbar}{\overline{\mathcal M}}
\newcommand{\Cbar}{\overline{\mathcal C}}
\newcommand{\vdim}{\normalfont\text{vdim}}
\newcommand{\delbar}{\bar\partial}
\newcommand{\del}{\partial}
\newcommand{\Hom}{\normalfont\text{Hom}}
\newcommand{\End}{\normalfont\text{End}}
\newcommand{\Aut}{\normalfont\text{Aut}}

\newcommand{\bC}{\mathbb{C}}
\newcommand{\bD}{\mathbb{D}}

\newcommand{\bL}{\mathbb{L}}

\newcommand{\bN}{\mathbb{N}}

\newcommand{\bP}{\mathbb{P}}
\newcommand{\bQ}{\mathbb{Q}}
\newcommand{\bR}{\mathbb{R}}

\newcommand{\bT}{\mathbb{T}}

\newcommand{\bX}{\mathbb{X}}
\newcommand{\bY}{\mathbb{Y}}
\newcommand{\bZ}{\mathbb{Z}}
\newcommand{\cA}{\mathcal{A}}
\newcommand{\cB}{\mathcal{B}}
\newcommand{\cC}{\mathcal{C}}
\newcommand{\cD}{\mathcal{D}}
\newcommand{\cE}{\mathcal{E}}
\newcommand{\cF}{\mathcal{F}}
\newcommand{\cG}{\mathcal{G}}

\newcommand{\cJ}{\mathcal{J}}
\newcommand{\cK}{\mathcal{K}}
\newcommand{\cL}{\mathcal{L}}
\newcommand{\cM}{\mathcal{M}}
\newcommand{\cN}{\mathcal{N}}
\newcommand{\cO}{\mathcal{O}}
\newcommand{\cP}{\mathcal{P}}

\newcommand{\cR}{\mathcal{R}}

\newcommand{\cT}{\mathcal{T}}
\newcommand{\cU}{\mathcal{U}}
\newcommand{\cV}{\mathcal{V}}
\newcommand{\cW}{\mathcal{W}}

\newcommand{\cZ}{\mathcal{Z}}

\newcommand{\fD}{\mathfrak{D}}

\newcommand{\fL}{\mathfrak{L}}

\newcommand{\ff}{\mathfrak{f}}
\newcommand{\fg}{\mathfrak{g}}

\newcommand{\fj}{\mathfrak{j}}

\newcommand{\fl}{\mathfrak{l}}

\newcommand{\fo}{\mathfrak{o}}
\newcommand{\fp}{\mathfrak{p}}

\newcommand{\fs}{\mathfrak{s}}
\newcommand{\ft}{\mathfrak{t}}
\newcommand{\fu}{\mathfrak{u}}

\newcommand{\virt}{\normalfont\text{vir}}

\newcommand{\C}{\bC}
\newcommand{\ol}{\overline}
\newcommand{\wt}{\widetilde}
\newcommand{\wh}{\widehat}

\newcommand{\indo}{\normalfont\text{ind}}

\newcommand{\obj}{\normalfont\text{Ob}}

\newcommand{\multr}[1]{m_{#1}}

\newcommand{\orl}{\hspace{0.5pt}\fo}

\newcommand{\wch}{\widecheck}
\newcommand{\cBR}{\cB^{P}}
\newcommand{\scBR}{\scB^{P}}
\newcommand{\cBRB}{\wt{\cB}^{P}}
\newcommand{\wkpd}{W^{k,p,\delta}}
\DeclareMathOperator{\Glue}{Glue}
\DeclareMathOperator{\proj}{proj}
\DeclareMathOperator{\Maps}{Maps}
\DeclareMathOperator{\spec}{spec}
\DeclareMathOperator{\colim}{colim}

\newcommand{\graph}{\normalfont\text{graph}}

\newcommand{\stb}{\,\normalfont\text{st}}
\newcommand{\sft}{\normalfont\text{SFT}}
\newcommand{\scBS}{\scB^\bR}
\newcommand{\cBS}{\cB^\bR}
\newcommand{\cTR}{\cT^\bR}
\newcommand{\supp}{\normalfont\text{supp}}

\newtheorem{intthm}{Theorem}

\newcommand{\ov}{\overline}
\newcommand{\Nbar}{\overline{\cN}}
\newcommand{\qtimes}[1]{\underset{#1}{\times}}

\newcommand{\Flow}{\normalfont\text{Flow}}
\newcommand{\dOrb}{\normalfont\text{dOrb}}

\usepackage[utf8]{inputenc}
\usepackage[english]{babel}
\usepackage{amsmath}
\usepackage{amssymb}
\usepackage{mathabx}
\usepackage{amsfonts}
\usepackage{amsthm}
\usepackage[shortlabels]{enumitem}
\usepackage[dvipsnames]{xcolor}
\usepackage{tikz-cd}
\usepackage{subfiles}
\usepackage[mathscr]{euscript}
\usepackage{extarrows}
\usepackage[pagebackref=true] {hyperref}
\usepackage{appendix}
\usepackage{bm}
\usepackage{mathtools}
\usepackage{float}
\usepackage{pdfpages}
\usepackage{transparent}
\usepackage{changepage}
\usepackage{wrapfig}
\usepackage{relsize}
\hypersetup{
    colorlinks=true,
    linkcolor=Blue,
    filecolor=red,      
    urlcolor=Blue,
    citecolor=Blue,
}

\newtheorem{theorem}{Theorem}[section]
\newtheorem{lemma}[theorem]{Lemma}
\newtheorem{corollary}[theorem]{Corollary}
\newtheorem{proposition}[theorem]{Proposition}

\theoremstyle{definition}
\newtheorem{definition}[theorem]{Definition}
\newtheorem{construction}[theorem]{Construction}
\theoremstyle{remark}
\newtheorem{remark}[theorem]{Remark}

\newtheorem{convention}[theorem]{Convention}
\newtheorem{ex}[theorem]{Example}

\newtheorem*{notation*}{Notation}

\numberwithin{equation}{section}

\makeatletter
\@addtoreset{equation}{section}
\@addtoreset{theorem}{section}
\makeatother

\makeatletter 
 \def\l@subsection{\@tocline{2}{0pt}{2pc}{6pc}{}} \makeatother

\newcommand{\Addresses}{{
  \bigskip
  \footnotesize

  \textsc{Soham Chanda\\ \indent University of Southern California, Los Angeles, California, USA}\par\nopagebreak
  \text{ORCID}: \texttt{0000-0002-6932-2855}\\
  \textsc{}\par\nopagebreak
\textsc{Amanda Hirschi\\ \indent Université Paris Cité, Sorbonne Université, CNRS, IMJ-PRG, F-75005 Paris, France}\\
\indent \text{ORCID}: \texttt{0000-0002-2392-7875}
}}

\begin{document}

\title{A Contact Homotopy Type}
\author{Soham Chanda}\author{Amanda Hirschi}
\date{\today}
\begin{abstract}
    Adapting the construction of global Kuranishi charts to the contact setting, we associate to any non-degenerate closed contact manifold a flow category based on Reeb orbits and moduli spaces of pseudo-holomorphic buildings. The construction is natural in the sense that to any exact symplectic cobordism we can associate a flow bimodule between the flow categories of its ends. 
\end{abstract}
\maketitle
\tableofcontents

\section{Introduction}

\subsection{Context} 
Contact manifolds are smooth manifolds equipped with a maximally integral distribution, their contact structure. Although such manifolds appear in the literature as early as \cite{Lie72}, the systematic study of contact manifolds is younger than that of symplectic manifolds. While any contact manifold $(Y,\xi)$ is odd-dimensional, given a contact form $\lambda$ we can define the symplectization of $(Y,\lambda)$ to be the product $(\bR\times Y,d(e^s\lambda))$. While the symplectization is not a complete invariant of contact manifolds, \cite{Cou14}, it allows many tools from symplectic geometry to be imported into contact topology, in particular, the theory of pseudo-holomorphic curves.
Early applications were constructions of symplectic capacities, \cite{EH87}, and the Weinstein conjecture in dimension $3$, \cite{Hof93}.\par 
Symplectic field theory is the ambitious vision of \cite{EGH00}, proposing an intricate algebraic framework based on moduli spaces of punctured pseudo-holomorphic curves. We refer to \cite{HS24} for a survey on the history and possible application of SFT; we could not do them justice here.
The main problem in realizing the SFT framework is the lack of transversality of the relevant moduli spaces. The basic theory and estimates were worked out in \cite{HWZI,HWZII}, while compactness was established in \cite{BEH03,CM05}. The goal of a uniform theory to deal with the transversality issues motivated the development of polyfolds, \cite{HWZ21,FFGW16,FH18}. A construction of Kuranishi charts for all genera was given by \cite{Ish18}, but utilizing them to obtain algebraic invariants is yet to be done.\par

In genus zero, considering curves with one positive puncture, SFT postulates the existence of a Floer homology theory, called \emph{contact homology}, with generators given by Reeb orbits. In \cite{BH18} and \cite{BH23}, Bao--Honda give constructions that are suitable for computations, cf. \cite{Avd23}, while \cite{Par19} gives a more abstract construction of contact homology using the framework developed in \cite{Par16}. However, as in Hamiltonian Floer theory, contact homology uses only the information of rigid moduli spaces of curves, i.e., those of dimension zero, and of the existence of suitable moduli spaces in dimension one. Floer homotopy theory has the goal of extracting invariants also from higher-dimensional moduli spaces. It was proposed by \cite{CJS95} in 1995 and reworked in the Morse--Bott context by \cite{Z24,CK23,Bon24}. Both \cite{LT18} and \cite{AB24} give different approaches, respectively, the latter placing greater emphasis on bordism theories. Flow categories and the associated homotopical structures have been constructed in Hamiltonian Floer theory, \cite{BX22,Rez22}, symplectic cohomology, \cite{Rez24,CK23}, Lagrangian Floer theory, \cite{Lar21,PS24,BC25} as well as \cite{LS14} in a somewhat different context. Apart from \cite{TT25}, which uses generating families instead of pseudo-holomorphic curves, no Floer homotopy theoretic constructions have been given in the context of contact topology.

\subsection{Main results}
The construction of our flow category relies fundamentally on the global Kuranishi charts for moduli spaces of punctured pseudo-holomorphic curves that we build in \textsection\ref{sec:prelim-aux}. To keep the notation light here, we summarize the background on contact topology in \textsection\ref{subsec:background} and define our pseudo-holomorphic curves and their degenerations in \textsection\ref{subsec:buildings}. For this introduction, let us simply say that we consider the SFT compactifications $\Mbar^{\,J}_{\sft}(\Gamma^+,\Gamma^-;\beta)$ of moduli spaces of genus-zero $J$-holomorphic curves of relative homology class $\beta$ in the symplectisation of a contact manifold $(Y,\lambda)$ that are asymptotic at their positive and negative punctures to Reeb orbits in $\Gamma^+$ and $\Gamma^-$ respectively. 

\begin{intthm}[Theorem~\ref{thm:leveled-gkc}]\label{prop:gkc-exists}
    Let $(Y,\xi)$ be a closed contact manifold equipped with a non-degenerate contact form $\lambda$. Suppose $J$ is a $\lambda$-adapted almost complex structure on $(Y,\xi)$ and $\Gamma^\pm$ are finite sequences of Reeb orbits.
    \begin{enumerate}[ \normalfont 1),leftmargin=20pt,ref=\arabic*]
        \item The moduli space $\Mbar^{\,J}_{\sft}(\Gamma^+,\Gamma^-;\beta)$ admits a global Kuranishi chart with corners $\cK$ of the correct virtual dimension.
        \item If $\Gamma^+$ and $\Gamma^-$ consist of good Reeb orbits, then the orientation line $\fo_\cK$ of $\cK$ is canonically isomorphic to $\fo(\bR)\dul\otimes\bigotimes\limits_{\gamma\in \Gamma^+}\fo_{\gamma}\otimes\bigotimes\limits_{\gamma\in \Gamma^-}\fo\dul_{\gamma}$.
    \end{enumerate}
\end{intthm}

While much of the construction is similar to previous constructions of global Kuranishi charts, the crucial first step is different: the construction of very ample line bundles on the closed domain curves. The original idea in \cite{AMS21}, somewhat rephrased in \cite{HS22} and \cite{AMS23}, is to construct a very ample line bundle on the domain of pseudo-holomorphic curve by pulling back a sufficiently positive differential form, respectively, line bundle on the target manifold. In the contact case, where $Y$ could have vanishing $H^2$ and such a line bundle could only be pulled back to the punctured domain, this is not an option. However, in genus zero, all that is needed is the assignment of a positive degree to each irreducible component of a domain so that these degrees add correctly when a node is smoothed. This is carried out in \textsection\ref{sssec:framings_of_buildings}.

\begin{remark}[Technical remark]
    The proof of Theorem~\ref{prop:gkc-exists} passes through the moduli spaces considered in \cite{Par19}. In contrast to the compactifications usually considered in the SFT literature, these curves do not come with levels, that is, one quotients the maps on each irreducible component separately by the translation action on the target. In the process, we give a geometric way of recovering the usual moduli spaces considered in SFT from those used by \cite{Par19}, which may be of independent interest.
\end{remark}

Theorem~\ref{prop:gkc-exists} does not assume $\Gamma^+$ to be a single Reeb orbit. Hence, our construction yields the foundations necessary for rational SFT as outlined in \cite{Lat22} and used in \cite{Sie19,MZ20,MZ23}. We do not pursue this direction and instead establish that (at least genus zero) SFT moduli spaces fit into the flow category framework of \cite{AB24}. However, we have to generalize their definition of flow category in two ways. First, we allow the objects to be orbifolds and replace the cartesian product (in the definition of the compositions) by a form of fiber product. Secondly, the flow category comes with symmetric actions on the objects and morphism spaces. The details can be found in \textsection\ref{sec:sliced-flow-cats}. Restricting to the case of points and trivial symmetric actions recovers the classical definition. Tangential structures such as stable complex structures or framings can be defined for this definition of flow categories exactly as in \cite{AB24} and are discussed in \textsection\ref{subsec:sliced-flow-cat}. In particular, we show that these flow categories are the objects of a stable $\infty$-category (Theorem~\ref{prop:flow-stable}), extending one of the main results of \cite{AB24} to our setting. We expect that most of the flow category constructions in the literature, such as \cite{PS25}, can be adapted.

\begin{remark}
     Replacing the objects of the flow category by orbifolds was proposed by Mohammed Abouzaid. It differs from the definition of a Morse--Bott flow category in \cite{Z24,CK23,Bon24} in the way the composition maps are defined and because the group action depends on the object. This new definition appears naturally in our setup because we equip our curves with asymptotic markers at each puncture but do not constrain the markers to be mapped to a fixed base point on the Reeb orbit. Curiously enough, not choosing a base point is almost forced on us by the global Kuranishi chart construction.
\end{remark}

Given a contact manifold $(Y,\xi)$ with non-degenerate contact form $\lambda$, define $\cP$ to be the set of finite sequences of Reeb orbits. Given $L > 0$, we let $\cP_{\le L}\sub \cP$ be the subset of sequences where each element has action $\leq L$. We identify a Reeb orbit $\gamma$ with the associated orbifold $B\gamma$ given by the manifold 
$$E\gamma = \{\wt\gamma\mid \wt\gamma \text{ a constant-speed parametrization of }\gamma\}$$
equipped with the $S^1$-action that rotates the domain. This orbifold is equivalent to $[*/\bZ_m]$, where $m$ is the multiplicity of the Reeb orbit. To a sequence $\Gamma = (\gamma_1,\dots,\gamma_k)$ of Reeb orbits, we associate the product $B\Gamma = B\gamma_1\times \dots\times B\gamma_k$.

\begin{intthm}\label{thm:sft-flow-category} For any choice of $\lambda$-adapted almost complex structure $J$ there exists a stably complex rel--$C^1$ flow category $\scM^{Y,\lambda}$ with objects given by finite sequences of Reeb orbits and morphism spaces based on moduli spaces of punctured $J$-holomorphic curves.
\end{intthm}

We emphasize that we obtain a stable complex structure on $\scM^{Y,\lambda}$ despite including bad Reeb orbits in the objects of our symmetric flow category. The non-orientability of these Reeb orbits is compensated for by certain index bundles that are part of the stable complex structure; see \S\ref{subsec:lift-of-objects}. Theorem~\ref{thm:sft-flow-category} follows from a telescope construction using Theorem~\ref{thm:flow-cat} and Proposition~\ref{prop:colimit-of-flow-cats}. Concretely, in Theorem~\ref{thm:flow-cat} and Theorem~\ref{thm:stable-complex-structure}, we construct a symmetric flow category $\scM^{Y,\lambda}_{\leq L}$ and its stably complex lift with objects given by $\cP_{\le L}$. Proposition~\ref{prop:colimit-of-flow-cats} then asserts that for $L < L'$, we can find a stably complex symmetric bimodule from $\scM_{\leq L}^{Y,\lambda}$ to $\scM_{\leq L}^{Y,\lambda}$.

The flow category we construct should be regarded as an enhancement of contact homology after forgetting the natural dga structure on contact homology. A sketch of the relation is given in \S\ref{subsec:contact homology}. We expect that the dga structure can be encoded by utilizing the natural symmetric monoidal structure of concatenation on the objects $\scM^{Y,\lambda}_{\le L}$ but do not pursue it in this paper. 

\begin{remark}
    Instead of the flow category constructed here, one could use Theorem~\ref{prop:gkc-exists} to build a contact flow multi-category, where the objects are Reeb orbits, while moduli spaces of curves with one positive puncture and multiple negative punctures constitute the multi-morphisms. Since the foundational aspects of flow multi-categories are still under development, we chose to pursue a bar-construction-style flow category for contact manifolds.  
\end{remark}

\begin{remark} 
The global Kuranishi chart of $\Mbar^J_{\sft}(\Gamma^+,\Gamma^-;\beta)$ depends on a choice of \emph{perturbation datum}, see Definition~\ref{de:auxiliary-datum}, similar to other constructions of global Kuranishi charts, \cite{AMS21,HS22,BX22,Rez22}. One key step of Theorem~\ref{thm:sft-flow-category} is an inductive construction of such perturbation data, following \cite{BX22}. However, by carefully choosing these data and not using smoothing theory, we can avoid some of the additional steps.
\end{remark}


Restricting to cylinders, we can drop the restriction on the action in geometrically nice cases.

\begin{corollary}
    Suppose $(Y,\lambda)$ has no contractible Reeb orbits. Then, there exists a \emph{cylindrical contact flow category} $\scM^{cyl}$ with objects given by \emph{all} Reeb orbits of $\lambda$ and morphism spaces given by compactifications of moduli spaces of cylinders.
\end{corollary}

Let now $(\wh X,d\lambda)$ be an exact symplectic cobordism from $(Y^-,\lambda^-)$ to $(Y^+,\lambda^+)$ so that the primitive $\lambda$ is of the form $e^{\pm s}\lambda^\pm$ near the respective end of $\wh X$. Let $J$ be an $\omega$-compatible almost complex structure on the completion of $X$ so that $J = J^\pm$ over the completed ends for some cylindrical almost complex structure. 

\begin{intthm}[{Theorem~\ref{thm:flo_bim}}] Given an action bound $L$, any exact symplectic cobordism induces a rel--$C^1$ flow bimodule $\scN^{\wh X}$ from $\scM^{Y^-,\hspace{-0.2pt}\lambda^-}_{\le L}$ to $\scM^{Y^+,\hspace{-0.2pt}\lambda^+}_{\leq L}$.
\end{intthm}

\begin{corollary}[{Lemma~\ref{lem:trivial-cobordism-diagonal-bimodule}}]\label{cor:diagonal-bimodule}
     If $(\wh X,\omega)$ is the trivial cobordism $(\wh Y,d(e^s\lambda))$, then $\scN^{\wh X}$ is equivalent to the diagonal bimodule given suitable choices of auxiliary data.
\end{corollary}

The diagonal bimodule should be thought of as the identity morphism in $\infty$-category $\Flow^\Sigma$. Thus, the corollary can be seen as a (weak) naturality statement of our construction.

To prove invariance up to equivalence of the choice of contact structure and almost complex structure of the flow category $\scM^{Y,\lambda}$, one would have to further prove the flow-categorical analogs of the composition of chain homotopies and invariance under deformations using flow bordisms. This is delegated to future work.

\subsection*{Acknowledgments}
The authors thank Mohammed Abouzaid, John Pardon, and Chris Woodward for valuable input and discussions. They are grateful to Russell Avdek, Kristen Hendricks, Roman Krutowski, and Noah Porcelli for comments on an earlier draft.
 S.C. thanks Julian Chaidez, Sheel Ganatra, Srijan Ghosh, Eric Kilgore, Mohan Swaminathan, Kyler Siegel and Sushmita Venugopalan for valuable discussions. S.C. is grateful to Chris Kottke for clarifications on fibers of generalized blowup maps. S.C. also thanks the Max Planck Institute for Mathematics for its hospitality. A.H. is grateful to Russell Avdek for discussions on contact topology and SFT, to Janko Latschev for explaining his understanding of the algebraic framework of SFT, and to Nick Sheridan for sharing his method to deal with orientation signs.\par 
\noindent A.H. is supported by ERC grant ROGW, No. 864919.

\section{Moduli spaces of pseudo-holomorphic buildings}

\subsection{Background}\label{subsec:background}
Let $(Y^{2n-1},\xi)$ be a closed contact manifold equipped with a contact form $\lambda$, i.e., a $1$-form satisfying $\ker(\lambda) = \xi$ and $\lambda \wedge (d\lambda)^{n-1} > 0$. Let $R$ be the associated Reeb vector field, uniquely determined by the conditions 
$$d\lambda(R,\cdot) = 0 \qquad \qquad \lambda(R) = 1.$$
Its closed orbits are called Reeb orbits. We define
	$$\cP_1\coloneqq \cP_1(\lambda) := \modulo{\set{\gamma\in C^\infty(S^1,Y)\mid\exists T > 0 : \dot{\gamma} = T\,R(\gamma)}}{\sim}$$
    to be the set of unparametrized Reeb orbits of constant speed of $\lambda$. Let $\cP^*_1\sub \cP_1$ be the subset of simple Reeb orbits. Given $\gamma\in \cP_1$, we denote by 
	\begin{itemize}
		\item $\multr{\gamma}$ the multiplicity of $\gamma$,
		\item $\cc{\gamma}$ the underlying simple orbit and identify it with $\im(\gamma)$,
        \item $\wt\gamma$ any parametrization of $\gamma$,
            \item $\cA(\gamma) := \int_{S^1}\gamma^*\lambda$ the \emph{action} of $\gamma.$
	\end{itemize}
We assume that $\lambda$ is \emph{nondegenerate}, that is, that the return map $d\phi_R^T(\gamma(1))$ of the Reeb flow has no eigenvalue $1$ for a Reeb orbit $\gamma$ of action $T$. Equivalently, the set of simple unparametrized Reeb orbits is discrete in $Y$.

An almost complex structure $J$ on the \emph{symplectization} $(\wh Y,\omega)\coloneqq (\bR \times Y,d(e^s\lambda))$ is \emph{$\lambda$-adapted} if $J$ is invariant under the $\bR$- translation action, $J(\frac{\partial}{\partial t}) = R $ and $J|_\xi$ tames $d\lambda$. Any $d\lambda$-compatible almost complex structure $J_\xi$ on $\xi$ determines uniquely a $\lambda$-adapted almost complex structure $J$ on the symplectisation $\wh Y$.\par
A smooth map
$u\cl \dot C \to \bR \times Y$ on a punctured Riemann surface $(\dot{C},\fj)$ is \emph{$ J$-holomorphic} if it satisfies the Cauchy-Riemann equation $$\delbar_{ J} u \coloneq \frac12(du + J du \fj) = 0.$$ 
Note that due to exactness, all $J$-holomorphic maps from a closed Riemann surface to the symplectization are constant. The simplest type of non-constant $J$-holomorphic curves is given by \emph{trivial cylinders}. They are of the form $$u_\gamma\cl\bR \times S^1 \to \bR \times Y: (s,t)\mapsto (Ks,\wt\gamma(Kt)),$$
for some Reeb orbit $\gamma$ and a constant-speed parametrization $\wt\gamma$ of $\gamma$, where $K = \int_\gamma\lambda$ is the action of the Reeb orbit.\par 
Since $\wh Y$ is non-compact and the curves are punctured, the usual notion of energy does not make sense in this context. The analogue is the \emph{Hofer energy}, given by 
$$E_H(u) = \int_{\dot C} f ^* d\lambda + \sup_{\phi \in \cC}\int a \;da\wedge f^*\lambda $$ 
for map $u=(a,f)\cl \dot{C}\to \wh Y$, where $\cC$ is the set of non-negative smooth functions $\phi : \bR \to \bR$ with compact support and normalized $L^1$ norm (i.e., $\int_\bR \phi(s) ds = 1.)$  In the case of non-degenerate contact form $\lambda$, it follows from \cite{HWZI} that any $J$-holomorphic curve with finite Hofer energy asymptotically converges to a trivial cylinder over a Reeb orbit.\par
The compactness result \cite[Theorem 10.1]{BEH03} (or \cite{CM05}) proves that the moduli space of punctured curves with bounded Hofer energy and bounded topological type allows a natural compactification by adding boundary strata of $J$-holomorphic buildings. We will now discuss them in detail in genus zero.

\subsection{Buildings in symplectizations}\label{subsec:buildings}
The moduli space of genus zero buildings have a natural stratification modeled on trees. We recall some basics and develop relevant notations for moduli spaces of genus zero buildings. The dual trees underlying our buildings without levels take the following form.

\begin{definition}\label{de:decorated-tree} A \emph{decorated tree (or forest)} is a directed tree (or forest) $T$ with internal and external (respectively finite and exterior) edges together with the data of 
	\begin{itemize}
		\item a class $\beta_v \in H_2(Y,\{\gamma_{e}\}_{e\in E_v})$ for each $v \in V$,
		\item $\gamma\cl E(T)\to \cP_1$ associating to an edge a Reeb orbit,
	\end{itemize}
    We call a vertex $v \in V(T)$ \emph{trivial} if $\beta_v = 0$ and $v$ has at most two adjacent edges. A decorated tree $T$ is \emph{stable} if it has no trivial vertices.
\end{definition}

\vspace*{-1.5cm}
\begin{figure}[h]
    \centering
    
    \def\svgscale{0.95}
\begingroup%
  \makeatletter%
  \providecommand\color[2][]{%
    \errmessage{(Inkscape) Color is used for the text in Inkscape, but the package 'color.sty' is not loaded}%
    \renewcommand\color[2][]{}%
  }%
  \providecommand\transparent[1]{%
    \errmessage{(Inkscape) Transparency is used (non-zero) for the text in Inkscape, but the package 'transparent.sty' is not loaded}%
    \renewcommand\transparent[1]{}%
  }%
  \providecommand\rotatebox[2]{#2}%
  \newcommand*\fsize{\dimexpr\f@size pt\relax}%
  \newcommand*\lineheight[1]{\fontsize{\fsize}{#1\fsize}\selectfont}%
  \ifx\svgwidth\undefined%
    \setlength{\unitlength}{225bp}%
    \ifx\svgscale\undefined%
      \relax%
    \else%
      \setlength{\unitlength}{\unitlength * \real{\svgscale}}%
    \fi%
  \else%
    \setlength{\unitlength}{\svgwidth}%
  \fi%
  \global\let\svgwidth\undefined%
  \global\let\svgscale\undefined%
  \makeatother%
  \begin{picture}(1,1)%
    \lineheight{1}%
    \setlength\tabcolsep{0pt}%
    \put(0,0){\includegraphics[width=\unitlength,page=1]{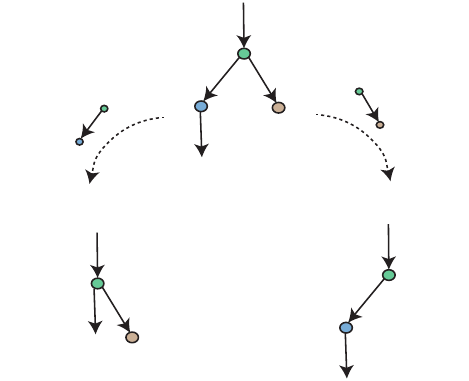}}%
    \put(-0.00592938,0.67403959){\makebox(0,0)[lt]{\lineheight{1.25}\smash{\begin{tabular}[t]{l}$\text{shrunk edge}$\end{tabular}}}}%
    \put(0.81129995,0.70430735){\makebox(0,0)[lt]{\lineheight{1.25}\smash{\begin{tabular}[t]{l}$\text{shrunk edge}$\end{tabular}}}}%
  \end{picture}%
\endgroup%

    \caption{Morphisms of a tree}
    \label{fig:tree_morph}
\end{figure}

\noindent Define the category $\scS$ (and $\scS^\bullet$) to have as objects decorated trees (respectively forests) as in Definition~\ref{de:decorated-tree} with morphisms given by contractions $p \cl T\to T'$ so that 
\begin{enumerate}
	\item for each non-contracted edge $e\in E(T)$ we have $\gamma'_{p(e)} = \gamma_{e}$,
	\item for each $v' \in V(T')$ we have $\beta_{v'} = \#_{p(v) = v'}\beta_v$.
\end{enumerate}
See Figure~\ref{fig:tree_morph} for an example.
The composition of two such morphisms is the obvious one. The category $\scS$ is the category $\scS_I$ in \cite{Par19}. We denote by $\Aut(T/T')$ the group of automorphisms of $T$ that leave $p$ invariant. We call a tree $T$ \emph{maximal} if any morphism $T\to T'$ is an isomorphism. Note for maps to a symplectization, any maximal decorated tree is a \emph{corolla}, at tree with a unique vertex and no interior edges.

\begin{notation*}\label{} Given a Riemann surface $C$ and a point $z \in C$, we denote by $$S_zC \coloneqq( T_zC\sm \{0\})/\bR_{> 0}$$ 
the boundary circle at infinity.
\end{notation*}

\begin{definition}\label{de:holomorphic-building} Let $T$ be a decorated tree. A \emph{pseudoholomorphic building of type $T$} consists of the following data
	\begin{enumerate}[label = (\roman*),leftmargin=20pt,ref= \roman*]
		\item for each $v \in V(T)$ a closed connected, possibly nodal, genus zero Riemann surface $(C_v,\fj_v)$ together with a set of pairwise distinct points $\{z_{v,e}\}_{e\in E_v}$;
		\item for each $v \in V(T)$ a smooth map $u_v \cl \dot{C}_v \coloneqq C_v\sm \{z_{v,e}\}\to \wh Y$, which  
		\begin{itemize}[leftmargin=15pt]
			\item is $C^0$-convergent to the positive trivial cylinder over $\gamma_{e}$ at $z_{v,e}$ for an incoming edge $e \in E_v^+$,
			\item is $C^0$-convergent to the negative trivial cylinder over $\gamma_{e}$ at $z_e$ for an outgoing edge $e \in E_v^-$
			\item represents $\beta_v$;
			\item satisfies $\delbar_{J,\fj_v}u_v = 0$. 
		\end{itemize}
		\item for each $e\in E^{\normalfont\text{ext}}_v$ an \emph{asymptotic marker} $\wt b_{e}\in S_{z_{v,e}}C_v$,
		\item for every interior edge $e$ from $v$ to $v'$, a \emph{matching isomorphism} $m_e \cl S_{z_{v,e}}C_v \to S_{z_{v',e}}C_{v'}$ intertwining $d{u_v}_Y$ and $d{u_{v'}}_Y$.
	\end{enumerate}
	We call such a building \emph{stable} if $T$ is stable.
    An \emph{isomorphism} of two such buildings consists of a collection $\{\iota_v \cl C_v \cong C'_v\}$ of biholomorphisms with $\iota_v(z_{v,e}) = z'_{v,e}$, preserving the matching isomorphisms so that $u_v = u'_v\g \iota_v$.
\end{definition}

\begin{definition}[Moduli spaces of buildings] Let $\wt\cM^J(T)$ be the moduli space of maps as in Definition~\ref{de:holomorphic-building} up to isomorphism. It admits a free $\bR^{V(T)}$-action given by translating the map on the corresponding sub-curve in the $\bR$-factor of $\wh Y$. The \emph{moduli space of $ J$-holomorphic buildings of type $T$} is
    $$\cM^J(T) \coloneqq \wt\cM^J(T)/\bR^{V(T)}.$$ 
    We define its \emph{SFT compactification} to be
	$$\Mbar^J(T) \coloneqq \djun{[T'\to T]}\cM^J(T')/\Aut(T'/T)$$
    equipped with the Gromov topology, \cite{Par19}.
\end{definition}

If $T$ is a corolla with incoming and outgoing exterior edges labeled by $\Gamma^+ = \{\gamma^+\}$ and $\Gamma^- = \{\gamma^-\}$, respectively, we write
    $$\Mbar^J(\Gamma^+,\Gamma^-;\beta) \coloneqq\Mbar^J(T).$$

We call $T$ \emph{effective} if $\Mbar^J(T)\neq \emst$.
By \cite[Theorem~10.1]{BEH03}, respectively \cite[Theorem 2.27]{Par19}, the moduli space $\Mbar^J(T)$ is compact for any $T$ and there exist only finitely many isomorphism classes of morphisms $T' \to T$ so that $T'$ is effective. Given $T$, define 
\begin{equation}\label{eq:tree-reeb-orbits}
    \cP(T) \coloneqq \union{T' \to T}{\{\gamma_e \mid e\in E(T')\}},
\end{equation}
to be the set of Reeb orbits which buildings in $\Mbar^J(T)$ can be asymptotic to, either at “internal" puncture or at an input or output puncture. Due to the compactness of $\Mbar^J(T)$, the set $\cP(T)$ is finite. In particular, for a fixed set of input and output Reeb orbits, the energy equations in \cite[\S 5.8]{BEH03} show that the Hofer energy is a constant that is linearly dependent on the action of the Reeb orbits.

\subsection{Leveled buildings}\label{subsec:leveled-buildings}
We now add level structures to our buildings by endowing the decorated trees defined in the last section with level structures. The derived orbifold of pseudo-holomorphic buildings with levels will be later used to construct the morphism spaces of the flow category in \textsection\ref{sec:sft-flow}. 

\begin{definition}\label{de:leveled-treee}
    A \emph{leveled tree} $(T,\ell)$ consists of a decorated tree $T$ as in Definition \ref{de:decorated-tree} equipped with a \emph{level function} $\ell : V(T) \to \bN$ satisfying 
    \begin{itemize}[leftmargin=20pt]
        \item if $\ell(v) = 1$, then $v$ has an incoming exterior edge,
        \item $\ell (w) \geq \ell(v)+1$ if $(v,w) \in E(T)$. 
     \end{itemize}
     We call $(T,\ell)$ \emph{stable} if each \emph{level} $\ell\inv(\{j\})$ is nonempty for $j \leq \max \ell$. The \textit{size} of a leveled tree is the maximal value of the level function. A \textit{morphism} of a leveled tree is obtained by simultaneous contraction of edges between adjacent levels.
\end{definition}

    Any leveled tree is from now on assumed to be stable.

\begin{lemma}\label{lem:pre-level-exists}
    Each decorated tree $T$ as in Definition~\ref{de:decorated-tree} admits a unique minimal level function, its \emph{pre-level function} $p\ell : V(T) \to \bN$ given by 
    \begin{itemize}[leftmargin=20pt]
        \item $p\ell(v) = 1$ if $v\in V^i$, the set of vertices with an incoming exterior edge,
        \item $p\ell (v) = \max \{ d(v,v_i) \mid v_i \in V^i \}$ 
     \end{itemize}
      where $d$ is the (edge-)distance function, which is well defined since $T$ is a tree.\qed
\end{lemma}

\begin{definition}
    A vertex $v$ in a decorated tree is \emph{trivial} if it has exactly two adjacent edges, both labeled by the same Reeb orbit, and if it carries the $0$ homology class.
\end{definition}

\begin{remark}
    Given a leveled tree $(T,\ell)$, we can construct a tree $T_\ell$ that contracts onto $T$ by adding a chain of $\ell(w)-\ell(v)-1$ trivial vertices between the vertices $v,w \in V(T)$ with $(v,w) \in E(T)$ and $\ell(w) > 1+ \ell(v)$. This recovers the usual notion of the underlying tree of SFT buildings.
\end{remark}

If the pre-level function is injective, it is the only level function on $T$. The simplest occurrence of non-injectivity is when the tree $T$ has three vertices (as in Figure \ref{fig:tree_morph}) with a pre-level containing two vertices. In this case, the fiber of the forgetful map 
\begin{equation}\label{eq:forget-levels}\Mbar_{\sft}^J (\Gamma^+,\Gamma^-;\beta)\to \Mbar^J(\Gamma^+,\Gamma^-;\beta)\end{equation}
is homeomorphic to a closed interval $[0,1]$. In particular, the interval $(0,1)$ records the relative height between the curves corresponding to each of the vertices, and the boundary $\{0,1\}$ of the fiber consists of the breaking of the pre-level $j$ into two levels and the appearance of trivial cylinder components.

\begin{figure}[H]
    \centering
    
    \def\svgscale{1}
    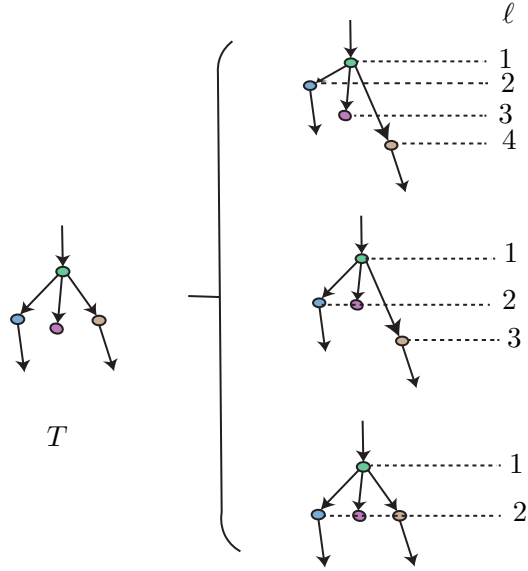

    \caption{Three different level functions for an unleveled tree. We do not draw the trivial vertices for the sake of clarity; one can recover the trivial vertices uniquely by the differences of levels between adjacent vertices.}
    \label{fig:leveled_tree}
\end{figure}

\begin{convention}
    We will denote the non-negative real line by $\bR_+ = [0,\infty).$
\end{convention}

\begin{definition}
    A leveled tree $(T,\ell)$ is \emph{maximally leveled} if $\#\ell\inv(j) \leq 1$ for each $j \in \bN$.
\end{definition}
    
\begin{proposition}\label{prop:maximally-leveled-tree}
 Given any decorated tree $T$, there exists a number, $N_T$, of maximally leveled trees $(T,\ell)$ over $T$, where $N_T$ is uniquely determined by the underlying tree.
\end{proposition}

The proof is by induction, for which we need another definition.

\begin{definition}
    A leveled tree $(T,\ell)$ is \emph{$k$-maximally leveled} if 
    \begin{itemize}
        \item $\#\ell\inv(j) \leq 1$ for $j \le k$;
        \item $\ell$ is injective on the set $p\ell\inv(\{1,\dots,k\})$.
    \end{itemize}
\end{definition}

\begin{lemma}\label{lem:increasing-maximality}
    Given a $k$-maximally leveled tree $(T,\ell)$, there exists $\#\ell\inv(k+1)!$-many leveled trees $(T,\ell')$ that are $(k+1)$-maximally leveled and contract onto $(T,\ell)$ and which are minimal with respect to this property.
\end{lemma}

\begin{proof} Set $N \coloneqq \#\ell\inv(k+1)$. Unraveling definitions, the conditions mean that we have to find $N$ level functions $\ell_1,\dots,\ell_N \cl V(T)\to \bN$ that satisfy
\begin{enumerate}[(a),leftmargin=25pt,ref=\alph*]
    \item\label{i:maximally-leveled} $\ell_i$ is $(k+1)$-maximal;
    \item\label{i:extends} $\ell_i$ agrees with $\ell$ on $\ell\inv(\{1,\dots,k\})$;
    \item\label{i:remains} for any $v,w \in V(T)$ that can be connected by a (directed) path, we have 
    \begin{equation}\label{eq:same-difference} \ell_i(v)-\ell_i(w) = \ell(v)-\ell(w) \end{equation}
    if $\ell(v)$, $\ell(w) \ge k+1$.
\end{enumerate}
    Condition~\eqref{i:extends} means that $\ell_i$ is determined on $\ell\inv(\{1,\dots,k\})$, while Condition~\eqref{i:remains} implies that the values of $\ell_i$ on $\ell\inv(\{m \geq k+2\})$ are determined by the values of $\ell_i$ on $\ell\inv(k+1)$, due to the minimality requirement. The $(k+1)$-maximality together with the requirement that $(T,\ell_i)$ be stable means that we have $N$ elements that have to distributed over $N$ levels. Thus, we have $N!$ options for extending $\ell$.
\end{proof}

\begin{proof}[Proof of Proposition~\ref{prop:maximally-leveled-tree}] The leveled tree $(T,p\ell)$ is tautologically $0$-maximal. Now we apply Lemma~\ref{lem:increasing-maximality} inductively until we obtain $\#V(T)$-corollas. The minimality of Lemma~\ref{lem:increasing-maximality} ensures that each $\#V(T)$-corolla can only be obtained through a unique path of $k$-corollas for $k \leq \# V(T)$. Thus, the number of such trees is determined inductively by the values $\#\ell\inv(k+1)!$ by Lemma~\ref{lem:increasing-maximality}.
\end{proof}

\begin{definition}\label{de:buildings-with-levels}
   Let $(T,\ell)$ be a leveled tree with underlying decorated tree $T$. A \emph{$J$-holomorphic building } $(u,C,b_*,m_*)$ \emph{of type $(T,\ell)$} is a $J$-holomorphic building of type $T_\ell$ as in Definition~\ref{de:holomorphic-building}. An isomorphism between two such buildings is defined as before. 
\end{definition}

\begin{remark}
    Note that for any such building $u$ and any trivial vertex $v$ of $T$, the map $u_v$ is forced to be a trivial cylinder.
\end{remark}

We write $\cM^J(T,\ell)$ for the space of leveled buildings of type $(T,\ell)$ up to reparametrization and translation, and we define the \emph{SFT moduli space of buildings of type at least $(T,\ell)$} to be 
   \begin{equation}
       \Mbar^J_{\sft}(T,\ell) \coloneqq \djun{(T',\ell')\to (T,\ell)}\cM^J(T',\ell')/\Aut(T'_{\ell'}/T_\ell)
   \end{equation}
   equipped with the Gromov topology as defined in \cite[\textsection7.3]{BEH03}.

\subsection{Buildings in exact symplectic cobordisms}
We now discuss the trees, which stratify the moduli space of buildings in exact symplectic cobordisms. 

\begin{definition}\label{de:exact-cobordism}
    An \emph{exact symplectic cobordism} from $(Y_+,\lambda_+)$ to $(Y_-,\lambda_-)$ is an exact symplectic manifold $(\wh X,\omega = d\lambda)$ together with embeddings
    \begin{gather*}
        \Theta^+\cl (-N,\infty)\times Y^+\to \wh X\\
         \Theta^-\cl (-\infty,N)\times Y^-\to \wh X
    \end{gather*} 
    so that $\wh X\sm \im(\Theta^+)\cup \im(\Theta^-)$ is compact and $(\Theta^\pm)^*\lambda = e^{s}\lambda^\pm$.
\end{definition}

An almost complex structure $J$ on a symplectic cobordism $(X,\omega)$ is \emph{compatible} if it is $\omega$-compatible and ${\Theta^\pm}^*J$ is translation invariant. In particular, this implies that ${\Theta^\pm}^*J$ is an adapted almost complex structure.

\begin{ex}\label{ex:interpolating-between-contact-structures}
    Suppose $(Y,\xi)$ is a contact manifold with two contact forms $\lambda^\pm$ and two adapted almost complex structures $J^\pm$. Since the space of contact forms is contractible, we can find a smooth path $\{\lambda_s\}_{s\in \bR}$ of contact structures so that $\lambda_s= \lambda^-$ for $s \leq -m$ and $\lambda_s = \lambda^+$ for $s \ge m$. We can then find a path $\{J_s\}_{s\in \bR}$ that agrees with $J^-$ on $\bR_{\leq -m}$ and $J^+$ on $\bR_{\ge m}$ and so that $J_s$ is $\lambda_s$-adapted for each $s$. By taking a $m\gg0$ and rescaling the homotopy $\{\lambda_s\}_{s\in \bR}$ we can ensure that $\partial_s\lambda_s$ is small, whence $d(e^s\lambda_s)$ is symplectic. Then, $J$ is a compatible almost complex structure on the symplectic cobordism $(\bR\times Y,\omega = d(e^s\lambda_s))$.
\end{ex}

\begin{notation*}
    Given an exact symplectic cobordism $(\wh X,\omega)$ from $(Y^+,\lambda^+)$ to $(Y^-,\lambda^-)$, we abbreviate
    $$X^{00} := \bR \times Y^+, \quad  X^{01}:=X, \quad X^{11}:= \bR \times Y^-.$$
    We set $\cP_{i} := \cP(X^{ii})$ and $\cP_{01} := \cP_0\sqcup \cP_1$.
\end{notation*}

\begin{definition}\label{de:decorated-cob-tree} A \textit{decorated cobordism tree (or forest)} is a tree (or forest) $T_c$ equipped with maps 
\begin{equation*}
    \ast \cl E(T_c)\to \{0,1 \}\qquad \qquad \ast_\pm\cl V(T_c) \to \{0,1\}
\end{equation*}
so that 
\begin{itemize}[leftmargin=20pt]
    \item $\ast(E^{\text{ext},+}) = 0$ and $\ast(E^{\text{ext},-}) = 1$,
    \item $\ast_+\leq \ast_-$,
    \item for any exterior edge $e\in E^{\text{ext},\pm}$ adjacent to the vertex $v$ we have $\ast(e) = \ast_\pm(v)$,
\end{itemize}
together with the data of
\begin{itemize}[leftmargin=20pt]
    \item a Reeb orbit $\gamma_e\in \cP_{\ast(e)}$ for each $e\in E(T_c)$,
    \item  a homology class $\beta_v\in H_2(X^{\ast_-(v)\ast_+(v)},\{\gamma_e\}_{e\in E_v})$ for each $v \in V(T_c)$.
\end{itemize}
We call $v$ a \emph{symplectisation vertex} if $\ast_+(v )= \ast_-(v)$. The cobordism tree is \emph{stable} if none of the symplectization vertices are trivial.
\end{definition}

Similar to the category $\scS$, there is a category of decorated cobordism trees $\scS^{c}$ whose objects are decorated cobordism trees and whose morphisms are contractions $T' \xrightarrow[\pi]{} T$ such that $\ast_+(\pi(v)) \leq \ast_+(v), \ast_-(\pi(v)) \geq \ast_-(v)$ and $\ast(\pi(e))= \ast(e)$ for any non-contracted edge $e$.
\vspace*{-1cm}
\begin{figure}[H]
    \centering
    
    \def\svgscale{1.1}
\begingroup%
  \makeatletter%
  \providecommand\color[2][]{%
    \errmessage{(Inkscape) Color is used for the text in Inkscape, but the package 'color.sty' is not loaded}%
    \renewcommand\color[2][]{}%
  }%
  \providecommand\transparent[1]{%
    \errmessage{(Inkscape) Transparency is used (non-zero) for the text in Inkscape, but the package 'transparent.sty' is not loaded}%
    \renewcommand\transparent[1]{}%
  }%
  \providecommand\rotatebox[2]{#2}%
  \newcommand*\fsize{\dimexpr\f@size pt\relax}%
  \newcommand*\lineheight[1]{\fontsize{\fsize}{#1\fsize}\selectfont}%
  \ifx\svgwidth\undefined%
    \setlength{\unitlength}{225bp}%
    \ifx\svgscale\undefined%
      \relax%
    \else%
      \setlength{\unitlength}{\unitlength * \real{\svgscale}}%
    \fi%
  \else%
    \setlength{\unitlength}{\svgwidth}%
  \fi%
  \global\let\svgwidth\undefined%
  \global\let\svgscale\undefined%
  \makeatother%
  \begin{picture}(1,1)%
    \lineheight{1}%
    \setlength\tabcolsep{0pt}%
    \put(0,0){\includegraphics[width=\unitlength,page=1]{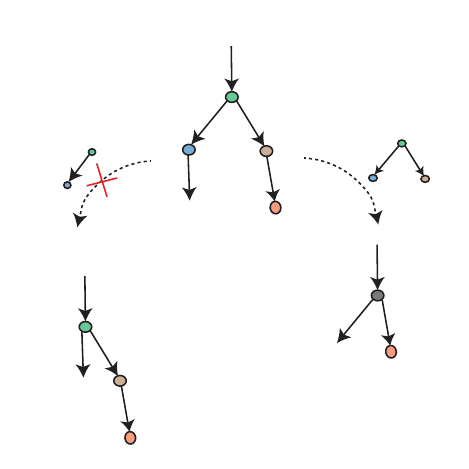}}%
    \put(0.55771398,0.7857545){\makebox(0,0)[lt]{\lineheight{1.25}\smash{\begin{tabular}[t]{l}$00$\end{tabular}}}}%
    \put(0.60585587,0.67736486){\makebox(0,0)[lt]{\lineheight{1.25}\smash{\begin{tabular}[t]{l}$01$\end{tabular}}}}%
    \put(0.43437698,0.65878914){\makebox(0,0)[lt]{\lineheight{1.25}\smash{\begin{tabular}[t]{l}$01$\end{tabular}}}}%
    \put(0.61991287,0.54651204){\makebox(0,0)[lt]{\lineheight{1.25}\smash{\begin{tabular}[t]{l}$11$\end{tabular}}}}%
    \put(0.84685593,0.3677446){\makebox(0,0)[lt]{\lineheight{1.25}\smash{\begin{tabular}[t]{l}$01$\end{tabular}}}}%
    \put(0.88905935,0.24710648){\makebox(0,0)[lt]{\lineheight{1.25}\smash{\begin{tabular}[t]{l}$11$\end{tabular}}}}%
    \put(0.2229616,0.29289322){\makebox(0,0)[lt]{\lineheight{1.25}\smash{\begin{tabular}[t]{l}$01$\end{tabular}}}}%
    \put(0.29741486,0.18141246){\makebox(0,0)[lt]{\lineheight{1.25}\smash{\begin{tabular}[t]{l}$01$\end{tabular}}}}%
    \put(0.33125722,0.06276499){\makebox(0,0)[lt]{\lineheight{1.25}\smash{\begin{tabular}[t]{l}$11$\end{tabular}}}}%
  \end{picture}%
\endgroup%

    \vspace*{-0.5cm}
    \caption{Morphisms of a cobordism tree}
    \label{fig:cob_tree_morph}
\end{figure}
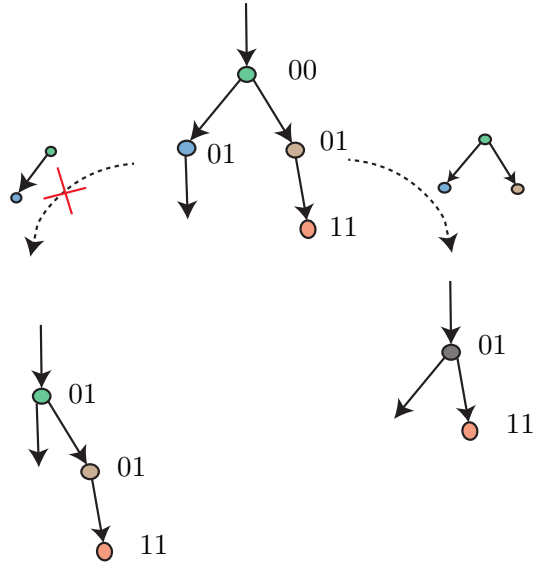

\begin{definition}\label{de:leveled-cob-tree}
    A \emph{leveled cobordism tree (or forest)} $(T_c,\ell_c)$ is a decorated cobordism tree (or forest) with a level function $\ell_c$ such that $(T_c,\ell_c)$ is a leveled tree as in Definition \ref{de:leveled-treee} and there is an integer $\mathfrak{c}>0$ so that $\ell_c\inv(\mathfrak{c}) = \ast_+\inv(0) \cap \ast_-\inv(1)$ and $(\ast_+(v),\ast_-(v))=(0,0)$ if and only if $\ell(v) < \mathfrak{c}$, while $(\ast_+(v),\ast_-(v))=(1,1)$ if and only if $\ell(v) > \mathfrak{c}$.
\end{definition}

\begin{definition}
    A leveled cobordism tree is \textit{maximally levelled} if $\ell\inv(j)\leq 1$ for all $j\in \bN \sm \{ \mathfrak{c}\}.$
\end{definition}

The proof of Proposition~\ref{prop:maximally-leveled-tree} carries over to leveled cobordism trees.

\begin{proposition}\label{prop:maximally-leveled-cob-tree}
 Given any decorated cobordism tree $T_c$, there exists a number $N_{T_c}$ of maximally leveled trees $(T_c,\ell)$ over $T_c$, where $N_{T_c}$ is uniquely determined by the underlying tree.\qed
\end{proposition}

\begin{definition}\label{de:cob-holomorphic-building} Let $T_c$ be a decorated cobordism tree. A \emph{pseudoholomorphic building of type $T_c$} consists of the same data as in Definition \ref{de:holomorphic-building} but with the property that the target of the map $u_v$ is $X^{\ast_-(v)\ast_+(v)}$.
	We call such a building \emph{stable} if $T_c$ is stable.\par
    An \emph{isomorphism} of two such buildings consists of a collection $\{\iota_v \cl C_v \cong C'_v\}$ of biholomorphisms with $\iota_v(z_{v,e}) = z'_{v,e}$ so that $u_v = u'_v\g \iota_v$.
\end{definition}

\begin{definition}[Moduli space of cobordism buildings] Let $\wt{\cM}^J(T_c)$ be the moduli space of maps as in Definition \ref{de:cob-holomorphic-building} up to isomorphism. It admits a free $\bR^{V_s(T)}$-action, given by translations of maps on each symplectization vertex. The \textit{moduli space of $J$-holomorphic cobordism buildings of type $T_c$} is $$\cM^J(T_c):= \wt\cM^J(T_c)/ \bR^{V_s(T)}.$$
We define 
$$\ol\cM^J(T_c) \coloneqq \bigsqcup_{[T'_c \to T_c]} \cM^J(T'_c)/\Aut(T'_c/T_c)$$ equipped with the Gromov topology as in \cite{BEH03}.
\end{definition}

\section{The global Kuranishi chart construction}\label{sec:prelim-aux}

\subsection{Base space}\label{subsec:base-space} Fix integers $n\geq 0$ and $d > 0$. We write
$$\cB = \cB(d) \sub \Mbar_0(\bP^d,d)$$ for the locus of (equivalence classes of) regular embedded stable holomorphic maps $\varphi \cl C\to \bP^d$ of genus $0$ and degree $d$. In particular, $\im(\varphi)$ is not contained in a complex hyperplane of $\bP^d$ and $\varphi$ admits no nontrivial automorphisms. Thus, $\cB$ is a smooth quasi-projective variety and served as the base space of previous global Kuranishi chart constructions, \cite{AMS21,HS22,AMS23}. Given nonnegative integers $n^\pm\ge 0$, we let $$\cB_{n^+,n^-} = \cB_{n^+,n^-}(d)$$ be the preimage of $\cB$ under the forgetful map $\Mbar_{0,n^+,n^-}(\bP^d,d)\to \Mbar_0(\bP^d,d)$, where the $(n^++n^-)$-many marked points are divided into positive and negative marked points. It is an easy verification that this is again a complex manifold.\par 
While the spaces $\cB$ and $\cB_{n}$ capture the domain breaking of stable maps, they do not capture the degenerations of buildings, where nodes are constrained by the asymptotic conditions. Thus, we construct a real-oriented blow-up of $\cB$ that will serve as the base space for the Kuranishi chart of Pardon buildings; see \textsection\ref{subsec:framings}. It agrees with the base space used in Hamiltonian Floer theory, \cite{BX22,Rez22,AB24}. In \textsection\ref{subsec:base-with-levels}, we perform a generalized corner blow-up to obtain the base space for the Kuranishi chart of leveled buildings.

\subsubsection{Without levels}\label{subsec:framings} We will perform real-oriented blow-ups on the space $\cB_{n^+,n^-}(d)$ following \cite{BX22}. These spaces will serve as the base spaces for the global charts of Pardon buildings in \ref{subsec:construction}. Assume we are given an integer $d_i^\pm\geq 1$ for each marked point $z_i^\pm$ so that
    $$d =  p\lbr{\s{i}{d^+_i}-\s{j}{d^-_j}} + n - 2$$
for some uniform integer $p \ge 1$ where $n$ is the number of marked points.
The following definition will allow us to modify $\cB_{n^+,n^-}(d)$ to obtain the right boundary stratification for curves in symplectizations. 

\begin{definition}\label{de:type-for-symplectisation} Given a stable map $\varphi \cl C\to \bP^d$ whose domain has a unique node $x$, we say that $x$ is of \emph{type $0$} if it is non-separating or if it separates $C$ into irreducible components $C_0$ and $C_1$ of degree $d_0$, respectively $d_1$ so that 
\begin{equation}\label{eq:degree-comparison} \mathrm{d}_x:=
(d_0-p\hspace{-3pt}\s{z^+_i \in C_0}{d_i^+}+p\hspace{-3pt}\s{z^-_j \in C_0}{d_j^-}  -\deg(\omega_{C}|_{C_0})) - (d_1-p\hspace{-3pt}\s{z^+_i \in C_1}{d_i^+}+p\hspace{-3pt}\s{z^-_j \in C_1}{d_j^-} -\deg(\omega_{C}|_{C_1})) = 0.
\end{equation}
We say $x$ is of \emph{type $1$} with order $|\mathrm{d}_x|$ otherwise.\end{definition}

The following does not depend on the way we determine the type of the node. Given $[\varphi,C,z_*] \in \cB_{n^+,n^-}(d)$ whose domain has two nodes $x_1,x_2$, we call $x_i$ of \emph{type $j$} if $[\varphi,C,z_*]$ is the limit of a sequence of maps with one-nodal domains so that the node converges to $x_i$ and each of them is of type $j$. Inductively, this yields a decomposition $N_C = N_C\inn \sqcup N_C^\bullet$ of the nodes of the domain of any $[\varphi,C,z_*]\in \cB$.

\begin{definition}\label{de:sft-base-space}
    The space $\cBR_{n^+,n^-}(d)$ consists of elements of $\cB_{n^+,n^-}(d)$ equipped with the following additional data
 \begin{itemize}[leftmargin=20pt]
	\item an asymptotic marker $b_{k}^\pm \in S_{z_k^\pm}C$ at each marked point $z_k^\pm$ determining a real line $\ell_k^\pm \sub T_{z_k^\pm}C$,
	\item an isomorphism $m_z \cl S_zC_v \to S_zC_{v'}$, or, equivalently, an element of $(T_zC_v\otimes T_zC_{v'}\sm \{0\})/S^1$, for each $z\in \cN_C^\bullet$ between the irreducible components $C_v$ and $C_{v'}$ of $C$.\footnote{To see this equivalence, note that $\ell = [\xi \otimes \xi']$ yields a map $m_\ell([\eta]) := [\lspan{\xi,\wt\eta}\xi']$ where $\lspan{\cdot}$ is any Hermitian product on $T_zC$, while conversely given $m$ we can set $\ell_m := [\wt\xi \otimes \wt m(\xi)]$ where $\xi \in S_zC$ is arbitrary and the tilde means that we take a lift in the respective tangent space.} 
\end{itemize}
\end{definition}

For the next assertion, we recall the construction of the real-oriented blow-up of a complex manifold $Z$ along an effective normal crossing divisor $D$ from \cite[\textsection8.2]{Sab13}. Suppose first $D$ is smooth and principal, and let $S_D \to Z$ be the $S^1$-bundle of the line bundle $L_D \to Z$ associated to $D$. Let $f \cl Z\to L_D$ be a holomorphic section. Then $f\inv(0) = D$, so the composition $\wt f$
$$Z\sm D\xra{f} L_D\sm 0 \to S_D$$
is well-defined. The \emph{real-oriented blow-up} $\text{Bl}_Z(D)$ of $Z$ along $D$ is the closure of the image of $\wt f$ and the blow-down map $\text{Bl}_Z(D) \to Z$ is simply given by the restriction of the projection $S_D\to Z$.
By \cite[Lemma~8.1]{Sab13}, the space $ \text{Bl}_Z(D)$ admits a unique smooth structure with 
\begin{equation}\label{eq:boundary-blow-up}
    \del\text{Bl}_Z(D) \cong S_D|_D.
\end{equation}

\par 
If $\{D_i\}_i$ is a normal crossing divisor with finitely many irreducible components, then there exists a real blow-up of $Z$ along $\{D_i\}_i$. By \cite[Lemma~8.2]{Sab13}, it is given by
\begin{equation}\label{eq:ncd-blowup}\text{Bl}_Z(\{D_i\}_i) = \text{Bl}_Z(D_1)\times_{Z}\dots \times_Z\text{Bl}_Z(D_k).\end{equation}
which we can take as the definition for the purposes of this paper.

\begin{remark}
    In our construction we also encounter the case of blowing up a divisor $D$, which has normal crossings self-intersections, i.e., any point of $D$ admits a \emph{holomorphic} chart $\phi\cl U \to Z$ where $U$ is an open neighborhood of $0$ in $\bC^n$ and $\phi\inv(D) = U \cap \{z_1\dots z_k = 0\}$. 
    We define the smooth structure on the real-oriented blow-up, $\text{Bl}_Z(D)$ via Equation \eqref{eq:ncd-blowup} in each such coordinate chart $U$. As the transition functions on the base are holomorphic, they lift to diffeomorphisms between the blow-ups.
\end{remark}

\begin{lemma}\label{lem:building-base-smooth} $\cBR_{n^+,n^-}(d)$ is a smooth $\PGL_{d+1}(\bC)$-manifold with corners so that the depth in the boundary of $(\varphi,C,z_*,b_*,m_*)$ is given by the number of elements in $\cN^\bullet_C$. Moreover, each corner stratum is invariant under the $\PGL_{d+1}(\bC)$-action. 
\end{lemma}

\begin{proof}
	We abbreviate $\cBR = \cBR_{n^+,n^-}(d)$ and $\cB = \cB_{n^+,n^-}(d)$. Recall that $\bL_i \to\cB$ is the complex line bundle with fiber given by $T_{z_i}^*C$ at $(\varphi,C,z_*)$. Let 
	$$\cB^{am}:= \bigoplus\limits_{i =1}^{k^+}S(\bL\dul_{i})\oplus  \bigoplus\limits_{j =1}^{k^-}S(\bL\dul_{j})$$
	be a direct sum of $U(1)$ bundles corresponding to the line bundles $\bL_*$ and let $\pi_{am} \cl \cB^{am}\to \cB$ be the forgetful map. 
    Then $\cB^{am}$ is clearly smooth (without corners) and $\pi_{am}$ is a submersion.\par Let $D_1,\dots,D_\ell$ be the divisors in $\cB$ given by curves with at least one node in $\cN^\bullet$. They form a normal crossing divisor, so we can define $\cB^{mi} := \text{Bl}_\cB(\{D_i\}_i)$ with blow-down map $\pi_{mi}\cl \cB^{mi}\to \cB$. We claim that 
    \begin{equation}
        \cBR = \cB^{am}\times_\cB\cB^{mi}
    \end{equation}
    and that the forgetful map $\cBR\to \cB$ is induced by the blow-down map $\cB^{mi}\to \cB$. To see this, we note that the asymptotic markers at the marked points correspond exactly to the additional data of elements of $S(\bL\dul_{j})$. Forgetting them yields the canonical map $\cBR\to \cB^{mi}$. Meanwhile, the normal bundle of $D_i$ has fiber $N_{\varphi} = \bL_{z^+}\dul\otimes \bL_{z^-}\dul$ by \cite[\textsection XI.3]{ACG11}. Thus an element of its sphere bundle is exactly a matching isomorphism, so the claim without the group action follows from \eqref{eq:boundary-blow-up}.
    Finally, by \cite[Theorem~5.1]{AK10}\footnote{To be precise, the statement is for projective blow-up while we work with the spherical one. However, as mentioned in the paragraph above Remark~1.1 op. cit., their results also hold for spherical blow-ups.}, the smooth $\PGL_{d+1}(\bC)$-action on $\cB$ lifts to a smooth action on $\cBR$. For the last assertion, it suffices to show that each stratum of codimension $1$ is invariant under the $\PGL$-action. Let $S$ be such a stratum, and note that $S$ is a connected component of the space $S_1(\cBR_d)$ of codimension-$1$ points in $\cBR_d$. Since $S_1(\cBR_d)$ is preserved by the action and $\PGL_{d+1}(\bC)$ is connected, the orbit of $S$ under this action is connected as well. Hence, it has to agree with $S$.
\end{proof}

Returning to the notation of \S\ref{subsec:buildings}, recall that $\scS$ is the category of decorated trees defined in Definition~\ref{de:decorated-tree}. Let $\scS\inn$ be the category obtained from $\scS$ by forgetting the data of the Reeb orbits associated to the edges and replacing the relative homology class $\beta_v \in H_2(Y,\{\gamma_{v,e}\}_{e\in E_v})$ by an integer $d_v \geq 1$, corresponding to the degree of the map $\wch\varphi_v \cl  C_v \to \bP^d$, where $\wch\varphi$ is the image of $\varphi$ under the blow-down map $\scBR\to \cB$. We equip $\scS\inn$ with the dimension function 
\begin{equation}\label{de:dimension-tree}
    \dim(T) = 2(d-3)+ 2(d+1)d+3(\# \Gamma^- + \# \Gamma^+)  -\# E^{\text{int}}(T)
\end{equation}

\vspace*{1.5pt}

\begin{lemma}\label{lem:stratification-base-space}
There exists a canonical stratification $P\cl \cBR_{n^+,n^-}(d) \to \scS\inn$ which assigns the tree type of the domain to an element of $\cBR_{n^+,n^+}(d)$. It is cell-like in the sense that $\cBR_{/T} := P(\scS\inn_{/T})$ is a smooth manifold with corners of dimension $\dim(T)$ with interior given by $P\inv(\{T\})$.
\end{lemma}

\begin{proof} Define the function $P^\bR$ by letting $T_{\wh\varphi} = P^\bR(\wh\varphi)$ be the underlying dual graph of the domain of $\pi(\wh\varphi)$, where we have collapsed all edges corresponding to nodes of type $0$ (and added the degrees of the associated vertices). Given an element in (the preimage of) a gluing chart (under $\pi$) near $\wh \varphi$, we obtain a unique morphism $T_{\wh\varphi}\to T_{\wh\phi}$ in $\scS\inn$. Since $\cB$ is unobstructed, the induced germ of a map $\cBR\to \scS\inn_{T_{\wh\varphi}/}$ near $\wh\varphi$ is a stratification. The last claim follows from the dimension formula for $\cB$ and the description of the corner strata of the real blow-up above. Observing that $\cBR\to \cB^{mi}$ is a torus bundle over $\cB^{mi}$, this completes the proof.\end{proof}

\subsubsection{Base space with levels}\label{subsec:base-with-levels}
In this subsection, we construct a generalized blow-up of $\cBR$ that will serve as the base space for the global Kuranishi chart of $\Mbar_{\sft}^{\, J}(\Gamma^+,\Gamma^-;\beta)$. For this, we use the generalized blow-up of \cite{KM15} and the stratification of $\cBR$ by decorated trees.
\vspace*{-1pt}
\subsubsection*{Generalized blow-up} We give a quick recap of the generalized blow-up of a manifold $X$ with corners as defined in \cite{KM15}. We denote the set of closures of connected boundary (codimension 1) components by $\cM_1(X)$, assuming that no $F\in \cM_1(X)$ self-intersects. We can then define the set of codimension-$k$ faces $\cM_k(X)$ to consist of intersections $F_I = F_{i_1}\cap \dots\cap F_{i_k}$ of $k$ distinct elements of $\cM_1(X)$. We now associate the following combinatorial data to $X$. 
\begin{itemize}[leftmargin=17pt]
    \item  To each face $F \in \cM_k(X)$, we associate the freely generated monoid $$\sigma_F \coloneqq \bigoplus_{\substack{H\in \cM_1{X},\\ F \sub H}} \bN e_H$$
    \item  The \emph{monoidal complex} $\cP_X$ of $X$ consists of the collection of monoids $\sigma_F$ for every face $F$ of codimension 1 and higher, together with the canonical maps $i_{GF} : \sigma_G \to \sigma_F$ induced by inclusion of the faces $F\sub G$.
\end{itemize}

\begin{definition}[{\cite[Definition~2.2]{KM15}}]
    A \emph{refinement} $\cR_\sigma$ of a single monoid $\sigma$ is a collection $\cR_\sigma = \{\tau\mid \tau \subset\sigma\}$ of submonoids such that 
\begin{enumerate}[label=\roman*),leftmargin=22pt,ref=\roman*]
    \item\label{maximal-monoids-suffice} if $\tau\in \cR_\sigma$ and $\tau'\sub \tau$, then $\tau'\in \cR_\sigma$,
    \item for any $\tau_1,\tau_2\in \cR_\sigma$, the intersection $\tau_1\cap \tau_2$ is a face of both $\tau_1$ and $\tau_2$,
    \item $\text{span}_{\bR_+}(\sigma) = \union{\tau\in \cR_\sigma}{\text{span}_{\bR_+}(\tau)}$.
\end{enumerate}
\end{definition}
Note that \eqref{maximal-monoids-suffice} implies that a refinement is uniquely determined by the maximal monoids it contains. A classical example of a refinement of a monoid is given by subdivison. We refer to \cite[\textsection2]{KM15} for more details and examples.\par 
The more general notion of the \emph{refinement of monoidal complex} $\cR_Q$ amounts to a collection $\cR_Q = \{\cR_\sigma\mid \sigma\in Q\}$ of refinements of the monoids of $Q$ together with compatibilities between these refinements. We refer the reader to \cite[Definition 4.7]{KM15} for the precise definition.

The main construction in \cite{KM15} can be paraphrased as follows.

\begin{theorem}[{\cite[Theorem A]{KM15}}]\label{thm:corner-blowup}
    For any smooth refinement $\cR \to \cP_X$, there is a manifold $Y = [ X ; \cR] $ with corners with a blow-down map $b : Y \to X$ such that $\cP_Y  = \cR$ and the blow-down map induces the refinement $\cR \to \cP_X$. 
\end{theorem}

We refer the reader to the excellently written \cite{KM15} for the proof and just add some description of the exceptional divisors, which is missing from \cite{KM15}. This will be useful for understanding the difference between real-oriented blow-ups and generalized blow-ups as well as\textsection\ref{subsec:leveled-gluing}, which shows that the generalized blow-up yields the `correct' base space.\par 
Recall that for the real-oriented blow $\text{Bl}_D(X)$ of a smooth quasi-projective variety along a normal crossing divisor $D = \{D_i\}_i$ in $X$, we intuitively replace $D_i$ by the
\emph{spherical projectivization} 
\[\bP_{> 0}(N_{D_i/X}) \coloneqq (N_{D_i/X}\sm 0)/\bR_{>0}.\]
Given now a smooth manifold $Y$ with corners and two embedded codimension-$1$ boundary strata $Z_1$ and $Z_2$ intersecting in $Z_{12}$, the generalized blow-up of $Y$ along $Z_{12}$ replaces $Z_{12}$ by the \emph{positive part} 
\begin{equation}\label{eq:positive-part}
    \bP_+(N_{Z_{12}/Y}) = \set{(y,[v])\in \bP_{> 0}(N_{Z_{12}/X})\mid v  =v_1\oplus v_2 \text{ with } v_i\in N_{Z_i/Y}\text{ inward pointing}}.
\end{equation}
of the spherical projectivization, where we use that $N_{Z_{12}/Y}$ splits as the direct sum $N_{Z_1/Y}\oplus N_{Z_2/Y}$. In the picture below we show the simplest case.

\begin{figure}[h]
    \centering
    
    \def\svgscale{0.6}
\begingroup%
  \makeatletter%
  \providecommand\color[2][]{%
    \errmessage{(Inkscape) Color is used for the text in Inkscape, but the package 'color.sty' is not loaded}%
    \renewcommand\color[2][]{}%
  }%
  \providecommand\transparent[1]{%
    \errmessage{(Inkscape) Transparency is used (non-zero) for the text in Inkscape, but the package 'transparent.sty' is not loaded}%
    \renewcommand\transparent[1]{}%
  }%
  \providecommand\rotatebox[2]{#2}%
  \newcommand*\fsize{\dimexpr\f@size pt\relax}%
  \newcommand*\lineheight[1]{\fontsize{\fsize}{#1\fsize}\selectfont}%
  \ifx\svgwidth\undefined%
    \setlength{\unitlength}{450bp}%
    \ifx\svgscale\undefined%
      \relax%
    \else%
      \setlength{\unitlength}{\unitlength * \real{\svgscale}}%
    \fi%
  \else%
    \setlength{\unitlength}{\svgwidth}%
  \fi%
  \global\let\svgwidth\undefined%
  \global\let\svgscale\undefined%
  \makeatother%
  \begin{picture}(1,0.5)%
    \lineheight{1}%
    \setlength\tabcolsep{0pt}%
    \put(0,0){\includegraphics[width=\unitlength,page=1]{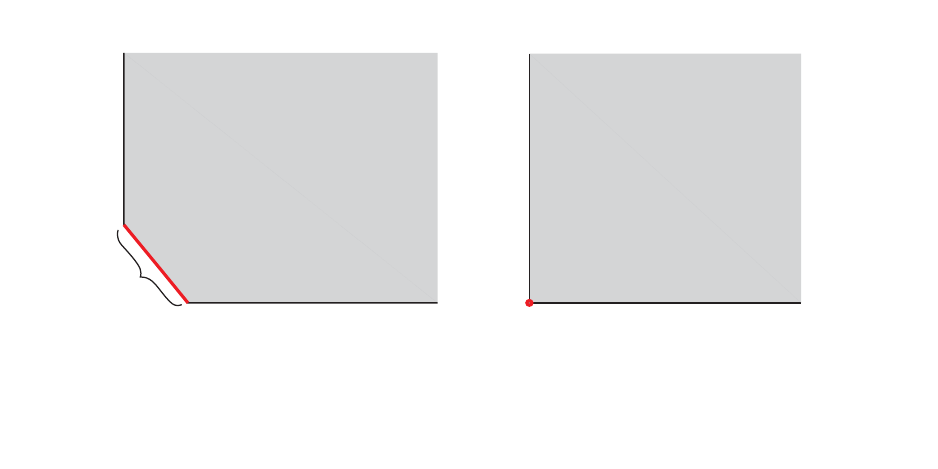}}%
    \put(0.38804408,0.07686628){\makebox(0,0)[lt]{\lineheight{1.25}\smash{\begin{tabular}[t]{l}$[\bR_+^2;\cR] \longrightarrow \bR_+^2 $\end{tabular}}}}%
    \put(0,0){\includegraphics[width=\unitlength,page=2]{blo2woT.pdf}}%
  \end{picture}%
\endgroup%

    \caption{Corner blowup of $\bR_+^2$.}
    \label{fig:blowup1}
\end{figure}

\begin{remark}\label{rem:neighbourhood-of-exceptional-boundary}
    Given a Riemannian metric $g$ on $Y$, we can identify $\bP_+(N_{Z_i/Y})$ with a subset of the sphere normal bundle $SN_{Z_i/Y}$ of $Z_i$ and extend this isomorphism to an isomorphism 
    \begin{equation*}
        [0,\epsilon)\times \bP_+(N_{Z_{12}/Y})\cong U \sub \wt Y,
    \end{equation*}
    where $\wt Y$ is the blow-up. Then, the blow-down map $\beta$ becomes the map
    \begin{equation*}
        [0,\epsilon)\times \bP_+(N_{Z_{12}/Y})\to Y : (t,y,[v]) \mapsto \exp_y(t\wt v),
    \end{equation*}
    where $\wt v$ is the unique lift of $[v]$ to the sphere normal bundle.
\end{remark}

\noindent Suppose $D$ is a normal crossing divisor in $X$ with two irreducible smooth components $D_1,D_2$. Applying the above observation to the case of $Y = \text{Bl}_D(X)$ with $Z_i$ the preimage of $D_i$ under the blow-down map, we have that the normal bundle of (the interior of) $Z_i$ is canonically isomorphic to the hyperplane line bundle 
\begin{equation}
    L_i\,\coloneqq \, \cO_{\bP_{>0}(N_{D/X})}(1) \coloneqq \bP_{>0}(N_{D/X})\times_D N_{D/X}.
\end{equation}
On the other hand, the (interior of the) intersection $Z_{12}$ is canonically identified with 
\begin{equation}\label{eq:intersection-and-projectivization}
    Z_{12}\,\cong\,\lbr{\bP_{+}(N_{Z_1/X})\times \bP_{+}(N_{D_2/X})}|_{D_{12}}
\end{equation}
with normal bundle corresponding to 
$$N_{Z_{12}/Y} \;\cong\; L_{1}|_{Z_{12}}\,\oplus\, L_{2}|_{Z_{12}}$$
under the identification~\eqref{eq:intersection-and-projectivization}. The generalized blow-up of $Y$ along $Z_{12}$ now replaces $Z_{12}$ with $\bP_{+ }(L_{1}|_{Z_{12}}\oplus L_{2}|_{Z_{12}})$. Hence, a point $p$ in $\beta\inv(Z_{12})$ corresponds to a tuple 
\begin{equation}\label{eq:point-in-blow-up}
p = \lbr{y,[v_1],[v_2],[v'_1\oplus v'_2]},\end{equation}
where $y \in D_{ij}$ and $v_i,v'_i \in (N_{D_i/X})_y$ with $[v_i'] = [v_i]$. The brackets denote the equivalence class under the $\bR_{>0}$-action on the respective bundle. The key point of~\eqref{eq:point-in-blow-up} is that the ``added data'' of $[v'_1\oplus v'_2]$ yields a ratio obtained by choosing a Riemannian metric on $X$ and lifting $v_i$ and $v_1'\oplus v_2'$ to unit vectors with respect to that metric. Another consequence is that the normal bundle of the embedding $\beta\inv(Z_{12})\hkra \beta\inv(Z_1)$ is exactly the pullback of $L_{2}|_{Z_{12}}$.

\subsubsection{Blow-up of $\cBR$}\label{subsubsec:blowup_to_leveled_base} We now return to our situation at hand, the real-oriented blow-up $\cBR \coloneqq\cBR_{n^+,n^-}(d)$ of $\cB_{n^+,n^-}(d)$. The construction of the generalized blow-up of \cite{KM15} as written assumes that the codimension one boundary components are embedded. In our situation we do not have this property due to the fact that the contraction maps $T'\to T$ of trees in $\scS$ can be fixed by a nontrivial automorphism of $T'$, see e.g., \cite[Figure~8]{Par19}. We therefore split up the construction into two steps.
\begin{itemize}[leftmargin=15pt]
     \item Fix a neighborhood $U$ of the closure of $\cBR_2$, the union of all codimension 2 strata in $\cBR$, such that every boundary (closure of codimension $1$ points) component in $U$ is embedded. 
    \item Replace $U$ with a corner blow-up $[U;\cR]$ corresponding to a suitable refinement of the monoidal complex $\cR \to \cP_{U} $.
\end{itemize}

The first step, the choice of $U$ we can do immediately. Let us now define the refinement $\cR$, recalling that the strata of $\cBR$ are indexed by decorated trees.

\begin{definition}\label{de:leveled-monoid}
    Given a decorated tree $T$ and an enumeration $E(T) = \{e_1,\dots,e_n\}$ with associated monoid $\sigma_T = \bN\lspan{e_1,\dots,e_n}$ and a maximal level function $\ell$ on $T$, the associated monoid is 
    \begin{equation}\label{} \sigma_{T,\ell} = \bN\lspan{e'_1,\dots,e'_n},\end{equation}
    where 
    \begin{equation}
        e'_i \coloneqq e_i + \s{\substack{p\ell(e_j)\leq p\ell(e_i)\\e\ell(e_j) < e\ell(e_i)}}{e_j},
    \end{equation}
	with $e\ell((v,w)) = \ell(w)$ for any edge $e = (v,w)$.
\end{definition}

\begin{definition}\label{} Given a decorated tree $T$, we let $L_T$ be the set of maximally leveled level functions on $T$. The \emph{refinement} $\cR(\sigma_T)$ of $\sigma_T$ is the refinement generated by $\set{\sigma_{T,\ell}\mid \ell\in L_T}$.
\end{definition}

\begin{lemma}\label{} $\cR(\sigma_T)$ is a smooth refinement, and the refinements $\cR(\sigma_T)$ form a refinement $\cR$ of the monoidal complex $\cP_U$.
\end{lemma}

\begin{proof} 
It suffices to show that each $\sigma_{T,\ell}$ is a smooth monoid in the sense of \cite[]{KM15}. Since $\sigma_T$ is smooth, it suffices to show that $e'_1,\dots,e'_n$ are linearly independent in $\bR\lspan{e_1,\dots,e_n}$. This follows from the definition of a level function. The second claim is a direct consequence of the construction.
\end{proof}

\begin{definition}\label{de:sft-base-sapce} We define the \emph{leveled base space}  to be
	\begin{equation}\label{}\cBS = \cBS_{n^+,n^-}(d) \coloneqq \frac{[U;\cR]\sqcup (\cBR\sm\cBR_2)}{\sim}\end{equation} 
	where we identify the interior of $[U;\cR]$ with $U \sm\cBR_2$ via the blow-down map $[U;\cR]\to U$.
\end{definition}

\begin{lemma}
The space $\cBS_{n^+,n^-}(d)$ is a smooth oriented manifold with corners whose codimension-$k$ boundary strata correspond to $(k+1)$-leveled trees.
\end{lemma}

\begin{proof}
The proof will be based on \cite[ Proposition 3.2]{KM15} which states that the dimension of a monoid $\tau$ in a refinement is equal to the codimension of a face $F_\tau$ corresponding to the monoid in the blow-up. We know that in the complement of $\cBR_2$, the codimension-$1$ strata correspond to trees with exactly two vertices, and there are no strata of higher codimension. Thus the result follows easily in $\cBR\sm\cBR_2.$ Definition \ref{de:sft-base-sapce} shows that it remains to consider $[U;\cR]$.
It follows from Definition \ref{de:leveled-monoid} that every monoid of dimension $k$ in the smooth refinement corresponds to a leveled tree with $(k+1)$-many levels. The result then follows directly from \cite[Proposition~3.2]{KM15}.
\end{proof}

\begin{ex}We briefly describe the preimage of a (simple) corner stratum of $\cBR$ under the blow-down map. Let $T_k$ be the tree with a unique vertex, one input and $k$ outputs, and let $\cBR_{T_k}$ be the stratum corresponding to $T_k$. For the standard simplex $\Delta^n $, we define the \emph{maximally blown-up simplex} $\wt \Delta^n$ by iteratively blowing up the faces of the simplex, starting with the zero-dimensional faces and ending at a $(n-2)$-dimensional face. Thus, we obtain a sequence 
    \begin{equation}\label{eq:max-blown-up-simplex} \Delta^n \xleftarrow[\text{ blow up vertices}]{\pi_1}\Delta^n_1  \xleftarrow[\text{ blow up edges  }]{\pi_2} \Delta^n_2  \xleftarrow[\text{ blow up $2$-faces}]{\pi_3} \dots \xleftarrow[\text{ blow up } (n-2) \text{-faces}]{\pi_{n-2}} \wt \Delta^n \end{equation}
    of real-oriented and generalized blow-ups. We write $\wt\pi \cl \wt\Delta^n \to \Delta^n$
    for the composition of the maps in \eqref{eq:max-blown-up-simplex}.
Now, a small enough neighborhood of $\cBR_{T_{k}}$ is diffeomorphic to $\cBR_{T_{k}} \times [0,1)^k$. The corner blow-up replaces $\cBR_{T_{k}} \times [0,1)^k$ with $\cBR_{T_{k}} \times \wt{\Delta}^{k-1} \times [0,1).$ Explicitly, in the case of the tree $T_2$ we have 

\begin{figure}[h]
    \centering
    
    \def\svgscale{1}
\begingroup%
  \makeatletter%
  \providecommand\color[2][]{%
    \errmessage{(Inkscape) Color is used for the text in Inkscape, but the package 'color.sty' is not loaded}%
    \renewcommand\color[2][]{}%
  }%
  \providecommand\transparent[1]{%
    \errmessage{(Inkscape) Transparency is used (non-zero) for the text in Inkscape, but the package 'transparent.sty' is not loaded}%
    \renewcommand\transparent[1]{}%
  }%
  \providecommand\rotatebox[2]{#2}%
  \newcommand*\fsize{\dimexpr\f@size pt\relax}%
  \newcommand*\lineheight[1]{\fontsize{\fsize}{#1\fsize}\selectfont}%
  \ifx\svgwidth\undefined%
    \setlength{\unitlength}{450bp}%
    \ifx\svgscale\undefined%
      \relax%
    \else%
      \setlength{\unitlength}{\unitlength * \real{\svgscale}}%
    \fi%
  \else%
    \setlength{\unitlength}{\svgwidth}%
  \fi%
  \global\let\svgwidth\undefined%
  \global\let\svgscale\undefined%
  \makeatother%
  \begin{picture}(1,0.5)%
    \lineheight{1}%
    \setlength\tabcolsep{0pt}%
    \put(0,0){\includegraphics[width=\unitlength,page=1]{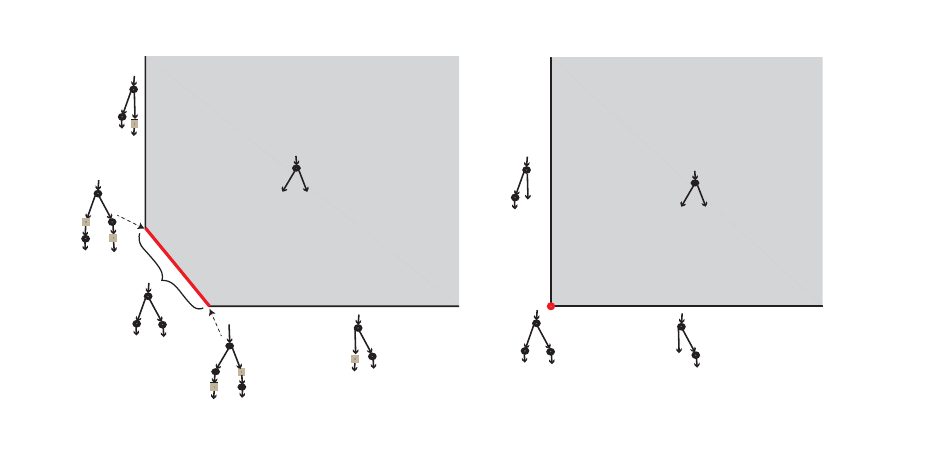}}%
    \put(0.71338582,0.0448819){\makebox(0,0)[lt]{\lineheight{1.25}\smash{\begin{tabular}[t]{l}$\bR_+^2$\end{tabular}}}}%
    \put(0.26771652,0.04330709){\makebox(0,0)[lt]{\lineheight{1.25}\smash{\begin{tabular}[t]{l}$[\bR_+^2;\cR]$\end{tabular}}}}%
  \end{picture}%
\endgroup%

    \vspace*{-1cm}
    \caption{Corner blowup of $\bR_+^2$ corresponding to the tree with 3 vertices.}
    \label{fig:blowup2}
\end{figure}

\end{ex}

 \subsubsection{Base space for disconnected domains}\label{subsec:base-for-disconnected} We now construct the base space for leveled buildings with disconnected domains, where each component has a unique incoming vertex. While the base space for disconnected Pardon buildings is simply given by the product of base space for (connected) Pardon buildings, the construction of $\cBS$ is more subtle because the level structure is defined for the whole forest, not each tree separately. The first lemma reduces this to the connected setting.

\begin{lemma}\label{lem:from-disconnected-to-connected}
    There exists a functor $\Phi$ from the category $\scS^\bullet$ of forests with $k$ components, each with one incoming edge, to the full subcategory $\scS_{k,\gamma_0}\sub \scS$ of decorated trees with one incoming edge labeled by a fixed Reeb orbit $\gamma_0$, so that the root vertex has energy zero and $k$ outgoing edges.
\end{lemma}

\begin{proof} The category $\scS^\bullet$ splits into a disjoint union of full subcategories, indexed by the Reeb orbits at the incoming exterior edges. Fix one such subcategory $\scS^\bullet_{\lc{\gamma}}$, labeled by Reeb orbits $\gamma_1,\dots,\gamma_k$. Let $T_0$ be the decorated corolla with zero energy, one incoming exterior edge, labeled by an arbitrary Reeb orbit $\gamma_0$, and $k$ outgoing exterior edges, labeled by $\gamma_1,\dots,\gamma_k$. Then, the functor 
\begin{gather*}
     \scS^\bullet_{\lc{\gamma}}\;\to\; \scS_{k,\gamma_0}\\
     T = T_ 1\sqcup \dots \sqcup T_k \;\mapsto\; T_0 \#_{i} T_i
\end{gather*}
is well-defined and an isomorphism by inspection.
\end{proof}

Intuitively, one can think of the isomorphism by introducing a \emph{ghost vertex} to which all incoming exterior edges of the components of the forest are connected. We write $e_j$ for the edge connecting $T_j$ to the ghost vertex in the construction of Lemma~\ref{lem:from-disconnected-to-connected}.
The functor $\Phi$ can be upgraded to a functor 
\begin{equation}\label{eq:upgraded-functor}
    P \cl \scS^\bullet_k\times (\{1,\dots,k\},\super)\to \scS
\end{equation}
by letting $P(T,I)$ be the tree obtained from $\Phi(T)$ by contracting all edges $e_j$ with $j \notin I$.\\

To make the blow-up construction concrete, fix Reeb orbits $\Gamma^+$ and $\Gamma^-$ and let $\Lambda\cl \Gamma^-\to \Gamma^+$ be a function. It induces a partition of $\Gamma^-$ into the sets $\Lambda_\gamma := \Lambda\inv(\{\gamma\})$. Fix for each $\gamma\in \Gamma^+$ a corolla $T_\gamma$ with a unique positive edge labeled by $\gamma$ and negative edges labeled by $\Lambda_\gamma$. Suppose the objects in the category $\scS_{/T_\gamma}$ come equipped with a function $d_\gamma \cl V(T'_\gamma)\to \bZ_{> 0}$ that is additive under contractions of edges. Set $d_\gamma \coloneqq d_\gamma(T_\gamma)$. Then, the product 
$$\cBR_\Lambda \coloneq \p{\gamma\in\Gamma^+}{\cBR_{\gamma,\Lambda_\gamma}(d_\gamma)}$$ 
is going to be the base space for the construction of the global Kuranishi chart for Pardon buildings with disconnected domains stratified by forests that contract onto the forest $\sqcup_\gamma T_\gamma$. Define the category $\wt\scS$ of leveled decorated forests similarly.

\begin{construction}\label{con:disconnected-leveled-base} Fix sequences $\Gamma^\pm$ of Reeb orbits and a partition $\Lambda\cl \Gamma^+\to \Gamma^-$. Then, we define the base space $\cBS_{\Lambda}$ as follows. Equip $[0,1)^{|\Gamma^+|}\times \cBR_\Lambda$ with the product stratification. Let $U \sub \cBR_\Lambda$ be a neighborhood of the strata of codimension at least $2$, so that each hyperplane in $U$ is embedded. The strata of $U$ are indexed by a decorated forest $T$ and a subset $I\sub \{1,\dots,k\}$, where $k = |\Gamma^+|$. Then, the monoid $\sigma_{T,I}$ associated to the stratum $S_{T,I}$ of $U$ is given by 
$$\sigma_{T,I} \coloneqq \bigoplus_{e\in E(T)} \bN e \;\oplus\;\bigoplus\limits_{i \in I}\,\bN e_i \;=\bigoplus_{e\in E(P(T,I))} \bN e.$$
The second equality holds by the definition of $P(T,I)$. We may now apply the algorithm of Proposition~\ref{prop:maximally-leveled-tree} along with Definition~\ref{de:leveled-monoid} for $\sigma_{T,I} = \sigma_{P(T,I)}$ to obtain a refinement $\wt\cR$ of the monoidal complex of $[0,1)^{|\Gamma^+|}\times U$. Then, we define $\wt U$ to be the pushout
\begin{equation*}\label{eq:disconnected-refinement}\begin{tikzcd}
			(U\sm  S_2(\cBR_{\lc{d}}))\times [0,1)^{|\Gamma^+|} \arrow[r,""] \arrow[d,""]&{[}U;\wt\cR{]} \arrow[d,""]\\ 
            {[0,1)^{|\Gamma^+|}} \times \cBR_\Lambda\arrow[r,""] & \wt U\end{tikzcd}\end{equation*}
           where the upper horizontal map is the inverse of the blow-down map, restricted to the complement of the blown-up locus. This admits a canonical smooth structure. Moreover, $\wt U$ is equipped with a canonical smooth map $\wt\beta\cl \wt U \to [0,1)^{|\Gamma^+|}$. We define the \emph{base space of disconnected buildings}
       \begin{equation}\label{eq:disconnected-base}
           \cBS_{\Lambda} \coloneqq \wt\beta\inv(\{0\}^{|\Gamma^+|})
       \end{equation}
       to be the preimage of the stratum of highest codimension in $[0,1)^{|\Gamma^+|}$.
\end{construction}

\begin{lemma}\label{lem:disconnected-base-well-defined}
    $\cBS_{\Lambda}$ is a smooth manifold with corners of dimension $\dim (\cBR_{\Lambda})+|\Gamma^+|-1$. 
\end{lemma}

\begin{proof}
By the universal property of the generalized blow-up, we can do the blow-up iteratively. Another way to see this is that the refinement we obtain can be realized as a sequence of iterated star subdivisions; then it follows from \cite[Corollary 7.3]{KM15}. By Construction~\ref{con:disconnected-leveled-base}, the first step is to blow-up $[0,1)^{|\Gamma^+|}$ to obtain a space $Y$, while the second step is to blow-up $Y\times\cBR_\Lambda$ according to the refinement described above. Using the explicit description in \cite[\textsection3]{KM15}, we see that the preimage $Z\sub Y$ of $\{0\}^{|\Gamma^+|}\sub [0,1)^{|\Gamma^+|}$ under the blow-down map is a codimension-$1$ face of $Y$ and thus a manifold with corners. Then, $\cBS_{\lc{d}}$ is the preimage of $Z$ under the generalized blow-up, or equivalently, the blow-up of $Z\times \cBR_\Lambda$ according to the given refinements. Thus, it is a smooth manifold with corners by \cite[Theorem~A]{KM15} and has dimension $\dim (\cBR_\Lambda)+|\Gamma^+|-1$.
\end{proof}

\subsection{Families of buildings}
Let $(Y,\lambda)$ be a nondegenerate closed contact manifold, and fix a $\lambda$-adapted almost complex structure $J$ on the symplectization $\wh Y$ of $Y$. We now explain how to obtain \emph{framings} of our punctured domain curves. This will lead to the definition of the (infinite-dimensional) family of buildings, Definition~\ref{de:family-of-tree}, out of which we cut our thickening by a perturbed Cauchy--Riemann equation. 

\subsubsection{Framings of buildings in symplectizations}\label{sssec:framings_of_buildings}  

\begin{definition}\label{de:approximation-in-symplectisation} 
Let $\cF$ be a finite set of Reeb orbits. We call $\wt\lambda\in \Omega^1(Y)$ an \emph{$\cF$-integral approximation} of $\lambda$ if 
	\begin{enumerate}[label = (\roman*),leftmargin=25pt,ref=\roman*]
		\item\label{i:contact} $\wt\lambda \wedge (d\wt\lambda)^n > 0$, 
		\item\label{i:reeb-orbits} $\int_{S^1}\wt\gamma^*\wt\lambda \in \bN$ for any $\gamma\in \cF$,
		\item\label{i:positive} For any subsets $\Gamma^+,\Gamma^-\sub \cF$, we have
        \begin{equation}
            \s{\gamma\in \Gamma^+}\cA_\lambda(\gamma)-\s{\gamma\in \Gamma^-}\cA_\lambda(\gamma) > 0 \qquad \rimp \qquad \s{\gamma\in \Gamma^+}\cA_{\wt\lambda}(\gamma)-\s{\gamma\in \Gamma^-}\cA_{\wt\lambda}(\gamma) > 0
        \end{equation}
	\end{enumerate}
\end{definition}

\begin{lemma}\label{lem:contact-approximation} There exists an $\cF$-integral approximation of $\lambda$ for any finite set $\cF\sub \cP(\lambda)$.
\end{lemma}

\begin{proof} Due to the finiteness of $\cF$ and compactness of $S^1$, all conditions are open save for the requirement that $\int_{S^1}\wt\gamma^*\wt\lambda\in \bN$. Thus there exists $\lambda'$ satisfying the other conditions and with $\int_{S^1}\wt\gamma^*\lambda'\in \bQ$ for any $\gamma\in \cF$. We can ensure this by modifying $\lambda$ separately in disjoint neighborhoods of all the orbits $\gamma \in \cF$.  Multiplying $\lambda'$ with a sufficiently high positive integer, we obtain the desired $\cF$-integral approximation $\wt\lambda$ of $\lambda$.
\end{proof}

Given a decorated tree $T$ we can thus fix a $\cP(T)$-integral approximation $\wt\lambda$ of $\lambda$, where $\cP(T)$ was defined in Equation~\eqref{eq:tree-reeb-orbits}, and an integer $p \geq 3$. Given a $J$-holomorphic building $(u,C,z_*,m_*)$ with underlying dual graph $T_u$ that contracts to $T$, we will now construct the line bundle that will serve as our ``reference line bundle" later on. 
For $v \in V(T_u)$, set 
\begin{equation}\label{eq:linebundle-construction}
	L_{u,v} \coloneqq \cO_{C_v}\lbr{\s{e\in E_v^+}{\int_{S^1}\gamma_e^*\wt\lambda \; z_{v,e}} - \s{e\in E_v^-}{\int_{S^1}\gamma_e^*\wt\lambda\; z_{v,e}}}^{\otimes p}.
\end{equation}
By Stokes' Theorem, 
\begin{equation}\label{eq:positive-degree}
\deg(\omega_{C_v}(D_v)\otimes L_v) = -2 + |D_v|+ p\int_{\dot{C}_v}u_v^*d\wt \lambda\;> 0 
\end{equation}
\noindent In particular, we obtain a very ample holomorphic line bundle 
\begin{equation}
    \wt{\fL}_u := \omega_{\wt C}(D)\,\otimes \wt{L}_u
\end{equation}
on the normalization $\wt C $ of $C$, where $D$ is the divisor of special points on $\wt C$. Since $C$ has genus $0$, the bundle $\wt L_u$ descends to a unique (up to holomorphic isomorphism) holomorphic line bundle $\fL_u$ on $C$. A basis of the global holomorphic sections of $\fL_u$ defines an automorphism-free stable map $\iota \cl C\to \bP^d$ so that $H^1(C,\iota^*\cO(1)) = 0$.
In particular, if the positive/negative exterior edges of $T_u$ are labeled by $\Gamma^\pm$, then $[\iota]\in \cB_{\Gamma^+,\Gamma^-}(d)$, as defined in \textsection\ref{subsec:base-space}.

Fix now sequences $\Gamma^\pm$ of Reeb orbits and a relative homology class $\beta\in H_2(Y,\Gamma^+\sqcup\Gamma^-)$. Let $T$ be the corolla with degree $\beta$ and positive/negative exterior edges labeled by $\Gamma^\pm$. Fix $L > \cA_\lambda(\Gamma^+),\cA_\lambda(\Gamma^-)$. Then, choose a $\cP(T)$-integral approximation $\wt\lambda$ and let $d$ be the integer of Definition~\ref{de:auxiliary-datum}. Let $\cBR = \cBR_{\Gamma^+\sqcup \Gamma^-}(d)$ be the space defined in \textsection\ref{subsec:base-space}. Recall that it is a principal torus bundle over the real-oriented blow-up of $\cB_{0,\Gamma^+\sqcup \Gamma^-}(d)$. Let $\cC\to \cBR$ be the pullback of the universal family of $\cB_{0,\Gamma^+\sqcup \Gamma^-}(d)$.
Due to the action bound, we may fix $\kappa_L > 0$ sufficiently small so that non-trivial cylinders have $\lambda$-energy at least $2\kappa_L$. 

\begin{definition}\label{de:family-of-tree}
We define $\cZ = \cZ_{\wt\lambda}(T)$ be the space of tuples $(\varphi,u)$ where 
\begin{enumerate}[label=\roman*),leftmargin=20pt,ref=\roman*]
    \item $\varphi\in \cBR$ lies in the stratum associated to a tree $T_\varphi$ admitting a contraction $T_\varphi\to T$,
    \item $u = ([u_v])_{v\in V(T_\varphi)}$ is a collection of equivalence classes of smooth maps $u_v \cl \dot{C}_v\to \wh Y$ up to translation, where
	\begin{itemize}[leftmargin=15pt]
    \setlength\itemsep{2.5pt}
	\item $u_v$ is $J$-holomorphic near the punctures of $\dot{C}_v$,
	\item if $x \in C_v$ is a positive/negative puncture (mapping to a node) of type $1$, then $u_v$ is positively/negatively asymptotic to the Reeb orbit $\gamma_e$, where $e\in E(T')$ is the associated edge;
    \item the matching isomorphism of $\varphi$ at the edge $e =(v,v')$ intertwines $(\wh u_v)_{z_{v,e}}$ and $(\wh u_{v'})_{z_{v',e}}$,
    \item $\int_{C_v}{u_v}^*_Yd\lambda \geq 0$,
	\item if $C_v$ is unstable, then $\int_{C_v}u^*_Yd\lambda \geq \kappa_L$.
	\end{itemize}
    \end{enumerate}
\end{definition}

\noindent We assume $\cBR$ and $\cC$ are equipped with $G$-invariant metrics $d_\cB$ and $d_\cC$ respectively, the choice of which is irrelevant. Define 
\begin{equation}\label{eq:curlyCnbd}
    \cC_{\ge \epsilon} \coloneqq \{ \zeta \in \cC\mid d_\cC(\zeta,\text{Crit}(\pi)) \ge\epsilon\}
\end{equation}
for $\epsilon > 0$. 
We equip $\cZ$ with the topology generated by the following $\epsilon$-neighborhoods for $\epsilon > 0$. Given $(\varphi,u)\in \cZ$, define $\cN_\epsilon(\varphi,u)$ to be the subset of points $(\varphi',u')$ such that
\begin{itemize}[leftmargin=20pt]
\setlength\itemsep{1pt}
	\item $d_\cB(\varphi,\varphi') < \epsilon$;
	\item the (orbits of the) graphs satisfy 
    \begin{equation*}
    d_{\text{H}}\lbr{\bR^{V(T_\varphi)}\cdot\graph(\varphi,u)|_{\cC_{\ge\epsilon}},\bR^{V(T_{\varphi'})}\cdot\graph(\varphi',u')|_{\cC_{\ge\epsilon}}} < \epsilon
	\end{equation*} 
	in the Gromov-Hausdorff metric, where $T_\varphi$ is the dual graph of the fiber $\cC_\varphi$ and we choose any representatives of the classes $[u_v]$ and $[u'_{v'}]$; 
	\item for $e \in E(T_\varphi)$ with associated Reeb orbit $\gamma_e$ and corresponding node $x_e \in \cC_\varphi$ we have $$d_Y(u'_Y(z),\gamma_e) < \epsilon$$ 
	for any $z \in \cC_{\varphi'}$ with $d_\cC(z,x_e) \leq \epsilon$.
\end{itemize}

\subsubsection{Determining unitary framings}\label{subsec:domain-metrics} 
Abusing notation, we also denote by $\cC \to \cZ$ the pullback of the universal family $\cC\to \cB_{\Gamma^+\sqcup \Gamma^-}(d)$. The $\cG$-action on $\cBR$ lifts to a $\cG$-action on $\cZ$. In contrast to the action on $\cBR$, which is not proper, the additional data of the building means that the lifted action is proper, which will be crucial to reduce to the action of the compact group $G$ later on.

\begin{lemma}\label{lem:palais-proper} The $\cG$-action on $\cZ$, given by 
	$$g \cdot (s,u) = (g\cdot s,u\g g\inv),$$
	is proper in the sense of Palais, \cite{Pal61}.
\end{lemma}

\begin{proof} Let $\wt \cC$ be the real blowup at the images of the sections and the nodes of the fibers of $\cC\to \cV$. Fix a smooth metric on $\wt\cC$. This restricts to a smooth metric on $\cC\inn$. Given now $(s,u)\in \cZ$, we see that the projection $u_Y$ of $u$ to $Y$ extends to a continuous function $\bar{u}_Y\cl \wt\cC_s \to Y$. Given any Riemannian metric on $Y$, we obtain an inequality of Lipschitz numbers 
	$$L(u_Y)\leq L(\bar{u}_Y).$$
	The Lipschitz number of $\bar{u}_Y$ is finite since its differential is bounded away from the ``puncture circles" and we know its behavior near the ``puncture circles". Thus, we may use the same argument as in \cite[Lemma 4.13]{AMS23} to conclude.
 \end{proof}

 \begin{lemma}\label{lem:zero-locus-right}
     Writing $\cZ_{\delbar}\sub \cZ$ for the closed subspace of $J$-holomorphic building, the quotient map 
    \begin{equation}\label{eq:footprint}
        \psi \cl \cZ_{\delbar}/\cG\to \Mbar^{\, J}(T)
    \end{equation}
    is an isomorphism of orbi-spaces.
 \end{lemma}

 \begin{proof} The surjectivity of~\eqref{eq:footprint} follows from the construction of the very ample line bundles in \textsection\ref{subsec:framings}, while injectivity (including the statement about isotropy groups) can be shown as in \cite[Discussion~3.16]{HS22}. Continuity follows from the definition of the metric on $\cZ_{\delbar}$. Thus, it remains to show that~\eqref{eq:footprint} is open or, equivalently, closed.\par
 \vspace{-4pt}
 This can be checked locally on the target. To this end, let $[u,C]\in \Mbar^{\, J}(T)$ be arbitrary.
 By \cite[Proposition~3.26]{Par19}, we can find a divisor $D\sub Y$ so that $u_Y\pf D$ and adding $u_Y\inv(D)$ as marked points stabilizes the domain $C$ of $u$, and so that this is the minimal number of marked points required to stabilize the domain.
 Then, there exists a neighborhood $U \sub \Mbar^{\, J}(T)$ of $[u,C]$ so that for any $[u',C']\in U$, the map $u'_Y$ intersects $D$ transversely and the added points ${u'}\inv(D)$ stabilize $C'$.
 This yields a continuous map 
 $$\ff\cl U \to \Mbar_{0,\#\Gamma^-+\#\Gamma^++ m}/S_m,$$ 
 where $m = \#u\inv(D)$ and the symmetric action permutes the last $m$ marked points.
 Let $\rho_1,\dots,\rho(d)$ be local sections of the universal family 
 $$\cc{\cC} := \cC_{0,\#\Gamma^-+\#\Gamma^++ m}\to \Mbar_{0,\#\Gamma^-+\#\Gamma^++ m}/S_m$$ near $p = \ff([u,C])$, whose images do not meet the nodal points of the fibers and so that for each irreducible component $C_v$ of $C$ we have 
 \begin{equation}\label{eq:number-of-sections}
 \#\{i\mid \rho_i(p) \in C_v\} = \deg(\fL_u|_{C_v}^{\otimes p}).
 \end{equation}
Shrinking $U$ if necessary, we may assume the equality in \eqref{eq:number-of-sections} holds for any point in $U$. Define the holomorphic line bundle $\cc{\cL}\coloneqq \cO_{\cc{\cC}}(\rho_1 +\dots +\rho(d))$. It pulls back to a complex orbi-line bundle over the universal family of $\Mbar^{\, J}(T)$, and has the same multi-degree as $\fL_u^{\otimes p}$ when restricted to the fiber over $[u,C]$. Therefore, as in \cite[Lemma~4.8]{HS22}, the forgetful map $\cZ_{\delbar}\to \Mbar^{\, J}(T)$ is locally the projectivization of a continuous orbi-bundle. In particular, the map is closed.
\end{proof}

\begin{remark}
    The above proof strongly relies on the fact that our curves have genus zero. 
\end{remark}

By \cite[Lemma~4.4]{HS22}, the $\PGL_{d+1}(\bC)$-action on $\cB_{n}(d)$ is proper when restricted to the locus $\cB^{\stb}_{n}(d)\sub\cB_{n+ \Gamma^+,\Gamma^-}(d)$ of curves with stable domain (for any $n \ge 0$). Therefore, \cite[Corollary~4.6]{HS22} asserts the existence of a $\PGL_{d+1}(\bC)$-invariant map 
\begin{equation}\label{eq:reducing-structure-group}
    \zeta\cl \cB^{\stb}_{3d'}(d)\to \PGL_{d+1}(\bC)/\PU(d+1)\cong \fp\fu_{d+1}.
\end{equation}
where the isomorphism is induced by the polar decomposition, \cite[Lemma~3.8(iii)]{HS22}. By averaging, we can choose $\zeta$ to be invariant under the $S_{3d'}$-action on $\cB_{3d'+ \Gamma^+,\Gamma^-}^{\stb}(d)$ given by permuting the marked points labeled by $\{1,\dots,3d'\}$.
Lemma~\ref{lem:zero-locus-right} and \cite[Proposition~3.26]{Par19}, refined as in the proof of \cite[Lemma~4.3]{HS22}, imply that we can find a finite set $\{D_i,r_i\}$ of compact codimension-$2$ submanifolds with boundary $D_i \sub Y$ so that the sets 
$$U_i := \set{(\varphi,u) \in \cZ\mid u\pf D_i,\,\forall C'\sub \cC|_\varphi: \# C'\cap u\inv(D_i) = 3\deg(L_u|_{C'}),\, (\cC|_\varphi,u\inv(D_i)) \text{ is stable}}$$
form an open cover of $\cZ_{\delbar}$. We replace $\cZ$ with $\union{i}{U_i}$ without further mention. Since $\cZ$ is metrizable, there exists a $\cG$-invariant partition of unity $\{\chi_i\}\iI$ subordinate to $\{U_i\}\iI$ by Lemma \ref{lem:palais-proper} and \cite[Corollary~4.12]{AMS23}. Let 
$$\Phi_i \cl U_i \to \cB_{3d'}^{\stb}(d)/S_{3d'}$$ 
be given by forgetting the marked points labeled by $\Gamma^\pm$, adding the intersections as marked points, mapping to $\cBR_{3d'}(d)/S_{3d'}$ and then applying the blow-down map. Finally, define
\begin{equation}\label{}
    \zeta_\cU \cl \cZ\,\to\, \fp\fu_{d+1} 
\end{equation}
by 
\begin{equation}\label{eq:construction-of-slice-map}
\zeta_\cU(\varphi,u) = \s{i}{\chi_i(\varphi,u)\,\zeta(\Phi_i(\varphi,u))}.\end{equation}

In \cite{HS22}, the collection $\{(U_i,D_i,\chi_i)\}_i$ was called a good covering; see Definition~3.10 op. cit. While the above discussion shows that we can do the same in our setting, this definition is too restrictive for the inductive construction in \textsection\ref{subsec:unstructured-flow-cat}.
A prototype of the issue one faces is that once we fix divisors for one-leveled buildings, we need to find new divisors for two-leveled buildings such that the restriction to the one leveled components has the same intersection combinatorics with the already chosen divisors.
Thus, we replace it with the following variation.

\begin{definition}\label{de:good-covering}
    A \emph{good covering} of $\cZ_{\delbar}$ consists of 
    \begin{enumerate}[label=\roman*),leftmargin=20pt,ref=\roman*]
        \item a finite collection $\{U_i\}_i$ of $\PGL_{d+1}(\bC)$-invariant open subsets of $\cZ$ that cover $\cZ_{\delbar}$,
        \item for each $i$ smooth $\cG$-equivariant sections $\sigma_{i,j}\cl U_i \to \cC\inn|_{U_i}$ for $1 \leq j \le d+2$ together with divisors $D_{ij}\sub Y$ so that 
        \begin{itemize}[leftmargin=5pt]
            \item $\sigma_{i,j}$ and $\sigma_{i,j'}$ have disjoint images for $j \neq j'$
            \item for each $j$ we have $u(\sigma_{i,j}(\varphi,u)) \in D_{i,j}$ and $u\pf D_{i,j}$ near $\sigma_{i,j}(\varphi,u)$ for any $(\varphi,u)\in U_i$,
        \end{itemize}
        \item a continuous $\cG$-invariant function $\chi_i \cl \cZ \to [0,1]$ with support in $U_i$
    \end{enumerate}
     so that $\s{i}{\chi_i}$ is positive on $\cZ_{\delbar}$.
\end{definition}

\begin{remark}
    Note that the divisors need not be distinct, i.e., in the construction above, we can take $D_{i,j} = D_i$ for any $j$. Thus the existence of a good covering follows from the discussion above Definition \ref{de:good-covering}.
\end{remark}

The discussion before Definition~\ref{de:good-covering} carries over to the definition of good coverings and yields the following statement.

\begin{lemma}\label{lem:map-to-lie-algebra}
    A good covering $\cU$ together with the map in \eqref{eq:reducing-structure-group} determines a $\cG$-invariant map $\zeta_\cU\cl \cZ \to \fp\fu_{d+1}$.\qed
\end{lemma}

\subsection{Kuranishi charts for buildings in symplectizations}\label{subsec:construction}
We restate our main theorems more precisely here and prove them in the following subsections. Let $\Gamma^+,\Gamma^-$ be finite collections of Reeb orbits of action $\leq  L$, and let $\beta\in H_2(Y,\Gamma^+\sqcup \Gamma^-)$ be a relative homology class. We define $T$ to be the decorated corolla as in Definition \ref{de:decorated-tree} with positive/negative exterior edges labeled by $\Gamma^+$ and $\Gamma^-$, respectively, and with degree $\beta$. We write $\Mbar^{\,J}(T) = \Mbar^{\,J}(\Gamma^+,\Gamma^-;\beta)$.

\begin{definition}\label{de:pre-perturbation} A \emph{pre-perturbation datum} $\fD = (\wt\lambda,\conn,p)$ for $\Mbar^J(T)$ consists of 
	\begin{itemize}[leftmargin=20pt]
		\item a $\cP(T)$-integral approximation $\wt\lambda$ of $\lambda$ as in Definition~\ref{de:approximation-in-symplectisation};
		\item a translation-invariant $J$-linear connection $\conn$ on $T\wh Y$;
		\item an integer $p\gg 1$.
        \end{itemize}
\end{definition}       

\noindent To such a pre-perturbation datum, we can associate the following spaces. Set
\begin{equation}\label{eq:auxiliary-degree}
    d' \coloneqq p\lbr{\s{\gamma\in \Gamma^+}\cA_{\wt\lambda}(\gamma)-\s{\gamma\in \Gamma^-}\cA_{\wt\lambda}(\gamma)}
\end{equation}
and $d \coloneqq d'-2$. Let $\cBR \coloneqq \cBR_{\Gamma^+,\Gamma^-}(d)$ be the smooth manifold with corners defined in \textsection\ref{subsec:framings} and define the groups
\begin{equation}\label{eq:covering-groups}
    G  \coloneqq \PU(d+1) \qquad \qquad \cG \coloneqq \PGL_{d+1}(\bC). 
\end{equation}
We let $\cZ =\cZ_{\wt\lambda}(T)$ be the family of curves over $\cBR$ defined in Definition~\ref{de:family-of-tree}. As before, let $\cC \to \cZ$ be the pullback of the universal family of $\cBR$.

\begin{definition}\label{de:auxiliary-datum} A \emph{perturbation datum} $\alpha = (\fD,\cU,\zeta,E,\mu)$ extending $\fD$ is the data of
        \begin{itemize}[leftmargin=20pt]
        \item a good covering $\cU = \{(U_i,\sigma_i,\chi_i)\}_{i\in I}$ on a subset of $\cZ$ and a $\cG$-equivariant map 
        \[\zeta \cl \cB^{\text{st}}_{3d'}(d)/S_{3d'}\to \cG/G; \]
		\item a finite-dimensional $G$-representation $E$ equipped with an equivariant linear map $\mu \cl E\to \cV$, where
        \begin{equation}\label{de:full-perturbation-space}
            \cV \coloneqq \set{\eta\in C^\infty(\cC^{\circ}\times \wh Y,\cc{\Hom}_\bC(\pr_1^*T_{\cC^{\circ}/\cBR},\pr_2^*T\wh Y))^\bR\mid \text{ supp}(\eta)/\bR \text{ is compact}},
        \end{equation} 
        $\cC^{\inn}$ denoting the complement of the special points of the fibers, so that for any $(\varphi,u)\in \cZ_{\delbar}$ with $\zeta_\cU(\varphi,u) = 0$, the Cauchy--Riemann operator 
    \[D\delbar_{J}(u) +\mu(\cdot)|_{\graph(\varphi,u)} \cl C^\infty_c(\dot{C},u^*T\wh Y)^\bR\oplus E_k \to \Omega^{0,1}_c(\dot{C},u^*T\wh Y)\]
    is surjective, where $\zeta_\cU$ is the map given by Lemma~\ref{lem:map-to-lie-algebra}.\end{itemize}
\end{definition}

\begin{theorem}\label{thm:pardon-gkc} Let $T$ be a decorated tree as at the beginning of the subsection.
	\begin{enumerate}[\normalfont 1),leftmargin=15pt,ref=\arabic*]
		\item\label{gkc-unobstructed-aux} Any pre-perturbation datum $\fD$ can be completed to a perturbation datum $\alpha = (\fD,\cU,\zeta,E.\mu)$ for $\Mbar^{\,J}(T)$ as in Definition~\ref{de:auxiliary-datum}.
		\item\label{gkc-rel-smooth} If $\alpha$ is a perturbation datum, then Construction~\ref{con:thickening} and Definition~\ref{con:obstruction} yield a rel--$C^1$ global Kuranishi chart with corners for $\Mbar^J(T)$.
		\item\label{gkc-orientation} If $\Gamma^\pm$ consists of good Reeb orbits, there exists a canonical isomorphism 
        \begin{equation}
\fo_{\cK_\alpha}\,\cong\,\fo(\bR)\dul\otimes\bigotimes\limits_{\gamma\in \Gamma^+}\fo_{\gamma}\otimes\bigotimes\limits_{\gamma\in \Gamma^-}\fo\dul_{\gamma}
        \end{equation} 
        of orientation lines, where $\fo_\gamma$ is the orientation line of the Reeb orbit, defined in \textsection\ref{subsec:orientation}.
	\end{enumerate}
\end{theorem}

Before giving the construction of the thickening, Construction~\ref{con:thickening}, we have to briefly discuss Cauchy-Riemann operators on punctured Riemann surfaces. More details can be found in \textsection\ref{sec:gluing} or \cite[\textsection7]{Wen16}. Fix $k \geq 4$ and $0 < \delta < 1$ so that 
\begin{equation}\label{eq:small-delta}
    \delta< \mini{\gamma\in \cP(T)}{\inf|\sigma(A_\gamma)|}
\end{equation}
This is well-defined since $\cP(T)$ is finite and $\lambda$ is nondegenerate. For our purposes here, it suffices to consider the linearized Cauchy--Riemann equation without variation of the domain. Fix thus a possibly nodal Riemann surface $(C,\fj)$ with underlying graph contracting to $T$. Choose also a complex linear translation-invariant connection $\conn$ on $T\wh Y$ and a Riemannian metric on $\wh Y$. Then, we can associate to any smooth map $u \cl \dot{C}\to \wh Y$, which is $J$-holomorphic and asymptotic to a trivial cylinder near the punctures, the operator 
\begin{equation}
    D^{\conn}_u \cl W^{k,2,\delta}(C,u^*T\wh Y)\to W^{k-1,2,\delta}(\wt C,\Omega^{0,1}_{\wt C}\otimes_\bC u^*T\wh Y)
\end{equation}
given by the derivative of 
\begin{equation}
    \scF_u(\xi) = \Phi_{\exp_u(\xi) \to u}(d\exp_u(\xi)^{0,1}_{J,\fj}).
\end{equation}
where $\Phi$ is the parallel transport along sufficiently short geodesics. The operator $D_u$ is independent of the choice of connection and metric if $u$ is $J$-holomorphic. Since the exponential weight $\delta$ satisfies \eqref{eq:small-delta}, we have that $D_u$ is Fredholm by \cite[Lemma~7.10]{Wen16}
with index 
\begin{equation}
    \indo(D_u^\conn) = n\chi(C) -n(|\Gamma^+|+|\Gamma^-|) + 2c_1^\tau(u^*T\wh Y) + \s{\gamma\in \Gamma^+}{\mu^\tau_{\text{CZ}}(\gamma)}- \s{\gamma\in \Gamma^-}{\mu^\tau_{\text{CZ}}(\gamma)}.
\end{equation}

\begin{construction}[Thickening]\label{con:thickening} The thickening $\cT = \cT_\alpha$ consists of tuples $(\varphi,u,w)\in \cZ\times E_k$ where
\begin{enumerate}[label = \roman*),leftmargin=20pt,ref= \roman*]
    \item the image $\wch\varphi$ of $\varphi$ under the blow-down map $\cBR\to \cB$ satisfies 
		\begin{equation}
		    d_v \coloneqq\deg(\wch\varphi_v) =|D_v|-2+ p\,\lbr{\s{e' \in E^{\normalfont\text{int},+}_v}{\int\gamma_{v,e'}^*\wt\lambda}-\s{e' \in E^{\normalfont\text{int},+}_v}{\int\gamma_{v,e'}^*\wt\lambda}};,
		\end{equation} 
	for each $v\in V(T_\varphi)$, where $D_v$ is the divisor of special points on $C_v$ and $\gamma_{v,e}$ is the Reeb orbit to which $u_v$ converges at the puncture $z_{v,e}$;
    \item the matching isomorphism $m_e$ associated to $\varphi$ at the edge $e = (v,v')$ intertwines $u_v$ and $u_{v'}$ in the sense that $\wh u_{v'}\g m_{(v,v')} = \wh u_v$;
    \item the perturbation $w$ satisfies
    \begin{equation}
			\hpd_J\, u_v + \mu_{k}(w) |_{\graph(\varphi_v,u_v)}  = 0
		\end{equation}
		for each $v \in V(T_\varphi)$\footnote{By the assumption that elements of $\cV$ are invariant under translation, this equation is well-defined, i.e., does not depend on the choice of representative $u_v$.}, and the map
		\begin{equation}\label{eq:perturbations-suffice}
		     E \to \coker(D_u^\conn): w \mapsto [\mu_{k}(w) |_{\graph(\varphi,u)}]
		\end{equation}
		is surjective.
\end{enumerate}  
\end{construction}

We take the covering group to be 
\begin{equation}
    \wh G := G\times \p{\gamma\in \Gamma^+\sqcup\Gamma^-}S^1,
\end{equation}
where the torus acts by rotating the respective asymptotic marker. The thickening $\cT$ admits a continuous $\wh G$-equivariant map $\Pi\cl\cT\to \cBR$. In general, $\cT$ is not compact and $\Pi$ is not surjective. 

\begin{definition}[Obstruction bundle and section]\label{con:obstruction} We define the obstruction bundle $\cE = \cE_\alpha$ to be the trivial bundle 
	\begin{equation}
		\cE\coloneqq E\, \oplus\, \fp\fu 
	\end{equation}
	and define the (pre-)obstruction section $\fs_{pre} \cl \cT \to \cE$ by 
	\begin{equation}\label{} \
		\fs_{pre}(\varphi,u,w) = (w,\zeta_\cU(\varphi,u)). 
	\end{equation}
\end{definition}

\begin{remark}\label{} The projection of $\fs_{pre}$ to $\fp\fu$ is only continuous. We will replace this part of the obstruction section by an equivariant section with the same zero locus that is of class rel--$C^1$, see Lemma~\ref{lem:better-obstruction-section}.\end{remark}

\begin{proof}[Proof of Theorem \ref{thm:pardon-gkc}\hspace{0.2pt}\eqref{gkc-unobstructed-aux}]
Given the work done in \textsection\ref{sec:prelim-aux}, it remains to find a perturbation space $(E,\mu)$. Let us summarize why. By Lemma~\ref{lem:contact-approximation}, we can find a $\cP(T)$-integral approximation $\wt\lambda$, while the existence of a translation-invariant $J$-linear connection $\wh \conn$ on $T\wh Y$ is immediate. Since our curves have genus zero, any integer $p \gg 3$ is sufficiently large for $\fL_u := \eqref{eq:linebundle-construction}$ to be very ample for any $(\varphi,u)\in\cZ$ close to $\cZ_{\delbar = 0}$ and $H^1(C,\fL_u) = 0$. Good coverings, as in Definition~\ref{de:good-covering}, were constructed in \textsection\ref{subsec:domain-metrics}. To obtain the existence of $(E,\mu)$, we recall from \cite[Definition~4.1]{AMS23} that a \emph{finite-dimensional approximation scheme} of a smooth $G$-vector bundle $V\to B$ over a smooth $G$-manifold is a sequence $(E_k,\iota_k)$ of finite-dimensional $G$-representations with $G$-equivariant linear maps $\iota_k \cl E_k \to C^\infty_c(B,V)$ so that
 \begin{itemize}[leftmargin=20pt]
    \setlength\itemsep{2pt}
     \item $E_k \sub E_{k+1}$ is a sub-representation with $\iota_{k+1}|_{E_k} = \iota_k$,
     \item $\union{k \geq 1}{\im(\iota_k)}$ is dense in $C^\infty_c(B,V)$ in the $C^\infty_{\text{loc}}$-topology.
 \end{itemize}

\begin{lemma}\label{lem:invariant-approximations-exist}
    Finite-dimensional approximation schemes of $\cV$ exist.
\end{lemma}

\begin{proof}
    Apply \cite[Lemma~4.2]{AMS23} to the manifold $B \coloneqq \cC\inn\times \wh Y/\bR$ and the vector bundle $V \coloneqq \cc{\Hom}_\bC(\pr_1^*T_{\cC^{\circ}/\cBR},\bR\oplus\pr_2^*T(\wh Y/\bR))$.
\end{proof}

\begin{lemma}\label{lem:transversality-pointwise} Let $y = (\varphi,u) \in \cZ$ be arbitrary. Then there exists $k_y\geq 0$ so that 
	\begin{equation}
		E_{k_y} \to \coker(D_u) : w\mapsto [\mu_{k}(w) |_{\graph(\varphi,u)}]
	\end{equation}
	is surjective.
\end{lemma}

\begin{proof} This follows from the facts that $D_u$ is Fredholm and that an element in the cokernel is identically zero if it vanishes on an open subset of $C$. See also \cite[Proposition~3.26]{Par19}.
\end{proof}

\begin{lemma}[Openness of transversality]\label{lem:regular-locus-open} Given $y\in \cZ_\alpha$ and $k_y$ as in Lemma \ref{lem:transversality-pointwise}, there exists a neighborhood $W_y \sub \cZ$ of $y$ so that for any $y' = (\varphi',u',w')\in W_y$ the map 
	\begin{equation}
		E_{k_y} \to \coker(D_{u'}) : w\mapsto [\mu_{k}(w') |_{\graph(\varphi',u')}]
	\end{equation}
	is surjective.
\end{lemma}

\begin{proof} Over a given stratum, this follows from the fact that regularity is an open condition. To see that it is also an open condition under gluing, refer to \cite[Lemma~5.8]{Par19} and the discussion loc. cit.
\end{proof}

\noindent Define the function $\wt k \cl \cZ \to \bN$ by 
	\[\wt k(\varphi,u) = \inf\{\ell \mid E_\ell \to \coker(D_u)\text{ is surjective}\}.\]
By Lemma \ref{lem:transversality-pointwise}, the function $\wt k$ is well-defined and, by Lemma \ref{lem:regular-locus-open}, it is upper semi-continuous. Thus it achieves a maximum $k$ on the compact set $\set{(\varphi,u)\in\cZ_{\delbar}\mid \zeta_\cU(\varphi,u) = 0}$. Setting $E = E_k$, we obtain the desired perturbation space.\end{proof}

We can now prove Theorem~\ref{thm:pardon-gkc}\eqref{gkc-rel-smooth}. Let $\alpha$ be a perturbation datum. We use the same notation as above. Write $\beta \cl \cBR\to \cB$ for the forgetful map.

\begin{proposition}\label{prop:pre-thickening-smooth} The forgetful map $\cT\to \cBR$ is a $G$-equivariant $C^1_{loc}$-fiber bundle. In particular, it is a topological submersion whose fibers carry a canonical smooth structure. 
\end{proposition}

\begin{proof}This follows from Theorem~\ref{thm:gluing-main} and Proposition~\ref{prop:c1loc}.
\end{proof}

\begin{lemma}\label{lem:group-action} The action $G\times \cT/G\times \cBR\to \cT/\cBR$ is fiberwise locally linear and of class rel--$C^1$.
\end{lemma}

\begin{proof} The action map is continuous by the definition of the topology on $\cT$. The fact that it is of class rel--$C^1$ follows from the gluing result in \textsection\ref{sec:gluing}, while the fiberwise local linearity is a consequence of relative $C^1$-regularity and the same argument as in \cite[Lemma~5.9]{HS22}.
\end{proof}

\begin{lemma}\label{lem:zero-locus-correct} The forgetful map $\fs_{\text{pre}}\inv(0)\to \Mbar^J(T)$ descends to an isomorphism $$\fs_{\text{pre}}\inv(0)/G\,\cong\, \Mbar^J(T)$$ 
    of orbispaces.
\end{lemma}

\begin{proof} By Lemma~\ref{lem:zero-locus-right}, it suffices to show that $\fs_{\text{pre}}\inv(0)/G\,\cong\, \cZ_{\delbar}/\cG$. This follows from the $\cG$-equivariance of $\zeta_\cU$.
\end{proof}

The obstruction bundle $\cE$ is trivially of class rel--$C^1$; however, the obstruction section $\fs_{\text{pre}}$ is only continuous since the cut-off functions we construct in \textsection\ref{subsec:domain-metrics} are only continuous. 

\begin{lemma}\label{lem:better-obstruction-section} We can replace $\fs_{\text{pre}}$ by a $\wh G$-equivariant obstruction section $\fs$ of class rel--$C^1$ with $$\fs_{\text{pre}}\inv(0) = \fs\inv(0).$$
\end{lemma}

\begin{proof} We only have to replace the projection of $\fs_{\text{pre}}$ to the second summand. This follows from a mollification argument as in \cite[Lemma~4.55]{AMS23}.
\end{proof}

\subsection{Leveled buildings}\label{subsec:leveled-gluing} For the flow category, it is important to use leveled buildings and not Pardon buildings. Thus, we prove here that our global Kuranishi charts constructed in \textsection\ref{subsec:construction} also yield global Kuranishi charts for the buildings classically used in SFT.

\begin{theorem}\label{thm:leveled-gkc}
 Suppose $\Gamma^\pm$ are collections of Reeb orbits of total action $ < L$, $\beta$ is a relative homology class, and $\alpha$ is a perturbation datum for $\Mbar^{\, J}(\Gamma^+,\Gamma^-;\beta)$. Then $$\cK^\bR \coloneqq\cBS\times_{\cBR}\cK_\alpha$$ is a global Kuranishi chart with corners for the moduli space $\Mbar_{\sft}^{\, J}(\Gamma^+,\Gamma^-;\beta)$ of leveled SFT buildings. 
\end{theorem}

The proof relies heavily on our adaptation of the gluing theorem in \cite{Par19}, developed in \textsection\ref{sec:gluing}. We first start with a discussion of local charts in the corner blow-up $\cBS$ of $\cBR$ and then prove Proposition~\ref{prop:gluing-level} which is the key ingredient for Theorem~\ref{thm:leveled-gkc}.

\subsubsection{Local charts for $\cBS$}\label{dsc: gluing}
Let $\varphi \in \cBR$ and by abuse of notation let $\cBR|_{T_\varphi}$ denote a small neighborhood of $\varphi$ in the stratum corresponding to the tree type $T_\varphi$. Recall that the gluing results \ref{ssc:base_gluing} prove that for a small enough $\cBR|_{T_\varphi}$, there is a local chart $$g_b\cl G_{T_\varphi} \times \cBR|_{T_\varphi} \to \cBR.$$
Recall from the discussion below Theorem~\ref{thm:corner-blowup} that points in $\cBS$ correspond to points in $\cB$ along with additional data coming from the normal bundles to the corner strata that are blown up. Concretely, an element $p \in \cBS$ lying over the intersection $D_I = \cap_{i\in I} D_i$ of divisors is given by 
\begin{equation}\label{eq:lift-corner-blow-up}
p = (y, [v_1],\dots,[v_I], [v'_{1,1}\oplus\dots\oplus v'_{1,k_1}],, \dots,  [v'_{\ell,1}\dots\dots\oplus v'_{\ell,k_\ell}]) 
\end{equation}
where $y\in D_I$ and $v_i,v_i' \in (N_{D_i/\cB})_y$ are related by a positive scaling. We thus can make the following observation.

\begin{lemma}\label{lem:tangent-vectors-give-lift}
    Given $\varphi\in \cBR$ with underlying decorated tree $T$, a choice of tangent vector $\{s_{v,e}\}_{v\in V(T),e\in E^{\normalfont\text{int}}_v}$ at each marked point determines a lift $\wt\varphi\in \cBS$ of $\varphi$.\qed 
\end{lemma}

Let $(T_\varphi,\ell)$ be the underlying leveled tree of $\wt\varphi$. Then, a sufficiently small neighborhood of $ \wt\varphi$ in the fiber $\cBS|_{\varphi}$ of the blow-down map is diffeomorphic to $ \prod_{j = 1}^{\max\ell}(\Delta^{\#\ell\inv(j)-1})^{\text{int}}$ by the natural isomorphism 
\begin{equation}\label{eq:simplex-interior}
    (\Delta^{k} )^{\text{int}}\, \cong \,(\bR_{>0})^{k+1} /\bR_{>0}.
\end{equation}
The stratum carries a natural action of $(\prod_{i=1}^k\bR^{e\ell\inv(i)}_{>0},\times)$, given by scaling the $v_*'$ terms in~\eqref{eq:lift-corner-blow-up}. We will show that a local chart $g_b\cl G_{T_\varphi} \times \cBR|_{T_\varphi} \to \cBR$ centered at $\varphi$ can be lifted to give a local chart centered at $\wt\varphi$ in the leveled base space $\cBS$. To this end, let $k  = \max\ell -1 $ be the number of `gaps' between levels. The level function induces a function $el \cl E(T_\varphi) \to \bN$ by the relation 
$$e\ell(e) = \ell(v') $$
when $e =(v,v')$ with $\ell(v') > \ell(v)$.
There is a natural embedding,
$$ \varDelta\cl [0,1)^k \times \prod_{j = 1}^{\max\ell}(\Delta^{\#\ell\inv(j)-1})^{\text{int}}\hookrightarrow G_{T_\varphi/}$$  
defined by $$\varDelta(t_1,\dots,t_k,w_1,\dots,w_{k}) = (\Delta^{w_1}_{e\ell\inv(1)} (t_1) , \Delta^{w_2}_{e\ell\inv(2)} (t_2), \dots , \Delta^{w_k}_{e\ell\inv(k)} (t_k))   $$  where $\Delta^{w_i}_{e\ell\inv(i)}(x)$ is the linear embedding of the ray $[0,1) \mapsto [0,1)^{e\ell\inv(i)}$  corresponding to the point $w_i$ in the simplex $\Delta^{\# \ell\inv(i)-1}$. Here we use the extension of~\eqref{eq:simplex-interior} to an isomorphism $$\bR_+^k / \bR_{>0} \cong \Delta^{k-1}.$$ 
See Figure~\ref{fig:blowup2} for a pictorial description.


\begin{lemma}
There exists a diffeomorphism $lg_b$ making
\begin{equation}\label{dig:square-base}
\begin{tikzcd}
 {[}0,1)^k\times \prod_{j = 1}^{\max\ell}(\Delta^{\#\ell\inv(j)-1})^{\text{int}}\times \cBR|_{T_\varphi}\arrow[rr,  "lg_b"] \arrow[d,"\varDelta\times\ide"] && V^\bR \arrow[d] \\
 G_{T_\varphi/}\times  \cBR|_{T_\varphi}\quad\arrow[rr, "g_b"] && \cBR 
\end{tikzcd}
\end{equation}
into a Cartesian square, up to shrinking $\cBR|_{T_{\varphi}}$.
\end{lemma}

\begin{proof}
    This follows from the explicit local description of the generalized blow-up in \cite[\textsection3]{KM15} and our identification of the fiber $\cBS|_{\varphi}$ with the product of interiors of simplices.
\end{proof}

\subsubsection{Leveled thickening}\label{leveled-thickening}
Fix two collections $\Gamma^\pm$ of Reeb orbits of total action $< L$ and let $\alpha$ be a perturbation datum for the moduli space $\Mbar^{\, J}(\Gamma^+,\Gamma^-;\beta)$ of Pardon buildings of degree $\beta$. Let $\cK = \cK_\alpha$ be the associated global Kuranishi chart constructed in \textsection\ref{subsec:construction} and recall that its base space $\cBR$ is a torus bundle over the blow-up of a subset of regular stable maps in some complex projective space. By \textsection\ref{subsec:base-with-levels}, there exists a further generalized blow-up $\cBS\to \cBR$ based on refining decorated trees to decorated trees with levels.
We now use $\alpha$ to construct a `leveled thickening' $\cTR$ and will show that it is homeomorphic to the pullback of $\cT$ along $\cBS \to \cBR$, see Corollary~\ref{cor:map-to-leveled-base}. 

\begin{definition}\label{de:leveled-thickening}
    The \emph{leveled thickening} $\cTR \coloneq \cc{\cT}^\bR(T)$ is given set-wise by the disjoint union 
    \begin{equation}\label{eq:leveled-thickening}
        \cTR \coloneqq \djun{T'_{\ell'}\to T_{\ell}}\modulo{\cTR(T'_{\ell'})}{\Aut(T'_{\ell'}/T_{\ell})}
    \end{equation}
    where $\cTR(T'_{\ell'})$ is the quotient $\cTR(T'_{\ell'}) = \wt\cTR(T'_{\ell'})/\bR^{\max\ell'}$ of the space $\wt\cTR(T'_{\ell'}) = \{(\varphi,u,w)\}$, where 
    \begin{itemize}[leftmargin=20pt]
        \item $\varphi \in \cBR|_{T'}$,
        \item $u = (u_v)_{v\in V(T'_{\ell'})}$ is a sequence of maps, where 
        \begin{itemize}[leftmargin=15pt]
            \item for $v \in V_{nt}(T'_{\ell'}) = V(T') $ 
        $$u_v \cl (\dot{\cC}|_{\varphi})_v \to \bR\times Y$$ 
        is a smooth map representing $\beta_v$, which is positively/negatively asymptotic near the puncture $z_{v,e}\in (\cC|_{\varphi})_v$ to the trivial cylinder over $\gamma_e$, while
        \item for $v \in V_t(T'_{\ell'})$, $u_v$ is the trivial cylinder over the associated Reeb orbit;
        \end{itemize}
        \item $w\in V_r$ is a vector so that for any $v \in  V_{nt}(T'_{\ell'}) $ we have 
        \begin{equation}
            \delbar_{J}\, u_v\ + \ \mu_k(w)|_{\graph(\check{\varphi}_v,u_v)} \ = \ 0.
        \end{equation}
    \end{itemize}
    We require that $(\varphi,u,w)$ satisfy the regularity assumption~\eqref{eq:perturbations-suffice} and equip $\cTR$ with the Gromov topology.
\end{definition}

\begin{remark}
    An unbranched trivial cylinder is regular by \cite[Lemma~2.40]{Par19}, so it is not an issue that we do not perturb them.
\end{remark}

\begin{proposition}\label{prop:gluing-level}
Fix a point $(\varphi,u_0,w) \in \cTR$. Let $g\cl   G_{T_\varphi/} \times  \cBR|_{T_\varphi} \times N\hookrightarrow  \cT$ denote the gluing map as constructed in the proof of Theorem \ref{thm:gluing-main}. Then, there is a lift $lg$ of the map $g$ to a gluing map for the leveled thickening $\cTR$ such that the following diagram is Cartesian,
\begin{equation}\label{dig:blow-up-square}
\begin{tikzcd}
 {[}0,1)^k\times \cBR|_{T_\varphi}\times \prod_{j = 1}^{\max\ell}(\Delta^{\#\ell\inv(j)-1})^{\text{int}}\times N \arrow[rr,  "lg"] \arrow[d,"\rho"] && U^\bR \arrow[d] \\
 G_{T_\varphi/}\times  \cBR|_{T_\varphi} \times  N\arrow[rr, "g"] && U 
\end{tikzcd}
\end{equation}
where $U^\bR$ and $U$ are neighborhoods of the points $(\varphi,u_0,w)$ and $(\varphi,[u_0],w)$ in $\cTR$ and $\cT$ respectively. 
The vertical arrow on the right is the natural projection map induced by forgetting the level data. The left vertical arrow $p$ is defined by  $$\rho(t,b,x,u,w) = (\varDelta(t,x),b,\cc{u},w),$$
where the map $u\mapsto \cc{u}$ forgets the trivial cylinders in $u$ and quotients by more translations. 
\end{proposition}

\begin{proof}
We first choose the lift $\wt\varphi$ of $\varphi$, using the leveled building $u_0$. The key observation is that a cylindrical chart at a puncture $p_{v,e}$ determines a tangent vector in $T_{p_{v,e}}C_v$. Indeed, a cylindrical chart can be compactified to obtain a disc chart $\vartheta \cl \bD^2_r \to C_v$ so that $\vartheta(0)=p_{v,e}$ and the associated tangent vector is $d\vartheta(0)1 \in T_{p_{v,e}}C_v$. Using these tangent vectors, Lemma~\ref{lem:tangent-vectors-give-lift} gives us a lift $\wt\varphi$ of $\varphi$ in $\cBS$. We will see that the leveled building determines cylindrical charts up to first order and thus a lift. We first fix a representative of $u_0$.
As in \S \ref{ssc:base_gluing}, choose sections $q'_v$ of the smooth locus of the universal family $\cC\to \cBR$ near $\varphi$. This determined a unique representative $\ol{u}_0 \cl \cC^o_\varphi \to \bR \times Y$ by requiring that $\ol{u}_0(q'_v) \in \{0\}\times Y$ for each $v$. We can similarly add marked points $q'_v$ to get a representative $\wt{u}_0$ of the leveled building $u_0$ such that $\wt{u}_0(q'_v) = 0$. Fix cylindrical charts near each puncture $p \in \cC|_\varphi$ such that $u_0$ is given by 
\[u_0(s,t) = (L_\gamma s,\wt\gamma(t)) + O(e^{-ks}) \]
in these local coordinates, where $\wt\gamma$ is a parametrization of $\gamma$ given by the base point (which is non-unique if $\gamma$ is multiply-covered). We call such a chart a \textit{normalizing end} and use them together with the above observation to obtain a lift $\wt \varphi$ of $\varphi$ by Lemma~\ref{lem:tangent-vectors-give-lift}. A direct computation shows that $\wt \varphi$ does not depend on the choice of representative of the leveled building $u_0$ or the sections $q'_v$. Using this lift $\wt \varphi$, we get a local chart $g_b\cl G_{T_\varphi} \times \cBR|_{T_\varphi} \to \cBR$ and its lift $lg_b$ as in \eqref{dig:square-base}.

We first construct $lg$ on $\{0\}^k \times \cBR|_{T_\varphi}\times \prod_{j = 1}^{\max\ell}(\Delta^{\#\ell\inv(j)-1})^{\text{int}}\times N$, we can extend it to the whole space by a similar construction as in Theorem \ref{thm:gluing-main}. 
We define ${lg}$ on $\{0\}^k \times \cBR|_{T_\varphi}\times \prod_{j = 1}^{\max\ell}(\Delta^{\#\ell\inv(j)-1})^{\text{int}}\times N$ as 
$${lg} \cl \{0\}^k \times \cBR|_{T_\varphi}\times \prod_{j = 1}^{\max\ell}(\Delta^{\#\ell\inv(j)-1})^{\text{int}}\times N  \to U^\bR|_{(T_\varphi,\ell)}$$
given by the relation
$${lg}\lbr{lg_b\inv(\wt\varphi), [u],w'} := (\varphi,u,w'),$$ 
where $(\varphi,u,w')$ lies in a small enough neighborhood $U^\bR|_{(T_\varphi,\ell)} \subset  \cTR|_{(T_nu,\ell)} $ in the restriction of $\cTR$ to the strata of leveled tree type $(T_\varphi,\ell).$   

Finally, we extend ${lg}$ to obtain the required lift $lg$ of $g$ as stated in the diagram \eqref{dig:blow-up-square} by a construction similar to that of $\Glue$ in the proof of Theorem \ref{thm:gluing-main}. The construction of the diagonal map $\Delta$ ensures that while performing the gluing of the target as in \S \ref{ssc:targ_glu}, we glue cylinders of the same length between two levels. The rest of the gluing construction can be followed verbatim to obtain the leveled gluing map $lg$. 
\end{proof}

We can now prove the key result that will allow us to show that the global Kuranishi chart for $\Mbar^{\, J}(T)$ pulled back along $\cBS\to\cBR$ yields a global Kuranishi for $\Mbar^{\, J}_{\sft}(T_\ell)$. 

\begin{corollary}\label{cor:map-to-leveled-base} There exists a canonical map $\pi^\bR\cl\cTR\to \cBS$ so that the following square 
    \begin{equation}\label{dig:blow-up-square-thickening}
    \begin{tikzcd}
    \cTR \arrow[r,"\Pi"] \arrow[d,"\pi^\bR"] & \cT \arrow[d,"\pi"] \\
    \cBS\arrow[r, ""] & \cBR 
    \end{tikzcd}\end{equation}
is a pullback square. 
\end{corollary}

\begin{proof} 
In order to construct $\pi^\bR$, fix open covers $\{U_i^\bR\}_{i\in I^\bR}$ of $\cTR$ and $\{U_i\}_{i\in I}$ of $\cT$ with maps $i \mapsto i'$ so that there exist local charts $lg_i$ and $g_{i'}$ as in~\eqref{dig:blow-up-square} and so that we have diffeomorphisms $lg_{b,i}$ and $g_{b,i'}$ onto $V_i^\bR\sub \cBS$ for $i\in I^\bR$ and $V_{i'} \sub \cBR$ for $i'\in I$ as in \eqref{dig:square-base}, with $\pi(U_{i'}) \sub V_{i'}$ for each $i'$. Moreover, choose the $U_i^\bR$ so that 
$$\Pi\inv(U_{i'}) \,=\, \union{i\in I^\bR_{i'}}{U_{i}^\bR}.$$
for each $j \in I$ and, shrinking the sets $V_i^\bR$, that $V_i^\bR\neq V_{j}^\bR$ for $i\neq j$ with $i'\neq j'$.
Increasing the cover $\{U_i\}_{i\in I}$ trivially, we may assume $I^\bR = I$. Thus, we may define $\pi^\bR_i\cl U_i^\bR\to V_i^\bR$ to be the composition 
\begin{multline}
    U_i^\bR\,\xra{lg_i\inv }\,[0,1)^{k_i}\times \cBR|_{T_i}\times  \prod_{j = 1}^{\max\ell_i}(\Delta^{\#\ell_i\inv(j)-1})^{\text{int}}\times N_i\\\longrightarrow\, [0,1)^{k_i}\times \cBR|_{T_i}\times  \prod_{j = 1}^{\max\ell_i}(\Delta^{\#\ell_i\inv(j)-1})^{\text{int}} \,\xra{lg_{b,i}} V_i^\bR.
\end{multline}
This map makes the square
\begin{equation}\label{dig:blow-up-square-open-subset}
    \begin{tikzcd}
    U_i^\bR \arrow[r,""] \arrow[d,"\pi_i^\bR"] & U_i \arrow[d,"\pi|_{U_i}"] \\
    V_i^\bR\arrow[r, ""] & V_i 
    \end{tikzcd}\end{equation}
    commute due to the commutativity of~\eqref{dig:square-base} and~\eqref{dig:blow-up-square-thickening}. That $U_i^\bR$ is the fiber product can be checked using the local coordinates $lg_i$ and $g_i$ as well as their counterparts on the base, whence it becomes immediate. The universal property of the fiber product implies that $\pi_i^\bR$ agrees with $\pi_j^\bR$ on the intersection $U_i^\bR\cap U_j^\bR$, yielding the map $\pi^\bR$.
    Moreover, it follows from~\eqref{dig:blow-up-square} that
\begin{equation}
    U_i^\bR \,=\, \Pi\inv(U_i) \cap (\pi^\bR)\inv(V_i^\bR).
\end{equation}
    Thus we may conclude by a formal argument, using that the squares~\eqref{dig:blow-up-square-open-subset} are Cartesian.
\end{proof}

\subsubsection{Disconnected buildings}
Let us now here also record the construction of a global Kuranishi chart for moduli spaces of disconnected leveled buildings.  

\begin{definition}\label{de:disconnected-domains}
    Suppose $\Lambda\cl \Gamma^-\to \Gamma^+$ is a function, and let $\beta = (\beta-\gamma)_\gamma\in \Gamma^+$ be a sequence of relative homology classes. We define $\Mbar^{\, J}_{\sft}(\Gamma^+,\Gamma^-;\beta)_\Lambda$ to be the moduli space of leveled buildings with $|\Gamma^+|$ components, where the curve restricted to a connected component is positively asymptotic to a unique Reeb orbit $\gamma\in \Gamma^+$ and negatively asymptotic to the Reeb orbits in $\Lambda_\gamma\coloneqq \Lambda\inv(\gamma)$.
\end{definition}

\noindent The corresponding moduli space $\Mbar^{\,J}(\Gamma^+,\Gamma^-;\beta)_\Lambda$ of Pardon buildings is simply the product 
$$\Mbar^{\,J}(\Gamma^+,\Gamma^-)_\Lambda\coloneqq \p{\gamma\in \Gamma^+}{\Mbar^{\,J}(\gamma,\Lambda_\gamma;\beta_\gamma)}.$$
However, in the case of leveled buildings, we have to incorporate the relative translations between components.

\begin{proposition}\label{prop:disconnected-buildings}
   Let $\cK_{\gamma,\Lambda_\gamma}$ be the global Kuranishi chart constructed for $\Mbar^{\,J}(\gamma,\Lambda_\gamma;\beta_\gamma)$ in Theorem~\ref{thm:pardon-gkc} with base space $\cBR_{\gamma,\Lambda_\gamma}$. Then, letting $\cBS_\Lambda$ be the base space constructed in \S\ref{subsec:base-for-disconnected}, the pullback chart  
   $$\cK^\bR_\Lambda \,\coloneqq\, \cBS_\Lambda\times_{\p{\gamma}{\cBR_{\gamma,\Lambda_\gamma}}}\p{\gamma}{\cK_{\gamma,\Lambda_\gamma}}$$ 
   is a global Kuranishi chart with corners for $\Mbar^{\, J}_{\sft}(\Gamma^+,\Gamma^-;\beta)_\Lambda$.
\end{proposition}

\begin{proof}
    The proof is analogous to the proof of Theorem~\ref{thm:leveled-gkc}, using the correspondence between forests and trees with a `ghost vertex' discussed at the beginning of \S\ref{subsec:base-for-disconnected}.
\end{proof}
\subsection{Kuranishi charts for buildings in symplectic cobordisms}\label{subsec:gkc-for-cobordism} 
In this subsection we construct global charts for moduli spaces of (leveled) buildings in symplectic cobordisms. Fix thus an exact symplectic cobordism $(\wh X,\omega = d\lambda)$ from $(Y^-,\lambda^-)$ to $(Y^+,\lambda^+)$ as well as open embeddings 
\begin{gather*}\label{eq:contact-ends} 
\Theta^+ \cl (N,\infty) \times Y^+\to \wh X\\ 
	\Theta^- \cl (-\infty,-N) \times Y^-\to \wh X
\end{gather*}
for some $N \gg 0$, so that $(\Theta^\pm)^*\lambda = e^s \lambda^\pm$ and $X \coloneqq \wh X\sm (\im(\Theta^+)\cup \im(\Theta^-))$ is compact.\par 
Fix an $\omega$-compatible almost complex structure $\wh J$ on $\wh X$ whose pullback under $\Theta^\pm$ is an adapted almost complex structure $J^\pm$ on $\wh Y^\pm$.

\subsubsection{Base space}\label{subsec:cobordism-base} The construction of the base space for buildings in symplectic cobordisms is similar.

\begin{definition}\label{de:type-for-exact-cobordism} Given a stable map $\varphi \cl C\to \bP^d$ of genus zero whose domain has a unique node $x$, we say that $x$ is of \emph{type $0$} if it is non-separating or if it separates $C$ into irreducible components $C_0$ and $C_1$ of degree $d_0$, respectively $d_1$ so that 
\begin{equation}\label{eq:degree-comparison-cobordism} 
\mathrm{d}_x:=(d_0-p^+\s{z^+_i \in C_0}{d_i^+}+p^-\s{z^-_j \in C_0}{d_j^-}) - (d_1-p^+\s{z^+_i \in C_1}{d_i^+}+p^-\s{z^-_j \in C_1}{d_j^-}) = 0.
\end{equation}
We say $x$ is of \emph{type $1$} with order $|\mathrm{d}_x|$ otherwise.\end{definition}

We let $\cB'_c$ be the real-oriented blow-up at the normal crossing divisor with irreducible components given by the divisors corresponding to curves with a unique node that has type $1$. Then, we define $\cBR_c$ to be the total space of the torus bundle over $\cB'_c$ obtained by adding asymptotic markers at the marked points.

Recall from Lemma \ref{lem:stratification-base-space} that there is a stratification $P\cl  \cBR \to \scS^o$ which assigns to a map the tree type of its domain. Here, we can upgrade this to a stratification that keeps track of the targets by using the convention discussed in Remark \ref{rem:distinguish-from-degree}. For a vertex $v\in V(T)$ let $D_{C_v}\sub C_v$ be the divisor of special points on $C_v$. Then, define $\ast^\pm(v)$ by setting

\begin{itemize}[leftmargin=20pt]
    \item  $\ast^\pm(v) = 0$ if \[ p^+ \mid \;
\underbracket[0.8pt]{\left(\deg(\varphi\!\mid_{C_v})-\deg(\omega_{C_v}(D_{C_v}))\right)}_{\neq 0}
\]

    \item $\ast^\pm(v) = 1$ if \[ p^- \mid \underbracket[0.8pt]{\left(\deg(\varphi\!\mid_{C_v})-\deg(\omega_{C_v}(D_{C_v}))\right)}_{\neq 0}
\]

    \item if $(\deg(\varphi|_{C_v})-\deg(\omega_{C_v}(D_{C_v}))) = 0$, then \begin{equation}\ast^\pm(v)=
        \begin{cases}
            1 & \text{ if }p^-\mid |\mathrm{d}_x|\\
            0 & \text{ if }p^+\mid |\mathrm{d}_x|,
        \end{cases}
    \end{equation}
    \item if $p^\pm \not \mid  (\deg(\varphi|_{C_v})-\deg(\omega_{C_v}(D_{C_v})))$, then $\ast^+(v)=0$ and $\ast^-(v) = 1$.
\end{itemize}

\begin{remark}
    Notice that the discussion above does not lift the stratification to  $\scS^c$ since not all morphisms of $\scS$ induce morphisms in $\scS^c$. Moreover, at this stage there is not always a function $\ast$ on $E(T)$ for which $(T, \ast, \ast^\pm)$ is a cobordism tree.
    This issue can be resolved by a corner-blow-up construction as explained below.
\end{remark}

In order to obtain the leveled base space $\cBS_c$, we need a preliminary definition. We call $(T,\ast_\pm)$ a \emph{pre-cobordism tree (or forest)} if the functions $\ast_\pm\cl V(T)\to \{0, 1\}$ satisfy the conditions in Definition \ref{de:decorated-cob-tree}. Similarly, $(T,\ast_\pm,\ell)$ is a \emph{leveled pre-cobordism tree} if the level function $\ell$ satisfies the conditions in Definition \ref{de:leveled-cob-tree}.\par
Then, we define the \emph{leveled cobordism base space} $\cBS_c$ to be the space obtained by applying the generalized blow-up on $\cBR_c$ associated to the refinement arising from maximally leveled pre-cobordism trees. This construction is similar to the construction for base space for leveled buildings. We may assume that $\cBS_c$ carries a stratification by leveled cobordism trees by removing the locus corresponding to leveled pre-cobordism trees, which do not support a leveled cobordism tree as in Definition~\ref{de:leveled-cob-tree}.

\subsubsection{Framings}\label{subsec:cobordism-framings} In order to obtain framings of buildings in $\wh X$, we use a similar construction as in \textsection\ref{subsec:framings}. However, we vary it somewhat to ensure that (in most cases) we can already see from the base space which `part' of the cobordism an irreducible component is mapped to.

\begin{definition}\label{de:approximation-cobordism}
    Let  $\cF = \cF^+\sqcup \cF^-$ be a finite set of Reeb orbits of $\lambda^+$ and $\lambda^-$. We call a $1$-form $\wt\lambda$ on $\wh X$ an \emph{$\cF$-integral approximation} if 
\begin{enumerate}[label=(\roman*),leftmargin=25pt,ref=\roman*]
	\item $(\Theta^\pm)^*\wt\lambda = e^s\wt\lambda^\pm$, where $\wt\lambda^\pm$ is an $\cF^\pm$-integral approximation,
    \item $\forall \Gamma^\pm \sub \cF\,:\, \cA_\lambda(\Gamma^+)-\cA_\lambda(\Gamma^-) > 0 \,\rimp \, \cA_{\wt\lambda}(\Gamma^+)-\cA_{wt\lambda}(\Gamma^-) > 0$.
\end{enumerate}
\end{definition}

The existence of $\wt \lambda$ follows from the same argument as in Lemma \ref{lem:contact-approximation}. Fix two prime numbers $p^\pm$ so that 
\begin{equation}\label{eq:primes-for-base}
    p^- > \s{\gamma\in \scF^+}{\cA_{\wt\lambda}(\gamma)}\qquad \text{and}\qquad p^+ \gg p^-
\end{equation}
Given a smooth stable building $u = (u_v,C_v,z_{v,*}^\pm)$ with dual graph $T$, recall that the vertices of $T$ are decorated with a pair of symbols $*^\pm(v)\in \{+,-\}^2$\footnote{In Definition \ref{de:decorated-cob-tree} we had that $\ast^\pm$ took values in $\{ 0,1\}$. We abuse notation in this subsection by identifying $0\sim +, 1\sim -$ to reduce clutter in \eqref{eq:cob-bundle}.}
so that $u_v$ maps to $Y^{*^\pm(v)}$ if the two symbols agree and to $\wh X$ if they disagree. Define for $v \in V(T)$ the line bundle 
\begin{equation}\label{eq:cob-bundle} 
L_{u,v} \coloneqq \cO_{C_v}\lbr{p^{*^+(v)}\s{e \in E_v^+}{\int\gamma_e^*\wt\lambda}-p^{*^-(v)}\s{e \in E_v^-}{\int\gamma_e^*\wt\lambda}}. 
\end{equation}
As before, this yields a holomorphic line bundle 
\begin{equation}
    \fL_u \coloneqq \omega_C(D) \otimes L_u
\end{equation} 
on $C$, where $D$ is the divisor of marked points, which is unique up to holomorphic isomorphism. By the definition of the $\cF$-integral approximation and the choice of $p^\pm$, $\deg(\fL_u|_{C'}) > 0$ for each irreducible component $C'$ of $C$. Thus, we obtain a $\PGL_{d+1}(\bC)$-orbit in $\cB_{\Gamma^+,\Gamma^-}(d)$ as before, where $d = \deg(\fL_u)$.

\begin{remark}\label{rem:distinguish-from-degree}
    The numbers $p^\pm$ were chosen in this specific way so that one can see from the framing which irreducible components are mapped to the ends $\wh Y^\pm$ and which are mapped to $\wh X$ proper. Concretely, given a framing $\varphi \cl C\hkra \bP^d$ associated to the line bundle $\fL_u$, an irreducible component $C'\sub C$ is mapped to $\wh Y^\pm$ if and only if  
    \[ p^\pm \mid \underbracket{ \deg(\varphi|_{C'})-\deg(\omega_C|_{C'})}_{\neq0}\]
    \noindent while in the case where $\deg(\varphi|_{C'})-\deg(\omega_{C}|_{C'})=0$, the component $C'$ is mapped to $\wh Y^\pm$ if and only if $p^\pm \mid |\mathrm{d}_x|$ where $x$ is any type 1 node in $C'$.
\end{remark}

\subsubsection{Families and local models}\label{subsec:families-and-local-models} We can now define the analogue of the infinite-dimensional spaces $\cZ$ of \textsection\ref{subsec:domain-metrics}. Fix a decorated corolla $T$ in $\scS^c$ and let $\wt\lambda$ be a $\cP(T)$-integral approximation, where we can define $\cP(T)$ for a cobordism tree as in Equation~\eqref{eq:tree-reeb-orbits}. Let $\Gamma^\pm$ be the Reeb orbits labeling the positive and negative edges of $T$, and set $L =\cA_\lambda(\Gamma^+)$. Fix a constant $\kappa_L > 0$ so that any non-trivial $J$-holomorphic map between Reeb orbits of action at most $L$ has $\omega$-energy at least $2\kappa_L$.

\begin{definition}\label{de:family-of-cobordism-buildings}
We define $\cZ^c = \cZ^c_{\wt\lambda}(T)$ to consist of tuples $(\varphi,T',u)$ of the form 
\begin{enumerate}[label=\roman*),leftmargin=20pt,ref=\roman*]
    \item $\varphi\in \cBR$ with image $\wch\varphi\in \cB$;
    \item $T'\in \scS^c_{/T}$ is mapped to $T_\varphi$ by the forgetful functor $\scS^c\to \scS$ and $\wch\varphi\in \cB$ satisfies 
    \begin{equation*}
        \deg(\wch\varphi_v) = |D_v|-2 + p^{\star^+(v)}\s{\gamma\in E^+_v}{\int \gamma_e^*\wt\lambda} -  p^{\star^-(v)}\s{\gamma\in E^-_v}{\int \gamma_e^*\wt\lambda}
    \end{equation*}
    for each $v \in V(T_\varphi)$, where $D_v$ is the divisor of special points on $C_v$;
    \item $u = (u_v)_{v\in V(T')}$ is a collection of maps where 
    \begin{enumerate}[leftmargin=15pt]
        \item $u_v$ is an equivalence class of smooth maps $\dot{C}_v \to \wh Y^+$ up to translation if $*^\pm(v) = 0$,
        \item $u_v$ is an equivalence class of smooth maps $\dot{C}_v \to \wh Y^-$ up to translation if $*^\pm(v) = 1$,
        \item $u_v$ is a smooth map $\dot{C}_v \to \wh X$ whenever $v$ is a nontrivial vertex with $\ast^+(v) = 0$ and $*^-(v) =1$ and a trivial cylinder over the associated Reeb orbit if $v$ is trivial, 
    \end{enumerate}
    such that 
	\begin{itemize}[leftmargin=20pt]
    \setlength\itemsep{2.5pt}
	\item $u_v$ is $J$-holomorphic near the punctures of $\dot{C}_v$;
	\item if $x \in C_v$ is a positive/negative node of type $1$, then $u_v$ is positively/negatively asymptotic to a Reeb orbit $\gamma\in \cP^\pm$ near $x\in C$;
    \item $\int_{C_v}{u_v}^*\omega \geq 0$;
	\item if $C_v$ is unstable, then $\int_{C_v}u^*_v\omega \geq \kappa_L$.
	\end{itemize}
    \end{enumerate}
     We write $\cZ^c_{\delbar}$ for the locus of $J$-holomorphic elements of $\cZ^c$.
\end{definition}

\begin{remark}
    $V(T')$ is only bigger than $V(T_\varphi)$ if $\wh X$ is the trivial cobordism from $(Y,\lambda)$ to itself. In this case, the base does not capture the full stratification of $\cZ^c$ and the resulting thickening will be a rel--$C^1$ manifold with boundary (in the fibers). This will be important in \textsection\ref{subsec:stable-complex}. 
\end{remark}

\begin{lemma}\label{lem:stratification-cobordism-family}
    $\cZ^c$ carries a canonical topology, which is stratified by $\scS^c$.
\end{lemma}

\begin{proof}
 We equip $\cZ^c$ with the following topology. We use $\cC_{\ge \epsilon}$ as defined in \eqref{eq:curlyCnbd}. Given $(\varphi,T', u)\in \cZ^c$, define $\cN_\epsilon(\varphi,T',u)$ to be the subset of points $(\varphi',T'',u')$ such that
\begin{itemize}[leftmargin=20pt]
\setlength\itemsep{1pt}
	\item $d_\cB(\varphi,\varphi') < \epsilon$;
        \item 
        $T' \xrightarrow{\pi_{\textbf{g}}}T''$ for $\textbf{g}$ in an $\epsilon$-neighborhood of $0$ in $G^c_{T'/}$
	\item the (orbits of the) graphs satisfy 
    \begin{equation*}
    d_{\text{H}}\lbr{\bR\cdot\graph(\varphi_v,u_v)|_{\cC_{\ge\epsilon}},\bR\cdot\graph(\varphi'_{\pi_\textbf{g}(v)},u'_{\pi_\textbf{g}(v)})|_{\cC_{\ge\epsilon}}} < \epsilon
	\end{equation*} 
	in the Gromov-Hausdorff metric for every vertex $v$ such that $\pi_\textbf{g}(v)$ is a symplectization vertex, and we choose any representatives of the classes $[u_v]$ and $[u'_{\pi_\textbf{g}(v)}]$; 
    \item the graphs satisfy
    $$d_{\text{H}}\lbr{\graph(\varphi_v,\pi_g\circ u_v)|_{\cC_{\ge\epsilon}},\graph(\varphi'_{\pi_\textbf{g}(v)},u'_{\pi_\textbf{g}(v)})|_{\cC_{\ge\epsilon}}} < \epsilon$$
    for vertices $v$ such that $\ast(\pi_g(v)) = 01$,
	\item for $e \in E(T'')$ with associated Reeb orbit $\gamma_e$ and corresponding node $x_e \in \cC_\varphi$ we have $$d_Y(u'_Y(z),\gamma_e) < \epsilon$$ 
	for any $z \in \cC_{\varphi'}$ with $d_\cC(z,x_e) \leq \epsilon$.
\end{itemize}
Similar to $G_{\text{II}}$ in \cite[\S 2.5]{Par19}, the gluing parameter space 
\begin{equation}\label{eq:local-model-cobordisms}
G^c_{T/}:= 
\left\{
\bigl(\{g_{e}\}_{e}, \{g_{v}\}_{v}\bigr) \in (0,\infty]^{E^{\mathrm{int}}(T)} \times (0,\infty]^{E_{00}(T)}
\;\middle|\;
g_{v} = g_{e} + g_{v'} \;\; \text{for }e= (v,v') \text{ with } v \in E_{00}(T)
\right\}
\end{equation}
 models neighborhoods of cobordism buildings. In the equation above, we interpret $g_{v'}=0$ if $v'\not \in E_{00}(T)$. There is a natural stratification of $G^c_{T/} \to (\scS^c)_{T/}$ obtained by sending $(\{ g_e\}, \{ g_v\})$ to the map $\pi: T \to T'$
such that
\begin{itemize}
    \item the edge $e$ is contracted if $g_e< \infty$
    \item $\ast(v)=00$ is changed to $\ast(v)=01$ if $g_v<\infty$.
\end{itemize}
Conceptually, we view $G^c$ as the parameter space of gluing the \textit{target} of the cobordism buildings. Any element $\textbf{g}\in G^c_{T/}$ naturally induces a collection of maps $$\{ \pi_* : \wh{X}_v \to \wh{X}_{\pi(v)} \}_{v\in V(T)}$$ 
where \begin{itemize}
    \item maps of the type  $\wh Y^\pm \to \wh Y^\pm$ are allowed to be
any $\bR$-translation
\item maps of the type $\wh X  \to \wh X$ must be the identity
\item maps of the type $\wh Y^\pm \to \wh X$ are the pre-composition of the relevant boundary collar identifying the ends of $\wh X$ with $\wh Y^\pm$ with any $\bR-$translation of $\wh Y^\pm$.
\end{itemize}
We denote by $0$ the unique element in $G^c$ corresponding to performing no gluing. An $\epsilon$-neighborhood of $0$ is defined via the natural identification $[0,1) \cong (1,\infty]$ via $t \mapsto 1/t$.
 It follows from the construction of $\cZ^c$ that this topology is equipped with a natural stratification by $\scS^c$.
\end{proof}

\noindent The following properties are shown by the same arguments as Lemma~\ref{lem:palais-proper} and Lemma~\ref{lem:zero-locus-correct}.

\begin{lemma}
    The $G_\bC$-action on $\cZ^c$ is Palais proper.\qed
\end{lemma}

\begin{lemma}
The induced map $$\psi\cl \modulo{\cZ^c_{\delbar}}{\PGL_{d+1}(\bC)} \to \Mbar^{\wh X,\,J}(T)$$ is an isomorphism.\qed
\end{lemma}

\subsubsection{Construction}\label{subsec:construction-cobordism} Recall that we have fixed a corolla $T$ with degree $\beta$ and positive/negative edges labeled by $\Gamma^+$ and $\Gamma^-$ respectively. We will write $\Mbar^{\wh X,J}(\Gamma^+,\Gamma^-;\beta) = \Mbar^{\wh X,J}(T)$ from now on to make the step to leveled buildings in Corollary~\ref{cor:leveled-gkc-cobordism} notationally easier.

\begin{definition} A \emph{pre-perturbation datum} $\fD = (\wt\lambda,\conn,p^\pm)$ for $\Mbar^{\wh X,J}(\Gamma^+,\Gamma^-;\beta)$
consists of 
	\begin{itemize}[leftmargin=20pt]
		\item a $\cP(\Gamma^+,\Gamma^-)$-integral approximation $\wt\lambda$ of $\lambda$ as in Definition~\ref{de:approximation-in-symplectisation};
		\item a $J$-linear connection $\conn$ on $T\wh X$ so that $(\Theta^\pm)^*\conn$ is translation invariant;
		\item primes $p^\pm$ satisfying \eqref{eq:primes-for-base}.
        \end{itemize}
 \end{definition}       
Given this, we set 
\begin{equation}\label{eq:auxiliary-degree-cobordism}
    d' \coloneqq p^+\s{\gamma\in \Gamma^+}\cA_{\wt\lambda}(\gamma)-p^-\s{\gamma\in \Gamma^-}\cA_{\wt\lambda}(\gamma)
\end{equation}
and $d \coloneqq d' +|\Gamma^+|+|\Gamma^-|-2$. Let $\cBR_c$ be the smooth manifold with corners defined in \textsection\ref{subsec:cobordism-base} and set $G := \PU(d+1)$ and $\cG := \PGL_{d+1}(\bC)$. 
We let $\cZ^c =\cZ^c_{\wt\lambda}(T)\to \cBS_c$ be the family of buildings of Definition~\ref{de:family-of-cobordism-buildings} with $T$ being the corolla with positive/negative edges labeled by $\Gamma^\pm$.
\begin{definition}\label{de:auxiliary-datum-cobordism}  A \emph{perturbation datum} extending $\fD$ is a tuple $\alpha= (\fD,\cU,\lambda,E,\mu,\mu^\pm)$, consisting of
        \begin{itemize}[leftmargin=20pt]
        \item a good covering $\cU := \{(U_i,\sigma_{ij},D_{ij},\chi_i)\}_{i\in I}$ of $\cZ^c$ and a $\cG$-equivariant map 
        \[\zeta\cl \cB^{\text{st}}_{d'}(d)/S_{d'}\to \cG/G \]
        yielding a $\cG$-equivariant map $\zeta_\cU \cl \cZ\to \fs\fu(d+1)$ (by Lemma~\ref{lem:map-to-lie-algebra});
	\item a \emph{joint perturbation space} $(E,\mu,\mu^\pm)$ consisting of a finite-dimensional $G$-representation $E$ equipped with $G$-equivariant linear maps
    $$\mu\cl E\to C^\infty_c(\cC\inn\times\wh X,\Lambda^{*,0,1}_{\cC\inn/\cBR_d}\otimes_\bC T\wh X)$$
    and 
    $$\mu^\pm\cl E\to C^\infty_c(\cC\inn\times\wh Y^\pm,\Lambda^{*,0,1}_{\cC\inn/\cBR_d}\otimes_\bC T\wh Y^\pm)^\bR$$
    so that $(\Theta^\pm)^*\mu(e)$ agrees with $ \mu^\pm(e)$ restricted to the respective end in $\cC\inn\times \wh Y^\pm$. We require that the map 
    \begin{equation}
        E\to \coker(D_u) : e\mapsto [\mu(e)|_{\graph(\varphi,u)}]
    \end{equation}
    is surjective for any $(\varphi,u)\in \cZ^c_{\delbar}$ with $\zeta_\cU(\varphi,u) = 0$.
\end{itemize}
\end{definition}

\begin{construction}\label{con:gkc-cobordism} Given a perturbation data $\alpha$, we define $$\cK^c_\alpha := (\bT^{\Gamma^+\sqcup \Gamma^-}\times G,\cT^c/\cBR,\cE,\fs)$$ 
by letting $\cT\sub \cZ\times E$ be the space of tuples $(\varphi,\wt T,u,w)$ such that 
\begin{enumerate}[label=\alph*),leftmargin=20pt,ref=\alph*]
    \item for each nontrivial vertex $v \in V(\wt T)$ the associated map $u_v$ (respectively a representative thereof) satisfies 
    \begin{equation}\label{eq:cobordism-perturbed-cr}
        \delbar_{\wh J} \, u_v + \mu_k^{\ast(v)}(w)|_{\graph(\varphi_v,u_v)} = 0,
    \end{equation}
    on $\dot{C}_v$
    \item the linearized operator of~\eqref{eq:cobordism-perturbed-cr} is surjective (without variation of the framing $\varphi$).
\end{enumerate}
    The obstruction bundle $\cE\to \cT$ is the trivial bundle $$\cE =E\oplus \fp\fu(d+1),$$ while the obstruction section $\fs$ is given by a mollification of $\wh\fs(\varphi,u,w) = (w,\lambda_\cU(\varphi,u))$ as in Lemma~\ref{lem:better-obstruction-section}.
\end{construction}

\begin{theorem}\label{thm:cobordism-gkc} Given an exact symplectic cobordism $(\wh X,d\lambda)$ from $(Y^+,\lambda^+)$ to $(Y^-,\lambda^-)$, a compatible almost complex structure $J$ on $\wh X$, and a tree $T$ in $\scS^c$, the following holds
\begin{enumerate}[\normalfont 1),leftmargin=15pt,ref=\arabic*]
		\item\label{cobordism-gkc-unobstructed-aux} The moduli space $\Mbar^{\wh X\,J}(T)$ of buildings in $\wh X$ admits perturbation data. Construction~\ref{con:gkc-cobordism} associates to each perturbation datum $\alpha$ a rel--$C^1$ global Kuranishi chart $\cK^c_\alpha$ with corners for $\Mbar^{\wh X\,J}(T)$.
		\item\label{cobordism-gkc-orientation} If the Reeb orbits $\Gamma^+$ and $\Gamma^-$ labeling the exterior edges of $T$ consist of good Reeb orbits, then there exists a canonical isomorphism $$\fo_{\cK^c_\alpha}\cong \bigotimes\limits_{\gamma\in \Gamma^+}\fo_\gamma\otimes\bigotimes\limits_{\gamma\in \Gamma^-}\fo\dul_{\gamma}$$ of orientation lines.
	\end{enumerate}    
\end{theorem}

As in the case of buildings in symplectizations, this yields a chart for leveled buildings.

\begin{corollary}\label{cor:leveled-gkc-cobordism}
    The pullback Kuranishi chart 
    \begin{equation}
        \cK^{c,\bR} \coloneqq \cBS_c\times_{\cBR_c}\cK^c_\alpha
    \end{equation}
    is a global Kuranishi chart for the moduli space $\Mbar^{\wh X,\,J}_{\sft}(\Gamma^+,\Gamma^-;\beta)$.
\end{corollary}

\begin{proof}
    The proof is analogous to the proof of Theorem~\ref{thm:leveled-gkc}.
\end{proof}

The proof of the first assertion is analogous to the proof of Theorem~\ref{thm:pardon-gkc}\eqref{gkc-unobstructed-aux} and \eqref{gkc-rel-smooth}. We simply have to replace Lemma~\ref{lem:invariant-approximations-exist} with the following definition and existence result. Then the arguments carry over verbatim. The proof of Claim~\eqref{cobordism-gkc-orientation} follows from the arguments of \textsection\ref{subsec:orientation}.

\begin{definition}
Suppose $V\to B$ and $E^\pm \to B^\pm$ are three smooth $G$-vector bundles and that $B^\pm$ admits a free $\bR$-action, which commutes with the $G$-action and lifts to $E^\pm$. Suppose there exist open $G$-invariant subsets $B^\pm_{\circ} \sub B^\pm$ whose orbit under the $\bR$-action covers all of $B^\pm$ and which admit smooth $G$-equivariant open embeddings $j^\pm \cl B^\pm_{\circ} \hkra B$ with disjoint image lifting to embeddings of vector bundles. Assume additionally that the quotients $B^\pm/\bR$ and $B\sm \im(j^+)\sqcup \im(j^-)$ are compact. Then, a \emph{joint finite-dimensional approximation scheme} of $(V,E^\pm)$ is a sequence $(E_\ell,\mu_\ell,\mu^\pm_\ell)$ of finite-dimensional $G$-representations together with $G$-equivariant linear maps 
$$\mu_\ell \cl E_\ell \hkra C^\infty(B,V)$$
and 
$$\mu^\pm_\ell \cl E_\ell \hkra C^\infty_c(B,V)^\bR := \{\eta\in C^\infty(B^\pm,E^\pm)^\bR\mid \supp(\eta)/\bR \text{ is compact}\}$$
satisfying 
\begin{enumerate}[label=\roman*),leftmargin=20pt]
    \setlength\itemsep{2pt}
     \item $E_\ell$ is a subrepresentation of $E_{\ell+1}$ with $\mu_{\ell+1}|_{E_\ell} = \mu_\ell$ and $\mu^\pm_{\ell+1}|_{E_\ell} = \mu^\pm_\ell$,
     \item $\union{}{\im(\mu^\pm_\ell)}$ is dense in $C^\infty_c(B^\pm,E^\pm)^\bR$ in the $C^\infty_{\text{loc}}$-topology,
     \item $\supp(\mu_\ell(v)|_{\im(j_*^\pm)}-j_*^\pm\mu_\ell^\pm(v)|_{B_{\circ}^\pm})$ is precompact in $B$ for any $\ell\ge 1$ and $v\in E_\ell$,
     \item $\union{}{\im(\mu_\ell-{j^+}_*\mu^+_\ell-{j^-}_*\mu^-_\ell)}$ is dense in $C^\infty_c(B,V)$ in the $C^\infty_{\text{loc}}$-topology.
 \end{enumerate}
 Note that the last property makes sense due to the third one.
\end{definition}

\begin{lemma} \label{lem:joint-fin-dim-scheme}
    Given finite approximation schemes $\mu^\pm_*$, there exists a choice of $\mu_*$ such that $(\mu^+,\mu, \mu^-)$ forms a joint finite-dimensional approximation scheme.
\end{lemma}

\begin{proof}
    We adapt the proof of \cite[Lemma~4.2]{AMS23}. Fix $G\times \bR$-invariant connections $\conn^\pm$ on $E^\pm$ and let $\conn$ be a $G$-invariant connection on $V$ so that ${j^\pm}^*\conn$ agrees with $\conn^\pm|_{B^\pm_{\circ}}$ away from a subset $K \sub B^\pm_\circ$, which is precompact in $B^\pm$.
    Let $A \coloneqq B\,\sm \;  (\im(j^+)\cup \im(j^-))$ and let $(A_n)_n$ be a countable exhaustion of $A$ so that each $A_n$ is a smooth $G$-invariant manifold with boundary. Fix an increasing sequence$(\rho_{n,k})_k$ of $G$-invariant smooth bump functions with support in $A_n$ and $A_n = \union{}{\rho_{n,k}\inv(1)}$ for each $n$.
    
    Similarly, let $(B^\pm_n)_n$ be a countable exhaustion of $B^\pm_{\circ}$ by $G$-invariant smooth manifolds with boundary, and let $\rho^\pm_n$ be a $G$-invariant smooth bump function that is identically $1$ on $B^\pm_n$ and supported in $B^\pm_{\circ}$.
    Then, let $\wch\conn^\pm$ be the induced connection on $E^\pm/\bR\to B^\pm/\bR$ and let $(\lambda_\ell^\pm)_\ell$ be the increasing sequence of non-negative eigenvalues of the Laplacian associated to $\wch\conn^\pm$. 
    
    Let $\wch W_\ell^\pm$ be the preimage of the space of eigenfunctions associated to the non-negative eigenvalues $\lambda_j^\pm$ with $j \leq \ell$ and let $W^\pm_\ell \sub C^\infty_c(B^\pm,V^\ell)^\bR$ be the preimage of $\wch{W}_\ell^\pm$. Define 
    $$E^\pm_\ell := \bigoplus\limits_{n\leq \ell} W^\pm_n$$ 
    and let $\mu^\pm_\ell \cl E^\pm_\ell \to C^\infty_c(B^\pm,E^\pm)^\bR$ be the inclusion on each summand. Then, define 
    $$\mu_\ell \cl E^\pm_\ell \to C^\infty(B,V)$$ 
    by 
    $$\mu_\ell((v_n)_n) = \s{n\leq \ell}{\rho_n^\pm\,j^\pm_*\mu^\pm_n(v_n)}.$$
    Finally, doubling $A_{n+1}$ and $V|_{A_{n+1}}$ and considering the eigenspaces of the Laplacian of the induced connection on the doubled vector bundle, we obtain for each $n$ a sequence of vector spaces $(E_{n,k}\inn)_k$ together with maps 
    $$\mu_{n,k}\cl E_{n,k}\inn\to C^\infty_c(B,V) : v \mapsto \rho_{n,k}\, v.$$
    Define for $\ell \ge 1$ the vector space 
    $$E\inn_\ell := \bigoplus\limits_{n,k\leq \ell} E_{n,k}\inn$$
    and let $\mu_\ell \cl E\inn_\ell \hkra C^\infty_c(B,V)$ be the canonical map induced by the maps $\mu\inn_{n,k}$. We finally define 
    $$E_\ell \coloneqq V^+_\ell \oplus E\inn_\ell \oplus V^-_\ell$$
    and let $\mu_\ell$ be given by the sum of the maps $\mu_\ell$ defined above. Extend $\mu^\pm_\ell$ to $E_\ell$ by letting it be $0$ on $E\inn_\ell\oplus V^\mp_\ell$.
\end{proof}

\subsubsection{Disconnected buildings in symplectic cobordisms}
Given a symplectic cobordism $(\wh X,\omega)$ as above and sequences $\Gamma^\pm$ of Reeb orbits of $\lambda^\pm$ as well as a partition $\Lambda\cl \Gamma^-\to \Gamma^+$ and a sequence $\beta (\beta_\gamma)_{\gamma\in \Gamma^+}$ of relative homology classes, we define the moduli space $\Mbar^{\wh X,J}_{\sft}(\Gamma^+,\Gamma^-;\beta)$ of disconnected leveled buildings in $\wh X$ exactly as in Definition~\ref{de:disconnected-domains} except that (one level of) the buildings now maps to the symplectization.\par 

Given $\gamma\in \Gamma^+$ let $\cK_\gamma^c$ be the global Kuranishi chart for $\Mbar^{\wh X,J}(\gamma,\Lambda_\gamma;\beta_\gamma)$ with base space $\cBS_{c,\gamma}$ given by Theorem~\ref{thm:cobordism-gkc}. Recall that $\cBS_{c,\gamma}$ was defined in \S\ref{subsec:cobordism-base} as the corner blow-up of $\cBR_{*}$ corresponding to refinement using maximally leveled pre-cobordism trees. Similarly, let $\cBS_{c,\Lambda}$ be the corner blow-up of $\p{\gamma\in \Gamma^+}{\cBR_{c,\gamma}}$ corresponding to refinements as in \textsection\ref{subsubsec:blowup_to_leveled_base} but using maximally leveled pre-cobordism forests.

\begin{proposition}\label{prop:disconnected-in-cobordims}
 The pullback global Kuranishi chart
   $$\cK^\bR_{c,\Lambda} \,\coloneqq\, \cBS_{c,\Lambda}\times_{\p{\gamma}{\cBR_{c,\gamma}}}\p{\gamma}{\cK^c_{\gamma}}$$ 
   is a global Kuranishi chart for $\Mbar^{\wh X, J}_{\sft}(\Gamma^+,\Gamma^-;\beta)_\Lambda$.
\end{proposition}

\begin{proof}
    The proof is analogous to the proof of Theorem~\ref{thm:leveled-gkc}.
\end{proof}
\subsection{Orientations}\label{subsec:orientation}

The \emph{determinant line} of a global Kuranishi chart $\cK = (G,\cT,\cE,\fs)$, is the real line bundle
\begin{equation}
    \det(\cK) := \det(T\cT) \otimes \det(\fg)\dul \otimes \det(\cE)\dul
\end{equation}
on $\cT$ restricted to the zero locus $\fs\inv(0)$. 
Its \emph{orientation line} is the sheaf of $\bZ_2$-torsors 
\begin{equation}
\orl(\cK) := (\det(\cK)\sm 0)/\bR_{> 0}.\end{equation}
in degree $\vdim(\cK)$. Given $n \in \bZ$, we define $\orl(n) \coloneqq \bZ_2[-n]$.

\begin{definition}
    An \emph{orientation} of $\cK$ is an isomorphism $\orl(\vdim\cK)\cong \orl(\cK)$.
\end{definition} 

 We will throughout use the isomorphism 
 \begin{equation}\label{eq:multiplied-with-dual}
     \orl\otimes\orl\dul\cong\fo(0): v\otimes f \mapsto f(v)
 \end{equation}
 to trivialize the multiplication of an orientation line with its dual. Given a finite-dimensional vector space $V$, we let $\orl(V)$ be the $\bZ_2$-torsor in degree $\dim V$ associated to $H_{\dim V}(V,V\sm\{0\};\bZ)$. Given a Cauchy--Riemann operator $D$ we define its orientation line to be 
\[\orl(D)\coloneqq \orl(\ker D)\otimes \orl(\coker D)\dul.\]

\begin{remark}
   Given a finite-dimensional vector space $V$, the zero map $D_0 = 0 \cl V\to V$ has orientation line $\orl(D_0) \cong \orl(V)\orl(V)\dul$. On the other hand, it is homotopic to the identity $D_1 = \ide$ with orientation line $\orl(D_1) =\orl(0)$. Our choice of trivialization in~\eqref{eq:multiplied-with-dual} ensures that the canonical isomorphism $\orl(D_0) \cong \orl(D_1)$ is orientation-preserving.
\end{remark}

\begin{lemma}
    Let $\cK_\alpha$ be the global Kuranishi chart of Theorem~\ref{thm:pardon-gkc}. Then there exists a canonical isomorphism 
    \begin{equation}
        \orl(\cK) \cong \orl(\delbar_J) \otimes \orl(\bR)\dul\otimes\orl(2|\Gamma^+|-2|\Gamma^-|-6).
    \end{equation} 
\end{lemma}

\begin{proof}
    Observe first that $\cK_\alpha$ admits a well-defined vector lift of its tangent micro-bundle, given by $$T\cT_\alpha = T_{\cT_\alpha/\cBR_\alpha}\oplus \pi^*T\cBR_\alpha,$$ 
    where $\pi \cl \cT_\alpha\to \cBR_\alpha$ is the forgetful map. We will call it from now on simply the tangent bundle and will omit the subscript $\alpha$. Recall that $\cBR$ is a torus bundle over a blow-up of the complex manifold $\cB\sub \Mbar_{0,|\Gamma^+|+|\Gamma^-|}(\bP^d,d)$. Thus, 
    $$\orl(\cBR)\cong \orl(\cB)\otimes \orl(S^1)^{\otimes \Gamma^+\sqcup \Gamma^-}\cong \orl(\fp\fg\fl)\otimes\orl(\Mbar_{0,|\Gamma^+|\sqcup \Gamma^-})\otimes \orl(S^1)^{\otimes \Gamma^+\sqcup \Gamma^-}$$
    canonically. Meanwhile, for $(\varphi,u,w)\in \cT$ we have
    $$(T_{\cT/\cBR})_{(\varphi,u,w)} = \ker\lbr{D_u^\conn + \mu_k(-)|_{\graph(\varphi,u)}\cl C^\infty(\dot{C},u^*T\wh Y) \oplus E_k \to \Omega^{0,1}(\dot{C},u^*T\wh Y)}$$
    which agrees with the index of the Cauchy--Riemann operator $D_u^\conn + \mu_k(-)|_{\graph(\varphi,u)}$. By \cite[Lemma~3.2]{Bao23}, there exists a canonical isomorphism 
    \begin{equation}\label{eq:orientation-vertical-tangent-bundle}
        \orl(D_u^\conn + \mu_k(-)|_{\graph(\varphi,u)}) \cong \orl(D_u^\conn) \otimes \orl(E_k).
    \end{equation}
    Combining these two isomorphisms with the polarization isomorphism $\cG\cong G \times \fg$, we obtain the canonical isomorphisms (over the locus of curves with smooth domains) 
   \begin{align*}
       \orl(\cK_\alpha) \,&\cong\, \orl(D^\conn) \orl(E_k)\orl(\bR)\dul\orl(\cBR)\orl(\wh\fg)\dul\orl(\fg)\dul\orl(E_k)\dul\\
       &\cong\, \orl(D^\conn) \orl(E_k)\orl(\bR)\dul\orl(\fp\fg\fl)\orl(\Mbar_{0,|\Gamma^+|+|\Gamma^-|})\orl(S^1)^{\otimes \Gamma^+\sqcup \Gamma^-}(\orl(S^1)^{\otimes \Gamma^+\sqcup \Gamma^-})\dul\orl(\fg)\dul\orl(\fg)\dul\orl(E_k)\dul\\
       &\cong\,\orl(D^\conn)\orl(\bR)\dul \orl(\Mbar_{0,|\Gamma^+|+|\Gamma^-|})\orl(\fp\fg\fl)\orl(\fp\fg\fl)\dul\orl(E_k)\dul\orl(E_k)\\&\cong\,\orl(D^\conn)\orl(\bR)\dul \orl(\Mbar_{0,|\Gamma^+|+|\Gamma^-|}),
   \end{align*}
  where we omitted the tensor product. Note that we use the Koszul sign rule when switching two orientation lines.
\end{proof}

We associate to any based Reeb orbit $(\gamma,b_\gamma)$ a virtual vector space $V_{\gamma} =(V_\gamma^+,V_\gamma^-)$ as follows.

\begin{definition}[{\cite[Definition~2.46]{Par19}}] Let $\wt\gamma$ be the constant-speed parametrization determined by a base point $b\in \ol\gamma$. Pull back the complex bundle $\wt\gamma^*\xi \oplus \bC \to S^1$ to a bundle $\cV\to \bC\units$. The Lie derivative $\cL_R$ and the trivial connection on $\bC$ pull back to yield a connection $\conn$ on $\cV$. Let $(\wt\cV,\delbar)$ be an extension of $(\cV,\conn^{0,1})$ to all of $\bC$. We define 
\begin{equation}
    V_{\gamma,b} \coloneqq \text{Ind}(\wt\cV,\delbar)
\end{equation}
    to be the index bundle of this extension.
\end{definition}

By \cite[Lemma~2.47]{Par19}, such extensions $(\wt\cV,\delbar)$ exist, and any two extensions differ by the direct sum with a complex vector space (up to isomorphism). In particular, the associated orientation line $\fo(V_{\gamma,b})$
is independent of the choice of extension. Recall that a Reeb orbit $\gamma$ is \emph{good} if it is \emph{not} an even multiple cover of a simple Reeb orbit $\ol\gamma$ with $|\sigma(\cA_{\ol\gamma})\cap (-1,0)| \equiv 1$ mod $2$. Equivalently, $\gamma$ is good if the action of $\bZ/m_\gamma$ on $\orl(V_{\gamma,b})$ is trivial. Thus, for good Reeb orbits, we have a canonical isomorphism $\orl(V_{\gamma,b})\cong \orl(V_{\gamma,b'})$ for any two base points $b,b'$ and we can set 
\begin{equation}
    \orl_\gamma \coloneqq \orl(V_{\gamma,b}).
\end{equation}
By a straightforward generalization of \cite[Lemma~2.51]{Par19} to the case with several positive punctures, the orientation line $\orl(\delbar_{J}) = \orl(D^\conn)$ is canonically isomorphic to 
$$\orl(\Gamma^+;\Gamma^-) \coloneqq \bigotimes\limits_{\gamma\in \Gamma^+}\orl_\gamma\otimes \bigotimes\limits_{\gamma\in \Gamma^-}\orl_\gamma\dul$$
This completes the proof of Theorem~\ref{thm:pardon-gkc}\eqref{gkc-orientation}.

\section{A contact flow category and bimodules}\label{sec:sft-flow}

In this section we associate a flow category to a contact manifold $(Y,\lambda)$ and a flow bimodule to an exact symplectic cobordism $(X,\omega)$. The objects are finite sequences of Reeb orbits in either case, and morphisms are buildings of genus zero. Due to a technical obstruction, we can only construct the flow category and bimodule after restricting the action of the Reeb orbits. Thus, in \textsection\ref{subsec:colimit-proof}, we show that the ``full'' contact flow category can be obtained via a colimit.

\subsection{Flow categories}\label{sec:sliced-flow-cats} Our flow categories are more general than the flow categories of \cite{AB24}; their objects are orbifolds, and the composition is defined on a certain fiber product instead of the usual product. The precise definition is given in \S\ref{subsec:sliced-flow-cat}.

\subsubsection{Preliminaries} Recall that a Lie groupoid $X = [X_1\rightrightarrows X_0]$ is a groupoid where the set of objects and morphisms carry the structure of smooth manifolds, all structure maps are smooth, and the source and target maps $s,t \cl X_1\to X_0$ are submersions (whence the multiplication is a well-defined smooth map). We write $\phi \cl x\to y$ for $\phi\in X_1$ with source $x = s(\phi)$ and target $y = t(\phi)$. We will also abuse notation and write $y =\phi(x)$.\par
We call $X$ \emph{\'etale} if $s$ (and thus $t$) is a local diffeomorphism and \emph{proper} if $(s,t)\cl X_1 \to X_0\times X_0$ is proper. By \cite[Corollary~1.4]{Par24}, any \'etale proper Lie groupoid with a finite number of isotropy types is equivalent to a \emph{transformation groupoid} $[M/G] := [G\times M\rightrightarrows M]$ given by the action of a compact Lie group $G$ on a smooth manifold $M$ so that all points have finite isotropy. All Lie groupoids we consider in our main application are of this form. We need a weakening of the notion of smoothness. See also \cite{Swa21} for more details.

\begin{definition}
    A map $\pi\cl M\to B$ to a smooth manifold is \emph{of class rel--$C^1$} if there exist local charts $\{\phi_i\cl U_i \to B\}$ and $\{\varphi_i\cl U_i\times N\to M\}$ so that $\pi\g \varphi = \phi\g \pr_U$ , the transition maps $\varphi_{ij}(u) =\varphi_i\inv\varphi_j(u,\cdot)$ is of class $C^1$ for $u \in U_i\cap U_j$, and $u\mapsto \varphi_{ij}(u)$ is continuous in the $C^1_{loc}$-topology.
\end{definition}

A \emph{rel--$C^1$ Lie groupoid} is a topological groupoid $X$ equipped with a morphism $X\to B$ of groupoids so that $X_1/B_1$ and $X_0/B_0$ are rel--$C^1$ manifolds and the structure maps are of class rel--$C^1$. The following notions can be defined verbatim in the rel--$C^1$ setting. We refrain from doing so here for the sake of clarity, but we will use them in that generality in \S\ref{subsec:unstructured-flow-cat}.

\begin{definition}
    A \emph{slicing} of a morphism of Lie groupoids $f \cl X\to [Y/G]$ is a $G$-action on $X_0$, whose action map factors through $G\times X_0 \to X_1$ and which satisfies $f_1(g,x) = (g,f_0(y))$. 
    We say $f$ is \emph{sliced} if it is (implicitly) equipped with a slicing and \emph{sliceable} if it admits a slicing. We say a slicing is \emph{free} if the $G$-action on $X_0$ is free.
\end{definition}

\begin{ex}
    A morphism $f\cl G\to G'$ of groups is sliceable if there exists an inclusion $G'\hkra G$ of groups, which is a section of $f$. 
\end{ex}

In our main application, $X$ is the groupoid associated to an action of a compact Lie group $G_X$ on a smooth manifold $\wt X$ and similarly for $Y$, with $G_X = G_Y\times G'_X$ canonically.

\begin{definition}\label{de:sliced-fiber-product}
    Given two morphisms $f \cl X\to Z = [Z_0/G]$ and $g\cl Y\to Z$ that intersect transversely in $Z_0$ and are equipped with free slicings $j_f$ and $j_g$, we define the \emph{quotient fiber product} $X\ov{\times}_Z Y$ to be the Lie groupoid with objects given by 
    \begin{equation}
        (X\bar{\times}_Z Y)_{0} \coloneqq X_0\times_{Z_0} Y_0/G 
     \end{equation}
    and morphisms $(X\ov{\times}_Z Y)_1 \coloneqq f_1\inv(\ide)\times g_1\inv(\ide)$.
\end{definition}

\begin{ex}\label{ex:quotient-fiber-product}
    Suppose 
    $$f\cl X=[M/ G_Z]\to [Z_0/G_Z]\qquad\text{and}\qquad g\cl Y= [N/ G_Z]\to [Z_0/G_Z]$$ 
   are submersions on the level of objects, and $G_Z$ acts freely on $M\times Z$. Then, $X\ov{\times}_Z Y$ is the manifold $(M\times_{Z_0} N)/G_Z$, where $G_Z$ acts diagonally via the slicings on $M$ and $N$. If, moreover, the action of $G_Z$ on $Z_0$ is transitive, then we have a canonical isomorphism $\fg_Z \cong TZ_0$ and therefore, a canonical identification 
   \begin{equation}\label{eq:tangen-bundle-sliced-fiber-product}
       T(X\,\ov{\times}_{Z}\, Y) \,\cong\, TX\oplus TY
   \end{equation}
   of $G_Z$-vector bundles on $M\times_{Z_0} N$. The same is true if $M$ and $N$ are equipped with almost free actions by $G\times G_Z$ and $G'\times  G_Z$, respectively.
\end{ex}

This quotient fiber product commutes with filtered colimits in the following sense.

\begin{lemma}\label{lem:commute-with-colimits}
	 Suppose $\{X_\alpha\}_{\alpha\in A}$ and $\{Y_\beta\}_{\beta\in B}$ are filtered diagrams in $\normalfont\text{dOrb}_{/\cdot}$ with $S_{X_\alpha} = S_{X_{\alpha'}}$ and $T_{X_\alpha} = S_{Y_\beta}$ for all $\alpha,\beta$ and $T_{Y_{\alpha}} = T_{Y_\beta}$. If the colimits $X$ and $Y$ exist, then the colimit of $\{(X_\alpha\ov{\times} Y_\beta)\}_{(\alpha,\beta)\in A\times B}$ exists and is given by $X\ov{\times} Y$.
\end{lemma}

\begin{proof}
    Write $S = S_{X_\alpha}$, $Z = T_{X_\alpha} = [Z_1\rightrightarrows Z_0]$ and $T = T_{Y_\beta}$ for some $\alpha$ and $\beta$. Then, we have that $X_\alpha\ov{\times} Y_\beta$ is a quotient of the fiber product $X_\alpha\times_{Z_0} Y_\beta$. Since finite limits commute with filtered colimits, we have that 
    $$\colim\limits_{(\alpha,\beta)}X_\alpha\times_{Z_0} Y_\beta = X\times_{Z_0} Y.$$
    Since the morphisms in $\dOrb_{/\cdot}$ are compatible with the slicings of $t_{X_\alpha}$ and $s_{Y_\beta}$, the claim follows now directly from the definition.
\end{proof}

\subsubsection{A more general notion of flow category}\label{subsec:sliced-flow-cat}
Given the setup in the previous subsection, the definition of a flow category looks almost as before, except that we replaced the Cartesian product by the quotient fiber product and that we have a symmetric action on the objects. Note that in the case where each object is simply a point and the symmetric action is trivial, this recovers the original definition of a flow category as in \cite{AB24} or \cite{CJS95}.\par 
Recall from \cite{AB24} that a \emph{strong equivalence} $f \cl \bX\to \bY$ of derived orbifolds is a morphism of derived orbifolds (that is, spaces equipped with global Kuranishi charts) together with a morphism $\pi \cl Y\to X$ of the underlying thickenings so that $X\to Y$ corresponds to the inclusion of the zero section.

\begin{definition}\label{de:derived-orbifold-category}
	We define the category $\dOrb_{/\cdot}$ to have as objects derived orbifolds $\bX$ with corners equipped with two maps $s_{\bX}\cl \bX\to S_\bX = [\wt{ S}_{\bX}/G^S_\bX]$ and $t_{\bX}\cl \bX\to T_\bX= [\wt{ T}_{\bX}/G^T_\bX]$ to compact transitive orbifolds so that $s_\bX$ and $t_\bX$ are freely sliced submersions. The morphisms are given by strong equivalences $f \cl \bX\to \bY$ so that $s_\bY\g f = s_\bX$ and $t_\bY\g f = t_\bX$.
\end{definition}

Given $\bX,\bY$ in $\dOrb_{/\cdot}$ with $T_\bX = S_\bY$, we define
\begin{equation}\label{eq:new-product}
	\bX\,\ov{\times}\, \bY \,\coloneqq\, (\bX\,\ov{\times}_{S_\bY}\, \bY,S_\bX,T_\bY)
\end{equation}
with the canonical structural maps. This yields one part of our generalization. In order to encode the symmetric action concisely, we introduce the following definition. 

\begin{definition}\label{de:symmetric-set}
    Let $\Delta^*$ be the groupoid whose objects are pairs $(n,\prec)$, where $n\in \bZ_{\ge 0}$ is an integer and $\prec$ is a total ordering of $\{1,\dots,n\}$, and whose morphisms are order-preserving isomorphisms $(n,\prec)\to (n,\prec')$. A \emph{symmetric set} in $\cC$ with \emph{orbit set} $I$ is a functor $P \cl I \times \Delta^*\to \cC$, where we consider the set $I$ as a discrete category. We set $|P(i,\{(n,\prec)\})|:= n$ and will identify $P$ with its image in $\obj(\cC)$.
\end{definition}

In the future we will repeatedly use the observation that a disjoint union of symmetric sets is canonically a symmetric set.

\begin{definition}\label{de:sliced-flow-cat}
    A \emph{symmetric flow category} $\scM$ consists of a symmetric set $\cP$ of closed orbifolds, the objects of $\scM$, and for any $\alpha,\beta\in \cP$ a derived orbifold $\scM(\alpha,\beta)$ of morphisms equipped with 
    \begin{itemize}
        \item a proper function $E\cl \scM(\alpha,\beta)\to [0,\infty)$,
        \item free sliced submersions $s_{\alpha\beta} \cl \scM(\alpha,\beta)\to \alpha$ and $t_{\alpha\beta} \cl \scM(\alpha,\beta)\to \beta$
        \item isomorphisms 
        \begin{equation}\label{eq:symmetric-isomorphisms}
            \scM(\sigma\cdot\alpha,\beta) \cong \scM(\alpha,\beta)\cong \scM(\alpha,\sigma'\cdot\beta)
        \end{equation}for any $\sigma\in S_{|\alpha|}$ and $\sigma'\in S_{|\beta|}.$
    \end{itemize}
    The composition functions, defined on the quotient fiber product
     \begin{equation}\label{eq:composition}  
    \scM(\alpha,\beta)\,\ov{\times}_{\,\beta} \;\scM(\beta,\gamma) \;\to \;\scM(\alpha,\gamma),
    \end{equation}
    are smooth embeddings onto a codimension-$1$ boundary stratum $\del_{\beta}\scM(\alpha,\gamma)$ of $\scM(\alpha,\gamma)$. They are \begin{itemize} 
    \item additive with respect to $E$.
        \item compatible with the symmetric action in the sense that 
        \begin{center}\begin{tikzcd}
		\scM(\alpha,\beta)\,\ov{\times}_{\,\beta} \;\scM(\beta,\gamma) \arrow[rr,"\cong"] \arrow[dr,""]&&\scM(\alpha,\sigma\cdot\beta)\,\ov{\times}_{\,\sigma\cdot\beta} \;\scM(\sigma\cdot\beta,\gamma) \arrow[dl,""]\\ & \scM(\alpha,\gamma) \end{tikzcd} \end{center}
        commutes for any $\beta$ and $\sigma\in S_{|\beta|}$, whence
        $$\del_{\beta}\scM(\alpha,\gamma) = \del_{\sigma\cdot\beta}\scM(\alpha,\gamma).$$ 
        \end{itemize}
        Thus, we can write $\del_{[\beta]}\scM(\alpha,\gamma)$, using the orbit of $\beta$ under the symmetric action to indicate the boundary stratum. We require that these strata cover the boundary of $\scM(\alpha,\gamma)$ and that
    \begin{equation}\begin{tikzcd}
\scM(\alpha,\beta)\,\ov{\times}_{\,\beta}\,\scM(\alpha,\beta')\,\ov{\times}_{\,\beta'}\,\scM(\beta',\gamma)  \arrow[r,""] \arrow[d,""]&\scM(\alpha,\beta)\,\ov{\times}_{\,\beta}\,\scM(\beta,\gamma) \arrow[d,""]\\ 
        \scM(\alpha,\beta')\,\ov{\times}_{\,\beta'}\,\scM(\beta',\gamma) \arrow[r,""] &  \scM(\alpha,\gamma) \end{tikzcd} \end{equation}
is a pullback square for any $\alpha,\beta,\beta',\gamma\in \cP$.
\end{definition}

\begin{definition}\label{de:rel-c^1-flow-cat}
    We say a symmetric flow category $\scM$ is \emph{of class rel--$C^1$} if the morphism spaces are derived orbifolds of class rel-$C^1$, the maps $s_{\alpha\beta}$ and $t_{\alpha\beta}$ are of class rel--$C^1$ as described in \S\ref{sec:gluing} as well as the symmetric actions and composition maps are of class rel--$C^1$.
\end{definition}

\begin{remark}\label{rem:extends}
    Since this generalization only makes the notation heavier, we will phrase all remaining proofs in terms of smooth flow categories. However, the definitions and proofs carry over verbatim to the rel--$C^1$ setting.
\end{remark} 

\begin{ex}
    In our main example, each object of $\scM$ is of the form $B\Gamma =[E\Gamma/\bT_\Gamma]$, where $\Gamma = (\gamma_1,\dots,\gamma_k)$ is a sequence of Reeb orbits and $\bT_\Gamma = (S^1)^\Gamma$. The symmetric action permutes the ordering of the sequence, and the morphism spaces are manifolds $\wt\scM(\Gamma^-,\Gamma^+)$ equipped with an action by the Lie group $\bT_{\Gamma^-}\times\bT_{\Gamma^+}\times G_{\Gamma^-,\Gamma^+}$. The action of $\bT_{\Gamma^-}\times\bT_{\Gamma^+}$ comes from rotating the asymptotic markers. Then, the quotient fiber product over $\Gamma$ is the quotient of 
    $$\wt\scM(\Gamma^-,\Gamma)\times_{E\Gamma} \wt\scM(\Gamma,\Gamma^+)$$
    by the (free) diagonal $\bT_\Gamma$ action; it carries the induced action of $\bT_{\Gamma^-}\times G_{\Gamma^-,\Gamma}\times \bT_{\Gamma^+}\times G_{\Gamma,\Gamma^+}$.
\end{ex}

We briefly discuss stable complex structures on the flow categories of Definition~\ref{de:sliced-flow-cat}; see also \cite[Definition~3.8]{AB24}. The definition of framed structures as in \cite{AB24} can be adapted similarly. 

\begin{definition}\label{de:stable-complex-flow-cat}
    A \emph{stably complex lift} $\scM^U$ of a symmetric flow category $\scM$ consists of
    \begin{enumerate}[leftmargin=20pt]
        \item a symmetric set of virtual orbi-bundles $V_\alpha \to \alpha$, for object $\alpha$ of $\scM$, lifting the symmetric set $\cP$ of objects,
        \item a complex virtual vector bundle $I_{\alpha\beta}$ on $\scM(\alpha,\beta)$, 
        \item a vector bundle $W_{\alpha\beta}$ on $\scM(\alpha,\beta)$,
        \item a virtual vector space $U_{\alpha\beta}$ of the form $U_{\alpha\beta} = (0,\bR^{\{\beta\}})$
        \item an equivalence
        \begin{equation}\label{eq:equivalence-stable-complex}
            T\scM(\alpha,\beta)\oplus  V_\beta\oplus \bR^{\{\beta\}} \,\simeq\, (W_{\alpha\beta},W_{\alpha\beta})\oplus I_{\alpha\beta}\oplus V_\alpha
        \end{equation}
        of virtual vector bundles;
        \item compatible equivalences 
        \begin{equation}\label{eq:compatible-complex}
            I_{(\sigma\cdot\alpha)\beta}\cong I_{\alpha\beta}\cong I_{\alpha(\sigma'\cdot\beta)}\end{equation}
            \begin{equation}
           \label{eq:compatible-stabilising} W_{(\sigma\cdot\alpha)\beta}\cong W_{\alpha\beta}\cong W_{\alpha(\sigma'\cdot\beta)}
        \end{equation}
        of complex virtual vector bundles, respectively vector bundles, that lift~\eqref{eq:symmetric-isomorphisms} and intertwine the equivalences~\eqref{eq:equivalence-stable-complex}.
    \end{enumerate}
    Moreover, for any objects $\alpha,\beta,\gamma$ of $\scM$ we have split embeddings and isomorphisms
    \begin{gather}
        I_{\alpha\beta}\oplus I_{\beta\gamma}\,\to\, I_{\alpha\gamma},\\
        W_{\alpha\beta}\oplus W_{\beta\gamma}\,\to\, W_{\alpha\gamma},\\
        U_{\alpha\beta}\oplus U_{\beta\gamma}\,\cong\, (0,\bR^{\{\beta\}})\oplus U_{\alpha\gamma},
    \end{gather}
    respectively, that cover the composition map~\eqref{eq:composition} and are compatible with the equivalences~\eqref{eq:compatible-complex} and~\eqref{eq:compatible-stabilising}, so that the square
    \begin{equation}\label{dig:associativity-stable-complex}
    \begin{tikzcd}
		T\scM(\alpha,\beta)\oplus V_\beta\oplus \bR^{\{\beta\}}\oplus T\scM(\beta,\gamma)\oplus V_\gamma\oplus \bR^{\{\gamma\}} \arrow[r,""] \arrow[d,"\simeq"]& T\scM(\alpha,\gamma)\oplus V_\gamma\oplus \bR^{\{\gamma\}}
        \arrow[d,"\simeq"]\\ 
    (W_{\alpha\beta},W_{\alpha\beta})\oplus I_{\alpha\beta}\oplus V_\alpha\oplus (W_{\beta\gamma},W_{\beta\gamma})\oplus I_{\beta\gamma}\oplus V_\beta \arrow[r,""] & (W_{\alpha\gamma},W_{\alpha\gamma})\oplus I_{\alpha\gamma}\oplus V_\alpha \end{tikzcd} \end{equation}
    commutes, where we implicitly use the isomorphism~\eqref{eq:tangen-bundle-sliced-fiber-product}
\end{definition}

Since we will need the main result of \cite{AB24} that symmetric flow categories are the $0$-simplices of a stable $\infty$-category, we introduce the relevant adaptions of their definitions here. Given a finite set $\cP = (\cP_0,\dots,\cP_n)$ of sets, we define the category $\cP^{\bR_{\ge 0}}(p,q)$ for $p\in \cP_i$ and $q \in \cP_j$ with $i \leq j$ to have 
\begin{itemize}[leftmargin=20pt]
    \item objects linear trees $T$ with two exterior edges so that 
    \begin{itemize}[leftmargin=15pt]
        \item each inter edge labeled by an element of $\cP_k$ for $i\leq k \leq j$,
        \item the incoming edge is labeled by $p$ and the outgoing one by $q$,
        \item each vertex is labeled by $m \in \bR_{\ge 0}$ and a subset of $\{k+1,\dots,\ell-1\}$, where $k$ and $\ell$ are the labels of the edges adjacent to $v$;
    \end{itemize}
    \item morphism from $T$ to $T'$ given by collapsing a (possibly empty) sequence of consecutive edges so that the labels of the collapsing of $T$ agree with those of $T'$ as described in \cite[Definition~4.2]{AB24}.
    \end{itemize}
This yields a strict $2$-category $\cP^{\bR_{\ge 0}}$, where the horizontal composition is given by gluing two trees along the `common' exterior edge. If $\cP$ is a symmetric set, we can define the $2$-category $\cP^{\bR_{\ge 0}}$ analogously. When $\cP$ is a symmetric set, then there are canonical isomorphisms $\cP^{\bR_{\ge 0}}(\sigma\cdot p,q)\cong\cP^{\bR_{\ge 0}}( p,q)\cong \cP^{\bR_{\ge 0}}(p,\sigma'\cdot q)$ induced by the symmetric actions on the labels.

\begin{definition}\label{de:enriched-in-complicated}
    A \emph{(non-unital graded) symmetric category} $\cC$ \emph{enriched in $\dOrb_{/\cdot}$} consists of
    \begin{itemize}[leftmargin=20pt]
        \item a symmetric set $\obj(\cC)$ in the category of closed orbifolds, associating to $x$ the orbifold $B_x$,
        \item for any pair of objects $x,y$ an object $(\cC(x,y),B_x,B_y)$ of $\dOrb_{/\cdot}$ together with a proper energy map 
        \begin{equation}\label{} 
        	E_{xy}\cl \cC(x,y)\to \bR_{\ge0},
        \end{equation}        
        \item for any $\sigma\in S_{|x|}$ and $\sigma'\in S_{|y|}$ isomorphisms 
        \begin{equation}
        	\cC(\sigma\cdot x,y) \xra{\sigma^*} \cC(x,y)\xra{\sigma'_*} \cC( x,\sigma'\cdot y)
        \end{equation}
    that are compatible with the symmetric actions on $\obj(\cC)$ and which intertwine the energy maps.
        \item for any triple $x,y,z$ the composition map is a strong equivalence 
        \begin{equation}\label{eq:composition-general}
            \cC(x,y)\,\ov{\times}_{B_y}\,\cC(y,z)\to \del_{y}\cC(x,y), 
        \end{equation}
        which is a component of the boundary $\del_y \cC(x,z)$ so that $E_{xz}$ restricts to $E_{xy} + E_{yx}$ on the image and
        \begin{center}\begin{tikzcd}
        		\cC(x,y)\,\ov{\times}_{B_y}\,\cC(y,z) \arrow[rr,""] \arrow[dr,""]&&\cC(x,\tau\cdot y)\,\ov{\times}_{B_{\tau\cdot y}}\,\cC(\tau\cdot y,z) \arrow[dl,""]\\ & \cC(x,z) \end{tikzcd} \end{center}
        commutes for any $\tau\in S_{|y|}$. We can thus let $\del_{[y]}\cC(x,z)$ denote the image of~\eqref{eq:composition-general}, where $[y]$ is the orbit of $y$ under the symmetric action.
    \end{itemize}
We require that 
\begin{equation}\begin{tikzcd}
		\cC(x,y)\,\ov{\times}_{\,B_y}\,\cC(y,y')\,\ov{\times}_{\,B_{y'}}\,\cC(y',z)  \arrow[r,""] \arrow[d,""]&\cC(x,y)\,\ov{\times}_{\,B_y}\,\cC(y,z) \arrow[d,""]\\ 
		\cC(x,y')\,\ov{\times}_{\,B_{y'}}\,\cC(y',z) \arrow[r,""] &  \cC(x,z) \end{tikzcd} \end{equation}

is a pullback square and that $\del \cC(x,z) = \djun{[y]}\,\del_{[y]}\cC(x,z).$
\end{definition}

\begin{definition}[{cf. \cite[Definition~4.8]{AB24}}]\label{de:elementary-flow-simplex}
    An \emph{elementary symmetric $n$-flow simplex} consists of a sequence $\cP = (\cP_0,\dots,\cP_n)$ of symmetric sets, a closed orbifold $B_p$ for each $p \in \djun{i}\cP_i$, and the data of a symmetric category $\scX$ enriched in $\dOrb_{/\cdot}$ with objects given by the symmetric set $\obj(\scX) = \djun{i}\cP_i$, together with a strict $2$-functor $\cP_\scX\to \cP^{\bR_{\ge 0}}$, where $\cP_\scX$ is the stratifying category of the corners of $\scX$. We require the energy map 
    \begin{equation}
        E\cl \djun{q}\scX(p,q) \to \bR
    \end{equation}
    to be proper for any $p \in \obj(\scX)$, with a uniform lower bound independent of $p$.
\end{definition}

For the next definition, observe that we can identify the strata of the standard simplex $\Delta^n$ with subsets of $\{0,\dots,n\}$. Thus, letting $I = \{i_1,\dots,i_k\}$ be the subset associated to a stratum $\sigma\sub \Delta^n$, we write $\del^\sigma\cP = (\cP_{i_1},\dots,\cP_{i_k})$. We write $\epsilon_i$ for the stratum corresponding to the complement of the singleton $\{i\}$. 
Given an elementary symmetric flow simplex $\scX$ as in Definition~\ref{de:elementary-flow-simplex}, we let $\del^\sigma\scX$ be the restriction of $\scX$ to $\del^\sigma\cP$.
As pointed out in \cite[Remark~4.12]{AB24}, Definition~\ref{de:elementary-flow-simplex} is not quite sufficient since the simplicial set obtained from these elementary symmetric flow simplices might not satisfy the horn-filling property and is thus not an $\infty$-category. This is remedied by the following condition.

\begin{definition}[{cf. \cite[Definition~4.10]{AB24}}]\label{de:flow-simplex}
    A \emph{symmetric $n$-flow simplex} consists of a sequence $\cP = (\cP_0,\dots,\cP_n)$ of sets, an orbifold $B_p$ for each $p \in \djun{i}\cP_i$, an elementary symmetric flow simplex $\scX^\sigma$ lifting $\del^\sigma\cP$ for each stratum $\sigma$ of $\Delta^n$, and a functor $\scX_\tau\to \del^\tau X_\sigma$ enriched in $\dOrb_{/}$ for any strata $\tau \sub \sigma$. They lift the isomorphisms of stratifying categories and satisfy the usual associativity condition for any triple $\rho\sub \tau\sub \sigma$ of strata. For each $i \in \{0,\dots,n\}$ we define
    \begin{equation}\label{eq:face-map}
        \del_i\scX := (\scX_{\sigma})_{\sigma\subseteq \epsilon_i}.
    \end{equation}
\end{definition}

The definition of a \emph{structured symmetric flow simplex} in our setting is equivalent to Definition~\ref{de:flow-simplex} with the changes of \cite[Definition~4.18 and Definition~4.19]{AB24}. This uses the isomorphism~\eqref{eq:tangen-bundle-sliced-fiber-product}.
We can now define the semisimplicial set underlying the stable $\infty$-category $\normalfont\text{Flow}^\Sigma$. 

\begin{lemma}[Flow]\label{lem:flow-semi-simplicial} Letting $\text{Flow}^\Sigma_n$ be the set of symmetric $n$-flow simplices for $n \ge 0$ and taking $\del_i\cl \normalfont\text{Flow}^\Sigma_n \to \normalfont\text{Flow}^\Sigma_{n-1}$ to be the map given by~\eqref{eq:face-map}, defines a semisimplicial set $\normalfont\text{Flow}^\Sigma$.\qed
\end{lemma}

The proof is a straightforward verification. As the set-theoretic problems facing this definition are exactly the same as in \cite{AB24}, we refer to \cite[Remark~4.14]{AB24} for an approach on how to deal with them. We can now state the main result of this subsection, generalizing \cite[Theorem~1.6]{AB24}.

\begin{proposition}\label{prop:flow-stable}
    $\normalfont\text{Flow}^\Sigma$ admits the structure of a stable $\infty$-category whose morphisms are symmetric bimodules. The same is true for the stably complex case $\normalfont\text{Flow}^{\Sigma,U}$.
\end{proposition}

\begin{proof}
    The proof follows from observing that the arguments of \cite{AB24} carry through. We first observe that the arguments of \cite[\textsection6]{AB24}, in particular, Proposition~6.4 loc. cit., which lift their semisimplicial set $\Flow^\Sigma$ to a simplicial set, carry over verbatim to our setting. Indeed, the reasoning is formal, using \cite{Stei18} adapted as in \cite{AB24}, once one has constructed the inital and terminal degeneracies
    \begin{gather*}
        s_0 \cl \Flow^\Sigma_n \to \Flow^\Sigma_{n+1}\\
        s_n\cl \Flow^\Sigma_n \to \Flow^\Sigma_{n+1},
    \end{gather*}
    and the constructions of \cite[\textsection6.2]{AB24} extends to our setting by replacing the usual product by the sliced fiber product of Definition~\ref{de:sliced-fiber-product} and keeping track of the symmetric action on objects and morphism spaces.\par 
    The proof of the horn-filling property is the part that is the least obvious to adapt to our setting as it requires delicate geometric arguments. However, they are always about one morphism space at at a time. Thus, these arguments carry through when having symmetric sets of objects, where each object is simply a point, due to the naturality of the constructions and by replacing the products appearing in the definition of the map~(5.5) loc. cit. by quotient fiber products. The proof of \cite[Lemma~5.8]{AB24}, showing the claim of \cite[Theorem~5.1]{AB24} under the simplifying Assumption~1 loc. cit. requires Lemma~\ref{lem:commute-with-colimits} in our setting. The proof without the assumption (cf. \cite[Proposition~5.12]{AB24}) is about a single morphism space and thus is not affected by our generalization of the definition of a flow category. This shows that $\Flow^\Sigma$ as in Lemma~\ref{lem:flow-semi-simplicial} is an $\infty$-category.\par
    Recall that an $\infty$-category $C$ is \emph{stable} by \cite[Theorem~1.1.2.14 and Remark~1.1.2.15]{Lur09}) if 
    \begin{itemize}[leftmargin=25pt]
        \item $C$ has a zero object $*$,
        \item the suspension $\Sigma \cl C\to C$, taking $x$ to the pushout $\Sigma x$ of $*\to x\leftarrow *$, exists and is an auto-equivalence,
        \item any morphism in $C$ admits a cofiber.
    \end{itemize}   
    The unit of $\Flow^\Sigma$ is the flow category $\emst$ whose set of objects is empty, just as in \cite[\S7.2]{AB24}. The proof that the suspension $\Sigma$ exists and is an auto-equivalence is the same as the proofs of Lemma~7.8 and Lemma~7.9 in \cite{AB24}, by associating to the ``additional'' elements of $s_i\cP$ in Equations~(7.31) and~(7.32) the obvious orbifolds, and by replacing the products in Equation~(7.36) with quotient fiber products. The last property follows from the constructions of \cite[\S7.4]{AB24} by noting that they do not require the boundary strata to be products of other morphism spaces. This shows the claim in the unstructured case. The arguments of \cite[\S7.5]{AB24} allows us to lift the assertion to structured flow categories.
\end{proof}

\begin{lemma}\label{lem:all-colimits}
    The stable $\infty$-category $\normalfont\text{Flow}^\Sigma$ of unstructured flow categories admits all $\aleph_0$-small (homotopy) colimits. The same is true for stably complex flow categories.
\end{lemma}

\begin{proof}
By \cite[Proposition~4.4.3.2]{Lur09} and Theorem~\ref{prop:flow-stable}, it suffices to show that $\normalfont\text{Flow}^\Sigma$ has all $\aleph_0$-small coproducts. Its coproducts are given by disjoint unions of flow categories: if $\{\scX_i\}_{i\in \bN}$ is a set of symmetric flow categories with object sets $P_i$, let $\scX$ be the flow category with objects the symmetric set $P = \djun{i}P_i$ and morphism spaces 
\begin{equation*}
    \scX(p,q) = \begin{cases}
        \scX_i(p,q) & p,q\in P_i\\
        \emst & \text{otherwise,}
    \end{cases}
\end{equation*}
equipped with the obvious symmetric actions and composition maps. It is a straightforward verification that this is indeed a coproduct. The argument carries over verbatim to the stably complex case.
\end{proof}
\subsection{A symmetric flow category with bounded action}\label{subsec:unstructured-flow-cat}
We can now construct a flow category using the Kuranishi charts of the previous section. Throughout, $(Y,\lambda)$ is a closed contact manifold equipped with a non-degenerate contact form $\lambda$. We denote its Reeb vector field by $R$ and let $(\wh Y,\omega) = (\bR\times Y,d(e^s\lambda))$ be the symplectisation of $(Y,\lambda)$. 
Given $L > 0$, let 
$$\cP_L\coloneqq \{\Gamma = (\gamma_1,\dots,\gamma_k)\mid \gamma_i \text{ is a Reeb orbit, with action }\cA_\lambda(\gamma_i)\leq L\}$$ 
be the set of finite sequences of unparametrized Reeb orbits of action at most $L$.  Recall that given by a function $\Lambda\cl \Gamma^-\to \Gamma^+$ and a sequence $\beta =(\beta_\gamma)_{\gamma\in \Gamma^+}$ of relative homology classes, we defined the moduli space $\Mbar_{\sft}^{\, J}(\Gamma^+,\Gamma^-;\beta)_\Lambda$ of buildings with disconnected domains in Definition~\ref{de:disconnected-domains}.

\begin{theorem}\label{thm:flow-cat}
    Given $L > 0$ and a choice of $\lambda$-adapted almost complex structure $J$, there exists a symmetric flow category $\scM^{Y,\lambda}_{\leq L}$ of class rel--$C^1$ whose objects are elements of $\cP_L$ and whose morphism spaces are 
    \begin{equation*}\label{eq:morphisms}
        \scM_{\leq L}^{Y,\lambda}(\Gamma^-,\Gamma^+)\,\coloneqq \,\djun{\Lambda}\djun{\beta}\Mbar_{\sft}^{\, J}(\Gamma^+,\Gamma^-;\beta)_\Lambda
    \end{equation*}
    for any $\Gamma^+,\Gamma^-\in\cP_L(Y)$, where the disjoint unions range over functions $\Lambda\cl \Gamma^-\to \Gamma^+$ and sequences $\beta =(\beta_\gamma)_{\gamma\in \Gamma^+}$ of relative homology classes.
\end{theorem}

Observe that the order of $\Gamma^-$ and $\Gamma^+$ is opposite the usual one in the morphism space of the symmetric flow category. We do so to be compatible with the conventions of \cite{AB24}. We use the energy functions
$$E = E_{\Gamma^-\Gamma^+}\cl \scM^{Y,\lambda}_{\leq L}(\Gamma^- , \Gamma^+) \to \subset [0,\infty) : \scA_\lambda(\Gamma^+)- \scA_\lambda(\Gamma^-),$$ 
which are clearly additive under composition and proper by \cite{BEH03}. The composition is given by the canonical ``stacking'' of buildings, as shown in 
\begin{figure}[H]
    \centering
    
    \def\svgscale{0.7}
    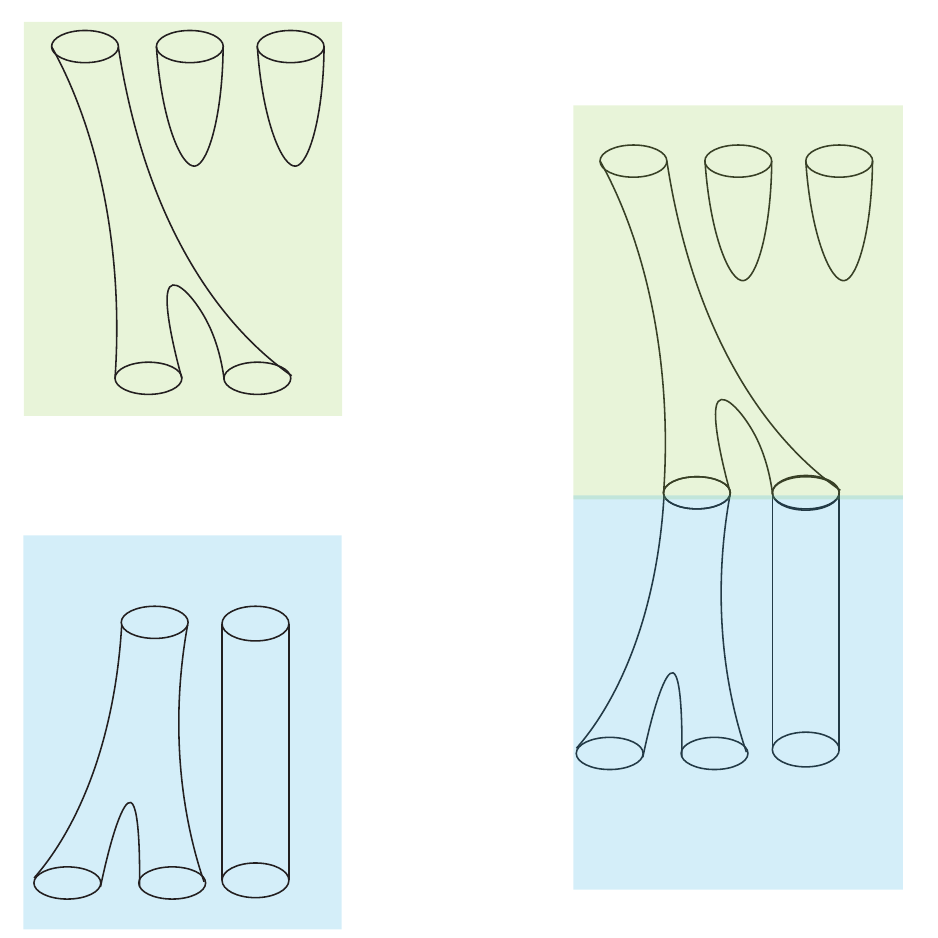

    \caption{Composition in $\scM_{\le L}^{Y,\lambda}$ } 
    \label{fig:flocomp}
\end{figure}

\begin{remark}
The bound by $L$ selects Reeb orbits with action less than $L$ but does not constrain the length (and action) of $\Gamma$. Any Reeb orbit with action greater than $L$ will not occur in the moduli space $\Mbar^{\, J}_{\sft}(\Gamma^+,\Gamma^-,\beta)_\Lambda$ if $\cA_\lambda(\gamma_i^\pm) \leq L$ for all $\gamma_i^+$. This is due to the fact that each connected component of the holomorphic buildings we consider has exactly one positive puncture. 
\end{remark}
 
In order to prove Theorem~\ref{thm:flow-cat}, we construct the global Kuranishi charts for the moduli spaces~\eqref{eq:morphisms} inductively as in \cite{BX22}. Fix a pre-perturbation datum $\fD$ consisting of
\begin{itemize}
    \item an action bound $L > 0$;
    \item a $\cP_L$-integral approximation $\wt\lambda$ of $\lambda$,
    \item a translation-invariant $J$-linear connection $\wh\conn$ on $T\wh Y$,
    \item a prime number $p \gg \max \{\cA_{\wt\lambda}(\gamma) \mid \gamma \in \cP_L \}$. 
\end{itemize}
Recall that $\Mbar_{\sft}^{\, J}(\Gamma^+,\Gamma^-;\beta)_\Lambda$ is the moduli space of disconnected leveled buildings, where each connected component has one positive puncture asymptotic to some $\gamma\in \Gamma^+$ and negative punctures asymptotic to the elements $\gamma'\in \Lambda_\gamma \coloneqq\Lambda\inv(\gamma)$ and of degree $\beta_\gamma$.
 
\noindent We will from now on drop the mention of degree to make the notation more tractable. Define a partial order $\prec$ on $\cP_L$ by  
$$\Gamma \prec \Gamma' \qquad \dimp \qquad\exists\Lambda: \Mbar^{\, J}_{\sft}(\Gamma',\Gamma)_\Lambda \neq \emst$$
and define the ``norm"
$$\norm{(\Gamma,\Gamma')} := \sup\set{k\in \bN_0\mid \exists \Gamma_0,\dots,\Gamma_k \in \cP_L: \Gamma = \Gamma_0 \prec \Gamma_1 \prec\dots\prec \Gamma_k \prec \Gamma'}$$
 on $\cP_L\times \cP_L$. Note that $\Gamma \nprec \Gamma'$ if and only if $\norm{(\Gamma,\Gamma')} = -\infty$.
\noindent We use induction on $\norm{(\Gamma^-,\Gamma^+)}$ to construct perturbation data for the moduli spaces $\Mbar_{\sft}^{\, J}(\Gamma^+,\Gamma^-)_\Lambda$. 
 If $\norm{(\Gamma^-,\Gamma^+)} = -\infty$, there is nothing to do. Given a pair $\Gamma^\pm$ of norm $0$ and a partition $\Lambda\cl \Gamma^-\to \Gamma^+$, extend $(\wt\lambda,\wh\conn,p)$ to an arbitrary perturbation datum $\alpha_{\Lambda}$ for $\Mbar^{\, J}(\Gamma^+,\Gamma^-)_\Lambda$. Let 
 $$\cK_\Lambda =  (\wh G_{\Lambda},\cT_{\Lambda},\cE_{\Lambda},\fs_{\Lambda})$$ 
 be the associated global Kuranishi chart given by Proposition~\ref{prop:disconnected-buildings}. Given permutations $\sigma^\pm$, we define $\cK_{\sigma^+\g \Lambda\g \sigma^-}$ to be the global Kuranishi chart obtained from $\cK_{\Lambda}$ by permuting the labels of the positive/negative marked punctures according to $\sigma^+$ and $\sigma^-$ respectively. Note that we also have to change the partition $\Lambda$.\par 
 Returning to $\cK_{\Lambda}$, recall that its thickening (and obstruction bundle) are rel--$C^1$ over the base space $\cBS_\Lambda$, constructed in \textsection\ref{subsec:base-for-disconnected}.
It admits canonical smooth maps
\begin{equation}\label{eq:base-total-degree}
    \cBS_\Lambda\to \p{\gamma\in \Gamma^+}{\cBR_{\gamma,\Lambda_{\gamma}}} \to \p{\gamma\in \Gamma^+}{\cB(d_{\Lambda_\gamma})}
\end{equation}
where $\cB(d) \sub \Mbar_0(\bP^d,d)$ was defined in \textsection\ref{subsec:base-space} and the degree is given by
\[d_{\Lambda_\gamma} = 1+\#\Lambda_\gamma-2 +  p\,\lbr{\cA_{\wt\lambda}(\gamma)-\s{\gamma'\in \Lambda_\gamma}{\cA_{\wt\lambda}(\gamma')}}.\]

In particular, the evaluations maps on the products of $\cB(d)$ induce smooth $\PU(d_{\Lambda_\gamma}+1)$-equivariant evaluation maps 
\begin{equation}\label{eq:evaluation-positive}
    \eva^+_\gamma\cl\cBR_{\gamma,\Lambda_{\gamma}}\to \bP^{d_{\Lambda_{\gamma}}}
\end{equation}
for $\gamma\in \Gamma^+$ and 
\begin{equation}\label{eq:evaluation-negative}
    \eva^-_{\gamma'}\cl \cBR_{\gamma,\Lambda_{\gamma}}\to \bP^{d_{\Lambda_{\gamma}}}
\end{equation}
for $\gamma'\in \Lambda_\gamma$.

\subsubsection{Embeddings of base spaces}\label{subsec:embeddings-base} We will lift the evaluation maps $\eva^\pm_{\gamma'}$ to smooth maps to the spheres $S^{2d_{\Lambda_\gamma}+1}$ as in \cite[Definition~5.2.1]{BX22} but via a different construction. Then, we will use these lifts to construct embeddings of base spaces that will induce the composition maps of the symmetric flow category later on. 
Fix thus $\gamma\in \Gamma^+$ and set $d \coloneqq d_{\Lambda_\gamma}$. Let $J_0$ be the standard complex structure on 
\begin{equation}\label{eq:symplectisation-sphere}
    \bR \times S^{2d+1}\,\cong\, \cO_{\bP}(-1)\sm 0 \,\cong\,\bC^{d+1}\sm\{0\},
\end{equation}
considered as the symplectization of the contact manifold $S^{2d+1}$ equipped with the round (Morse-Bott) contact form. We denote the moduli space of Pardon buildings in $\bR \times S^{2d+1}$ by $\Mbar^{J_0}_{\gamma\sqcup\Lambda_\gamma}(S^{2d+1})$ and we equip it with the Gromov topology as described in \cite{Par19}. Note that we do not fix the Reeb orbits the punctures have to be asymptotic to, nor any base point of the Reeb orbits. The quotient map $S^{2d+1}\to \bP^d$ induces a continuous map 
\begin{equation}\label{eq:map-to-stable-maps}
    \Mbar^{J_0}_{\gamma\sqcup\Lambda_\gamma}(S^{2d+1})\to \Mbar_{0,\gamma\sqcup\Lambda_\gamma}(\bP^d,d)
\end{equation}
and we denote by $\cBRB_{\gamma, \Lambda_\gamma}$ the preimage of $\cB_{\gamma\sqcup\Lambda_\gamma}(d)$ under \eqref{eq:map-to-stable-maps}.
By construction, the map~\eqref{eq:map-to-stable-maps} lifts to a continuous map $q\cl \cBRB_{\gamma, \Lambda_\gamma} \to \cBR_{\gamma,\Lambda_\gamma}$.

\begin{lemma}\label{lem:principal-bundle} The map
 $q\cl \cBRB_{\gamma, \Lambda_\gamma} \to \cBR_{\gamma, \Lambda_\gamma}$ is a principal $S^1$-bundle.  
 \end{lemma}

\begin{proof} The $S^1$-action on $S^{2d+1}$ induces a continuous $S^1$-action on $\cBRB_{\gamma, \Lambda_\gamma}$, with respect to which the evaluation $\eva^+_\gamma$ at the positive puncture is equivariant. Since $S^1$ acts freely on $S^{2d+1}$, the action on $\cBRB_{\gamma, \Lambda_\gamma}$ is free as well. The map $q$ is clearly $S^1$-invariant. Since $\cBRB_{\gamma, \Lambda_\gamma} \to \cBR_{\gamma, \Lambda_\gamma}$ is the pullback of the proper map~\eqref{eq:map-to-stable-maps}, it is proper, and thus so is $q$.\par
It remains to show that the descent of $q$ to the quotient is bijective. 
To see surjectivity, fix $[\varphi,C,b,m] \in \cBR_d$, where $b$ denotes the asymptotic markers and $m$ denotes the matching isomorphisms, and let $T$ be the underlying decorated graph. Recall that to every vertex $v$ with one incoming edge and $k$ outgoing edges we can associate numbers $d^+_v, d^-_{v,1}, \dots d^-_{v,k}$ such that 
\[\deg (\varphi_v)  = d^+_{v} - d^-_{v,1} - \dots d^-_{v,k} + (k+1)-2 = (d^+_{v}-1) - (d^-_{v,1}-1) - \dots (d^-_{v,k}-1).  \] 
Thus, there is a $\bC^*$-family of sections of $\varphi_v^* \cO(-1)$ with pole of order $d^+_v-1$ at $z_{v,e}\in \C_v$ for $e\in E^+_v$ and zeroes of order $d^-_{v,i}-1$ at $z_{v,e_i} \in \C_v$ for $e_i\in E^-_v$ with $i\in \{ 1, \dots ,k\}$. Note that the degrees $d_{v,i}^-$ and $d_v^+$ are always at least $3$ by our choice of integer $p$.
Using the identification \eqref{eq:symplectisation-sphere}, we can construct the lifts of $\varphi_v$ by choosing sections of $\varphi^*_v \cO(-1)$ satisfying the pole-zero arithmetic described above and then projecting to $S^{2d+1}$. Note that the choice of the meromorphic section over $\varphi_v^*\cO(-1)$ is only well defined up to the $\bC^*$-action given by scaling. Thus, writing $D_v$ for the divisor of punctures of $C_v$, a lift $\Phi_v  : B_{D_v} C_v  \to S^{2d+1}$ is only well defined up to the Hopf action. Fix a vertex $v \in V(T)$ and a lift $\Phi_v$. By using the matching condition at the punctures, we can find unique lifts $\Phi_{v'}$ for vertices $v'\neq v$ such that $\Phi \cl \cup_{v\in V(T)} C_v\sm D_v \to S^{2d+1}$ is a Pardon building. From the construction it is clear that there is an $S^1$ family of such lifts of $[\varphi,C,b,m]$. The injectivity of the descent of $q$ follows directly.
\end{proof}

\begin{lemma}\label{lem:sphere_pullback}
    For $p \in \{\gamma\}\sqcup \Lambda_\gamma$, let $\eva_{p}$ be the respective evaluation map of \eqref{eq:evaluation-positive} or \eqref{eq:evaluation-negative}. Then the following holds.
    \begin{enumerate}[label=\normalfont\arabic*),leftmargin=20pt,ref=\arabic*]
        \item $\cBRB_{\gamma,\Lambda_\gamma}$ is isomorphic as a topological principal $S^1$-bundle to $\eva_{\gamma'}^*S^{2d+1}$,
        \item If we equip $\cBRB_{\gamma,\Lambda_\gamma}$ with the smooth structure pulled back from $\eva_\gamma^*S^{2d+1}$, then there exists a $U(d+1)$-equivariant smooth submersion $\wt\eva_{\gamma'} \cl \cBRB_{\gamma,\Lambda_\gamma}\to S^{2d+1}$ so that 
        \begin{equation*}\begin{tikzcd}
		\cBRB_{\gamma,\Lambda_\gamma} \arrow[r,"\wt\eva_{\gamma'}"] \arrow[d,""]&S^{2d+1} \arrow[d,""]\\ 
        \cB_{\gamma,\Lambda_\gamma} \arrow[r,"\eva_{\gamma'}"] & \bP^d \end{tikzcd} \end{equation*}
        commutes for each $p \in \Lambda_\gamma$.
        \item The restriction of $\cBRB_{\gamma,\Lambda_\gamma}$ to \begin{equation}\label{eq:slice-of-base}
            \scBR_{\gamma,\Lambda_\gamma} := \eva_\gamma\inv(\{[1:0:\dots:0]\})
        \end{equation} is a trivializable principal bundle with a compatible $\normalfont\text{U}(d)$-action, where $\normalfont\text{U}(d) \hkra \PU(d+1)$ is the canonical embedding.
    \end{enumerate}
\end{lemma}

\begin{proof} The first assertion is a consequence of Lemma~\ref{lem:principal-bundle} and the fact that the induced map $\cBRB_{\gamma,\Lambda_\gamma}\to \eva_{\gamma'}^*S^{2d+1}$ is equivariant, hence a morphism of principal bundles. Taking $p = \gamma$, this allows us to pull back the smooth structure on $\eva_\gamma^*S^{2d+1}$ to $\cBRB_{\gamma,\Lambda_\gamma}$. Note that $\wt\eva_\gamma$ is smooth with respect to this smooth structure. Now, for $p \in \{\gamma\}\sqcup \Lambda_\gamma$, we obtain by \cite[Proposition~I.13]{mw09} a diffeomorphism $\psi_{\gamma'}\cl \cBRB_{\gamma,\Lambda_\gamma}\to \eva_{\gamma'}^*S^{2d+1}$, which is the canonical map if $p = \gamma$. By using an equivariant version of the Whitney approximation theorem, we can ensure that $\psi_{\gamma'}$ is $U(d+1)$-equivariant. Thus, we can define $\wt\eva_{\gamma'}$ to be the composition
\begin{equation}\label{eq:lift-of-evaluation}
    \cBRB_{\gamma,\Lambda_\gamma}\,\xra{\psi_{\gamma'}}\, \eva_{\gamma'}^*S^{2d+1}\,\to\, S^{2d+1}.
\end{equation}
Since the second map in~\eqref{eq:lift-of-evaluation} is a smooth submersion, the whole composition is a smooth submersion. It lifts $\eva_{\gamma'}$ by construction.
Now, let 
\[\wt{\eva}_\gamma\cl \cBRB_{\gamma,\Lambda_\gamma}\to S^{2d+1}\] 
be the evaluation map at the positive puncture. Then the restriction of $q$ to 
\[\wt{\eva}_\gamma\inv(\{(1,0,\dots,0)\}) \,\to \scBR_{\gamma,\Lambda_\gamma} \]
  is an isomorphism, whence the claim follows.  
\end{proof}

\begin{remark}
    Note that we do not use regularity of $\cBRB_{\gamma,\Lambda_\gamma}$ to obtain a smooth structure and that the lifted evaluation maps $\wt\eva_{\gamma'}$ might not agree with the natural evaluation maps on this moduli space of buildings. While we believe this to be true, we do not need and hence do not show it.
\end{remark}
 
\begin{corollary}[Spherical evaluation map]\label{cor:normalized-evaluation} 
There exists a lift of the evaluation map $\eva^-\cl \scBR_{d_{\Lambda_\gamma}}\to (\bP^{d_{\Lambda_\gamma}})^{\Lambda_\gamma}$ to a smooth $U(d_{\Lambda_\gamma})$-equivariant map 
\begin{equation}
    \wt\eva^-\cl \scBR_{\gamma,\Lambda_\gamma}\to (S^{2d_{\Lambda_\gamma}+1})^{\Lambda_\gamma}  .
\end{equation}
\end{corollary}

\begin{proof}
    We may assume without loss of generality that $\Lambda\inv(\gamma)\neq \emst$. Take $\wt\eva^-$ to be the product of the compositions
    \begin{equation*}
        \scBR_{\gamma,\Lambda_\gamma}\,\xra{\sim}\, \wt\eva_{\gamma'}\inv(\{(1,0,\dots,0)\}) \,\xra{\wt\eva_{\gamma'}}\, S^{2d_{\Lambda_\gamma}+1}
    \end{equation*}
    over all $p \in \Lambda_\gamma$. Since $\wt\eva_{\gamma'}$ is $U(d+1)$-equivariant, its restriction to $\wt\eva_{\gamma'}\inv(\{(1,0,\dots,0)\})$ is $U(d)$-equivariant. As the first isomorphism is the inverse of a $U(d)$-equivariant map, the resulting map is equivariant as desired.
\end{proof}

Set $\scBR_\Lambda\coloneqq\p{\gamma\in \Gamma^+}{\scBR_{\gamma,\Lambda_\gamma}}$
and let $\scBS_\Lambda = \cBS_\Lambda\times_{\cBR_\Lambda} \scBR_\Lambda$ be the pullback, where $\cBS_\Lambda$ and the map to $\cBR_\Lambda$ were constructed in \textsection\ref{subsec:base-for-disconnected}. Equivalently, $\scBS_{\Lambda}$ is a fiber of the (submersive) evaluation map $\cBS_{\Lambda}\to \p{\gamma\in \Gamma^+}{\bP^{d_{\Lambda_\gamma}}}.$ Abbreviate
\begin{equation}
    \cG_\Lambda \coloneqq \p{\gamma\in \Gamma^+}{\PGL_{d_{\Lambda_\gamma}+1}(\bC)_{[1:0:\dots:0]}}\qquad \qquad G_{\Lambda} \coloneqq \p{\gamma\in \Gamma^+}{\text{U}(d_{\Lambda_\gamma})}
\end{equation}
which acts smoothly on $\scBR_\Lambda$ and $\scBS_\Lambda$ via the embedding 
    $$A\in U(d) \;\mapsto\; \begin{pmatrix} 1 & 0 \\ 0 & A\end{pmatrix} \in \PU(d+1)$$
onto the stabilizer of $[1:0:\dots:0]\in \bP^d$. Since $\PU(d+1)$ acts transitively on $\bP^d$ for any $d\geq 1$, the inclusion induces an isomorphism 
\begin{equation}\label{eq:comparing-group-quotients}
    \cG_\Lambda/G_\Lambda \,\cong\, \p{\gamma\in \Gamma^+}{\PGL_{d_{\Lambda_\gamma}+1}(\bC)/\PU(d_{\Lambda}+1)}.
\end{equation}

\begin{proposition}\label{prop:embedding-base-spaces}
    Given a factorization $\Gamma^- \xra{\Lambda^0}\Gamma_0 \to \dots \to\Gamma_{k-1}\xra{\Lambda^k} \Gamma^+$, there exists a smooth embedding
\begin{equation}\label{eq:embedding-base-spaces}
    \Psi\cl \scBS_{\Lambda^0}\qtimes{\;\bT_{\Gamma_0}}\dots\qtimes{\;\bT_{\Gamma_{k-1}}}\scBS_{\Lambda^k}\,\hkra\, \scBS_{\Lambda},
\end{equation}
where $\scBS_{\Lambda^{j-1}}\qtimes{\;\;\bT_{\Gamma^{j-1}}}\scBS_{\Lambda^{j}}$ is the quotient of the product by the diagonal action of $\bT_{\Gamma^{j-1}}$. The embedding maps to a boundary stratum of codimension $k$ and is equivariant with respect to the block matrix inclusion
\begin{equation}\label{eq:inclusion-of-groups}
    G_{\Lambda^k}\times \dots \times G_{\Lambda^0}\hkra G_{\Lambda}.
\end{equation}
\end{proposition} 


\begin{proof}
We construct this map first at the level of the base spaces $\cBR$ and then lift it to~\eqref{eq:embedding-base-spaces}. We only discuss the case of $k = 1$, the more general case follows analogously. 
Suppose thus $\Lambda^1 \cl\Gamma\to \Gamma^+$ and $\Lambda^0 \cl\Gamma^-\to \Gamma$ are two partitions with $\Lambda = \Lambda^1\g\Lambda^0$. We will consider the case where $\Gamma^+ =\{\gamma^+\}$; the case of more components in the top level is a straightforward generalization. Assume without loss of generality that $\Gamma\neq \emst$.\par 
For any $d_0,\dots,d_n \geq 1$ we have a smooth map 
\begin{gather}
    \notag F_i\cl  S^{2d+1} \times \bP^{d_i}\,\to\, \bP^{d_0 +\dots + d_n} \\
    (a,[z]) \mapsto [a_0z_0:\dots: a_{d_0}z_0:0:\dots:0:z_1:\dots:z_{d_i}:0:\dots:0],
\end{gather}
considering $S^{2d+1}$ as a subset of $\bC^{d+1}$ and inserting $z_1,\dots,z_{d_i}$ in the positions $(d_0+\dots+d_{i-1}+1)$ to $(d_0+\dots+d_i+1)$. Now, given $\varphi_{\gamma^+} = [\check{\varphi}_{\gamma^+},C_{\gamma^+},m_{\gamma^+}]\in \scBR_{{\gamma^+},\Gamma}$ and $\varphi_{\gamma} = [\check{\varphi}_{\gamma},C_{\gamma},m_{\gamma}]\in \scBR_{\gamma,\Lambda^0_{\gamma}}$ for $\gamma\in \Gamma$, we define 
$$\wt{\varphi}\cl C_{\gamma^+}\sqcup \djun{\gamma}C_{\gamma}\,\to\, \bP^{d_{\Lambda_{\gamma^+}}}$$
by 

\begin{equation} \wt\varphi|_{C_{\gamma}} = 
    \begin{cases}
        j\g \varphi_{\gamma^+} \quad & \gamma = \gamma^+\\
        F_{\gamma^+}(\wt\eva_{\gamma}(\varphi_{\gamma^+},-)\g \varphi_{\gamma}\quad & \gamma\in \Gamma,
    \end{cases}
\end{equation}
where $j\cl \bP^{d_{\Gamma}}\hkra \bP^{d_{\Lambda_{\gamma^+}}} :[z]\mapsto [z:0\dots:0]$ is the inclusion into the first homogeneous coordinates.
By Corollary~\ref{cor:normalized-evaluation} and the definition of $\scB$, this descends to a holomorphic map $\check\varphi\cl C\to \bP^{d_{\Lambda_{\gamma^+}}}$ on the curve $C$ obtained from clutching $C_0$ and the $C_{\gamma}$ at the respective marked point. Using the lift of the clutching map to the real-oriented blow-up, we obtain the map \eqref{eq:embeddings-pardon-base-for-stable-complex}. 
Since 
\[j^*\cO_{\bP^{d_{\Lambda_{\gamma^+}}}}(1) = \cO_{\bP^{d_{\Gamma}}}(1)\qquad\quad \qquad F_{\gamma}(a_{\gamma},\cdot)^*\cO_{\bP^{d_{\Lambda_{\gamma^+}}}}(1) =\cO_{\bP^{d_{\Lambda^0_{\gamma}}}}(1),\] 
the map $\Psi_{\Lambda^1,\Lambda^0}$ is a well-defined map 
\begin{equation}\label{eq:embedding-pardon-base}
    \scBR_{\Lambda^0}\qtimes{\,\bT_\Gamma}\scBR_{\Lambda^1}\to\scBR_\Lambda
\end{equation}It is smooth and equivariant with respect to the inclusion~\eqref{eq:inclusion-of-groups} by construction. To lift it, observe that the universal family $\cC\to \cB(\Lambda)$ is canonically embedded in the product $\cB(\Lambda)\times \p{\gamma\in \Gamma}{\bP^{d_{\Lambda_\gamma}}}$. Pulling back the Fubini--Study metric on projective space, we obtain for any asymptotic marker a canonical lift to the normal bundle of the respective divisor. Using the explicit lift description in the proof of Lemma~\ref{lem:disconnected-base-well-defined} and in the discussion below Theorem~\ref{thm:corner-blowup}, we can lift the map~\eqref{eq:embedding-pardon-base} to~\eqref{eq:embedding-base-spaces} by incorporating the lengths into the matching isomorphisms. Concretely, recall that the matching isomorphisms at the newly created level jump are given by the sequence $(m_\gamma)_{\gamma\in \Gamma} = ([b^0_{\gamma}\otimes b^1_{\gamma}])_{\gamma}$. The $b^i_\gamma\in (T_{z_\gamma}C^i\sm0)/\bR_{>0}$ are the respective asymptotic markers, and $[b^0_{\gamma}\otimes b^1_{\gamma}]$ indicates their image in the quotient by the $S^1$-action. Writing $\wt b^i_\gamma$ for the lift to $T_{z_\gamma}C^i$ and given $a \in \bR^{\Gamma}_{>0}/\bR_{>0}$, we define the `refined matching isomorphism' $\wt m$ to be the image of $(\wt b^0_\gamma\,\otimes\, a_\gamma\wt b^1_\gamma)_{\gamma\in \Gamma}$ in the quotient
$$\bP_{>0}\lbr{\bigoplus\limits_{\gamma\in \Gamma} \sigma_\gamma^*T_{\scC\inn_{\Lambda^0}/\scB_{\Lambda^0}}\otimes \sigma_\gamma^*T_{\scC\inn_{\Lambda^1}/\scB_{\Lambda^1}}}/\bT_\Gamma,$$
where $\bP_{>0}$ of a vector bundle was defined in Equation~\eqref{eq:positive-part}. This completes the proof.
\end{proof}

\begin{lemma}\label{lem:associativity-of-embeddings} 
    For any factorization $\Gamma^-\xra{\Lambda^0}\Gamma_0\xra{\Lambda^1}\Gamma_1\xra{\Lambda^2}\Gamma^+$ of $\Lambda$, the square
    \begin{equation}\begin{tikzcd}
    \scBS_{\Lambda^0}\qtimes{\bT_{\Gamma_0}} \scBS_{\Lambda^1} \qtimes{\bT_{\Gamma_1}} \scBS_{\Lambda^2} \arrow[d,"\Psi\times \ide"] \arrow[rr,"\ide\times \Psi"]&& \scBS_{\Lambda^0}\qtimes{\bT_{\Gamma_0}}\scBS_{\Lambda^{21}}\arrow[d,"\Psi"]\\ 
    \scBS_{\Lambda^{10}}\qtimes{\bT_{\Gamma_{10}}}\scBS_{\Lambda^2} \arrow[rr,"\Psi"] && \scBS_{\Lambda} 
    \end{tikzcd} \end{equation}
    commutes and is a pullback square.
\end{lemma}

\begin{proof}
    The argument is analogous to the proof of \cite[Proposition~5.22]{BX22}.
\end{proof}

Given a sequence $\Lambda^* = \{ \Lambda_i \}_{i=0}^k$ of partitions composing to $\Lambda \cl\Gamma^-\to \Gamma^+$, define 

\begin{equation}
    \cc{\scB}^\bR_{\Lambda^*}\,\coloneqq\, G_\Lambda\times_{G_{\Lambda^*}} \im(\Psi_{\Lambda^*}),
\end{equation}
where $G_{\Lambda^*} = \prod_{j = 0}^k G_{\Lambda^j}$. For the next result, we have to recall some definitions from \cite[\S 5.2.5]{BX22}. For all $d\geq 0$, set ${\bm Q}_d \coloneqq \tilde {\bm Q}_d/ \bR_{> 0}$, where
\begin{equation*}
     \tilde {\bm Q}_d:= \set{ \tilde h \in \bC^{(d+1)\times (d+1)}\ |\ \tilde h^\ast   = \tilde h,\ \tilde h_{00} \neq 0 }.
\end{equation*}
We use the convention that the indices of the Hermitian matrix $\tilde h \in \tilde {\bm Q}_d$ range from $0$ to $d$. The multiplicative group ${\bR}_{> 0}$ acts on $\tilde {\bm Q}_d$ by scalar multiplication on each entry. The $\bR_{> 0}$-orbit of $\tilde h \in \tilde {\bm Q}_d$ is denoted by $[\tilde h]$. We identify ${\bm Q}_d$ with
\begin{equation*}
{\bm Q}_d^*:= \Big\{ h \in \bC^{(d+1)\times (d+1)}\ |\ h^\ast = h,\ h_{00} = 0 \Big\} 
\end{equation*}

in the way that a Hermitian matrix $h$ with $h_{00} = 0$ is identified with the $\bR_{> 0}$-orbit of $\tilde h = I_{d+1} + h$. Then ${\bm Q}_d$ is a real vector space with dimension equal to $d^2 + 2d$. 

\begin{lemma}\label{lem:boundary-up-to-stabilisation} There exists a smooth $G_\Lambda$-vector bundle 
\begin{equation}
    Q_{\Lambda^*}\to \cc{\scB}^\bR_{\Lambda^*}
\end{equation}
so that the total space admits an equivariant diffeomorphism $Q_{\Lambda^*}\to \del_{\Lambda^*}\scB_\Lambda$ to the corner stratum of $\scB_\Lambda$ extending the embedding~\eqref{eq:emb-base},\qed\end{lemma}

\begin{proof}
    This is similar to {\cite[Proposition~5.24]{BX22}} but since we have several punctures and asymptotic markers at outgoing as well as incoming punctures, we describe how the proof has to be adapted. We make a few simplifying assumptions to keep the exposition clear while still highlighting the essential modifications needed. Thus we will deal with the case of $\Lambda^* = \{ \Lambda^0,\Lambda^1 \}$, hence $\Lambda = \Lambda^1 \circ \Lambda^0 $. We further assume $\Gamma^+ =\{\gamma^+\}$. The general result will follow from a similar but combinatorially more involved argument. Write $\Gamma = (\gamma_1, \dots, \gamma_m).$
    
    Define the bundle $Q_{\Lambda^*}$ to be the $G_{\Lambda}$-equivariantization of the $G_{\Lambda^*}$ bundle $\bm Q_{\Lambda^*} \times \im(\Psi_{\Lambda^*})$ where $\bm Q_{\Lambda^*}$ is a distinguished subspace of $\bm Q_{{1+d_0 + d_{10}\dots +d_{1m}}}$. The required equivariant diffeomorphism $Q_{\Lambda^*} \to \partial_{\Lambda^*}\scB_\Lambda$ is determined uniquely by the map $$\rho \cl \bm Q_{\Lambda^*} \times \im (\Psi_{\Lambda^*}) \to \partial_{\Lambda^*}\scB_\Lambda$$ 
    such that $$\rho (h,\varphi)= (Id+\rho_h)\varphi,$$ 
    where $\rho_h$ is the upper diagonal matrix corresponding to $h$. The proof of the map $Q_{\Lambda^*} \to \partial_{\Lambda^*}\scB_\Lambda$  being injective follows exactly by the same arguments as in \cite[Proposition 5.24]{BX22}. 

    In order to prove surjectivity, recall that a rational stable map $\varphi$ of degree $\diamond \le d$ curve in $\bP^d$ lies in a unique minimal linear $\bP(W)\sub \bP^d$ where $\dim W = \diamond +1$. We call $W$ the \emph{linear span} of $\varphi$ \footnote{ We are slightly abusing terminology since classically the linear span is defined as the projectivization of the vector space we consider here}. Thus, any curve $\varphi \in \partial_{\Lambda^*}\scB_\Lambda$, determines a system of linear projective spaces $\bP(W_0), \bP(W_{10}),\,\dots ,\,\bP(W_{1m})$ where $W_{*}\hkra \bC^{1+d_0 + d_{10}\dots +d_{1m}}$ is a subspace so that
    \begin{itemize}
        \item $W_0$ is the linear span of $\varphi_{\gamma^+}$ and $W_{1i}$ is the linear span of $\varphi_{\gamma_i}$,
        \item $\dim (W_0 \cap W_{1i}) =1$,
        \item $W_{1i} \cap W_{1j} = \emst $ unless $i=j$,
        \item $W_0 + W_{10} \dots + W_{1m} = \bC^{1+d_0 + d_{10}\dots +d_{1m}}.$
    \end{itemize}
    The above system of spaces can be viewed as a `tree' generalization of the fans defined in \cite[\S 5.2.5]{BX22}. Thus, we call such a system a \textit{t-fan}. Set $L_i \coloneqq W_0 \cap W_{1i}$. We say a t-fan is \textit{normal} if $W_0$ is orthogonal to the orthogonal complement of $L_i \hkra W_{1j}$ for all $i$. A similar argument as in \cite[Proposition 5.21]{BX22} proves that $\cc{\scB}^\bR_{\Lambda^*}$ is exactly the set of curves with normal t-fans.

   The \textit{t-flag} corresponding to the t-fan $W_*$ is defined as \begin{itemize}
        \item $V_{0}= W_{0}$,
        \item $V_{1i} = W_{0} + W_{1i}.$
    \end{itemize}

By using the Gram--Schmidt process, any t-flag can be mapped to a standard t-flag under an action of $g\in G_\Lambda$ where the standard t-flag is defined as
    \begin{itemize}
        \item     $V_{0}= \bC^{1+d_0}$,
        \item $V_{1j} = \bC^{1+d_0} \times \{0\}^{d_{11} +\dots d_{1j-1}}\times \bC^{d_{1j}} \times \{0\}^{d_{1j+1} +\dots d_{1m}}.$
    \end{itemize}
Thus, given some $\varphi \in \partial_{\Lambda^*} \scB_\Lambda$, we can assume that the t-flag corresponding to it is the standard one. In particular we can assume that that linear span of $\varphi_{\gamma^+}$ is $\bC^{1+d_0} \hkra \bC^{1+d_0 + d_{10}\dots +d_{1m}}$, i.e., is given by the first $1+d_0$ coordinates. Let $y_i=\wt {ev}_{\gamma_i}(\varphi_{\gamma_+}) \in \bC^{1+d_0}$ and choose vectors $w^{1i}_1,w^{1i}_2,\dots ,w^{1i}_{d_i}$ such that $W_{1i}$ has a basis $(y_i,w^{1i}_1,w^{1i}_2,\dots ,w^{1i}_{d_i}).$ 

Let $\bm Q_{\Lambda^*}$ denote the elements of $\bm Q_{1+d_0+\dots}$ which are of the form $Id + A +A^*$ where $A$ is of the following block matrix form,

\begin{equation}\label{matrxform}
A=\begin{bmatrix}
0 & A^0_1 & A^0_2 & \cdots & A^0_m \\
0 & 0 & 0 & \cdots & 0 \\
0 & 0 & 0 & 0 & \cdots \\
\vdots & & & \ddots & \\
\end{bmatrix}.  
\end{equation}
We will construct an element $h\in \bm Q^*_{1+d_0+\dots}$ such that $I+h^u$ takes the t-fan $W_*$ to a normal t-fan where $h^u$ is the upper triangular matrix corresponding to $h$. Denote the orthogonal complement of $W_{0} \cap W_{1i}$ in $W_{1i}$ as $W_i^o$. Note that $W_i^o$ can be viewed as a graph of a linear function $T^o_i:\bC^{d_i} \to \bC^{1+d_0}$. Thus we can find a matrix $A$ of the form as described in \eqref{matrxform} such that $I+A+A^*$ takes the t-fan $W_*$ to the normal t-fan 
\begin{itemize}
    \item $\scW_0 =\bC^{1+d_0}\times\{0\}^{d_{11}+ d_{12}\dots}$,
    \item $ \scW_i = \bC \cdot \langle y_i \rangle \oplus \{0\}^{1+d_0+d_{11}\dots d_{1i-1}}\times \bC^{d_{1i}} \times \{ 0\}^{d_{1i+1}+\dots}$.
\end{itemize}
This shows that $\varphi$ lies in the image of the map $Q_{\Lambda^*} \to \partial_{\Lambda^*}\scB_\Lambda$. 
\end{proof}

\subsubsection{Embeddings of families of buildings}\label{subsec:embeddings-thickening} We first discuss the inductive construction when $\norm{(\Gamma^+,\Gamma^-)} = 1$. Let $\Lambda\cl \Gamma^-\to \Gamma^+$ be a partition and $\Gamma^-\xra{\Lambda^0}\Gamma\xra{\Lambda^1}\Gamma^+$ be a factorization. We will do the construction for the moduli space associated to $\Lambda$ and then endow any other moduli space in the orbit of the symmetric action with the same global Kuranishi chart after permuting the labels of the punctures.\par
By the inductive hypothesis, we have constructed global Kuranishi charts $\scK_{\Lambda^0}$ and $\scK_{\Lambda^1}$ with base spaces $\scBR_{\Lambda^0}$, respectively $\scBR_{\Lambda^1}$ and by Proposition~\ref{prop:embedding-base-spaces}, there exists an embedding
\begin{equation}\label{eq:emb-base}
    \Psi\cl \scBS_{\Lambda^0}\times_{S^1} \scBS_{\Lambda^1}\hkra \scBS_{\Lambda}
\end{equation} 
whose image is contained in a boundary stratum $\scBS_{\Lambda^{01}}$ of $\scBS_{\Lambda}$. Define the `restricted family'
$$\scZ_{\Lambda} \coloneqq\scBS_{\Lambda}\times_{\cB_{\Lambda}} \cZ_{\cBR_{\Lambda}}$$
of buildings, where $\cZ_{\cBR_{\Lambda}}$ was defined in Definition~\ref{de:family-of-tree}. To lift the embeddings~\eqref{eq:emb-base} to maps between these families, we need the following definition.

\begin{definition}[Orbifold associated to Reeb orbit]\label{de:reeb-orbifold}
    Given an unparametrized Reeb orbit $\gamma$, let 
\begin{equation}\label{eq:parametrisations-reeb-orbit}
    E\gamma\coloneqq\set{\sigma\in C^\infty(S^1,Y)\mid \dot{\sigma} = \cA_\lambda(\gamma)\,R(\sigma),\,\im(\sigma) = \im(\gamma)}
\end{equation}
be the space of parametrized Reeb orbits lying over $\gamma$. Then, $S^1$ acts transitively on $E\gamma$ with isotropy $\bZ/m_\gamma$ and we define 
\begin{equation}
    B\gamma := [E\gamma/S^1].
\end{equation}
Note that we have an equivalence $B\bZ/{m_\gamma} \to B\gamma$, where $m_\gamma$ is the multiplicity of $\gamma$.
\end{definition}

\noindent The asymptotic markers yield rel--$C^1$ evaluation maps
\begin{equation}\label{eq:evaluation-to-reeb-orbit}
    p_\gamma\cl \ol\scT_\Lambda\coloneqq [\scT_\Lambda/\wh G_\Lambda] \to B\gamma 
\end{equation}
induced by 
$$(\varphi,u,w) \mapsto [\theta\mapsto (\wh u)_{z_\gamma}(\theta\cdot b_\gamma)],$$
where $\theta\in S^1$ and for each $\gamma$ labeling an exterior edge $b_\gamma$ denotes the asymptotic marker associated to $\varphi$ and $\gamma$. Given a sequence $\Gamma =(\gamma_1,\dots,\gamma_k)$ of Reeb orbits, we set $$B\Gamma\coloneqq \prod_{i=1}^kB{\gamma}_i.$$
Then, the (smooth local) embeddings~\eqref{eq:emb-base} lift to (continuous) embeddings 
\begin{equation}\label{eq:emb-families-with-maps}
   \Psi_{\Lambda^{01}}\cl  \scZ_{\Lambda^0}\ov{\times}_{B\Gamma}\,\scZ_{\Lambda^1}\coloneqq \lbr{\scZ_{\Lambda^0}\times_{B\Gamma}\,\scZ_{\Lambda^1}}/S^1\hkra \scZ_{\Lambda},
\end{equation}
where $S^1$ acts on the fiber product via the diagonal embedding $S^1\hkra \bT_{\Gamma}\times\bT_\Gamma$. The lift uses the fact that the map~\eqref{eq:emb-base} is covered by a canonical isomorphism 
\begin{equation}\label{eq:emb-family}
\pr_1^*\scC_{\Gamma_i,\Gamma^+}\inn\sqcup\pr_2^*\scC_{\Gamma^-,\Gamma_i}\inn\cong\Psi^*\scC\inn_{\Gamma^+,\Gamma^-}.
\end{equation}
Note that we have already fixed a pre-perturbation datum $\fD$ for $\Mbar_{\sft}^J(\Gamma^+,\Gamma^-)_\Lambda$. Suppose $\cU_\Lambda$ is a good covering as in Definition~\ref{de:good-covering} and $\zeta$ a smooth map as in \eqref{eq:reducing-structure-group} yielding a $\cG_\Lambda$-equivariant map
\begin{equation}\label{eq:slice-map}
\zeta_\Lambda\cl \cZ_\Lambda \to \fp\fu_\Lambda \coloneqq\bigoplus\limits_{\gamma\in \Gamma^+}{\text{Lie}(\PU({d_{\Lambda_\gamma}+1}))},
 \end{equation}
 due to the isomorphism~\eqref{eq:comparing-group-quotients}. The pullback of~\eqref{eq:slice-map} to $\scZ_\Lambda$ yields a $\cG_{\Lambda}$-equivariant map $\zeta_\Lambda\cl\scZ_\Lambda\to\fg_\Lambda$. Thus, we may demand that $\zeta_\Lambda$ restricts to $\zeta_{\Lambda^1}\times\zeta_{\Lambda^0}$ over the respective boundary stratum. Assume also that we have found a perturbation space $(E_\Lambda,\mu)$ so that $(E_{\Lambda^1},\mu_{\Lambda^1})\times (E_{\Lambda^1},\mu_{\Lambda^1})$ admits a linear equivariant embedding into the pullback of $(E,\mu)$ along~\eqref{eq:emb-family}. Recall that the perturbation space depends on the choice of $\zeta_\Lambda$, since the regularity condition has to be satisfied over the locus $\{(\varphi,u)\mid \delbar_{J}u = 0,\, \zeta_\Lambda(\varphi,u) = 0\}$. Then, $\alpha_{\Lambda} = (\fD,\cU_\Lambda,\zeta,E_{\Lambda},\mu_\Lambda)$ is a well-defined perturbation datum. Thus, Proposition~\ref{prop:disconnected-buildings} yields a global Kuranishi chart 
$$\cK_\Lambda = \lbr{\bT_{\Gamma^+}\times\bT_{\Gamma^-}\times G'_\Lambda,\cTR_\Lambda,\cE^\bR_\Lambda,\fs_\Lambda}$$
for $\Mbar^{\, J}_{\sft}(\Gamma^+,\Gamma^-)_\Lambda$ equipped with a rel--$C^1$ map $\cT_\Lambda\to \cBS_\Lambda$. We let $G_\Lambda\sub G'_\Lambda$ be the product of stabilizers of $[1:0:\dots:0]$ and set $\wh G_\Lambda := \bT_{\Gamma^+}\times\bT_{\Gamma^-}\times G_\Lambda$. Define the $\wh G_\Lambda$-invariant subspaces
\begin{equation}
    \scT_\Lambda \coloneqq \scBS_\Lambda\times_{\cBS_\Lambda}\,\cTR_\Lambda\qquad \qquad   \scE_\Lambda \coloneqq \scBS_\Lambda\times_{\cBS_\Lambda}\,\cE^\bR_\Lambda
\end{equation}
and denote the pullback of $\fs_\Lambda$ to $\scT_\Lambda$ by the same symbol. By construction, the embeddings \eqref{eq:emb-families-with-maps} induce rel--$C^1$ embeddings
\begin{gather}\label{eq:emb-thickenings}
   \notag\wt\Psi\cl \scT_{\Lambda^0}\ov{\times}_{B\Gamma}\, \scT_{\Lambda^1}\hkra \scT_{\Lambda} \\ 
   ((\varphi_0,u_0,w_0),(\varphi_1,u_1,w_1)) \mapsto (\Psi((\varphi_0,u_0),\varphi_1,u_1)),w_0\oplus w_1).
\end{gather}
They are covered by embeddings of obstruction bundles, defined as follows. The original obstruction bundle is the direct sum $\cE_\Lambda = E_\Lambda \oplus \fp\fu(d_\Lambda +1)$. The embeddings $E_{\Lambda^1_i}\oplus E_{\Lambda^0_i}\hkra E_\Lambda$ exist by assumption, while the embeddings of Lie algebras are by the inclusion~\eqref{eq:inclusion-of-groups} of covering groups. We define 
\begin{equation}
     \scK_{\Lambda^0}\ov{\times}_{B\Gamma}\,\scK_{\Lambda^1}\coloneqq \lbr{\bT_{\Gamma^-}\times G_{\Lambda^0}\times G_{\Lambda^1}\times \bT_{\Gamma^+},\scT_{\Lambda^0}\ov{\times}_{B\Gamma}\,\scT_{\Lambda^1},\scE_{\Lambda^0}\boxplus\scE_{\Lambda^1},\fs_{\Lambda^0}\boxplus\fs_{\Lambda^1}}
\end{equation}
and summarize the construction in the following lemma.

\begin{lemma}\label{lem:embedding-inductive-step}
    There exists a global Kuranishi chart
    \begin{equation}\label{eq:slice-gkc}
        \scK_{\Lambda} \;\coloneqq\; (\wh{G}_\Lambda,\scT_{\Lambda},\scE_{\Lambda}, \fs_{\Lambda}) 
    \end{equation} 
    for $\Mbar_{\sft}^{\, J}(\Gamma^+,\Gamma^-)_\Lambda$ admitting strong equivalences 
     \begin{equation}\label{eq:emb-gkc}\scK_{\Lambda^0}\ov{\times}_{B\Gamma}\,\scK_{\Lambda^1} \hkra \scK_{\Lambda}
    \end{equation}
    onto a boundary strata of $\scK_{\Lambda}$ for any factorization $\Gamma^-\xra{\Lambda^0}\Gamma\xra{\Lambda^1}\Gamma^+$ of $\Lambda$.
\end{lemma}

\begin{proof}
   The discussion above shows that it remains to construct a suitable good covering and a suitable perturbation space. To keep the notation tractable, we assume that $\Gamma^+ = \{\gamma\}$ and $\Lambda^1 = \{\gamma'\}$. The general case is a straightforward generalization.\\
   
   \noindent\textbf{Step 1:} The usual clutching maps induce embeddings
   \begin{equation}\label{eq:clutching-on-stable}
   \psi\cl \scB^{\text{st}}_{1+\#\Lambda^1+3d'_{\Lambda^1}}(d_{\Lambda^1})\times \scB^{\text{st}}_{1+\#\Lambda^0+3d'_{\Lambda^0}}(d_{\Lambda^0})\hkra \scB^{\text{st}}_{1+\#\Lambda+3d'_{\Lambda}}
   \end{equation}
   where the superscript indicates that we only consider stable maps. The maps $\psi$ are equivariant with respect to the inclusions $\cG_{\Lambda^1}\times\cG_{\Lambda^1}\hkra \cG_\Lambda$. By \cite[Lemma~4.13]{AMS23} taking $X = \{\Pt\}$ in Definition~4.6 op. cit., the $\cG_\Lambda$-action on $\cB^{\stb}_{1+\#\Lambda + 3d'_{\Lambda}}$ is Palais proper. Thus, \cite[Theorem~4.3.1]{Pal61} asserts that $\scB^{\text{st}}_{\Lambda,+3d'_{\Lambda}}$ admits a $\cG_\Lambda$-invariant Riemannian metric as well as $\cG_\Lambda$-invariant bump functions. First, we can extend the product of the $\zeta_i := \zeta_{\cU^i}$ to the canonical  $\cG_\Lambda$-equivariant function 
   \[\zeta_{01}\cl \im(\psi)^\sim\coloneqq\cG_\Lambda\times_{\cG_{\Lambda^1}\times\cG_{\Lambda^0}}\im(\psi)\,\to\, \fg_\Lambda.\]
   Then, choosing a $\cG_\Lambda$-equivariant tubular neighborhood of $\im(\psi)^\sim$, we can extend $\zeta_{01}$ to a $\cG_\Lambda$-equivariant function on an open neighborhood of $\im(\psi)^\sim$. If we have several embeddings of the form~\eqref{eq:clutching-on-stable}, their images are disjoint, even under the $\cG_\Lambda$-action, due to our assumption on $\Gamma^+$ and $\Gamma^-$. Thus we can choose the tubular neighborhoods to be disjoint. Then, using $\cG_\Lambda$-invariant cut-off functions, we can extend the thus obtained functions to a smooth $\cG_\Lambda$-equivariant function 
   \begin{equation}\label{eq:suitable-slice-function}
    \zeta \cl\scB^{\text{st}}_{\Lambda,+3d'_{\Lambda}} \to \fg_\Lambda.
   \end{equation}
   It remains to extend the good coverings $\cU_{\Lambda^1}$ and $\cU_{\Lambda^0}$. Recall that such a good covering consists of a finite collection $\{(U_i,\sigma_{i,j},D_{i,j},\chi_i)\}_{i,j}$ of 
    \begin{enumerate}[ref=\roman*,leftmargin=20pt,label=\roman*)]
        \item\label{i:intersection-nice} open $\cG_{\Lambda'}$-invariant subsets $U_i \sub \scZ_{\Lambda'}$ that cover $\scZ_{\delbar}$
        \item a smooth $\cG_{\Lambda'}$-equivariant section $\sigma_{i,j}\cl U_i \to \cC\inn|_{U_i}$ for $1 \leq j \leq 3d_{\Lambda'}$ and divisors $D_{i,j}\sub Y$ so that for any $(\varphi,u)\in U_i$, we have $u\pf D_i$ and $u(\sigma_{i,j}(\varphi,u))\in D_{i,j}$ and
        $$\# \,C_v\cap \{\sigma_{i,j}(\varphi,u)\}_j = \frac{3}{p}\deg(L_{u,v})$$
        \noindent for any irreducible components $C_v\sub C$, allowing for the stabilisation map to lift to 
        \begin{equation}
            \text{st}_{U_i}\cl U_i \to \cB_{1+\#\Lambda + 3d'_{\Lambda'}}(d_{\Lambda'})
        \end{equation}
        \item $\cG_{\Lambda'}$-invariant functions $\chi_i \cl Z_i \to [0,1]$ with support contained in $U_i$
    \end{enumerate}
    so that $\scZ_{\Lambda',\delbar}$ is contained in the support of $\s{}{\chi_i}$.\par
    \vspace{-5pt}
    \noindent Write $\cU_{\Lambda^r}= \{(D_{i,j}^r,U_i^r,\chi_i^r)\}_{i\in I^r,j}$ and choose for $i \in I^0\times I^r$ an open $\cG_{\Lambda}$-invariant subset $U_i$ of $\scZ_\Lambda$ so that $U_i \cap \im(\Psi) = \Psi(U_{i^1}\times U_{i^0})$ and $U_i$ does not intersect any other other boundary stratum. Shrinking $U_i$ if neccessary, we may ensure that for any $1 \leq j \leq 3d'_{\Lambda^r}$ the section $\sigma_{i^r,j}^r$ extends to a $\cG_\Lambda$-invariant section $U_i\to\cC\inn|_{U_i}$ with $u(\sigma_{i,j}^r(\varphi,u))\in D^r_{i^r,j}$ and $u\pf D_{i,j}$ near $\sigma_{i^r,j}(\varphi,u)$ for any $(\varphi,u)\in U_i$. This requires the divisor $D_{i^r,j}$ and the fact that transversality is an open condition. Note that we do not need the whole map $u$ to intersect $D^r_{i^r,j}$ transversely. 
    In particular, this construction yields a $\cG_{\Lambda}$-equivariant map $\stb_{U_i}\cl U_i \to \cB^{\stb}_{1+\# \Lambda + 3d'_{\Lambda}}(d_\Lambda)$ so that 
\begin{equation}\label{eq:square-for-coverings}\begin{tikzcd}
    U_{i^1}^1\times_{S^1} U_{i^0}^0 \arrow[rr,"\Psi"] \arrow[d,"\stb_{U_{i^1}^1}\times\stb_{U_{i^0}^0}"]&&  U_i\arrow[d,"\stb_{U_i}"]\\ \scB^{\stb}_{1+\#\Lambda^1+3d'_{\Lambda^1}}\times\scB^{\text{st}}_{1+\#\Lambda^0+3d'_{\Lambda^0}}\arrow[rr,""] && \scB^{\text{st}}_{1+\#\Lambda+3d'_{\Lambda}}
    \end{tikzcd} \end{equation}
     commutes. Now, we can use the Tietze extension theorem, applied to $\cZ_\Lambda/\cG_\Lambda$, to extend the cut-off functions $\chi_{i^1}^1\times\chi^0_{i^0}$ on $\Psi(U_{i^1}^1\times U_{i^0}^0)^\sim$ equivariantly to obtain invariant continuous functions $\chi'_{i}\cl U_i\to [0,1]$. Since $A_i  := \cG\cdot\text{supp}(\chi_{i^1}^1\times\chi^0_{i^0})$ is closed in $\scZ_\Lambda$ and contained in $U_i$, we can use the Tietze extension theorem again to find a $\cG$-invariant function $\rho_i \cl \scZ_\Lambda\to [0,1]$ which is identically $1$ on $A_i$ and supported in $U_i$. Thus, we can extend $\rho_i\chi'_i$ to a $\cG$-invariant function $\chi_i$ on all of $\cZ_\Lambda$, extending the cut-off function on the boundary. Now, we can complete $\cU'$ to an invariant open cover $\cU$ of $\scZ_{\Lambda,\delbar}$ so that any $U \in \cU\sm \cU'$ does not meet the images of the embeddings $\Psi_i$. This yields the desired good covering $\cU_\Lambda$, so that $\zeta_{\cU_\Lambda}$ is an extension of the functions $\zeta_{01}$.\\
     
     \noindent\textbf{Step 2:} First, note that we can rephrase a perturbation space $(E,\mu)$ as the trivial $G$-vector bundle $E\to \cC\inn\times \wh Y$ equipped with a $\bR\times G$-equivariant vector bundle morphism $E\to \Lambda^{0,1*}_{\cC\inn/\cB}\otimes_\bC T\wh Y$. Thus, we can extend the trivial $G_{\Lambda_i^1}\times G_{\Lambda^0_i}$-vector bundle $E_{\Lambda^1}\times_{S^1} E_{\Lambda^0}\to \scB_0\times_{S^1}\scB_1$ to a $G_\Lambda$-vector bundle 
 \begin{equation*}
    E'_{01} := G_\Lambda\times_{G_{\Lambda_i^1}\times G_{\Lambda^0_i}} (E_{\Lambda^1}\times_{S^1} E_{\Lambda^0})
 \end{equation*}
 and we can extend $\mu_{\Lambda^1}\times\mu_{\Lambda^0}$ uniquely to $G_\Lambda$-equivariant map. By \cite[Proposition~1.1]{Las79}, we can take the direct sum with a $G_\Lambda$-vector bundle $W\to \im(\Psi)^\sim$, equipped with the zero map to $\Lambda^{0,1*}_{\cC\inn/\cB}\otimes_\bC T\wh Y$, to obtain a $G_\Lambda$-perturbation space $(E_{01},\mu_{01})$ for $\im(\Psi)^\sim$ that extends the perturbation space of the boundary stratum. Using the map of Lemma~\ref{lem:boundary-up-to-stabilisation} and multiplying $\mu_{01}$ with a suitable cut-off function, we can pull back this perturbation space to a perturbation space on the whole boundary stratum. Using a (sufficiently small) collar of $\im(\Psi_{\Lambda^0,\Lambda^1})^\sim$ and a bump function, we can extend $(E_{01},\mu_{01})$ to a perturbation space $(E_{01},\mu_{\Lambda,01})$ on all of $\cB_\Lambda$. Since these perturbation spaces are sufficient to achieve transversality for maps in $$\fs_{\text{pre}}\inv(0) = \{(\varphi,u)\in \cZ_\Lambda\mid \zeta_\cU(\varphi,u) = 0\}$$ 
 that lie near some boundary stratum, we can extend the direct sum $\bigoplus\limits_{\Gamma_i}(E_{01},\mu_{\Lambda,01})$ to a perturbation space $(E_\Lambda,\mu_\Lambda)$ for $\fs_{\text{pre}}\inv(0)$ so that over a boundary stratum $(E_\Lambda,\mu_\Lambda)$ is of the form $(E_{01},\mu_{01})\oplus (W,0)$. This yields the desired perturbation datum $\alpha =(\fD,\cU,\zeta,E,\mu)$.\par
      By definition (and \cite[Lemma~3.5]{AMS23}) the global chart $\scK_{\Lambda}$ is equivalent to $\cK_{\Lambda}$. In particular, it is a global chart for $\Mbar_{\sft}^{\, J}(\Gamma^+,\Gamma^-)_\Lambda$. The embeddings~\eqref{eq:emb-thickenings} exist by our choice of perturbation space $E_{\Gamma^+,\Gamma^-}$ and they fit into a commutative square 
    \begin{equation}\begin{tikzcd}
    \scE_{\Lambda^0}\ov{\times}_{B\Gamma}\, \scE_{\Lambda^1} \arrow[r,""] \arrow[d,""]& \scE_{\Lambda} \arrow[d,""]\\ 
    \scT_{\Lambda^0}\ov{\times}_{B\Gamma}\, \scT_{\Lambda^0} \arrow[r,""] & \scT_{\Lambda} 
    \end{tikzcd} \end{equation}
    This completes the proof.
\end{proof}

\begin{remark}
    In the construction of the perturbation space, we have used a strategy of \cite{Rez22}, instead of the construction in \cite{BX22}. This allows us to see the embeddings of boundary strata immediately as strong equivalences, while \cite{BX22} has to use an outer-collaring as well as some gluing results (\cite[Proposition~5.64]{BX22}). 
\end{remark}

It remains to prove the inductive step.

\begin{proposition}\label{prop:embedding-gkc}
    There exists a system of global Kuranishi charts $\{\scK_{\Lambda}\}_{\Lambda\cl\Gamma^-\to\Gamma^+}$ for the moduli spaces $\Mbar_{\sft}^J(\Gamma^+,\Gamma^-)_\Lambda$ admitting strong equivalences 
    \begin{equation}\label{eq:emb-gkc-first-step}
        \scK_{\Lambda^0}\ov{\times}_{B\Gamma}\,\scK_{\Lambda^1}\hkra \scK_{\Lambda}
    \end{equation} 
    onto the codimension-$1$ boundary strata of $\scK_{\Lambda}$ for any factorization $\Lambda = \Lambda^1\g\Lambda^0$ so that any boundary stratum of $\scK_{\Lambda}$ is the image of~\eqref{eq:emb-gkc} up to stabilization for some factorization and the squares 
    \begin{equation}\begin{tikzcd}
    \scK_{\Lambda^0}\ov{\times}_{B\Gamma_0}\, \scK_{\Lambda^1}\ov{\times}_{B\Gamma_1}\, \scK_{\Lambda^2} \arrow[r,""] \arrow[d,""]&\scK_{\Lambda^0}\ov{\times}_{B\Gamma_0}\,\scK_{\Lambda^{21}}\arrow[d,""]\\  \scK_{\Lambda^{10}}\ov{\times}_{B\Gamma_1}\,\scK_{\Lambda^2} \arrow[r,""] & \scK_{\Lambda} 
    \end{tikzcd} \end{equation}
    are pullback squares.
\end{proposition}

\begin{proof} 
    The construction of the embeddings $\scT_{\Lambda^1}\ov{\times}_{B\Gamma}\,\scT_{\Lambda^0}\hkra \scT_{\Lambda}$ of thickenings for respective partitions of Reeb orbits follows by induction, using the arguments of \textsection\ref{subsec:embeddings-thickening} and Lemma~\ref{lem:embedding-inductive-step}. The same holds for the embeddings of obstruction bundles. The additional difficulty in the general case is to ensure that our extensions of good coverings and perturbation spaces can be done compatibly over the corner strata.  We first extend the maps of~\eqref{eq:suitable-slice-function} from the boundary strata to the interior. Extending $\zeta_{\Gamma^+,\Gamma'}$ and $\zeta_{\Gamma',\Gamma^+}$ and extending to $\cG$-equivariant maps, we obtain $\cG$-equivariant functions 
    $$\wt\zeta_{\Gamma'}\cl \del_{\Lambda^{01}} {\cBS}^{\stb}_\Lambda\,\longrightarrow\,\fg_{\Lambda} $$
    that agree over the corner strata. By Lemma~\ref{lem:extending-from-boundary}, we may extend them to a $\cG$-equivariant smooth function $\zeta' \cl {\cBS}^{\stb}_{\Gamma^+,\Gamma^-}\to \fg$. By the inductive hypothesis, we assume that the good coverings of the codimension-$1$ strata over their respective boundary strata are obtained from good coverings of the moduli spaces forming the codimension-$2$ strata. Thus, we may use the same argument as in the discussion before the square~\eqref{eq:square-for-coverings} to extend the good coverings of the boundary strata to a good covering of a neighborhood of the boundary. In particular, the maps $\stb_{U_i}$ and $\stb_{U_j}$ agree over the intersection of the boundary with $U_i \cap U_j$. Thus, we may choose arbitrary extensions of $\cG$-invariant bump functions and extend these data to a good covering as at the end of the proof of Lemma~\ref{lem:embedding-inductive-step}.\par 
    To construct the perturbation space $(E_{\Gamma^+,\Gamma^-},\mu_{\Gamma^+,\Gamma^-})$, we observe that the beginning of the construction in the proof of Lemma~\ref{lem:embedding-inductive-step} applied to the codimension-$1$ boundary strata of $\cT_{\Gamma^+,\Gamma^-}$ yields a $G$-representation $E'_\Lambda$, equipped with a a map $$\mu^\del_\Lambda\cl C^\infty_c(\cC\inn|_{\del\scBS_{\Gamma^+,\Gamma^-}}\times\wh Y,\Lambda^{0,1*}_{\cC\inn/\scBS_{\Gamma^+,\Gamma^-}}\otimes_\bC T\wh Y)^\bR.$$
    whose restrictions to the closures of the codimension-$1$ strata agree over the codimension-$2$ strata. Thus, we can extend $\mu^\del_\Lambda$ to a $G_\Lambda$-equivariant map $\mu'_\Lambda\cl E'_\Lambda\to C^\infty_c(\cC\inn\times\wh Y,\Lambda^{0,1*}_{\cC\inn/\scBS_{\Gamma^+,\Gamma^-}}\otimes_\bC T\wh Y)^\bR$ by \cite{Kott}. As this is sufficient to achieve transversality for curves with domain near the boundary, we may extend $(E'_\Lambda,\mu'_\Lambda)$ to a perturbation datum, where $(E_\Lambda,\mu_\Lambda) = (E'_\Lambda,\mu'_\Lambda)\oplus (E\inn,\mu\inn)$, where $\mu\inn(e)$ is supported away from $\cC|_{\del\scB_\Lambda}\times \wh Y$. 
    This proves the first claim. Any boundary stratum is covered by such an embedding since any boundary stratum of $\scB_\Lambda$ is a vector bundle over the image of some embedding $\scB_{\Lambda^1}\times\scB_{\Lambda^0}\to \scB_\Lambda$. The last assertion follows from Lemma~\ref{lem:associativity-of-embeddings} and the construction of the perturbation spaces.
\end{proof}

\begin{lemma}\label{lem:extending-from-boundary}
    Suppose $\cG$ acts properly on a smooth manifold $M$ with corners and for each boundary stratum $S\sub \del M$ there exists a $\cG$-equivariant function $f_S \cl S\to V$ to a common finite-dimensional $\cG$-representation so that the restrictions agree over the codimension-$2$ corner strata. Then, there exists a $\cG$-equivariant smooth function $f \cl M\to V$ that extends the functions $f_S$.
\end{lemma}

\begin{proof}
    We first apply \cite{Kott} to the functions $\{f_S\}$ to obtain a smooth extension $\wt f\cl M\to V$. Averaging over $G$, we may assume $\wt f$ to be $G$-invariant. By \cite{Pal61}, we can find a locally finite open cover $\cU$ of $M$ so that for each $U \in \cU$ there exists a $G$-invariant subset $S\sub U$ with $U\cong\cG\times_G S$. Moreover, we can find a $\cG$-invariant partition of unity $\{\varsigma_U\}_U$ subordinate to $\cU$. Define $f_U \cl U\to V$ by $f_U(g\cdot s) = g\cdot \wt f(s)$ for $(g,s)\in \cG\times S$ and set $f := \s{U}{\varsigma_U\, f_U}$ to obtain the desired extension.
\end{proof}

This completes the proof of Theorem~\ref{thm:flow-cat}.
\subsection{Flow bimodules from symplectic cobordisms}\label{subsec:flow-bimodule} 
In this subsection we show that exact symplectic cobordisms induce flow bimodules between the flow categories constructed in \S\ref{subsec:unstructured-flow-cat}. Let $(\wh X,\omega)$ be an exact symplectic cobordism from $(Y^+,\lambda^+)$ to $(Y^-,\lambda^-)$, equipped with an $\omega$-adapted almost complex structure $J$. Suppose we are given action bounds $L^\pm\ge 0$ and pre-perturbation data $\fD^\pm = (\wt\lambda^\pm,\nabla^\pm, p^\pm )$ for $\cP_{L^\pm}$, where 
\begin{itemize}
    \item $\wt \lambda ^\pm$ is a $\cP_{L^\pm}(Y^\pm)$-integral approximations of $\lambda^\pm$,
    \item $\wh \nabla^\pm$ are complex linear connections on $\wh X$ and $\wh Y^\pm$, respectively, such that  $\nabla_o$ restricts to $\wh \nabla^\pm$ on the corresponding end and the connections on $Y^\pm$ are translation invariant;
    \item prime numbers $p^+, p^-$ satisfying~\eqref{eq:primes-for-base}.
\end{itemize}

\begin{theorem}\label{thm:flo_bim}
    Given a pre-perturbation datum $\fD$ extending $\fD^\pm$, there exists a flow bimodule $\scN^{\wh X}$ from the flow category $\scM_{\leq L^-}^{Y^-}$ to the flow category $\scM_{\leq L^+}^{Y^+}$. The morphism space from an object $\Gamma^- \in \cP_{L^-}(Y^-)$ to $\Gamma^+\in\cP_{L^+}(Y^+)$ is
    \begin{equation*}\label{eq:morphisms-bimodule}
        \scN^{\wh X}(\Gamma^-,\Gamma^+) = \djun{\Lambda}\,\Mbar_{\sft}^{\wh X,\, J}(\Gamma^+,\Gamma^-;\beta)_\Lambda
    \end{equation*}
   respectively, a global Kuranishi chart of said moduli space.
\end{theorem}

\begin{proof}
    We only sketch the construction of $\scN^{\wh X}$ because of its similarity to the construction of the flow category in \ref{subsec:unstructured-flow-cat}. There is a natural extension of the definition of the `norm' of $\|(\Gamma^+ , \Gamma^-) \|$ for $\Gamma^\pm \in \cP_{L^\pm}(Y^\pm)$ given by the height of the tallest building in $\Mbar_{\sft}^{\wh X,\, J}(\Gamma^+,\Gamma^-;\beta)_\Lambda$ (if it were unobstructed). The inductive construction of perturbation data for the moduli spaces begins similarly as before. By the choice of perturbation datum $\fD^\pm$ and the constructions of \textsection\ref{subsec:unstructured-flow-cat}, we are given global Kuranishi charts for the morphism spaces between objects in the same symplectization. Hence, we only need to construct charts for the moduli of buildings from an object of $\scM_{\leq L^-}^{Y^-}$ to an object of $\scM_{\leq L^+}^{Y^+}$ so that its boundary strata are compatible with the already chosen charts in the construction of $\scM_{\leq L^\pm}^{Y^\pm,\lambda^\pm}.$ 
    
    We do this inductively as well. However, both due to the formalism and due to the geometry of the moduli spaces, we have to make an artificial choice: we let 
    \begin{equation*}\label{eq:trivial-cobordism}
        \scN^{\wh X}(\emptyset_{Y^+},\emptyset_{Y_-})\coloneqq *
    \end{equation*}  
    be the trivial global Kuranishi chart for a point. This choice is forced on us due to the following phenomenon in cobordisms: there can be a family of holomorphic planes in $\wh X$ that escape off to $\wh Y^+$, resulting in a two-leveled building, which has a holomorphic plane in $\wh Y^+ $ and an empty level in the symplectization. This phenomenon can occur whenever curves have no negative punctures. 
    
    Given a partition $\Lambda\cl \Gamma^+\to \Gamma^-$ of norm $0$, extend $\fD$ to an arbitrary perturbation datum $\alpha_{\Gamma^\pm}$ for $\Mbar^{\wh X,\, J}_{\sft}(\Gamma^+,\Gamma^-)_\Lambda$ as in Definition~\ref{de:auxiliary-datum}. Let 
 $$\scK_{\Lambda}^c = \scK^c_{\alpha_{\Lambda}} = (\wh G_{\Lambda},\scT_{\Lambda},\scE_{\Lambda},\fs_{\Lambda})$$ 
 be the associated global Kuranishi chart for $\Mbar_{\sft}^{\wh X,\, J}(\Gamma^+,\Gamma^-;\beta)_\Lambda$ as in Lemma~\ref{lem:embedding-inductive-step}, obtained from the chart $\cK_{\alpha_{\Lambda}}$ of Proposition~\ref{prop:disconnected-in-cobordims}. The proofs of the counterparts of Propositions \ref{prop:embedding-base-spaces} and \ref{lem:associativity-of-embeddings} are the same except for the added notational complexity required to keep track of the targets.\par
 An important observation is that the category of trees that stratifies $\scBS_c$ naturally carries the information of an order $\mathrm{d}_e $ for every edge $e$, obtained from Definitions~\ref{de:type-for-exact-cobordism} and~\ref{de:type-for-symplectisation}. The other change one has to make is that the concatenation $T_e\#T_c$, of a leveled forest $T_e$ labeling a stratum in $\scBS$ with a leveled forest $T_c$ labeling a stratum in $\scBS_c$, is a leveled forest, which labels a stratum in $\scBS_c$. For the cobordism counterparts of Lemma \ref{lem:embedding-inductive-step} and Proposition \ref{prop:embedding-gkc}, the only difference lies in the construction of the embedding maps between thickenings; see \S \ref{subsec:embeddings-thickening}. While choosing finite-dimensional approximation scheme $E_*$, we use Lemma \ref{lem:joint-fin-dim-scheme} to obtain a joint finite-dimensional approximation scheme. The rest of the construction follows similarly. 
\end{proof}

An important bimodule from a flow category $\scX$ to itself is the \emph{diagonal bimodule} $\Delta_\scX$, which can be thought of as the identity morphism. It has objects given by two copies $\cP^\pm$ of the symmetric sets $\cP$ of objects of $\scX$ and morphisms given by 
\begin{equation}\label{eq:morphisms-diagonal-bimodule}
    \Delta_\scX(x,y) = \begin{cases}
        \cD \scX(x,y) \quad & x \neq y \\
        * \quad & x= y,
    \end{cases}
\end{equation}
where $\cD X$ is the conic degeneration of an orbifold, \cite[\textsection 6.1]{AB24}. It comes with a natural map $d\cl\cD X\to X$, and we write $\cD\cK = (\cD\cT,d^*\cE,d^*\fs)$ for the conic degeneration of a derived orbifold. In particular, $\Delta_\scX(x,y)$ defines an object of the category $\dOrb_{/\cdot}$ defined in Definition~\ref{de:derived-orbifold-category}

\begin{lemma}\label{lem:trivial-cobordism-diagonal-bimodule}
    If $(\wh X,\omega ) = (\wh Y,d(e^s\lambda))$ is the trivial cobordism equipped with the almost complex structure $J$, then $\scN_{\leq L}^{\wh X}$ of Theorem~\ref{thm:flo_bim} is equivalent to the diagonal bimodule.
\end{lemma}

\begin{proof}
    We will use slightly different perturbation data to construct the flow bimodule in this case. The proof of equivalence of global Kuranishi charts in \cite[Proposition~6.1]{HS22} then shows that the associated flow bimodule is equivalent to the one constructed in Proposition~\ref{thm:flo_bim}. Here we say two flow bimodules are equivalent if they have the same objects and their morphism spaces are equivalent compatibly with structure maps.\par
    Let $\fD$ be the pre-perturbation datum chosen for the construction of $\scM_{\le L}^Y$ and let $\{\alpha_\Lambda\}_\Lambda$ be the collection of perturbation data constructed inductively in \textsection\ref{subsec:unstructured-flow-cat}. Then, $\fD$ is also a pre-perturbation datum for the symplectic cobordism $\wh X$ and $\wt\alpha_\Lambda\coloneqq \alpha_\Lambda$ defines a perturbation datum for the moduli space $\Nbar^J_{\sft}(\Gamma^+,\Gamma^-)$ of buildings in $\wh X$, denoted by $\Nbar$ instead of $\Mbar$ in order to distinguish it from moduli spaces of buildings in the symplectisation. By definition, $\scN^{\wh X}(\emst,\emst)$ is a point, while for a nonempty sequence $\Gamma$, the moduli space of trivial cylinders in $\wh X$ is regular and a point, whence $\scN^{\wh X}(\Gamma,\Gamma) = *$ as well. Suppose $\Gamma^-\neq \Gamma^+$ and let $\Lambda\cl \Gamma^-\to\Gamma^+$ be a partition. By Lemma~\ref{lem:auxiliary-thickening}, the thickening $\scT^c_\Lambda$  of the global Kuranishi chart $\scK^c_\Lambda$ for $\Nbar^{\,J}(\Gamma^-,\Gamma^+)_\Lambda$ admits a canonical equivariant rel--$C^1$ map $q\cl \scT^c_\Lambda\to \scT_\Lambda$, which is a fiber bundle of intervals. Moreover, $\scE^c_\Lambda = q^*\scE_\Lambda$ and the obstruction section is pulled back as well. Since $\text{Homeo}_+([0,1])$ is contractible, one can lift $q$ to a homeomorphism $\scT^c_\Lambda\to \scT_\Lambda\times [0,1]$. Using the explicit description of the composition maps, this shows that the forgetful map $\scT^c_\Lambda\to \scB_\Lambda$ factors through the conic degeneration $\cD\scB_\Lambda$ of $\scB_\Lambda$ and that $\scT^c_\Lambda = \cD\scB_\Lambda\times_{\scB_\Lambda}\scT_\Lambda$. Since the map $\scT^c_\Lambda\to \cD\scB_\Lambda$ is compatible with the bimodule structure maps, the claim follows.
\end{proof}

\subsection{Stable complex structures}\label{subsec:stable-complex}
In this subsection we show that the flow categories of Theorem~\ref{thm:sft-flow-category} admit stable complex structures. Recall that we fixed a nondegenerate contact manifold $(Y,\lambda)$, a $\lambda$-adapted almost complex structure $J$, a real number $L > 0$. Given this, let $\scM^{Y}_{\leq L}$ the flow category of Theorem~\ref{thm:sft-flow-category}.

\begin{theorem}\label{thm:stable-complex-structure}
    The symmetric flow category $\scM^{Y,\lambda}_{\leq L}$ admits a lift to a stably complex symmetric flow category.
\end{theorem}

The proof is similar to \cite[\S 11.3]{AB21} and \cite[\S B]{AB24}. We abbreviate $\scM \coloneqq \scM^{Y,\lambda}_{\leq L}$. The morphism space is the disjoint union
$$\scM(\Gamma^-,\Gamma^+)\; =\; \djun{\Lambda}\,\scM(\Gamma^-,\Gamma^+)_\Lambda$$ 
over all functions $\Lambda\cl \Gamma^-\to \Gamma^+$. We will phrase every statement in terms of these partitions as we have done in \textsection\ref{subsec:unstructured-flow-cat}. The tangent bundle of the global Kuranishi chart satisfies 
\begin{equation}\label{eq:tangent-bundle-decomposition}
	T^+\scM(\Gamma^-,\Gamma^+)_\Lambda\,\oplus\, \wh\fg_\Lambda =\; T\scB_{\Lambda}\,\oplus\, T^v\scT_{\Lambda}
\end{equation}
while the obstruction bundle is given by
\begin{equation}\label{eq:obstruction-bundle-decomposition}
	T^-\scM(\Gamma^-,\Gamma^+)_\Lambda\, =\; E_{\Lambda}\,\oplus\, \fp\fu_\Lambda\,
\end{equation}
where $\wh\fg_\Lambda$ is the Lie algebra of the covering group $\wh G_\Lambda$, $\fp\fu_\Lambda$ is the Lie algebra of the product $\p{\gamma\in \Gamma^+}{\PU(d_{\Lambda_\gamma}+1)}$ of projective unitary groups, and $E_\Lambda$ is a finite-rank $G$-vector bundle.\par In \textsection\ref{subsec:lift-of-objects}, we define the lift of the objects $\Gamma^+$ to objects of a stably complex flow category. Subsequently, we construct the stable complex structures on the morphism spaces in \textsection\ref{subsec:stable-complex-base} and \textsection\ref{subsec:stable-complex-fiber}, summarizing the results in Proposition~\ref{prop:existence-stable-complex}.

\subsubsection{Lift of the objects}\label{subsec:lift-of-objects}
Recall that the objects of a lift of $\scM$ to a stably complex flow category $\scM^{U}$ consist of a finite sequence $\Gamma$ of Reeb orbits of action at most $L$ and a virtual vector space $V_\Gamma = (V^+_\Gamma,V^-_\Gamma)$. We will construct for each Reeb orbit $\gamma$ of $\lambda$ a virtual $S^1$-vector bundle $V_\gamma$ over the $S^1$-manifold $E\gamma$ defined in~\eqref{eq:parametrisations-reeb-orbit}.Then, we define $$V_\Gamma \coloneqq\p{\gamma\in \Gamma}{V_\gamma}\to B\Gamma$$
to be the product vector bundle. 
\begin{remark}
    If $\gamma$ is a good Reeb orbit, this virtual vector bundle is orientable. It is non-orientable otherwise.
\end{remark}

Recall that the pre-perturbation datum $\fD$  we chose for the construction of $\scM$ includes a choice of $J$-linear connection $\conn^Y$ on $T\wh Y = \xi \oplus \bC$. While the construction of global Kuranishi charts does not require any properties of $\conn^Y$ except linearity with respect to $J$, for the following constructions it will be useful to assume that $\conn^Y$ has trivial monodromy around any simple Reeb orbit of action at most $L$. 
Moreover, fix a smooth cutoff function $\chi$ on $\bR$ with 
\begin{equation}
	\chi(s) = \begin{cases}
		1\quad& s \ll 0\\
		0 \quad& s \gg 0.
	\end{cases}
\end{equation}
Given a Reeb orbit $\gamma$ and a parametrization $\wt\gamma\in E\gamma$, let $c_{\wt\gamma}\cl \bR\times S^1\to \wh Y$ be the trivial cylinder over $\wt\gamma$. Then, the pullback $c_{\wt\gamma}^*T\wh Y= c_{\wt\gamma}^*\xi \oplus \bC$, equipped with the chosen almost complex structure $J$, is a complex vector bundle over $\bR\times S^1$. It carries two canonical connections: the pullback of $\conn^Y$, which is complex linear, and the connection $\conn'$ induced by the pullback of $\cL_{R_\lambda}$ on $\xi$ and the trivial connection on $\bC$. We define the connection 
\begin{equation}
	\conn^{\wt\gamma} \coloneqq \conn^Y + \chi(s)(\conn'-\conn^Y)
\end{equation}
on $\bR\times S^1$ and let $\delbar^{\wt\gamma} = (\conn^{\wt\gamma})^{0,1}$ be the associated real Cauchy--Riemann operator. Since $c_{\wt\gamma}^*\xi\oplus \bC$ is trivializable, $(c_{\wt\gamma}^*\xi\oplus \bC,\delbar^{\wt\gamma})$ extends uniquely to a Cauchy--Riemann problem $(\cV_{\wt\gamma},\delbar^{\wt\gamma})$ on the capping off of $\bR\times S^1$ at the \emph{positive} end, using the trivialization induced by the connection $\conn^Y$.\footnote{We cap off at the positive instead of the negative end in order to obtain formulas compatible with the conventions in \cite{AB24}. This is for the same reason that we define $\scM(\Gamma^-,\Gamma^+)$ to be $\Mbar^J_{\sft}(\Gamma^+,\Gamma^-)$.} As all data that depend on the parametrization $\wt\gamma$ depend smoothly on it, we obtain a smooth $S^1$-vector bundle 
$$\cV\to E\gamma$$
with an $S^1$-invariant family $\delbar^\gamma = \{\delbar^{\wt\gamma}\}_{\wt\gamma}$ of Cauchy--Riemann operators. Fix an element $\wt\gamma\in E\gamma$ and choose a finite-dimensional complex representation 
\begin{equation}\label{eq:perturbation-for-point}
	\nu_0 \cl W'\to \Omega^{0,1}_c(\bR\times S^1,c_{\wt\gamma}^*T\wh Y)
\end{equation}
that is invariant under the isotropy of $\wt\gamma$ and is sufficiently large so that $\delbar^{\wt\gamma}\,\oplus \,\nu_0$ is surjective. We require that 
\begin{equation}\label{eq:support-of-orbit-perturbation}
    \supp(\nu_0(w)) \sub [-1,1]\times S^1
\end{equation}
for each $w \in W'$. Using the $S^1$-equivariance of $\delbar^\gamma$ and the transitivity of the $S^1$-action on $E\gamma$, we obtain a finite-rank $S^1$-vector bundle $W'_\gamma\to E\gamma$ with fiber-wise trivial $S^1$-action and a linear $S^1$-equivariant map $\nu$ from $W'_\gamma$ to the bundle with fibers $\Omega^{0,1}(\bR\times S^1,c_{\wt\gamma}^*\xi)$ which surjects onto the fiber-wise cokernel of $\delbar^\gamma$. Then, we define the vector bundles 
\begin{equation}
	V_\gamma^+ = \ker(\delbar^\gamma +\nu)\qquad \qquad V^-_\gamma = W'_\gamma.
\end{equation}

\subsubsection{Stable complex structures on base spaces}\label{subsec:stable-complex-base} This is essentially a more complicated version of \cite{AB24} since our base spaces go through an additional generalized blow-up compared to those in \cite{AB24}. 

\begin{lemma}\label{lem:stable-complex-corner-blow-up}
	The generalized blow-up of an almost complex manifold, equipped with the canonical smooth structure, admits a canonical almost complex structure up to contractible choice.
\end{lemma}

\begin{proof}
We will show that the generalized blow-up $\wt M \xrightarrow{\beta}M $ of a smooth manifold $M$ with corners comes equipped with a bundle isomorphism $T\wt M \to \beta^* TM$. The choice of this isomorphism is canonical up to a contractible set. In particular, if $M$ is almost complex, we can lift its almost complex structure to the generalized blow-up uniquely up to a contractible choice. Moreover, it follows from the argument in \cite[\S B.3.1]{AB24} and Remark~\ref{rem:neighbourhood-of-exceptional-boundary} that once the bundle isomorphism is determined on the exceptional boundary locus, then it can be extended up to a contractible choice using bump functions. Hence, it suffices to construct such an isomorphism over the blow-up locus.
For simplicity, we assume that $\wt M$ is generalized blow-up of the corner $$C = \lcap{i=1}{\ell+1} B_i,$$ 
where $B_1,\dots,B_{\ell +1}$ are boundary faces of $M$. Denote the exceptional boundary face $\beta\inv(C)$ by $E$. The construction begins with choosing any two metrics $\wt g, g$ on $\wt M$ and $M$ respectively. From hereon in the proof, we use canonical to mean \textit{canonical for a fixed pair $\wt g,g$.} A metric on $M$ determines a trivialization of the normal bundles $N_i \to B_i$. This in turn yields a canonical splitting 
\begin{equation}\label{eq:splitting-over-corner}
    TM|_C \cong TC \oplus\bigoplus N_i
\end{equation}
and we write $n_i$ for the inward-pointing unit normal vector in $N_i$.
By construction of the generalized blow-up, we have a canonical isomorphism $E \,\cong\, \bP_+ (\bigoplus_i N_i)$ to the positive part of the spherical projectivisation of the normal bundle at $C$, defined in Equation~\ref{eq:positive-part}. Using the canonical trivialization of $N_i$, this shows that $E \cong C\times \Delta^{\ell} $ where $\Delta^{\ell}$ is identified with the intersection of $S^{\ell}\cap [0,\infty)^{\ell+1}$. Thus, the pair $(g,\wt g)$ yields a canonical splitting 
\begin{equation*}
    T\wt M |_E \;\cong\; N_{E /\wt M}\,\oplus \,\pr_C^*TC \,\oplus\, \pr_{\Delta}^* T\Delta^{\ell},
\end{equation*} 
with $\ker(d\beta) = \pr_{\Delta}^*T\Delta^{\ell}$. In particular, $d\beta$ restricts to an isomorphism
\begin{equation*}
     (\pr_\Delta^*T\Delta^\ell|_{C\times\star})^\perp\,\xlongrightarrow{\simeq}\, \beta^*(TC\oplus N_1)|_{C\times \star},
\end{equation*}
where $\star = (1,0,\dots 0)\in \Delta^\ell$. The tangent space $T_\star \Delta^\ell $ is canonically identified with $ \{ 0\} \times \bR^\ell$ in $\bR^{\ell+1}$ and we write $e_1,\dots,e_\ell$ for the standard basis of $\bR^\ell$.
Fix an isomorphism 
\begin{equation*}
    \Phi \cl \pr_\Delta^*T \Delta^\ell|_{C\times\star} \to \im(d\beta|_{E})^\perp
\end{equation*}
such that $\Phi(c,\star,e_i) = n_{i+1}(c)$ for $i = 1,\dots,\ell$ under the isomorphism~\eqref{eq:splitting-over-corner}. The space of such isomorphisms is a contractible space. Therefore, the bundle isomorphism 
\begin{equation}
    T\wt M |_E \;\xra{\simeq}\; N_{\wt M / E} \oplus \pr_C^*TC \oplus \pr_{\Delta}^* T\Delta^{\ell} \;\xlongrightarrow{\beta_* \oplus \beta_* \oplus \Phi}\; \beta^* TM
\end{equation}
is canonical up to contractible choice.
\end{proof}

\begin{remark}\label{rem:real-oriented-stable-complex}
    In the case of a real-oriented blow-up $\text{Bl}_D(X)$, we do not require the choice of a metric on $X$ to be able to pull back an almost complex structure. Then, the preimage $E$ of $D$ is canonically isomorphic to $\bP_{> 0}(N_{D/X})$, whence we have a free $S^1$-action on $E$. Thus, a metric on $\text{Bl}_D(X)$ gives us a decomposition $$T\text{Bl}_D(X)|_E \cong \beta^*TD\oplus \bR\oplus L,$$
    where $L$ is the canonical line of $\bP_{> 0}(N_{D/X})$. On $\beta^*TD$ we have a canonical complex structure $J_D$ and we extend it to $J$ on $T\text{Bl}_D(X)|_E$ by mapping the unit section of $L$ to the unit vector of $\bR$ (corresponding to the canonical vector field of the action). 
\end{remark}

We are now going to apply this to the base space of the topological flow category $\scM$. Given a smooth manifold $M$ with the action of a compact Lie group $G$, it will be useful to write $T\ov{M}$ for the virtual vector bundle $TM -\fg$, where $\fg$ is the Lie algebra of $M$. In the result below we will not explicitly indicate the quotient by the group action by any tangent bundle should be considered in that sense.

\begin{lemma}\label{lem:stable-complex-base}
    For any $\Gamma^-,\Gamma^+\in \cP_{\leq L}$ and any partition $\Lambda\cl \Gamma^-\to \Gamma^+$ of $\Gamma^-$, there exists 
    \begin{enumerate}[label=\normalfont\arabic*),leftmargin=20pt,ref=\normalfont\arabic*]
        \item a complex $\wh G_\Lambda$-vector bundle $I^b_\Lambda\to \scBS_\Lambda$
        \item an equivalence 
        \begin{equation}
            T\cc{\scB}^\bR_\Lambda\,\oplus\,\fp\fu_\Lambda\,\oplus\, \bR \;\simeq\;I^b_\Lambda\,\oplus\, \bR^{\Gamma^+}
        \end{equation}
        of $\wh{G}_\Lambda$-equivariant virtual vector bundles on $\scBS_\Lambda$.
        \item for any factorization $\Gamma^-\xra{\Lambda^0 } \Gamma\xra{\Lambda^1}\Gamma^+$ of $\Lambda$ a split equivariant embedding 
        \begin{equation}
            I^b_{\Lambda^0}\,\oplus I^b_{\Lambda^1}\,\to\, I^b_\Lambda
        \end{equation}
        of complex equivariant vector bundles over $\scBS_{\Lambda^0}\qtimes{\bT_\Gamma}\scBS_{\Lambda^1}$.
        \end{enumerate}
        They satisfy the following compatibility conditions.
        \begin{itemize}[leftmargin=25pt]
            \item The diagram 
        \begin{equation}\label{compatiblity-1}\begin{tikzcd}
		   I^b_{\Lambda^0}\,\oplus \bR^{\Gamma}\,\oplus I^b_{\Lambda^1}\,\oplus \bR^{\Gamma^+}\arrow[r,""] \arrow[d,""]&I^b_\Lambda \oplus \bR^{\Gamma}\,\oplus \bR^{\Gamma^+}\arrow[dd,""]\\
          T\cc{\scB}^\bR_{\Lambda^0}\,\oplus\fp\fu_{\Lambda^0}\,\oplus \bR_{\Lambda^0}\,\oplus T\cc{\scB}^\bR_{\Lambda^1}\,\oplus\fp\fu_{\Lambda^1}\,\oplus \bR_{\Lambda^1}\arrow[d,"\simeq"]\\
         \ft_{\Gamma}\,\oplus T\cc{\scBS_{\Lambda^0}\qtimes{\bT_\Gamma}\scBS_{\Lambda^1}}\,\oplus\bR_{\Lambda^0}\,\oplus\bR_{\Lambda^1}\,\oplus  \fp\fu_{\Lambda^0}\,\oplus \fp\fu_{\Lambda^1} \arrow[r,""] &\ft_{\Gamma}\,\oplus T\scBS_{\Lambda}\,\oplus \fp\fu_\Lambda\oplus\bR_\Lambda\end{tikzcd} \end{equation}
         commutes, where $\ft_\Gamma$ is identified canonically with $\bR^{\Gamma}$, $\bR_{\Lambda^1}$ is identified with $\bR_{\Lambda}$, and $\bR_{\Lambda^0}$ is identified with the normal vector of the boundary stratum of $\scBS_\Lambda$. Moreover, we identify $\fp\fu_\Lambda$ with $\fp\fu_{\Lambda^0} \oplus \fp\fu_{\Lambda^1}\oplus \frac{\fp\fu_\Lambda}{\fp\fu_{\Lambda^0}\oplus \fp\fu_{\Lambda^1}}$ and the latter summand with the normal bundle of the image of $\cc{\scB^\bR_{\Lambda^0}\qtimes{\,\bT_\Gamma}\scB^\bR_{\Lambda^1}}$ in $\ov{\scB}^\bR_\Lambda$ using Lemma~\ref{lem:boundary-up-to-stabilisation}.
        \item For any factorization $\Lambda = \Lambda^2\g \Lambda^1\g \Lambda^0$ the square
        \begin{equation}\label{compatibility-2}\begin{tikzcd}
		   I^b_{\Lambda^0}\,\oplus I^b_{\Lambda^1}\,\oplus I^b_{\Lambda^2}\arrow[r,""] \arrow[d,""]&I^b_{\Lambda_0}\,\oplus I^b_{\Lambda^{12}} \arrow[d,""]\\
         I^b_{\Lambda^{01}}\,\oplus I^b_{\Lambda^2} \arrow[r,""] & I^b_\Lambda\end{tikzcd} \end{equation}
         over $\scBS_{\Lambda^0}\qtimes{\;\bT_\Gamma}\scBS_{\Lambda^1}\qtimes{\;\;\bT_{\Gamma'}}\scBS_{\Lambda^2}$ commutes.
         \end{itemize}
\end{lemma}

\begin{proof}
     Recall that the base space of $\scM(\Gamma^-,\Gamma^+)_\Lambda$ is the manifold $\scBS_\Lambda$ equipped with an action of $\wh G_\Lambda\coloneqq\bT_{\Gamma^-}\times\bT_{\Gamma^-}\times G_\Lambda$, where $\bT_\Gamma := (S^1)^\Gamma$ and $G_\Lambda = \p{\gamma\in \Gamma^+}{U(d_{\Lambda_\gamma})}$. In particular, it carries the structure of a principal $\bT_{\Gamma^-}\times\bT_{\Gamma^-}$-bundle $\pi_\Lambda\cl \scBS_\Lambda\to \lc{\scB}^\bR_\Lambda.$ The map $\pi_\Lambda$ is given by forgetting the asymptotic markers. The space $\lc{\scB}^\bR_\Lambda$ is a stratum of the (corner) blow-up of $\lc{\scB}^P_\Lambda\times[0,1)^{\Gamma^+}$, where $\lc{\scB}^P_\Lambda$ is a real oriented blow-up of the complex manifold 
    $$\scB(\Lambda)\sub \p{\gamma\in \Gamma^+}{\Mbar_{0,\gamma\sqcup \Lambda_\gamma}(\bP^{d_{\Lambda_\gamma}},d_{\Lambda_\gamma})_{\varphi(z_\gamma) = [1:0:\dots:0]}}.$$
Given any factorization $\Lambda = \Lambda^1\g \Lambda^0$, we have an embedding 
\begin{equation}\label{eq:embeddings-base-for-stable-complex}
    \scBS_{\Lambda^0}\qtimes{\;\bT_\Gamma}\scBS_{\Lambda^1}\to \scBS_\Lambda
\end{equation}
covering
\begin{equation}\label{eq:embeddings-pardon-base-for-stable-complex}
    \scBR_{\Lambda^0}\qtimes{\;\bT_\Gamma}\scBR_{\Lambda^1}\to \scBR_\Lambda,
\end{equation}
both constructed in \textsection\ref{subsec:embeddings-base}. Using \cite{Kott}, we can construct systems $\{\wt{g}_\Lambda\}_\Lambda$ and $\{g_\Lambda\}_\Lambda$ of invariant Riemannian metrics on $\scBS_\Lambda$ and $\scBR_\Lambda$, respectively, so that the embeddings~\eqref{eq:embeddings-base-for-stable-complex} and~\eqref{eq:embeddings-pardon-base-for-stable-complex} are isometric. By Remark~\ref{rem:real-oriented-stable-complex}, the metric $g_\Lambda$ induces a complex structure on $T\lc{\scB}^P_\Lambda$. Using $\wt g_\Lambda$ and $g_\Lambda$ in Lemma~\ref{lem:stable-complex-corner-blow-up}, we get a decomposition 
\begin{equation*}
    T\scBS_\Lambda \;\cong\; T\lc{\scB}^\bR_\Lambda\oplus \ft_{\Gamma^+}\,\oplus\, \ft_{\Gamma^-}\;\cong\; T\Delta^{|\Gamma^+|-1}\,\oplus\,\beta^*T\lc{\scB}^P_\Lambda\,\oplus \,\ft_{\Gamma^+}\,\oplus \,\ft_{\Gamma^-}
\end{equation*}
Identifying $T\Delta^{k-1}$ with $\bR^k/\bR$, this shows that 
\begin{equation}\label{eq:decomposition-1}
    T\scBS_\Lambda\,\oplus\, \bR \;\cong\; T\bR^{\Gamma^+}\,\oplus\,\beta^*T\lc{\scB}^P_\Lambda\,\oplus\, \ft_{\Gamma^+}\,\oplus\, \ft_{\Gamma^-}
\end{equation}
canonically. Thus, we can set $I^{b,+}_\Lambda \coloneqq \beta^*T\lc{\scB}^P_\Lambda$. The decomposition~\eqref{eq:decomposition-1} can be further transformed to 
\begin{equation}\label{eq:decomposition-2}
    T\cc{\scB}^\bR_\Lambda\,\oplus\,\fg\fl_\Lambda\,\oplus\, \bR \;\cong\; T\bR^{\Gamma^+}\,\oplus \,\beta^*T\lc{\scB}^P_\Lambda\,\oplus\, \fp\fu_\Lambda
\end{equation}
where $\fg\fl_\Lambda \stackrel{\eqref{eq:comparing-group-quotients}}{=}\ii\fu_\Lambda\oplus\fp\fu_\Lambda$ is the Lie algebra of the complex Lie group $\cG_\Lambda$. Therefore, we can set $$I^{b,-}_\Lambda \coloneqq\fg\fl_\Lambda.$$
\noindent The maps $I^{b,-}_{\Lambda^0}\,\oplus I^{b,-}_{\Lambda^1}\to I^{b,-}_{\Lambda}$ are induced by the block matrix inclusions $\cG_{\Lambda^0}\times\cG_{\Lambda^1}\hkra\cG_{\Lambda}$, whence their compatibility in the sense of the square~\eqref{compatibility-2} is immediate. Meanwhile, the embeddings $I^{b,+}_{\Lambda^0}\,\oplus I^{b,+}_{\Lambda^1}\to I^{b,+}_{\Lambda}$ come from the embeddings~\eqref{eq:embedding-base-spaces} of base spaces. The commutativity of~\eqref{compatibility-2} follows in their case from the inductive choice of Riemannian metrics. Since the vertical maps in Diagram~\eqref{compatiblity-1} are defined using these compatible metrics, the commutativity of this diagram follows as well.
\end{proof}

\subsubsection{Stable complex structures on vertical tangent bundle}\label{subsec:stable-complex-fiber}
We will first prove the existence of stable complex structures on each morphism space $\scM(\Gamma^-,\Gamma^+)$ separately before considering their compatibilities. Our strategy is similar to the one of \cite[\S 11.3.4]{AB21}. The idea is to construct a homotopy equivalence $\wt\scT_\Lambda\to \scT_\Lambda$ with two sections $a$ and $b$ and a family of Cauchy--Riemann operators over $\wt\scT_\Lambda$ whose restrictions to $\im(a)$ and $\im(b)$ are given by 
$$\delbar^{\gamma}\,\#\, D\delbar_J\qquad \text{and}\qquad D^\bC\,\# \,(\#_{\gamma'\in \Lambda_\gamma}\delbar^{\gamma'}),$$
respectively, where $D^\bC$ is a (non-canonical) complex-linear Cauchy--Riemann operator. 

\begin{notation*}
   Given a partition $\Lambda\cl \Gamma^-\to \Gamma^+$, we write $S = S_\Lambda$ for the set of Reeb orbits $\gamma$ (counted with multiplicities) so that $\gamma\in \Gamma\sm\im(\Lambda^0)$ for some factorization $\Gamma^-\xra{\Lambda^0}\Gamma\xra{\Lambda^1}\to \Gamma^+$ of $\Lambda$. Note that we allow $\Lambda^1 = \ide_{\Gamma^+}$.
\end{notation*}

\begin{lemma}\label{lem:auxiliary-thickening}
	There exists for each function $\Lambda\cl \Gamma^-\to \Gamma^+$, there exists a $G_\Lambda$-equivariant fiber bundle 
	$$q\cl \scT_{\Lambda}^c\to \scT_{\Lambda}$$ with fibers given by $[0,1]$. It admits two equivariant sections $a_\infty$ and $a_0$.
\end{lemma}

\begin{proof}
	Recall that we chose a pre-perturbation datum $\fD = (\wt\lambda,\conn,p)$ for the construction of $\scM$. We observe first that we can construct the global Kuranishi chart for buildings $\Nbar^{\,J}(\Gamma^-,\Gamma^+)_\Lambda$ in the trivial symplectic cobordism using the $\fD$. In this case, the base space agrees with the base space $\scBS_\Lambda$ of the global Kuranishi chart for $\Mbar^{\,J}_{\sft}(\Gamma^-,\Gamma^+)_\Lambda$, but the thickening is defined using the family $\cZ^c_\lambda$ in Definition~\ref{de:family-of-cobordism-buildings} instead of Definition~\ref{de:family-of-tree}. There exists a canonical rel--$C^1$ submersion 
    \begin{equation}\label{eq:forgetful-1}
        \scT^c_\Lambda\;\to\; \scT_\Lambda,
    \end{equation}
    which is the identity on the level of base spaces. Its fibers are canonically identified with $[0,1]$. The map $\scT^c_\Lambda\to\scT_\Lambda$ has two canonical sections
     \begin{equation*}
         a_0 \cl \Nbar^{\,J}(\Gamma^-,\Gamma^-)\ov{\times}_{B\Gamma^-}\scT_\Lambda \hkra \scT^c_\Lambda\qquad \quad a_\infty \cl \scT_\Lambda\ov{\times}_{B\Gamma^+}  \Nbar^{\,J}(\Gamma^+,\Gamma^+)\times \hkra \scT^c_\Lambda,
     \end{equation*}
     given by adding trivial cylinders (in the cobordism) at the positive puncture or at the negative one. If the compactified domain is given by the sphere in the middle of Figure~\ref{fig:auxiliary-spheres}, then the section $a_\infty$ is given by the configuration on the left, while $a_0$ is given by the right configuration on the right hand side:
     
	\begin{figure}[H]
		\centering
		\includegraphics[scale=1.4]{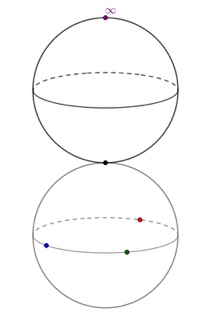}\hspace{0.8cm}
		\includegraphics[scale=1.2]{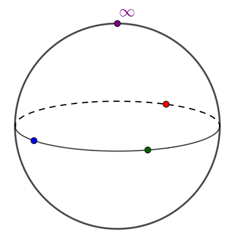}\hspace{.8cm}
		\includegraphics[scale=0.9]{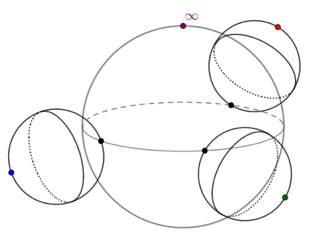}
		\label{fig:auxiliary-spheres}
		\caption{$\qquad\qquad\qquad$}
	\end{figure}
	\noindent Here the black spheres are those components that are mapped to the symplectic cobordism (being trivial cylinders in the left and right configuration), while the gray components are mapped to the symplectisation. The point labeled by $\infty$ is the positive puncture, while all other marked points correspond to negative punctures.
\end{proof}

Fix a Riemannian metric $h$ on $\wh Y$ and let $B_r(\gamma)$ be the ball of radius $r > 0$ around a Reeb orbit $\gamma$. Fix some number $\delta > 0$ that is less than the injectivity radius of $h$. Let $d$ be the associated distance and $\Phi$ be the parallel transport of $h$ and write $\Phi_{x\to y}$ for parallel transport along the unique geodesic connecting $x$ to $y$, whenever $d(x,y)$ is sufficiently small.

\begin{proposition}\label{prop:interpolating-to-complex}
	There exists a principal $G'_\Lambda$-bundle $\wt\scT_\Lambda\to \scT^c_\Lambda$, which admits a $G'_\Lambda\times \wh G_\Lambda$-vector bundle $\cW\to \wt\scT_{\Lambda}$, a vector bundle $W^v_\Lambda$, and a complex virtual vector bundle $I^v_{\Lambda}$ so that
	\begin{equation}
		a_\infty^*\cW \,\cong\, T^v\scT_{\Lambda}\,\oplus\,\bR^{\Gamma^+}\,\oplus\, V^+_{\Gamma^+}\,\oplus\, V^-_{\Gamma^-}\,\oplus I^{v,-}_\Lambda\,\oplus\, W^v_\Lambda
	\end{equation}
        and
    \begin{equation}
        a_0^*\cW \,\cong\,E_\Lambda\,\oplus\, I^{v,+}_{\Lambda}\,\oplus\, V_{\Gamma^-}^+\,\oplus\, V_{\Gamma^+}^-\,\oplus\, W^v_\Lambda,
    \end{equation}
	where $a_\infty$ and $a_0$ are the sections of Lemma~\ref{lem:auxiliary-thickening}.
\end{proposition}

\begin{proof}
	Let $p' \gg p$ be a large prime number, and let $\wt\fD = (\wt\lambda,\conn,p',p)$ be the pre-perturbation datum for the moduli spaces of buildings in the trivial symplectic cobordism. Then, we can construct the global Kuranishi chart for $\Nbar^{\,J}(\Gamma^-,\Gamma^+)$ using the base space $\scB^c_\Lambda$ defined in \textsection\ref{subsec:cobordism-base}. Let 
	$$\wt\cB_\Lambda \sub \p{\gamma\in \Gamma^+}{\Mbar_{0,\gamma\sqcup \Lambda_\gamma}(\bP^{d'_{\Lambda_\gamma}}\times \bP^{d_{\Lambda_\gamma}},(d'_{\Lambda_\gamma},d_{\Lambda_\gamma}))_{z_\gamma\mapsto (e_0,e_0)}}$$
	be the preimage of $\scB^c_\Lambda\times\scB_\Lambda$. The argument of \cite[Lemma~7.3]{HS22} shows that $\wt\cB_\Lambda$ is unobstructed, thus a complex manifold with a complex $\cG'_{\Lambda}$-action. We can thus define 
	\begin{equation}
		\wt\scB_\Lambda = \scB^c_\Lambda\times_{\cB_\Lambda}\wt\cB_\Lambda.
	\end{equation}
	The only difference between $\scB^c_\Lambda$ and $\wt\scB_\Lambda$ is that framings in the former are constant on the domains of trivial cylinders, while those in $\scB^c_\Lambda$ are not (they have degree $(p'-p)\cA_{\wt\lambda}(\gamma)$). Thus, the universal family $\wt\cC_\Lambda\to \wt\scB_\Lambda$ allows us to access these domains. In order to lift this to the thickenings, we fix a map $\zeta'_\cU$ as in~\eqref{eq:slice-map} that determines the unitary framings in $\scB^c_\Lambda$ and set 
	\begin{equation}
		\wt\scT_\Lambda \coloneqq \set{(\wt\varphi,\varphi,u,w)\in \wt\scB_\Lambda\times_{\scB^c_\Lambda}\scT^c_\Lambda\mid T_{\wt\varphi} = T_u\, \zeta'_\cU(\wt\varphi,u) = 0}.
	\end{equation}
	The conditions we impose are open in the fiber product, so this is a rel--$C^1$ manifold with corners. Moreover, the canonical forgetful map $\wt\scT_\Lambda\to \scT^c_\Lambda$ is a principal $G'_\Lambda$-bundle, where $G'_\Lambda$ is a product of unitary groups. Let $\cC^{sc}\sub \wt\cC_\Lambda\inn$ be the locus of points that are mapped to the symplectic cobordism. Then, the height functions of the respective map give us a rel--$C^1$ function
	\begin{equation}
	\eva_\bR\cl \wt \cC^{sc}\to \bR,
	\end{equation}
	which extends to a continuous function $\ov\eva_\bR\wt\cC_\Lambda\to [-\infty,\infty]$. 
	Using $\eva_\bR$, we can extend $\chi$ to a rel--$C^1$- map  $\wt\chi\cl \wt\cC_\Lambda\to [0,1]$. Note that $\wt\chi$ is invariant under the covering group action.\par
	Let $\cF\to \wt\scT_\Lambda$ and $\cE^{0,1}\to \wt\scT_\Lambda$ be the s bundles with
	$$\cF_{(\wt\varphi,u,w)} = W^{\ell,2,\delta}\lbr{\wt\cC\inn_{\wt\varphi},u^*T\wt Y} \qquad \qquad \cE^{0,1}_{(\varphi,u,w)} = W^{\ell,2,\delta}\lbr{\wt\cC\inn_{\wt\varphi},\Lambda^{0,1*}_{\wt\cC\inn_{\wt\varphi}}\otimes u^*T\wt Y},$$
	where $\ell \ge 6$ and $\delta > 0$ is a sufficiently small exponential weight.
	Let $B_u\in \Omega^{0,1}(\dot{C},\End(u^*T\wt Y))$ be the difference $D^{\conn^Y}_u - (\conn^Y\cdot)^{0,1}$. We define the rel--$C^1$ family $\wt D_{(\wt\varphi,u,w)}$ of Cauchy--Riemann operators by 
	\begin{equation}\label{eq:interpolating-cr-operator}
		\wt D_{(\wt\varphi,\varphi,u,w)}(\xi) = D^{\conn^Y}_u(\xi) -\wt\chi(\wt\varphi,\varphi,u,w,\cdot)B_u(\xi).
	\end{equation}
	Writing the domain of $a_\infty(\varphi,u,w)$ as $\djun{\gamma\in \Gamma^+}C^\infty_\gamma\vee C_\gamma$, where $C^\infty_\gamma$ is the `cylinder on top' (see Figure \ref{fig:auxiliary-spheres}), this operator is the restriction of
	\begin{equation}
		\wt D_{a(\varphi,u,w)} =\begin{cases}
			D^{\conn^Y}_u(\xi) \quad& \text{on }C^\infty(C_\gamma,u^*T\wt Y) \\
			\delbar^{\wt\gamma}\quad& \text{on }C^\infty(C_\gamma^\infty,c_{\wt\gamma}^*T\wt Y),
		\end{cases}
	\end{equation}  
	where $c_{\wt\gamma}$ is the trivial cylinder a parametrisation $\wt\gamma$ of $\gamma$, induced by $u$ and the asymptotic marker at $z_\gamma$. On the other hand, writing the domain of $a_0(\varphi,u,w)$ as $\djun{\gamma\in \Gamma^+}C_\gamma\vee \bigvee\limits_{\gamma'\in \Lambda_\gamma}C^0_{\gamma'}$, where $C^0_{\gamma'}$ are the `additional cylinders at the bottom', $\wt D_{b(\varphi,u,w)}$ is the restriction of 
	\begin{equation}
		\wt D_{b(\varphi,u,w)} =\begin{cases}
			(\conn^Y\xi)^{0,1} \quad& \text{on }C^\infty(C_\gamma,u^*T\wt Y) \\
			\delbar^{\wt\gamma'}\quad& \text{on }(C_{\gamma'}^0,c_{\wt\gamma'}^*T\wt Y),
		\end{cases}
	\end{equation}
	where $\wt\gamma'$ is the parametrization of $\gamma'$ determined by $u$ and the asymptotic marker at $z_{\gamma'}$.\par 
	We use finite-dimensional complex vector spaces of perturbations to achieve surjectivity of $\wt D$. However, first we extend the maps $\nu_\gamma\cl V^-_\gamma\to \Omega^{0,1}_c(\bR\times S^1,u_\gamma^*T\wt Y)^\bR$ of \textsection\ref{subsec:lift-of-objects} to maps $\nu_\gamma\cl p_\gamma^*V^-_\gamma\to \cE^{0,1}$ for $\gamma\in \Gamma^+\sqcup \Gamma^-$, where $p_\gamma$ is the evaluation map~\eqref{eq:evaluation-to-reeb-orbit}. We now explain how we can construct extensions of the perturbations~\eqref{eq:perturbation-for-point}.

    \begin{construction}\label{con:extending-perturbations}  We have the evaluation map  
    \begin{equation}
    \eva_Y\cl \wt\cC\inn_\Lambda \to Y : (\wt\varphi,u,w,z) \mapsto \pr_Y(u(z)),
    \end{equation} Choose now for $\gamma\in \Gamma^\pm$ a neighborhood $U_\gamma^\pm$ of the canonical section $\sigma_\gamma\sub \wt\cC_\Lambda$ given by the respective marked point so that
	\begin{itemize}
		\item $\eva_Y({U^\pm_\gamma})\sub B_\delta(\gamma)$,
		\item the closures are pairwise disjoint, and
		\item the closures do not meet the critical points of $\wt\cC_\Lambda\to \wt\scT_\Lambda$.
	\end{itemize}  
	Fix for each $\gamma$ a rel--$C^1$ invariant cut-off function $j_\gamma\cl \wt\cC_\Lambda\to [0,1]$ which is identically $1$ near $\sigma_\gamma$ and supported in $U_\gamma$. Recall that we have a canonical retraction $r \cl B_\delta(\gamma)\to \im(\gamma)$. Let $\Phi_x$ be the parallel transport along the normal geodesic from $r(x)$ to $x$ (using the Riemannian metric $h$ on $Y$ chosen above). Then, the vector bundle map
    \begin{equation}
    \wt\nu_\gamma \cl \wt\scT_\Lambda\times_{E\gamma} V_\gamma^- \,\to\, \cE^{0,1}
    \end{equation}
    at $y = (\wt\varphi,\varphi,u,w)$ is given by
    \begin{equation}\label{eq:extension-orbit-lift}
    	\wt\nu_\gamma(y,v) \coloneqq 
    \begin{cases}
        j_\gamma(y,\cdot)\wt\chi(\eva_\bR(y,\cdot))\,\Phi_{u_Y(\cdot)}[\nu_\gamma(p_\gamma(y),v)(r_\gamma(u_Y(\cdot)))]\qquad & \text{on }\wt\cC_y\cap U^\pm_\gamma,\\
        0 \quad & \text{otherwise}.
    \end{cases}
    \end{equation}
	where the superscript of $U^\pm_\gamma$ is determined by whether $\gamma$ labels a positive or negative puncture. This is well-defined even if $\gamma$ is multiply-covered by the definition of the map $\nu_\gamma$ in \textsection\ref{subsec:lift-of-objects}.
\end{construction}

   \noindent By~\cite{Kott}, we can arrange for the choices in Construction~\ref{con:extending-perturbations} to be compatible across boundary strata. We now do a similar extension for $\nu_\gamma$ whenever $\gamma$ is a Reeb orbit at which a building in $\Mbar_{\sft}^{\,J}(\Gamma^+,\Gamma^-)$ breaks. Since any element of $\Gamma^+$ has action $\leq L$ and each component has a unique positive puncture, we have $\cA_\lambda(\gamma)\leq L$. By the construction in Theorem~\ref{thm:flo_bim}, each boundary stratum of $\wt\scT_\Lambda$ admits a map
    \begin{equation}\label{eq:bottom-section}
        \wt\scT_{\Lambda^{01}}\to\scT^c_{\Lambda^{01}}\to G_\Lambda\times_{G_{\Lambda^0}\times G_{\Lambda^1}}\im\lbr{\scT^c_{\Lambda^0}\ov{\times}_{B\Gamma}\scT_{\Lambda^1}\hkra \wt\scT_\Lambda},
    \end{equation}
    respectively,
    \begin{equation}\label{eq:infinity-section}
        \wt\scT_{\Lambda^{10}}\to \scT^c_{\Lambda^{10}}\to G_\Lambda\times_{G_{\Lambda^0}\times G_{\Lambda^1}}\im\lbr{\scT_{\Lambda^0}\ov{\times}_{B\Gamma}\scT^c_{\Lambda^1}\hkra \wt\scT_\Lambda}
    \end{equation}
	for some $\Gamma$ and partitions $\Lambda^1 \cl \Gamma\to \Gamma^+$ and $\Lambda^0 \cl \Gamma^-\to \Gamma$ with $\Lambda^1\g\Lambda^0 = \Lambda$. The first map in each case is induced by the forgetful map $\wt\scT_\Lambda\to \scT^c_\Lambda$, while the second map is a vector bundle map (cf. Lemma~\ref{lem:boundary-up-to-stabilisation}). The evaluation map $p_\gamma\cl \scT_{\Lambda^1}\to E\gamma$ for $\gamma\in \Gamma$ is invariant under the $G_{\Lambda^1}$-action, so it induces a rel--$C^1$ map 
	\begin{equation}\label{eq:extension-orbit-evaluation}
		\wt p_\gamma\cl \wt\scT_{\Lambda^{01}}\to E\gamma
	\end{equation}
    that is equivariant with respect to the action of $G_\Lambda$.
    \begin{construction}\label{con:extending-breaking-perturabtions} Let $U_{ij}\sub \wt\scT_{\Lambda^{ij}}$ be an open subset of the boundary stratum so that for any curve in $U_{ij}$ the component that is mapped to the symplectization is either directly above the Reeb orbits $\gamma\in \Gamma$ or below it. In other words, $U_{ij}$ is the complement of a union of some of the boundary strata. We can use the same definition as in Construction~\ref{con:extending-perturbations} once we have constructed for each such $\gamma$ an equivariant extension of the map $\wt{p}_\gamma|_{U^\gamma_{ij}}$. We discuss the case of the embedding~\eqref{eq:infinity-section}, and will just write $U$. Let $\sigma_\gamma\cl \wt\scT_{\Lambda^{ij}}\to \wt\cC_\Lambda$ be the section given by the nodal point labeled by $\gamma\in \Gamma^+$. For each such $\gamma$, let $U''_\gamma$ be a neighborhood of $\sigma_\gamma(U)$ so that 
    \begin{itemize}
        \item $\eva_Y(U''_\gamma\sm \im(\sigma_\gamma))\sub B_\delta(\gamma)$,
        \item $\eva_\bR|_{U''_\gamma\cap \wt\cC\inn_\Lambda}$ is a fiber-wise submersion
        \item $\cc{U''_\gamma}$ does not intersect $\cc{U_{\gamma'}}$ for any $\gamma'\in \Gamma^+\sqcup \Gamma'$
        \item $\cc{U''_\gamma}\cap \cc{U''_{\gamma'}} = \emst$ if $\gamma\neq \gamma'$ are Reeb orbits at which buildings in $\wt\scT_\Lambda$ break,\footnote{Note that $\gamma$ might appear twice as a Reeb orbit at which a building breaks. These two occurrences are distinct, although we omit this from the notation.}
        \item $\cc{U''_\gamma}$ only meets the critical points of $\wt\cC_\Lambda\to \wt\scT_\Lambda$ in $\im(\sigma_\gamma)$.
    \end{itemize}
    In particular, we may assume that the intersection $\wt\cC_y\cap U''_\gamma$ with a fiber is contained in the unique irreducible component $C^\bullet_y$ of $\wt\cC_y$ that is mapped to the symplectic cobordism. Set $U'_\gamma := U''_\gamma\sm\im(\sigma_\gamma)$. The asymptotic marker at the positive puncture $\Lambda^1(\gamma)$ and the matching isomorphism on nodes induce for each $y\in U$ a path $L^\gamma_y\sub C_y^\bullet$ from the unique positive `puncture' (which is a point in $\wt\cC_\Lambda$) of $C_y^\bullet$ to $\sigma_\gamma(y)$. Since $\pi|_{U_\gamma}\cl U'_\gamma\to \pi(U'_\gamma)$ is a rel--$C^1$ submersion, we can use Ehresmann's Lemma and the connection induced by the chosen Riemannian metric and the Fubini-Study metric on complex projective space to obtain for each $y \in \pi(U'_\gamma)$ an interval $L_y \sub C^\bullet_y$, possibly replacing the whole thickening by an invariant neighborhood of the zero locus $\fs\inv(0)$. 
    Writing $d$ for the Gromov--Hausdorff metric on $\wt\scT_\Lambda$, we define 
    \begin{equation}
        R_\gamma\cl \pi(U'_\gamma) \to [0,\infty) : y \mapsto d(y,\wt\scT_\Lambda)
    \end{equation}
    be the distance from the boundary stratum, where $\pi \cl \wt\cC_\Lambda\to \wt\scT_\Lambda$ is the universal family. Then, we can extend $\sigma_\gamma$ to a rel--$C^1$ map $\wt\sigma_\gamma\cl \pi(U'_\gamma) \to U''_\gamma$ by letting $\wt\sigma_\gamma(y)=:z$ be the unique point in the fiber so that 
    $$\eva_\bR(y,z) = R_\gamma(y) \qquad\qquad z\in L_y.$$
    Such a point exists and is unique due to the asymptotic behavior of the curves near nodes. Finally, we can define $\wt{p}_{\gamma,\gamma'}\cl \pi(U'_\gamma)\to E\wt\gamma$ via the composition 
    \begin{equation*}
      \pi(U'_\gamma)\,\xra{\wt\sigma_\gamma}\, U''_{\gamma}\,\xra{(y,z)\,\mapsto r(\eva_Y(u,z))} \,\im(\gamma) \,\xra{\simeq} \, E\gamma,
    \end{equation*}
    where the first map associates to $y$ the point $r(\eva_Y(\sigma_{\gamma'}(y))$ and the second map is the inverse of $\wt\gamma\mapsto \wt\gamma(1)$. Now we may define the extension by the term~\eqref{eq:extension-orbit-lift} with $p_\gamma$ replaced by $\wt p_\gamma$.
    \end{construction}
    
    \noindent We define 
	\begin{equation}
		W^v_\Lambda \coloneqq \bigoplus\limits_{\Gamma^- \prec \Gamma\prec \Gamma^+}\bigoplus\limits_{\gamma\in \Gamma} V_\gamma^-
	\end{equation}
	equipped with the induced map $\wt\nu_\Lambda\cl W^v_\Lambda \to \cE^{0,1}$.
	Since $q \cl \wt\scT_\Lambda\to \scT_\Lambda$ is proper, the preimage $q\inv(\fs\inv(0))$ is compact. Thus, we may find a finite-rank complex $G$-vector bundle $I^{v,-}_{\Lambda}\to \scT_\Lambda$ over the `original' thickening and an equivariant map $\kappa_\Lambda\cl I^{v,-}_\Lambda\to \cE^{0,1}$ so that the operator
	\begin{equation}\label{eq:perturbed-homotopy-of-cr}
		\wt D_{(\wt\varphi,u,w)}+\nu_{\Gamma^+}+\nu_{\Gamma^-}+\wt\nu_\Lambda+\kappa_\Lambda \cl \cF_{(\wt\varphi,u,w)}\,\oplus\, V_{\Gamma^+}^-\oplus\, V_{\Gamma^-}^-\oplus\, W^v_\Lambda\oplus {I^{v,-}_\Lambda}_{_{(\wt\varphi,u,w)}}\to\, \cE^{0,1}_{(\wt\varphi,u,w)}
	\end{equation}
	is surjective for any ${(\wt\varphi,\varphi,u)}\in q\inv(\fs\inv(0))$. Shrinking $\scT_\Lambda$, we may assume~\eqref{eq:perturbed-homotopy-of-cr} is surjective on all of $\wt\scT_\Lambda$. We define 
	\begin{equation}
		\wt\cW \coloneqq \ker\lbr{\wt D + \mu_\Lambda+\nu_{\Gamma^+}+\nu_{\Gamma^-}+\wt\nu_\Lambda+\kappa_\Lambda} 
	\end{equation}
	where $\mu_\Lambda\cl E_\Lambda\to \cE^{0,1}$ is the perturbation space chosen in the construction of the global Kuranishi chart. Since $\wt\cW$ is invariant under the free $G'_\Lambda$ action, the quotient $\cW \coloneqq \wt\cW/G'_\Lambda$ admits a vector bundle map $\cW\to \scT^c_\Lambda$,
    Since $\wt\nu_\Lambda$ vanishes near $a_\infty$ and $a_0$, we have a canonical isomorphism
	\begin{equation*}
		a_\infty^*\cW \;\cong\; q^*T^v\scT_\Lambda\,\oplus\, \bR^{\Gamma^+} \,\oplus\, V_{\Gamma^+}^+\,\oplus\, V_{\Gamma^-}^-\,\oplus\, I^{v,-}_\Lambda\,\oplus\, W^v_\Lambda,
	\end{equation*}
	where $\ker(D_y + \mu_\Lambda(y)) = T^v_y\scT_\Lambda\oplus \bR^{\Gamma^+}$ because we quotient by $\bR$-translations in the target. Meanwhile, 
	\begin{equation*}
		a_0^*\cW \;\cong\; E_\Lambda\,\oplus\, I^{v,+}_\Lambda\,\oplus\, V_{\Gamma^-}^+\,\oplus\, V_{\Gamma^+}^-\,\oplus\, W^v_\Lambda,
	\end{equation*}
	whose first part $I^{v,+}_\Lambda \coloneqq \ker((\conn^Y\cdot)^{0,1} + \mu_\Lambda+\kappa_\Lambda)$ is the kernel of a complex-linear surjective Cauchy--Riemann operator and thus carries a canonical complex structure.
\end{proof}

\subsubsection{Compatibilities of stable complex structures}  We can now make Theorem~\ref{thm:stable-complex-structure} precise, phrasing the statement in terms of the partitions $\Lambda$ to keep the notation consistent. 

\begin{proposition}\label{prop:existence-stable-complex}
	Possibly after shrinking the thickenings of the morphism spaces of $\scM$, there exists for any pair of objects $\Gamma^-,\Gamma^+$ and any partition $\Lambda\cl \Gamma^-\to\Gamma^+$ an equivalence 
    \begin{equation}
     T\scT_{\Lambda} \,\oplus\, V_{\Gamma^+}\,\oplus\, W_{\Lambda}\,\oplus \bR\;\simeq\; I_{\Lambda}\,\oplus\, V_{\Gamma^-}\,\oplus\, W_{\Lambda}
    \end{equation}
    of virtual vector bundles that is compatible with the symmetric actions. For any factorization $\Gamma^- \xra{\Lambda^0} \Gamma\xra{\Lambda^1}\Gamma^+$ of $\Lambda$ we have an equivariant split embeddings
    \begin{equation}\label{eq:embeddings-index-bundles}
         I_{\Lambda^0}\,\oplus\, I_{\Lambda^1}\;\to \;I_{\Lambda}
    \end{equation}
     \begin{equation}\label{eq:embeddings-stabilising-bundles}
         W_{\Lambda^0}\,\oplus\, W_{\Lambda^1}\;\to \;W_{\Lambda}
    \end{equation}
    over $\scT_{\Lambda^0}\ov{\times}_{B\Gamma}\scT_{\Lambda^1}$ so that the square of a
    \begin{equation*}\begin{tikzcd}
		T\scT_{\Lambda^0}\,\oplus\, V_{\Gamma^-}\,\oplus\, W_{\Lambda^0}\,\oplus\, \bR_{\Lambda^0}\,\oplus T\scT_{\Lambda^1}\,\oplus\, V_{\Gamma}\,\oplus\, W_{\Lambda^1}\,\oplus\, \bR_{\Lambda^1} \arrow[r,""] \arrow[d,""]& T\scT_{\Lambda}\,\oplus  V_{\Gamma^+}\,\oplus\, W_{\Lambda}\,\oplus\, \bR_\Lambda \arrow[d,""]\\ 
        I_{\Lambda^0}\,\oplus\, V_{\Gamma^-}\,\oplus\, W_{\Lambda^0}\,\oplus  I_{\Lambda^1}\,\oplus\, V_{\Gamma}\,\oplus\, W_{\Lambda^1}\arrow[r,""] &  I_{\Lambda}\,\oplus\, V_{\Gamma^-}\,\oplus\, W_{\Lambda} \end{tikzcd} \end{equation*}
    commutes, where $\bR_{\Lambda^1}$ is mapped to $\bR_\Lambda$ and $\bR_{\Lambda^1}$ to the normal bundle of the boundary stratum. Moreover, the restrictions are compatible across the boundary strata of codimension $2$.
\end{proposition}

\begin{proof}[Proof of Proposition~\ref{prop:existence-stable-complex}]
    Recall that $$T\cT_\Lambda = (T\ov{\scB}^\bR_\Lambda\,\oplus\, T^v\scT_\Lambda,E_\Lambda\,\oplus\,\fp\fu_\Lambda),$$
    where $\fp\fu_\Lambda$ is the Lie algebra of $\p{\gamma\in \Gamma^+}{\PU_{d_{\Lambda_\gamma}+1}(\bC)}$.
    We define
    \begin{equation}
        I_\Lambda \coloneqq I^b_\Lambda\,\oplus\,I^v_\Lambda\quad\qquad U_\Lambda \coloneqq (0,\bR) \qquad\quad W_\Lambda \coloneqq \bR^{\Gamma^+}\oplus W^v_\Lambda.
    \end{equation}
    By Lemma~\ref{lem:obtain-isomorphism-from-homotopy} and Lemma~\ref{lem:stable-complex-base}, it suffices to show that we can choose the vector bundles of Proposition~\ref{prop:interpolating-to-complex} compatibly over boundary strata. 
    Due to the decompositions~\eqref{eq:tangent-bundle-decomposition} and~\eqref{eq:obstruction-bundle-decomposition}, we can discuss the compatibility of the stable complex structures on the tangent bundle of the base spaces and the vertical tangent bundles separately. The compatibility for the base spaces was shown in \S\ref{subsec:stable-complex-base}.
    The existence of the split embeddings~\eqref{eq:embeddings-index-bundles} follows from Lemma~\ref{lem:stable-complex-base} and by constructing the bundles $I^{v,-}_{\Lambda}$ in the proof of Proposition~\ref{prop:interpolating-to-complex} inductively as in Step 2 of the proof of Lemma~\ref{lem:embedding-inductive-step}, respectively, of Proposition~\ref{prop:embedding-gkc}. Meanwhile, the existence of the split embeddings $I^{v,+}_{\Lambda^0}\,\oplus\, I^{v,+}_{\Lambda^1}\to I^{v,+}_\Lambda$ follows immediately from the inductive construction of the perturbation spaces. This completes the proof.
\end{proof}

\begin{lemma}\label{lem:obtain-isomorphism-from-homotopy}
    Suppose $\scX$ is a symmetric flow category and $\scY$ a symmetric flow bimodule from $\scX$ to itself, admitting for each $x,y\in \scX$ a fibration $q_{xy}\cl \scY(x,y)\to \scX(x,y)$ with fibers given by the interval $[0,1]$, so that the system $\{q_{xy}\}$ is compatible with the structural maps of the flow bimodule, including the symmetric actions. Suppose also that $q_{x,y}$ admits two boundary sections $a_+,a_-$ so that the squares
    \begin{equation*}\begin{tikzcd}
		\scX(x,y)\,\ov{\times}_{y}\,\scX(y,z) \arrow[d,"a_-\times \ide"] \arrow[r,""]&\scX(x,z)\arrow[d,"a_-"]\\ 
       \scY(x,y)\,\ov{\times}_{y}\,\scX(y,z)  \arrow[r,""]&\scY(x,z) \end{tikzcd} \qquad \qquad \begin{tikzcd}
		\scX(x,y)\,\ov{\times}_{y}\,\scX(y,z) \arrow[d,"\ide\times a_+"] \arrow[r,""]&\scX(x,z)\arrow[d,"a_+"]\\ 
       \scX(x,y)\,\ov{\times}_{y}\,\scY(y,z)  \arrow[r,""]&\scY(x,z) \end{tikzcd} \end{equation*}
    commute. Then, for any vector bundle $\scW\to \scY$ that is compatibly a flow bimodule itself, there exists a contractible space of systems $\{\Phi_{x,y}\}_{x,y}$ of compatible isomorphisms 
    \begin{equation}\label{eq:compatible-isomorphisms}
        \Phi_{xy}\cl a_+^*\scW(x,y)\to a_-^*\scW(x,y)
    \end{equation}
\end{lemma}

\begin{proof}
    Using \cite{Kott}, we inductively choose a system of invariant Riemannian metrics so that the composition maps are isometric. Then, we use the Riemannian metrics to inductively construct countable locally finite open covers $\cU^{xy} = \{B_{r_i}(x_i)\}_i$ so that $\cU^{xz}\cap( \scX(x,y)\ov{\times}-y\scX(y,z))= \cU^{xy}\,\ov{\times}_{y}\,\cU^{yz}$. Inductively, choose trivializations $$\wt U_i := q\inv(U_i)\xra{\psi_i} U_i\times [0,1],$$ so that $\psi_i(a_-(b)) = (b,0)$ and $\psi_i(a_+(b)) = (b,1)$, and $$W|_{\wt U_i}\xra{\wt\psi_i} U_i\times[0,1]\times W_{a_+(x_i)},$$
    covering $\psi_i$ and which are compatible across boundary strata. Inductively, define orderings of the open covers $\cU^{xy}$ so that the inclusions $\cU^{xy}\,\ov{\times}_{y}\,\cU^{yz}\hkra \cU^{xz}$ are order-preserving. Then, the construction in the proof of \cite[Theorem~(14.3.1)]{tD08} carries over to yield the desired compatible isomorphisms.
\end{proof}

\subsection{A colimit and a cylindrical contact flow category}\label{subsec:colimit-proof}
Our construction of global Kuranishi charts for the moduli spaces $\Mbar^{J}_{\sft}(\Gamma^+,\Gamma^-;\beta)_\Lambda$ requires us to approximate the contact form $\lambda$ by a $1$-form with integral action on the Reeb orbits in order to construct framings. This forces our construction of the flow category $\scM^{Y,\lambda}_{\le L}$ to depend on the action bound $L$. To obtain the ``full flow category'' of $(Y,\lambda)$ we take a colimit of the induced diagram of flow categories. Since we do not prove that the composition of our flow bimodules is (homotopic to) the flow bimodule of the `larger' cobordism, the full flow category will be a telescope of the flow categories of Theorem~\ref{thm:sft-flow-category}.

\begin{proposition}\label{prop:colimit-of-flow-cats}
    There is a directed system  \[ \scM^{Y,\lambda}_1 \rightarrow  \scM^{Y,\lambda}_2 \rightarrow \dots \] of the symmetric stably complex flow categories constructed in Theorem \ref{thm:flow-cat}. 
\end{proposition}

\begin{definition}
    We call the colimit $$\scM^{Y,\lambda} \coloneqq \colim\limits\,\scM^{Y,\lambda}_n$$ in $\Flow^{\Sigma,U}$ ``the" \emph{contact flow category} of the contact manifold $(Y,\lambda).$ 
\end{definition}   

While we expect $\scM^{Y,\lambda}$ to be independent of the required auxiliary choices, we do not show this invariance here.

\begin{proof} From the discreteness of $\spec(\lambda) \sub \bR$, we can find an order-preserving enumeration $a\cl \bN \to \spec(\lambda)$. We choose $\cP_L$-integral approximations by taking any $\cP_{a(1)}$-integral approximation $\wt\lambda_1$ and selecting the $\cP_{a(n+1)}$-integral approximation $\wt\lambda_n$ such that the action $\cA_{\wt \lambda_n}(\gamma)$ divides $\cA_{\wt \lambda_{n+1}}(\gamma)$ whenever $\cA_\lambda(\gamma) \leq a(n)$.
Choose prime integers $p_n$ so that for any $n\geq 1$, the pair $(p^+,p^-) = (p_{n+1},p_n)$ satisfies~\eqref{eq:primes-for-base}. Let $\wh X = \wh Y$ be the trivial cobordism equipped with the chosen almost complex structure $J$ and let $\scN^{n}_{\wh Y}$ be the bimodule associated to $(\wh X,J)$ and the following pre-perturbation datum $\fD^\pm_n$ given by
\begin{itemize} 
    \item $L^+ = a(n+1), L^- = a(n)$,
    \item $\cP^-_n = \cP_{a(n)}$ and $\cP^+_n = \cP_{a(n+1)}$,
    \item $p^+ = p_{n+1}$, $p^- = p_n$.
\end{itemize}
The bimodules $\scN^{n}_{\wh Y}$ induce the directed system $\scM^{Y,\lambda}_1 \;\stackrel{\scN^{1}_{\wh Y}}{\longrightarrow}\;  \scM^{Y,\lambda}_2\;\xrightarrow[]{\scN^{2}_{\wh Y}} \;\dots$ of the claim.\par
It remains to show that they admit stable complex structures compatible with those of $\scM^{Y,\lambda}_i$ from Theorem~\ref{thm:stable-complex-structure}.
Fix thus $i\in \bN$. By \cite[Definition~4.17]{AB24}, respectively its adaption to our setting, we have to construct for any function $\Lambda\cl \Gamma^-\to \Gamma^+$ between a sequence of Reeb orbits considered as objects of $\scM^{Y,\lambda}_i$ and $\scM^{Y,\lambda}_{i+1}$, respectively, the data of
\begin{enumerate}
    \item a complex virtual vector bundle $\wt I_\Lambda$
    \item a vector bundle $\wt W_\Lambda$
    \item the vector space $\wt U_\Lambda = (\bR,\bR^{\{\Gamma^+\}})$
\end{enumerate}
 on $\scN^i_\Lambda\coloneqq\scN^i_Y(\Gamma^-,\Gamma^+)_\Lambda$ together with an equivalence
 \begin{equation}
     T\scN^i_\Lambda\,\oplus\, V^i_{\Gamma^+} \,\oplus\, \wt W_\Lambda\ \cong\ \wt I_{\Lambda}\,\oplus\, \wt U_\Lambda\,\oplus\, \wt{W}_\Lambda\,\oplus\,V^{i+1}_{\Gamma^+}
 \end{equation}
 where $V^j_\Gamma$ is the vector bundle constructed in \S\ref{subsec:lift-of-objects} with the perturbation data used for $\scM^{Y,\lambda}_j$. These are required to satisfy analogues of the compatibility conditions described in Definition~\ref{de:stable-complex-flow-cat}. As in \textsection\ref{subsec:stable-complex}, the construction of such a stably complex lift of $\scN^{\wh Y}_i$ can be split up into the construction of stable complex structures on the base spaces, cf. Lemma~\ref{lem:stable-complex-base}, and those on the vertical tangent bundle. The generalization of Lemma~\ref{lem:stable-complex-base} to the cobordism base spaces is straightforward, because the key input, Lemma~\ref{lem:stable-complex-corner-blow-up}, is a general statement that also holds for cobordism base spaces. The construction of the stable complex structures on the vertical tangent bundles is exactly the same as in \S\ref{subsec:stable-complex-fiber}. The only difference is that we have $T^v\scN^i_\Lambda = \ker(D(\delbar_J) + \mu_\Lambda)$, while $T^v\scM^{Y,\lambda}_i\oplus  \bR = \ker(D(\delbar_J) + \mu_\Lambda)$ due to the fact that we quotient by translations in the target for curves in the symplectization. This accounts for the difference between $U_\Lambda$ and $\wt U_\Lambda$. Choosing the auxiliary data needed for the construction of the stable complex structures inductively, their compatibility with the composition maps in the sense of Proposition~\ref{prop:existence-stable-complex} follows. 
\end{proof}

In the construction of the full contact flow category above, we were forced to take the colimit approach due to a lack of an integral approximation $\wt \lambda$ without first filtering through the action. However, we can construct a flow category for cylindrical contact homology directly. We collect the necessary modifications to obtain a cylindrical contact flow category.

\begin{theorem}\label{thm:cyl-flow-cat}
    Let $(Y,\lambda)$ be a hyper-tight contact manifold with a given compatible almost complex structure $J$ interpolating between cylindrical almost complex structures $J^\pm$. Then, there exists a flow category $\scM_{cyl}$ of class rel--$C^1$ whose objects are the Reeb orbits of $\lambda$ and whose morphism spaces are 
    \begin{equation*}\label{eq:morphisms-cylindrical}
        \scM_{cyl}^{Y,\lambda}(\gamma^-,\gamma^+) = \Mbar_{\sft}^{\, J}(\gamma^+,\gamma^-)
    \end{equation*}
    for any pair $(\gamma^+,\gamma^-)$ of Reeb orbits.
\end{theorem}

\begin{proof}
Recall that the construction of the global Kuranishi charts in \S\ref{subsec:construction} started off with a choice of a pre-perturbation datum $\fD = (\wt\lambda,\conn,p)$ as in Definition~\ref{de:pre-perturbation}. The integral approximation $\wt \lambda$ was only used to frame holomorphic buildings; cf. \S\ref{sssec:framings_of_buildings}. We can instead use the following scheme to frame cylindrical holomorphic buildings. Let $$\cA_e\cl \spec(\lambda) \to \bN$$ denote the ordered enumeration function on the spectrum of the contact form $\lambda$. Given a cylindrical holomorphic building $u$, we replace the complex line bundle in \eqref{eq:linebundle-construction} with the following cylindrical version
\begin{equation}\label{eq:cyl-linebundle-construction}
	L^c_{u_v} \coloneqq \cO_{C_v}\lbr{{\cA_e(\gamma_{e^+}) \; z_{v,e^+}} - {\cA_e(\gamma_{e^-})\; z_{v,e^-}}}^{\otimes p}.
\end{equation}
We use sections of such line bundles to frame cylindrical buildings. The map 
\begin{gather}
    E\cl \Mbar_{\sft}^J(\gamma^+,\gamma^-)^{cyl} \to \bN  \\
    [u] \mapsto \cA_e(\gamma^+) - \cA_e(\gamma^-)
\end{gather}
that sends a cylindrical holomorphic building to the difference of $\cA_e$-action of the incoming and outgoing Reeb orbit, satisfies the following properties:
\begin{itemize}
    \item $E$ is additive under cylindrical breaking, i.e., $E([u_1]\#[u_2])  =  E([u_1]) + E([u_2]), $
    \item $E\geq 0$ and $E([u])=0$ if and only if $[u]$ is a trivial cylinder.
\end{itemize}
    The proof of Theorem~\ref{thm:flow-cat} now carries over verbatim except for replacing every mention of $\cA_{\wt \lambda}$ with $\cA_{e}$, using the line bundles of \eqref{eq:cyl-linebundle-construction}, and assuming that each $\Gamma$ is a singleton set.
\end{proof}
\subsection{Recovering contact homology}\label{subsec:contact homology}
We sketch that our flow category recovers contact homology, a classical invariant of closed contact manifolds, following \cite{Par19}. Let $(Y^{2n-1},\xi)$ be a closed contact manifold equipped with a non-degenerate contact form $\lambda$ as before. Fix a choice of $\lambda$-adapted almost complex structure $J$ and action bound $L$ as well as the perturbation data required for the construction of the stably complex flow category $\scM^{Y,\lambda}_{\leq L}$. We will outline how this data yields a $\bZ/2$-graded chain complex $(CC_*(Y,\xi)_\lambda^{\leq L},\del)$, where
    \begin{equation}\label{eq:chains}
    CC^{\leq L}_*(Y,\xi)_\lambda \coloneqq \bigoplus\limits_{k\geq 0}\text{Sym}^k_\bQ\lbr{\bigoplus\limits_{\gamma\in (\cP_{\leq L})_{\text{good}}}\fo_\gamma\otimes\bQ} ,
\end{equation}
and the differential on the space generated by $\fo_{\Gamma^+} = \fo_{\gamma_1}\dots \fo_{\gamma_k}$ is induced by the map 
 \begin{equation}
     \partial_{\Gamma} \;\coloneqq\;\s{\Gamma^-}{\;\vfc{\scM^{Y,\lambda}_{\leq L}(\Gamma^-,\Gamma)_{(0)}}}
 \end{equation}
where $\scM^{Y,\lambda}_{\leq L}(\Gamma^-,\Gamma^+)_{(0)}$ is the part of the morphism space of virtual dimension $0$ and $\vfc{\cdot}$ is any choice of virtual count as in \cite[\S4]{Par19} or \cite[\S9]{BH23}. The fact that these virtual counts define well-defined maps between orientation lines follows from the existence of stable complex structures, Theorem~\ref{thm:stable-complex-structure} on the flow category $\scM^{Y,\lambda}_{\le L}$.\par 
For the sketch, recall that the orientation line $\fo_\gamma$ defined in \textsection\ref{subsec:orientation} and has parity 
$$|\fo_{\gamma}| = |\gamma| = \text{sign}(\det(\ide-A_\gamma)),$$ 
where $A_\gamma$ is the asymptotic operator associated to $\gamma$. We set $|\Gamma| = \sum_{i=1}^k|\gamma_i|$ for $\Gamma = (\gamma_1,\dots,\gamma_k)$ and use this to grade the chain complex~\eqref{eq:chains}. The key properties of the maps $\del_\Gamma$ are
\begin{enumerate}
    \item $\del_\Gamma$ has odd degree for each $\Gamma$
    \item the maps satisfy the Leibniz rule 
$$\partial_{(\gamma_1,\gamma_2)} = \partial_{(\gamma_1)}\otimes\ide + (-1)^{|\gamma_1|}\ide\otimes\partial_{(\gamma_2)}.$$
\end{enumerate}  
Both follow from studying the relevant moduli spaces of buildings. Then, the maps $\del_\Gamma$ descend to a map $\del$ on $CC_*(Y,\xi)_\lambda^{\leq L}$ and it remains to show that $\del^2 = 0$. Once this is proven, the \emph{contact homology of action at most $L$} of $(Y,\lambda)$ will be
    $$HC^{\le L}_*(Y,\xi)_\lambda := H_*(CC^{\leq L}_*(Y,\xi)_\lambda,\del).$$

\appendix
\section{Gluing Results}\label{sec:gluing}

This section develops the tools required for showing that the thickening $\cT$ is a topological manifold over the base-space $\cB^{\bR}$ where the fibers carry a smooth structure. Our presentation of the gluing results is similar to \cite[Appendix A]{AMS23} and \cite[\textsection 5]{Par19}.

Let $(\wh Y,J)$ be the symplectization of a contact manifold $(Y,\lambda)$ equipped with an $\bR$-invariant almost complex structure $J$, which satisfies $J(\partial_s) = R$ where $\partial_s$ is the infinitesimal vector field generated by the translation $\bR$ action and $R$ is the Reeb vector field. Let $\Mbar \to \cB^{\bR}$ be a smooth map and define $\cC$ via the pull-back square 
$$
\begin{tikzcd}
\Cbar \arrow{r} \arrow{d} & \overline{\cC_{\cB}} \arrow{d} \\
\Mbar \arrow{r}{}& \cB^{\bR}.
\end{tikzcd}
 $$

\noindent Recall that $\cB^{\bR}$ is the base space, which was obtained from a real-oriented blow-up of certain divisors of $\Mbar^*(\bP^{N}).$ In particular, each fiber in the universal bundle consists of the data of
\begin{itemize}
    \item a curve $C$
    \item matching conditions at certain nodes of $C$
\end{itemize}
Let $\Cbar^0$ denote the complement of the nodal points in $\Cbar$ and $\mathfrak{Y}_{\Cbar}:= \Omega^{0,1}_{\Cbar/\Mbar} \otimes_{\bC} T\wh Y$
be the bundle over $\Cbar^0 \times \wh Y$ which consists of $T \wh Y$-valued anti-holomorphic forms on the vertical bundle of $\Cbar \to \Mbar.$  Let $W$ be a real vector space with a linear map $$ A  \cl W \to \Gamma^\infty_c (\mathfrak{Y}_{\Cbar} ).$$

Given a decorated corolla $T$, we define the regular locus of moduli of buildings $\Mbar_T^{\,\normalfont\text{reg}}(\wh Y)$ to be 

\begin{align}\label{gluing-square} 
    \Mbar_T^{\,\normalfont\text{reg}}(\wh{Y}) = 
	 \left\{
		\begin{array}{l|l}
			\nu \in \Mbar & u \ \textnormal{is smooth and is of type }   T \\ 
			u \cl \Cbar|_\nu \rightarrow{} \wh Y & \delbar_J u + A(w)(\cdot ,u(\cdot)) = 0 \\
			w \in W & u \ \textnormal{is regular} \\
		\end{array} 
		\right\}.
\end{align}
We denote the fiber over a point $\nu \in \Mbar$ by $\Mbar^{\,\normalfont\text{reg}}_T(\wh{Y})|_\nu$. The regularity assumption in \eqref{gluing-square} implies that each fiber $\Mbar^{\,\normalfont\text{reg}}_T(\wh{Y})|_\nu$ carries a unique smooth structure. However, this smooth structure does not necessarily extend to a smooth structure on all of $\Mbar^{\reg}_T(\wh Y)$ due to the resolution of nodes. The next gluing statement shows a somewhat weaker assertion. It is the main result of this section.

\begin{theorem}\label{thm:gluing-main}
    For any $(\nu, u, w) \in  \Mbar_T^{\,\normalfont\text{reg}}(\wh{Y})$, there exists a neighborhood $N$ of $(\nu, u, w)$ in the fiber $\Mbar^{\,\normalfont\text{reg}}_T(\wh{Y})|_\nu$ and a neighborhood $B$ of $\nu$ in $\Mbar$ admitting an embedding $g\cl B \times N \to \Mbar^{\,\normalfont\text{reg}}_T(\wh{Y})$ that fits in the commuting square 
   $$ \begin{tikzcd}
			{B \times N} & {\Mbar^{\,\normalfont\text{reg}}_T(\wh{Y})} \\
			B & { \Mbar}
			\arrow[ from=1-2, to=2-2]
			\arrow["g",hook, from=1-1, to=1-2]
			\arrow["i"', hook, from=2-1, to=2-2]
			\arrow["\proj"', from=1-1, to=2-1]
    \end{tikzcd}$$
    such that the restriction of $g$ to each fiber of the trivial fibration $B\times N \to B$ is smooth.
\end{theorem}

We first prove a special case of Theorem \ref{thm:gluing-main}, assuming that $\Mbar \to \cB^{\bR}$ is an open embedding. Then, we use an exponentiation-along-pullback-of-the-tangent-bundle trick to extend it to any map $\Mbar \to \cB^{\bR}$.

\begin{proposition}\label{prop:c1loc}
    For any two gluing maps $g_1,g_2$ as constructed in Theorem \ref{thm:gluing-main} the restrictions to the fibers over $B_1 \cap B_2$ 
    \begin{equation}
        g_{21,b} \coloneqq g_2\inv \circ g_1|_{b} \cl (\{b\} \times N_1 )\cap g_1\inv(g_2(B_2 \times N_2)) \to \{b\} \times N_2 
    \end{equation}
     vary continuously in the $C^1_{loc}-$topology.
\end{proposition}

\begin{corollary}\label{cor:tangent-bundle}
    There exists a well-defined \emph{vertical tangent bundle}
\begin{equation}\label{eq:vert-tangent-bundle}T^v\Mbar_T^{\reg}(\wh Y) \to \Mbar_T^{\reg}(\wh Y)\end{equation}
with fibers given by
\begin{equation}\label{eq:vert-tangent-bundle-fiber}
T^v_{(b,u)}\Mbar_T^{\reg}(\wh Y) = \ker(D_\varphi+ P).\end{equation}
\end{corollary}

\subsection{Setup for Gluing} 
We will present the gluing treatment as done in \cite[\S 5]{Par19} for the sake of completeness. The main gluing theorem as stated in op. cit. is a local homeomorphism result, but actually the gluing map is smooth for a fixed gluing parameter. In our discussion we will point out how to extract that result from \cite[\S 5]{Par19} and use it to prove Theorem \ref{thm:gluing-main}.
To be consistent with notations in \cite{Par19} we recall some relevant concepts here. 

\subsubsection{Preliminaries on \textsection 2.2 and 2.5 of \cite{Par19}}
In \textsection\ref{subsec:buildings}, we defined the categories $\scS$ stratifying the moduli spaces $\Mbar^J_P(\Gamma^+,\Gamma^-;\beta)$ of Pardon buildings. The space of gluing parameters of decorated tree $T$ as in Definition~\ref{de:decorated-tree} is defined to be 
$$G_{T/} \coloneqq (0,\infty]^{E^{int}(T)}.$$
Given a gluing parameter $\mathbf{\ell}= \{\ell_e \}_e$, the tree type $T_\ell$ is obtained by contracting all the edges for which $\ell_e < \infty.$ In certain setups, it is conceptually beneficial to reparametrize $G_{T/} \cong [0,1)^{E^{int}(T)}$ by applying the function $x\to e^{-x}$ to each factor.

\begin{convention}
    In this section we will implicitly assume that the gluing parameter space $\bC^{N_0}$ actually denotes a small open neighborhood of $0$ in $\bC^{N_0}$.
\end{convention}

\subsubsection{Target Gluing}\label{ssc:targ_glu}

 Fix a point $(\nu,[u],w)\in \Mbar_T^{\reg}(\wh Y)$ in the fiber over $\nu \in \cBR$. For a gluing parameter $\ell \in G_{T_\nu/}$, we define the \textit{glued target} $\wh Y_\ell$ as follows. For each positive (resp. negative) end we truncate $[0,\infty) \times Y$ (resp. $(-\infty,0] \times Y$) to $[0,\ell_e]$ (resp. $[-\ell_e,0]$)and identify the truncated ends corresponding to the edge $e\in E^{int}(T)$ by translation by $\ell_e$. If $\ell_e = \infty$ for some edge $e$, we do nothing for that edge. The target constructed by these gluing operations is denoted as $\wh Y_\ell$. 

 We also select local sections $q_i'$ of the universal curve over $\cBR$ near $\nu$  for every curve corresponding to a vertex of $T_\nu$, such that $\pi_\bR (u(q_i'(\nu)) = 0.$  In our gluing construction, we will send these marked points to their corresponding $0-$levels in the glued target $\wh{Y}_{\textbf{g}}$. This choice of sections is intuitively a gauge-fixing for the $\bR$ action on the target. In particular, it allows us to pick a map $u\cl \cC^0_{\nu} \to \bR \times Y $ in the equivalence class of maps $[u]$.

\subsubsection{Gluing in the base}\label{ssc:base_gluing}

Before elaborating on the gluing chart in the base space, we show how we can reduce to the case of the real-oriented blow-up of the moduli space of stable curves.

\begin{lemma}\label{thm:stab_map_to_curv}
    For any $\varphi \in \cBR(d)$ there exist $d'=d(d+2)$ hyperplanes $D_i$ in $\bP^d$ such that $\wh\varphi$ intersects $D_i$ transversely away from the marked points and there is an open embedding 
    \begin{align*}
        U &\to \Mbar^\bR_{0,d'+ E^{\normalfont\text{ext}}} \\ 
        [\check{\phi},C,j,x,m]&\mapsto  [C,j,x\cup\check{\phi}\inv(D_1)\cup \dots\cup\check{\phi}\inv(D_{d'}),x]
    \end{align*}
    of a neighborhood of $\varphi_0$ into real-oriented blow-up $\Mbar^\bR_{0,d'+ E^{\normalfont\text{ext}}(T)}$ of $\Mbar_{0,d'+E^{\normalfont\text{ext}}(T)}$ at the boundary divisors.
\end{lemma}

\begin{proof}
    This is a generalization of \cite[Lemma~5.15(3)]{BX22}, which can be shown in the same way, using \cite[Proposition~6.5]{AMS21}.
\end{proof}

For the element $\nu = (\varphi_0,C_0,j_0, m)$ of $\cBR$, we will construct a neighborhood in $\cBR$ based on the asymptotics of the map $u$ . Recall that $(C_0,j_0)$ is a marked closed Riemann surface, $\varphi_0 \cl C_0 \hkra \bP^N$ is a holomorphic embedding into $\bP^N$ and $m$ is the additional decoration of asymptotic markers and matching conditions at certain nodes. Pick cylindrical charts of the form $[M,\infty) \times S^1$ or $(-\infty,M]\times S^1$ near each puncture of $C_0$ such that the map $u$ satisfies $$u(s,t) = (Ls,\gamma(t)) + O(e^{-ks}).$$
\noindent Note that such a choice of cylindrical end is equivalent to the choice of a tangent vector $v\in T_pC_v$ at the puncture $p$.

To get a gluing chart near $\nu \in \cBR$, we use the identification from Lemma~\ref{thm:stab_map_to_curv}. We will use the local model of the moduli space of genus $0$ curves with asymptotic markers and matching conditions as \cite[\S 2.6 - 2.7, 5.1.3]{Par19}. We denote the component of the curve $C_0$ corresponding to the vertex $v\in V(T)$ as $C_0^v$. Assume that $C_0^v$ has $N_v$ nodes. Recall that nodes do not carry matching conditions.
Pick a chart
$$\scJ_v :=  \bC^{2\#p_v + \# q_v  -N_v - 3} \to \Mbar_{0,q_{v,i}, {p_{v,i}}_{(2)} }^{\#\text{nodes}=\#N_v} $$ 
which sends $0$ to the curve $(C_0^v,j_0^v)$ and $z \in \scJ_v$ to $(C_0^v,j_z^v)$, where $j_z^v$ is a complex structure agreeing with $j_0^v$ near the special points. 
The subscript $(2)$ denotes that the points $p_{v,i}$ are doubly-marked, i.e., the tangent space $T_{p_{v,i}}C_0^v \sm \{0 \}$ carries a marking (equivalently, there is an associated complex isomorphism $\bC \to T_{p_{v,i}}C_0^v$).
Moreover, we can identify $\scJ_v$ with an open neighborhood of $0$ in the cokernel of the map  
\begin{align*}
&\set{ 
  X \in C^\infty(C^v_0, TC^v_0) \;\middle|\;
  \begin{array}{l}
    X \text{ is holomorphic near } p, q, \\
    X(p_{v,i}) = X(q_{v,i}) = 0, \\
    dX(p_i) = 0
  \end{array}}
\,\to\, C^\infty_c(C^v_0 \sm \cup \{ p_{v,i}\} \cup \{ q_{v,i}\},TC^v_0 \otimes \Omega^{0,1}_{C^v_0})\\
&\qquad\qquad \qquad \qquad \qquad \qquad \qquad\qquad\qquad \qquad \qquad \quad X\mapsto \scL_X{j_0}
\end{align*}

\noindent We can then obtain a local diffeomorphism $$ \modulo{\lbr{\prod_{ v\in V(T)} \scJ_v \times \bC^{\# N_v}}}{\sim} \; \; \to f\Mbar_{0,d'+ E^{\normalfont\text{ext}}(T)}|_{T_\nu}\to \cBR|_{T_\nu}.$$
where the quotient $\sim$ is taken over relations induced from the $\bR_>0$ action on the tensor at the doubly marked points. In particular, for every edge $e = (v,v')$, we quotient by the $\bR_{>0}$ action on $T_{p_v} C^v_0 \otimes  T_{p_{v'}} C^{v'}_0 $. Here, $\cBR|^{\text{nodes fixed}}_{T_\nu}$ is the locus of all framed curves of tree type $T_\nu$ with a fixed number of nodes.
Abusing notation, we denote the image of this diffeomorphism by $\cBR|^{\text{nodes fixed}}_{T_\nu}$.
We construct a local chart near $\cBR|_{T_\nu}$ of the form,
\begin{equation}\label{base-gluing-stratum}
    \text{Glue}_{\cBR}:\cBR|^{\text{nodes fixed}}_{T_\nu}\times G_{T_\nu / }\times \bC^{N^{\inn}_{C_0}} \to \cBR,
\end{equation}
where the map $\Glue_\cBR$ is defined by a slight variation of the usual gluing in moduli of genus $0$ curves.
In particular, over the nodes in $\cN_C^\bullet$ we restrict it to the angle $ \theta = (\theta_e)_e$ determined by the matching isomorphisms in $\nu$, while the gluing parameter $\ell = (\ell_e)_e \in G_{T_\nu/}$ determines the radial component of each coordinate $\textbf{y} = (y_e)_e \in \bC^{E^{int}(T_\nu)}.$
More precisely, suppose $(C_0,j_0,m) \in \cBR$ is a curve with matching conditions $m$, which is over the fiber $(C_0,j_0) \in \cB$ under the blow-down map $\cBR \to \cB$.
Then, 
\begin{equation}
    \text{Glue}_{\cBR} ((\varphi,C,j_0,m) ,\ell,\textbf{z})  = (\varphi_{\textbf{g}},C_{\normalfont\textbf{g}},j_{\normalfont\textbf{g}},m_{\normalfont\textbf{g}}),
\end{equation} 
\noindent where $\textbf{g} = ((y_e)_e,\textbf{z}), \; y_e = (\ell_e e^{i\theta_e})$ and $(C_{\normalfont\textbf{g}},j_{\normalfont\textbf{g}})$ is obtained by the usual truncation-followed-by-gluing construction. The decoration $m_{\normalfont\textbf{g}}$ for the glued curve $(C_{\normalfont\textbf{g}},j_{\normalfont\textbf{g}})$ is induced from $(C_0,j_0,m)$ by keeping the asymptotic markers fixed and remembering the matching conditions at the edges $e$ for which $\ell_e \neq 0.$

\begin{notation*}
    From now on, we will write $\textbf{g} = (\ell,z) \in G_{T_\nu / }\times \bC^{N^{\inn}_{C_0}}$ to denote the `total' gluing parameter, which accounts for gluing along both Reeb punctures and nodes. We also abbreviate $\text{Glue}_{\cBR} (\varphi,C,j,m) , \textbf{g})$ as simply $\nu_{\textbf{g}}=(\varphi_{\textbf{g}},C_{\textbf{g}}).$
\end{notation*}

We can rewrite the gluing map in \eqref{base-gluing} in a conciser form by replacing $\cBR|_{T_\nu} \times \bC^{N^o_{C_0}}$ with the locus of $\cBR$ corresponding to the tree type $T_\nu$,  $\cBR|_{T_\nu}$ .
\begin{equation}\label{base-gluing}
    \text{Glue}_{\cBR}:\cBR|_{T_\nu}\times G_{T_\nu / } \to \cBR,
\end{equation}

\subsubsection{Linearization with respect to varying domain complex structures}
Recall that the linearization of the section $\delbar (\cdot ) + A(\cdot ) (\cdot , \cdot ) $ at the point $(\nu,u,w) \in \Mbar$ is given as 

\begin{equation}
    D^v_0 \cl \wkpd(C_\nu, u^* T\wh Y) \oplus W \to W^{k-1,p,\delta } (C_\nu, u^* T\wh Y\otimes \Omega^{0,1}TC_\nu),
\end{equation}
where the superscript $v$ denotes that this is the `vertical' part of the total linearization with respect to the fibration $\Mbar^{reg}_T(\wh Y)\to \cBR.$
If $(\nu,u,w)$ is in $\Mbar^{reg}_T(\wh Y)$, then the vertical linearization $D^v_0$ is surjective. We also have the linearization $D_0$
\begin{equation*}
     D_0 \cl \wkpd(C_\nu, u^* T\wh Y)\oplus W \oplus \cJ_{T_v}  \to W^{k-1,p,\delta } (C_\nu, u^* T\wh Y\otimes \Omega^{0,1}TC_\nu)  
\end{equation*}
given by
\begin{equation*}
     D_0(\eta,w,\xi) = D_0^v(\eta,w) + J \circ du \circ\xi. 
\end{equation*} 
Due to the regularity assumption, we have the local diffeomorphisms
\begin{align}\label{eq:kernel_at_0_identif}
    \ker D^v_0 \xrightarrow{\sim} \Mbar^{reg}_T(\wh Y)_{\nu}  \\    \ker D_0 \xrightarrow{\sim} \Mbar^{reg}_T(\wh Y)_{T_\nu}, 
\end{align}
near the origin, where $ \Mbar^{reg}_T(\wh Y)_{\nu} $ is the restriction to fiber over $\nu \in \cBR$ and $\Mbar^{reg}_T(\wh Y)_{T_\nu}$ is the restriction to strata $\cBR_{T_\nu}$corresponding to the tree type $T_\nu$.

\subsubsection{Pre-Glued maps}  
For a gluing parameter, $\textbf{g} \in G_{T_\nu/} \times \bC^{N^{\inn}_{C_0}}$. Recall that $\textbf{g}$ defines a positive real parameter $ \ell_e$ for each internal edge $e$ and that there is a Reeb orbit $\widetilde\gamma_e$ for each edge $e$ such that $u_0$ is asymptotic to the trivial cylinder $(Ls,\widetilde \gamma_e(t))$ where $L$ is the $\lambda$-action of $\widetilde \gamma_e.$ We define the \textit{flattened} building $u_0|_{\normalfont\textbf{g}}$ as follows: for every internal edge $e,$ in the chosen cylindrical coordinates $(s,t)$ near the positive puncture corresponding to the edge $e$, we define

$$u_0|_{\normalfont\textbf{g}} (s,t) := \begin{cases}
   \begin{array}{cc}
       u_0(s,t)\quad& s< \frac {1}{6} \ell_e\\
        \exp_{(Ls,\widetilde\gamma_{e}(t))} [ \chi(s-\frac{1}{6}\ell_e). \exp_{(Ls,\widetilde\gamma_{e}(t))}^{-1} u_0(s,t)] \quad  &\frac{1}{6} \ell_e \leq s \leq \frac{1}{6} S + 1 \\
        (Ls,\widetilde\gamma_{e}(t))\quad   \frac{1}{6} &\ell_e + 1<s.
        
   \end{array}    
\end{cases}$$
Near the positive end of a nodal point $n$ of $(C_0,u_0)$, we define the flattening as

$$u_0|_{\normalfont\textbf{g}} (s,t) := \begin{cases}
   \begin{array}{cc}
       u_0(s,t)\quad& s< \frac {1}{6} S_{\mathsf{n}}\\
        \exp_{u_0(\mathsf{n})} [ \chi(s-\frac{1}{6}S). \exp_{u_0(\mathsf{n})}^{-1} u_0(s,t)] \quad&  \frac{1}{6} S_{\mathsf{n}} \leq s \leq \frac{1}{6} S_{\mathsf{n}} + 1 \\
        u_0(\mathsf{n})  \quad& \frac{1}{6} S_{\mathsf{n}} + 1<s.
        
   \end{array}

\end{cases}$$
where $S_{\mathsf{n}} = -\log |z_{\mathsf{n}}|$ for the gluing parameter $z_{\mathsf{n}}$ corresponding to the node $\mathsf{n}$. 
We define the flattening analogously in the negative ends of the punctures and nodes. Here the exponential maps are taken with respect to a fixed $\bR-$invariant metric on $\wh{Y}$.

\begin{definition}
    
We define the \textit{pre-glued} map $u_{\textbf{g}} \cl C_{\textbf{g}} \to \widehat Y_{\textbf{g}}$ to be the natural descent of the flattened map $u_0|_{\normalfont\textbf{g}}.$  
\end{definition}

\subsection{Gluing Estimates} 

We will now cover the required Fredholm setup and compare the linearization of the usual `Floer-function' on the pre-glued map with the linearization before gluing. We will also prove a `kernel-gluing' which is an isomorphism between kernels of linearization before and after pre-gluing. In this section our analysis slightly deviates from Pardon's setup since we consider the gluing map without variation of the domain curve.\\

\noindent\textbf{Fredholm Setup}
We recall the relevant Sobolev spaces required for gluing. We start off with selecting metrics and connections.

\begin{convention}

We fix an $\bR-$invariant metric on the target $\wh{Y}$ and  a $J-$linear connection
on $\wh{Y}$ that is induced from the pullback of the natural map  $\wh{Y} \to \{ 0  \} \times Y$.  On the domain $C$, we fix a metric that is equal to $ds^2 + dt^2$ in the cylindrical coordinates near each puncture $p_e$. We also equip the tangent bundle $TC$ with a connection for which $\partial_s$ is parallel in the cylindrical coordinates near each puncture.
\par
Different choices of metrics or connections are commensurable, so these choices do not affect the topology.
\end{convention}

\subsubsection{Weighted Sobolev norms.}

Recall that the weighted Sobolev space $W^{k,p,\delta} (E)$ of a bundle $E \to C_\g$ consists of those sections that decay at the rate $O(e^{\delta s})$ near each cylindrical end of $C_\g$. In particular, the weighted Sobolev norm $\| \cdot \|_{k,p,\delta}$ has the usual $W^{k,p}$ norm contribution away from the ends and near each cylindrical end of $C_\g$ the contribution is $$\int (\lvert \xi \rvert^p + \lvert \nabla \xi \rvert^p+ \cdots + \lvert \nabla^{k} \xi \rvert^p)e^{p\delta s} ds\; dt$$    
for a section $\xi$ supported in the cylindrical end. The norm is finally constructed by the usual bump function trick.

\subsubsection{Floer Function}
Given a point $(\nu, u_0, w_0) \in \Mbar_T^{\,\normalfont\text{reg}}(\wh{Y}),$ and a gluing parameter, $\textbf{g},$ we know that the pre-glued map, $u_{\normalfont\textbf{g}}$, does not satisfy the equation $\delbar u_{\normalfont\textbf{g}} + A(w_0) ( \cdot, u_{\normalfont\textbf{g}}(\cdot)) = 0.$ But the `defect' of being a true solution can be explicitly measured by the following function in a neighborhood of $(u_{\normalfont\textbf{g}}, w_{\normalfont\textbf{g}})$ in $\Maps(C_{\normalfont\textbf{g}},\wh{Y}).$

\begin{equation}
    \scF_{\textbf{g}}\cl \wkpd(C_{\normalfont\textbf{g}},u_{\normalfont\textbf{g}}^* T\wh Y_{\normalfont\textbf{g}}) \oplus \scJ_{T_v} \oplus W \to W^{k-1,p,\delta}(C_{\normalfont\textbf{g}},u_{\normalfont\textbf{g}}^* T\wh Y_{\normalfont\textbf{g}} \otimes \Omega^{0,1}TC_{\normalfont\textbf{g}}) \oplus \bR^{V(T_{\normalfont\textbf{g}})}
\end{equation}
$$\scF_{\textbf{g}}(\xi,y,w) :=  (\text{P{T}}_{\exp_{u_{\normalfont\textbf{g}}} \xi \to u_{\normalfont\textbf{g}}}\otimes I_{y})\left(d(\exp_{u_{\normalfont\textbf{g}}}\xi) + A(w+w_0)(\cdot, \exp_{u_{\normalfont\textbf{g}}}\xi(\cdot)) \right)^{0,1}_{j_{y}}\oplus \bigoplus_{v\in V(T_{\normalfont\textbf{g}}) } \pi_\bR (\exp_{u_{\normalfont\textbf{g}}} \xi)(q'(v)).  $$
In the above equation, the exponential maps are taken with respect to an $\bR-$equivariant metric on $\wh Y$, $I_y$ is the composition of the natural maps 
$$\Omega^{0,1}_{C_{\normalfont\textbf{g}},j_y} \to \Omega^{0,1}_{C_{\normalfont\textbf{g}}} \otimes_\bR \bC \to \Omega^{0,1}_{C_{\normalfont\textbf{g}},j_0},$$
and the parallel transport $\text{PT}$ is taken with respect to a fixed $\wh{J} _0$ linear connection which is $\bR$ equivariant.
We recall some estimates about the Floer function $\scF_{\normalfont\textbf{g}}$ from \cite[\S 5]{Par19}.

\begin{lemma}[{\cite[Lemma~5.5]{Par19}}]
    We have $\|\scF_{\normalfont\textbf{g}}(0)\|_{k-1,2,\delta} \to 0$ as $\normalfont\textbf{g} \to 0$ for all $k\geq 1.$
\end{lemma}

\begin{lemma}[{\cite[Proposition~5.6]{Par19}}]
    For $\left\|\zeta\right\|_{k,2,\delta},\left\|\eta\right\|_{k,2,\delta}\leq c_{k,\delta}'$, we have
\begin{equation}\label{quad}
\left\|\scF_{\normalfont\textbf{g}}'(0,\eta)-\scF_{\normalfont\textbf{g}}'(\zeta,\eta)\right\|_{k-1,2,\delta}\leq c_{k,\delta}\left\|\zeta\right\|_{k,2,\delta}\left\|\eta\right\|_{k,2,\delta}
\end{equation}
for constants $0 < c_{k,\delta}' , c_{k,\delta}<\infty$ which are bounded uniformly in $\normalfont\textbf{g}$ near $0$, for all $k\geq 4$.
\end{lemma}

We denote the linearization of $\scF_{\normalfont\textbf{g}}$ at $0$ by $D_{\normalfont\textbf{g}}$. Recall that, by assumption, the restriction of $D_0$ to $\wkpd(u_0^* T\wh Y) \oplus W$ is a surjective Fredholm operator. Thus, there is a right inverse $Q_0$. We now state the main kernel gluing and existence of right inverses.

\begin{proposition}
    There is a right inverse $Q_{\normalfont\textbf{g}}$ of $D_{\normalfont\textbf{g}}$ whose norm is bounded uniformly in $\textbf{g}$ and, for sufficiently small $\normalfont\textbf{g}$, an isomorphism of vector spaces 
    $$\Glue_{\ker} \cl \ker D_0 \to \ker D_{\normalfont\textbf{g}}.$$
\end{proposition}

\begin{proof}
    We omit the proof and refer the reader to \cite[\textsection 5.2.8]{Par19} for the construction of the right inverse $Q_{\normalfont\textbf{g}}$. The kernel gluing isomorphism is constructed in Equation (5.40) in op. cit.
\end{proof}

By Equation \eqref{quad}, the derivative $\scF_{\normalfont\textbf{g}}'(v,\cdot)$ is surjective for $\|v \|_{k,p,\delta} < c_{k,\delta} $. Thus $\scF_{\normalfont\textbf{g}}^{-1}(0)$ is a $C^{k-2}-$manifold, which is transverse to $\im \;Q_{\normalfont\textbf{g}}$.

\begin{proposition}\label{prop:proj_g}
    The projection map $\proj  :\scF_{\normalfont\textbf{g}}^{-1}(0) \to \ker D_{\normalfont\textbf{g}}$ along $\im \; Q_{\normalfont\textbf{g}}$ is a local diffeomorphism whose image contains $0 \in \ker D_{\normalfont\textbf{g}}.$
\end{proposition}

\begin{proof}
    This is shown in \cite[\S 5.3.1]{Par19}.
\end{proof}

\subsection{Gluing map}
We can now define a gluing map $\Glue(\textbf{g}, \cdot )\cl $The main gluing theorem in \cite{Par19} proves the following result about local homeomorphism.

\begin{proposition}\label{fixed-g-gluing} 
    Fix a point $\nu$ in the base-space $\cBR.$
    For a given $\normalfont\textbf{g}\in G_{T_\nu/} \times \bC^{N^{\inn}_{C_0}}$, the restriction of the gluing map 
    $$\Glue(\normalfont{\textbf{g}}, \;\underline{\;\;\;}\;) \cl \Mbar_{T}^{\,\normalfont\text{reg}}(\wh{Y})|_{T_\nu} \to \Mbar_{T}^{\,\normalfont\text{reg}}(\wh{Y})|_{T_{\Glue_{\cBR}(\nu,\textbf{g})}} $$ given by the equation 
    \begin{equation}\label{eq:gluing_map}
    \Glue(\textbf{g}, \;\underline{\;\;\;}\;) =  \exp_{u_{\normalfont\textbf{g}}} \circ \proj_{\ker D_{\textbf{g}}}^{-1}\circ \Glue_{\ker}\circ \proj_{\ker D_0}
\end{equation}
    is a local diffeomorphism.
\end{proposition}

\begin{proof}
    Fix a point $(u_0,w_0) \in  \Mbar_{T}^{\,\normalfont\text{reg}}(\wh{Y})|_\nu$. For a fixed $\textbf{g}$, the gluing map is the composition of local diffeomorphisms and thus itself is a local diffeomorphism. The map $\proj_{\ker D_0} \cl \Mbar_{T}^{\,\normalfont\text{reg}}(\wh{Y})|_\nu \to \ker D_0 $ is the inverse of the local diffeomorphism $\ker D_0 \to \Mbar_{T}^{\,\normalfont\text{reg}}(\wh{Y})|_\nu$ defined on a small neighborhood of the origin. The middle map $\Glue_{\ker}$ is the kernel gluing map $$\Glue_{\ker} \cl \ker D_0 \to \ker D_{\textbf{g}},$$ and $\proj_{\ker D_{\textbf{g}}}$ is the restriction of the projection map in Proposition \ref{prop:proj_g}. This finishes the proof.
\end{proof}

\begin{proof}[Proof of Theorem \ref{thm:gluing-main}]
We will only prove the case of $\Mbar \to \cB^{\bR}$ being an open embedding since the general result then follows from a similar argument as \cite[Corollary A.2]{AMS23}. Due to the strong regularity of $(\nu,u,w)$, there exists a neighborhood $B_{T_\nu}$ of $\nu$ in the submanifold $ \cBR|_{T_\nu} $ of domains of tree type $T_\nu$ and a neighborhood $N$ of $(\nu,u,w)$ in the fiber $\Mbar^{reg}_T(\wh Y)|_{(\nu,u,w)}$ such that there is an embedding $B_{T_\nu} \times N \to \Mbar^{reg}_T(\wh Y)|_{T_\nu}$. Fix a chart in a neighborhood of $B_{T_\nu}$ given by the map 
$$ \text{Glue}_{\cBR}:B|_{T_\nu}\times G_{T_\nu / } \to B \underset{open}{\subset} \cBR,$$
as described in \eqref{base-gluing}. Under the identification of a neighborhood of $(\nu,u,w)$ in the fiber $\Mbar^{reg}_T(\wh Y)|_{T_\nu}$  with $B|_{T_\nu} \times N$, we have that the gluing map, $\Glue$ can be written as 

$$\Glue \cl B|_{T_\nu} \times N \times G_{T_\nu/}  \to \Mbar^{reg}_T(\wh Y). $$

Now the result is a direct consequence of Proposition~\ref{fixed-g-gluing}.
\end{proof}

\begin{proof}[Proof of Proposition \ref{prop:c1loc}]
We rephrase the proposition using the language developed in this appendix. Let $(\nu_i,u_i,w_i) \in \Mbar^{reg}_T(\hat{Y}), \; i\in \{ 1,2\}$  be a pair of  points and let $\Glue^i$ be gluing maps constructed as above. We also assume that these gluing neighborhoods intersect non-trivially. In particular, let $N_i \sub \Mbar^{reg}_T(\hat{Y})|_{\nu_i}$ be neighborhoods of $(\nu_i,u_i,w_i)$ in the fiber and suppose $B_i$ are neighborhoods of $\nu_i$ in the $\cBR$ such that $B_1 \cap B_2 \neq \emptyset$ and $\Glue^1(B_1\cap B_2,  N_1) \cap \Glue^2(B_1\cap B_2,  N_2) \neq \emptyset.$ We can further assume that $N_i's$ are open neighborhoods of $0$ in the kernel of the respective linearization, $\ker D_0^i$. Thus, we can identify $N_1$ and $N_2$ with open neighborhoods of the origin in finite-dimensional Euclidean space. After potentially replacing $N_1$ with a subset, we have a map  $$G:=(\Glue^2)\inv \circ \Glue^1 \cl ( B_1 \cap B_2)\times N_1 \to ( B_1 \cap B_2)\times N_2 .$$  Now it is enough to check that the map $G(g,\;\underline{\;\;} \;) \cl N_1 \to N_2$ depends continuously on $g$ in the $C^1_{loc}$ topology. By viewing the map $G$ using the definition of gluing map as defined in Equation~\eqref{eq:gluing_map}, we see that it is enough to check continuity of derivative of 
$$ \Glue_{\ker D^2_0}\inv\circ\proj_{\ker D^2_g} \inv\circ\exp_{u_2,\textbf{g}}\inv \circ\exp_{u_1,\textbf{g}} \circ \proj_{\ker D^1_g}\inv \circ \Glue_{\ker D^1_0}\cl N_1 \to N_2$$ 
where $D^i_\textbf{g}$ is defined similarly to $D_\textbf{g}$ above. Now the result follows from the fact that the derivative of $\exp_{u_2,\textbf{g}}\inv \circ \exp_{u_1,\textbf{g}} $ depends continuously on $\textbf{g}$ and the construction of the right inverse $Q_\textbf{g}$ depends continuously on $\textbf{g}$, see \cite[\S 5.2.7-5.2.8]{Par19}.
\end{proof}

\bibliographystyle{amsalpha}
\bibliography{bib}

\Addresses

\end{document}